\documentclass[a4paper,12pt,oneside]{book}

\usepackage[english]{babel}
\usepackage{amssymb}
\usepackage{amsfonts}
\usepackage{amsmath}
\usepackage{amsthm}
\usepackage{geometry}
\usepackage{mathdots}
\usepackage{url}
\usepackage[bookmarks,
linktocpage,
pdftitle=Sorting\ permutations\ with\ pattern-avoiding\ machines,
pdfauthor= Giulio\ Cerbai,
pdfdisplaydoctitle,
]{hyperref}
\usepackage{caption}
\usepackage{booktabs}

\usepackage[all,cmtip]{xy}
\usepackage{cite}
\usepackage{tikz}
\usetikzlibrary{patterns, matrix}
\tikzstyle{NE-lines}=[pattern=north east lines, pattern color=black!45]
\usepackage[boxruled]{algorithm2e}

\theoremstyle{definition}
\newtheorem{theorem}{Theorem}[chapter]
\newtheorem{proposition}[theorem]{Proposition}
\newtheorem{lemma}[theorem]{Lemma}
\newtheorem{corollary}[theorem]{Corollary}
\newtheorem{definition}{Definition}[chapter]
\newtheorem{remark}{Remark}[chapter]
\newtheorem{example}{Example}[chapter]
\newtheorem{openproblem}{Open Problem}[chapter]

\newenvironment{DrawPerm}
{\begin{tikzpicture}[scale=0.4, baseline=20pt]}
{\end{tikzpicture}}
\newcommand{\fillPerm}[3]{
\draw [semithick] (0.001,0.001) grid ({#2},{#3});
\foreach \y [count=\xi] in {#1} ;
	\foreach \x in {1,...,\xi};
		\foreach \y [count=\x] in {#1} \filldraw (\x,\y) circle (6pt);
}
\newcommand{\meshBox}[2]{\fill[NE-lines] #1 rectangle #2;}

\newenvironment{DrawPath}
{\begin{tikzpicture}[scale=0.75, baseline=20pt]}
{\end{tikzpicture}}
\newcommand{\fillPath}[3]{
(0,0)
\foreach \dir in {#1}{
	\ifnum\dir=1
		-- ++(1,1)
	\else
		\ifnum\dir=0
			-- ++(1,0)
		\else
			-- ++(1,0)
		\fi
	\fi
} |- (#1);
\draw[help lines] (0,0) grid (#2,#3);
\coordinate (prev) at (0,0);
\filldraw (prev) circle (3pt);
\foreach \dir in {#1}{
	\ifnum\dir=1
		\coordinate (dep) at (1,1);
	\else
 		\ifnum\dir=0
			\coordinate (dep) at (1,0);
		\else
        	\coordinate (dep) at (1,-1);
		\fi
	\fi
\draw[thick] (prev) -- ++(dep) coordinate (prev);
\filldraw (prev) circle (3pt);
    };
}

\newcommand{\oneton}{[n]} 
\newcommand{\catalan}{\mathfrak{c}} 
\newcommand{\CatalanFun}{C} 
\newcommand{\narayana}{\mathfrak{n}} 
\newcommand{\Do}{\texttt{min}} 
\newcommand{\Co}{\texttt{cons}} 
\newcommand{\fishpattern}{\mathfrak{f}}
\renewcommand{\hat}{\widehat}
\renewcommand{\mod}{\omega}
\newcommand{\Perm}{\mathfrak{S}} 
\newcommand{\identity}{\mathsf{id}} 
\newcommand{\antiid}{\mathsf{ai}} 
\newcommand{\inverse}{\mathcal{I}}
\newcommand{\reverse}{\mathcal{R}}
\renewcommand{\complement}{\mathcal{C}}
\newcommand{\Des}{\mathrm{Des}} 
\newcommand{\Asc}{\mathrm{Asc}} 
\newcommand{\des}{\mathrm{des}} 
\newcommand{\asc}{\mathrm{asc}} 
\renewcommand{\max}{\mathrm{max}} 
\newcommand{\ltrmin}{\mathrm{LTRmin}} 
\newcommand{\ltrminsize}{\mathrm{ltrmin}} 
\newcommand{\ltrmax}{\mathrm{LTRmax}} 
\newcommand{\ltrmaxsize}{\mathrm{ltrmax}} 
\newcommand{\core}{\mathrm{core}} 
\newcommand{\rmost}{\mathrm{rm}} 
\newcommand{\lmost}{\mathrm{lm}} 
\renewcommand{\top}{\mathrm{top}} 
\newcommand{\std}{\mathrm{std}} 
\newcommand{\depth}{\mathrm{dep}} 
\newcommand{\nat}{\mathbb{N}} 
\newcommand{\Cay}{\mathfrak{C}\mathfrak{a}\mathfrak{y}} 
\newcommand{\Modasc}{\mathfrak{M}\mathfrak{A}} 
\newcommand{\Ascseq}{\mathfrak{A}} 
\newcommand{\RGF}{\mathfrak{R}\mathfrak{G}\mathfrak{F}} 
\newcommand{\rgf}{{\sc rgf}} 
\newcommand{\rgfs}{{\sc rgf}s} 
\newcommand{\Fish}{\mathfrak{F}} 
\newcommand{\Sort}{\mathrm{Sort}} 
\newcommand{\out}[1]{\mathcal{S}^{#1}} 
\newcommand{\fsigma}[1]{f^{#1}} 
\newcommand{\Fsigma}[1]{F^{#1}} 
\newcommand{\DSort}{\mathrm{Sort}^{\downarrow}}
\newcommand{\mapsigma}[1]{\mathcal{S}^{#1}}
\newcommand{\pathsigma}[1]{\mathcal{P}_{#1}}
\newcommand{\sorted}{\mathrm{Sorted}} 
\newcommand{\fert}[1]{\mathrm{fert}^{#1}} 
\newcommand{\U}{\mathtt{U}} 
\newcommand{\D}{\mathtt{D}} 
\renewcommand{\H}{\mathtt{H}} 
\newcommand{\Dyck}{\mathcal{D}}
\newcommand{\Motzkin}{\mathcal{M}}

\begin{document}
\pagestyle{myheadings}
\thispagestyle{empty}
\begin{center}
\includegraphics[width = 0.34\linewidth]{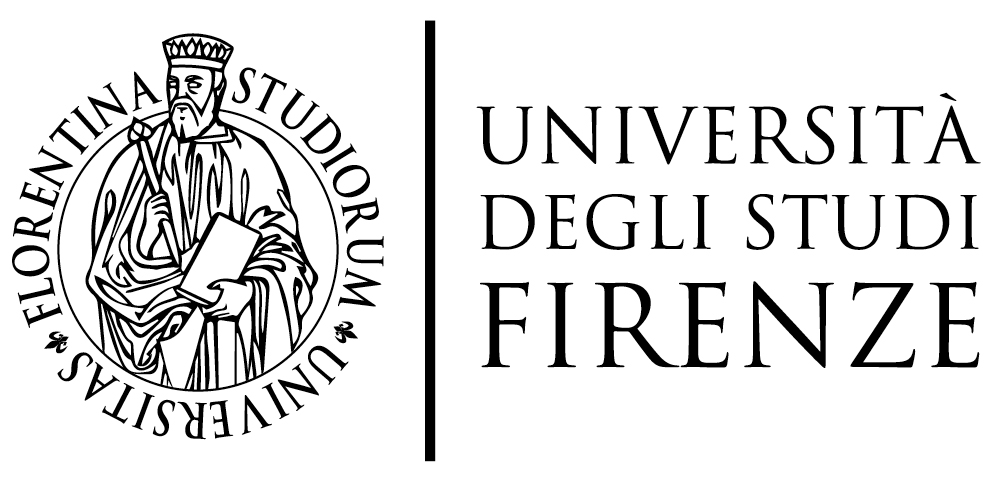}
\includegraphics[width = 0.34\linewidth]{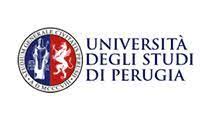}
\includegraphics[width = 0.29\linewidth]{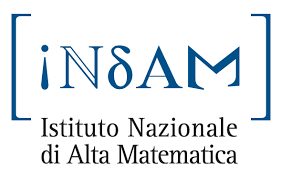}

\bigskip

Universit\`a di Firenze -- Universit\`a di Perugia -- INdAM  -- CIAFM

\bigskip
\bigskip

\large
\textbf{DOTTORATO DI RICERCA\\
IN MATEMATICA, INFORMATICA, STATISTICA}

\bigskip
\medskip

CURRICULUM IN INFORMATICA\\
CICLO XXXIII

\bigskip
\medskip

\textbf{Sede amministrativa: Universit\`a degli Studi di Firenze}\\

\bigskip

Coordinatore Prof.~Paolo Salani

\bigskip
\bigskip

\huge
\textbf{Sorting permutations with}

\medskip
\textbf{pattern-avoiding machines}

\bigskip
\bigskip

\large
Settore Scientifico Disciplinare INF/01
\end{center}

\bigskip
\medskip

\normalsize
\begin{minipage}{0.25\linewidth}
\textbf{Dottorando}

Dott. Giulio Cerbai

\bigskip
\medskip

\hrule
\end{minipage}
\hfill
\begin{minipage}{0.25\linewidth}
\textbf{Tutore}

Prof. Luca Ferrari

\bigskip
\medskip

\hrule

\bigskip
\medskip

\textbf{Coordinatore}

Prof. Paolo Salani

\bigskip
\medskip

\hrule
\end{minipage}

\bigskip
\medskip

\begin{center}
Anni 2017/2020
\end{center}

\newpage
\thispagestyle{empty}
\null\vspace{\stretch{3}}
\begin{flushright}
\textit{A mio babbo}
\end{flushright}
\vspace{\stretch{1}}\null

\restoregeometry
\pagenumbering{roman}

\chapter*{Acknowledgements}

First, I thank my advisor Luca Ferrari, for being the kindest and most helpful guide I could have asked for. His office door was always open, and the hours we spent discussing combinatorics (and everything else) are countless. Regardless how busy he was, he managed to carve out some time to answer my questions and hear my thoughts. I owe him all my academic results: it is not much, but I hope more will come in the future. 

I thank Anders Claesson for raising the question upon which this entire work of thesis is built, and for his precious teachings and insight. I thank him for having invited me in Iceland for a month. He shared his office and treated me as an equal from the first moment, although the difference in our understanding and knowledge was immense.

I thank Luca Ferrari, Antonio Bernini, Elena Barcucci and Renzo Pinzani for welcoming me in the combinatorics family in Firenze.

I thank all the people I've worked with during my PhD, as well as those that were kind enough to share their knowledge with me. Amongst them, Einar Steingr\'imsson, not only for his inspiring lessons and stimulating discussions, but also for his introduction to the study of patterns on ascent sequences. Christian Bean, for providing a lot of useful data. Jean-Luc Baril, Carine Khalil and Vincent Vajnovszki, for the paper we wrote together.

\chapter*{Introduction}
\addcontentsline{toc}{chapter}{Introduction}

To characterize and enumerate permutations that can be sorted by two consecutive stacks, connected in series, is amongst the most challenging open problems in combinatorics. It is known that the set of sortable permutations is a class with an infinite basis, but the basis remains unknown. The enumeration of sortable permutations is still missing as well. Several variants and weaker formulations have been discussed in the literature, some of which are particularly meaningful. For instance, West~$2$-stack sortable permutations are those that can be sorted by making two passes through an \textit{increasing} stack. To use an increasing stack means to always perform a push operation, in a greedy fashion, unless adding the next element would make the content of the stack not increasing, reading from top to bottom. As it is well known, an increasing stack is optimal for the classical problem of sorting with one stack, thus West's device is arguably the most straightforward generalization to two stacks of the orginal instance. This monotonicity requirement can be equivalently expressed by saying that the stack is~$21$-avoiding, again referring to the stack not being allowed to contain occurrences of the pattern~$21$. The present work of thesis is nothing more than an attempt to answer a question raised\footnote{As a follow-up question of a related talk given by Ferrari and the current author at Permutation Patterns 2018.} by Claesson:

\begin{center}
\textit{What happens if we regard the stack as~$\sigma$-avoiding during the first pass, for some pattern~$\sigma$, and~$21$-avoiding during the second pass?}
\end{center}

The resulting device is called the~$\sigma$\textit{-machine}. A~$\sigma$-machine is the natural generalization of West's device obtained by replacing~$21$ with any pattern~$\sigma$, during the first pass, and using the optimal algortithm, during the second pass. This approach can be pushed even further by replacing~$\sigma$ with a set of patterns~$\Sigma$, obtaining a family of sorting devices which we call \textit{pattern-avoiding machines}.

We devote this entire manuscript to the study of pattern-avoiding machines, aiming to gain a better understanding of sorting with two consecutive stacks. For specific choices of~$\sigma$, we characterize and enumerate the set of permutations that are sortable by the~$\sigma$-machine, which we call~$\sigma$\textit{-sortable}. The combinatorics underlying~$\sigma$-machines and~$\sigma$-sortable permutations is extremely rich. It often displays geometric structure and reveals links with a great deal of discrete objects. Certain patterns are particularly relevant. For instance, by setting~$\sigma=21$ we get West's device, while the~$12$-machine is similar to a device studied by Smith (but uses a different sorting algorithm). In order to sort the largest amount of permutations, a sensible try is to set~$\sigma=231$: indeed a~$231$-stack constantly aims to prevent the output of occurrences of~$231$, which is the well known requirement that the input of a classical (increasing) stack has to satisfy in order to be sortable.

The thesis has the following structure:

In Chapter~\ref{chapter_single_pattern} we provide results that cover families of patterns by analyzing how the choice of~$\sigma$ affects the structure of~$\sigma$-sortable permutations. The most relevant (and maybe surprising) is the proof that sets of $\sigma$-sortable permutations that are not permutation classes are enumerated by the Catalan numbers, contained in Chapter~\ref{chapter_single_pattern}. If~$\sigma$-sortable permutations form a class, it is the set of permutations avoiding~$132$ and the reverse of~$\sigma$. On the other hand, sets that are not classes are really hard to characterize and enumerate. Amongst them, the only patterns we are able to solve are~$123$ and~$132$. A description of~$231$-sortable permutations remains unknown, as well as their enumeration. We then prove that~$\sigma$-sortable permutations avoid a bivincular pattern~$\xi$ of length three, unless~$\sigma$ is the skew-sum of~$12$ minus a~$231$-avoiding permutation.

The pattern~$123$ is solved in Chapter~\ref{chapter_pattern123}. We provide a step by step construction of~$123$-sortable permutations that leads to a bijection with a class of pattern-avoiding Schr\"oder paths, whose enumeration is known.

The pattern~$132$ is solved in Chapter~\ref{chapter_pattern132}. We first characterize~$132$-sortable permutations as those avoiding~$2314$ and a mesh pattern of length three. The obtained description is then exploited to determine their geometric structure. More precisely, we show that~$132$-sortable permutations satisfy specific geometric constraints in the grid decomposition induced by left-to-right minima. Finally, we define a bijection with a class of pattern-avoiding set partitions, encoded as restricted growth functions.

In Chapter~\ref{chapter_sigma_tau_machine} we discuss a variant of pattern-avoiding machines where the first stack is~$(\sigma,\tau)$-avoiding, for a pair of patterns~$\sigma$ and~$\tau$. We solve several pairs of patterns of length three. We then determine an infinite family of pairs where sortable permutations are enumerated by the Catalan numbers.

In Chapter~\ref{chapter_sorted_perm_fertilities} we analyze some dynamical aspects of the~$\sigma$-stack operator. We introduce the notions of~$\sigma$\textit{-sorted permutations}, $\sigma$\textit{-fertility} and \textit{effective} pattern. We characterize effective patterns, then we describe~$\sigma$-sorted permutations and~$\sigma$-fertilities of the~$123$-machine.

In Chapter~\ref{chapter_sort_words_various_kind} we consider a natural generalization of~$\sigma$-machines where input sequences and forbidden patterns are chosen in different sets of words, namely Cayley permutations, ascent sequences and modified ascent sequences. In each case, we characterize sets of~$\sigma$-sortable sequences that are classes. By encoding the action of~$\sigma$-stacks as labeled Dyck paths, we then determine for which patterns~$\sigma$ the~$\sigma$-stack operator is bijective on the set of Cayley permutations.

Many aspects concerning pattern-avoiding machines are yet to be investigated thoroughly. Some of them, as well as some open problems and questions that motivated our research, will be mentioned at some point in the manuscript.

\tableofcontents

\listoffigures\addcontentsline{toc}{chapter}{List of figures}

\listoftables\addcontentsline{toc}{chapter}{List of tables}

\chapter*{Notations}
\addcontentsline{toc}{chapter}{Notations}

{\centering
\begin{tabular}{lll}
\toprule%
\textbf{Symbol} & \textbf{Meaning} & \textbf{Note}\\%
\midrule%
\textbf{Words} & & \\%
$\nat$ & Natural numbers & $\nat=1,2,3,\dots$\\%
$\oneton$ & Interval~$\lbrace 1,2,\dots,n\rbrace$ & \\%
$\nat^*$ & Words on~$\nat$ & \\%
$\Perm$ & Permutations & \\%
$\Cay$ & Cayley permutations & Sections~\ref{section_sequences_integers}, \ref{section_sorting_Cayley_perms}\\%
$\RGF$ & Restricted growth functions & Section~\ref{section_sequences_integers}\\%
$\fishpattern$ & Fishburn pattern & $\fishpattern=(231,\lbrace 1\rbrace,\lbrace 1\rbrace)$\\%
$\Fish$ & Fishburn permutations & $\Fish=\Perm(\fishpattern)$\\%
$\Ascseq$ & Ascent sequences & Sections~\ref{section_sequences_integers}, \ref{section_asc_seq_stack}\\%
$\Modasc$ & Modified ascent sequences & Sections~\ref{section_sequences_integers}, \ref{section_modasc_seq_stack}\\%
$X_n$ & Words in~$X$ having length~$n$ & For any set~$X$ and~$n\ge 1$\\%
$X(p)$ & Words in~$X$ avoiding~$p$ & For any set~$X$ and pattern~$p$\\%
$X(A)$ & Words in~$X$ avoiding patterns in~$A$ & For any sets~$X$ and~$A$\\%
$\identity$ & Identity permutation & $\identity_n=12\cdots n$\\%
\midrule%
$\sigma$-\textbf{machines} & & \\%
$\mapsigma{\sigma}$ & $\sigma$-stack operator & \\%
$\mapsigma{21}\circ\mapsigma{\sigma}$ & $\sigma$-machine & \\%
$\Sort(\sigma)$ & $\sigma$-sortable permutations & \\%
$\fsigma{\sigma}_n$ & Cardinality of~$\Sort_n(\sigma)$ & $\Fsigma{\sigma}(t)=\sum_{n\ge 0}\fsigma{\sigma}_nt^n$\\%
$\hat{\sigma}$ & $\hat{\sigma}=\sigma_2\sigma_1\sigma_3\cdots\sigma_k$ & \\%
$\sorted(\sigma)$ & $\sigma$-sorted permutations & \\%
$\fert{\sigma}(\pi)$ & $\sigma$-fertility of~$\pi$ & $\fert{\sigma}(\pi)=|\left(\mapsigma{\sigma}\right)^{-1}(\pi)|$\\%
\midrule%
\textbf{Miscellanea} & & \\%
$\inverse$, $\reverse$, $\complement$ & Inverse, reverse, complement & Trivial bijections on~$\Perm$\\%
$\Des(\cdot)$ & Set of descents & $\des(\cdot)=|Des(\cdot)|$\\%
$\Asc(\cdot)$ & Set of ascents & $\asc(\cdot)=|Asc(\cdot)|$\\%
$\ltrmin(\cdot)$ & Set of ltr-minima & $\ltrminsize(\cdot)=|\ltrmin(\cdot)|$\\%
$\ltrmax(\cdot)$ & Set of ltr-maxima & $\ltrmaxsize(\cdot)=|\ltrmax(\cdot)|$\\%
\bottomrule%
\end{tabular}
\par}

\chapter{Preliminaries}
\pagenumbering{arabic}

\section{Patterns on words and permutations}\label{section_pattern_avoidance}

The notion of pattern avoidance plays a central role in our work. A detailed introduction on patterns on words can be found in the books~\cite{Bo2} and~\cite{Ki}, while the paper~\cite{Be} contains a brief presentation on permutation patterns. In this section we recall some basic definitions and notations from the literature.

Let~$\nat^*$ be the set of words over the alphabet~$\nat=\left\lbrace 1,2,\dots\right\rbrace$ of positive integers. Let~$x=x_1\cdots x_n$ and~$y=y_1\cdots y_k$ be words in~$\nat^*$, with~$k\le n$. We say that~$y$ is a \textit{pattern} of~$x$ if there exist indices~$i_1<i_2<\cdots<i_k$ such that~$x_{i_1}x_{i_2}\cdots x_{i_k}$ is \textit{order isomorphic} to~$y$, that is:
\begin{itemize}
\item $x_{i_s}<x_{i_t}$ if and only if~$y_s<y_t$; and
\item $x_{i_s}=x_{i_t}$ if and only if~$y_s=y_t$,
\end{itemize}
for each pair of indices~$s,t$. The subsequence~$x_{i_1}\cdots x_{i_k}$ is an \textit{occurrence} of~$y$ in~$x$ and we write~$x_{i_1}\cdots x_{i_k}\simeq y$. If~$y$ is a pattern of~$x$, we say that~$x$ \textit{contains}~$y$ and we write~$x\ge y$. Otherwise, we say that~$x$ \textit{avoids}~$y$ (or~$x$ is~$y$\textit{-avoiding}) and write~$x\not\ge y$.

Let~$X$ be a set of words. Given a pattern~$y$, let~$X(y)$ be the set of words in~$X$ that avoid~$y$. For a set of patterns~$Y$, denote by~$X(Y)$ the set of words in~$X$ that avoid every pattern in~$Y$. If~$Y=\lbrace y_1,\dots,y_t\rbrace$, we write~$X(y_1,\dots,y_t)$ instead of~$X\left(\lbrace y_1,\dots,y_t\rbrace\right)$. We use the notation~$X_n$ to denote the set of words of length~$n$ in~$X$, where the length of a word is the number of letters it contains. The sets~$X_n(y)$ and~$X_n(Y)$ are defined accordingly.

Let~$n\ge 1$ and let~$\oneton=\lbrace 1,\dots,n\rbrace$. A \textit{permutation} of length~$n$ is a rearrangement~$x=x_1\cdots x_n$ of the integers~$\oneton$. Sometimes we regard the empty word as the only permutation of length zero. Denote by~$\Perm$ the set of all permutations. A permutation~$x=x_1\cdots x_n$ is often represented by plotting the points~$\lbrace(i,x_i)\rbrace_i$ in the Euclidean plane, as in Figure~\ref{figure_plot_perm}. The \textit{identity} (or \textit{increasing}) permutation of length~$n$ is~$\identity_n=12\cdots n$. The \textit{anti-identity} (or \textit{decreasing}) permutation of length~$n$ is~$\antiid_n=n\cdots 21$.

Pattern containment is a partial order on the set~$\Perm$ and the resulting poset is called the \textit{permutation pattern poset}. A set of permutations~$C$ which is closed downwards under this partial order is said to be a \textit{permutation class} (or simply \textit{class}). To be closed downwards means that for any pair of permutations~$x,y$, if~$x\in C$ and~$y\le x$, then~$y\in C$ too. A permutation class~$C$ is characterized completely by the minimal permutations in the complementary set~$\Perm\setminus C$. The \textit{basis} of~$C$ is the set of minimal avoided patterns. Note that the basis of a class is an antichain due to minimality. Conversely, if~$B$ is an antichain, then~$\Perm(B)$ is a class with basis~$B$. A class~$C$ is \textit{finitely based} if its basis is finite. If the basis is a singleton, then~$C$ is said to be \textit{principal}.

\begin{example}\label{example_pat_involv}
Let~$y=231$. A permutation~$x=x_1\cdots x_n$ contains~$y$ if there are three elements~$x_i,x_j,x_k$, with~$i<j<k$, such that~$x_k<x_i<x_j$. It is well known that there are~$\catalan_n$ $231$-avoiding permutations of length~$n$, where~$\catalan_n=\frac{1}{n+1}\binom{2n}{n}$ is the~$n$-th Catalan number (sequence~A000108 in~\cite{Sl}).
\end{example}

\subsection{Generalized pattern avoidance}

Pattern containment has been generalized in a variety of ways. Some notions of non-classical pattern are recalled below.

A \textit{barred} permutation~\cite{We2} is a permutation where some entries are barred. Let~$y$ be a barred permutation and let~$y'$ be the classical permutation underlying~$y$, that is the permutation obtained from~$y$ by ignoring the bars. Let~$w$ be the permutation order isomorphic to the non-barred entries of~$y$. For a permutation~$x$, to avoid the barred pattern~$y$ means that every occurrence of~$w$ in~$x$ is part of a classical occurrence of~$y'$.

A \textit{bivincular} pattern~\cite{BMCDK} of length~$k$ is a triple~$(y,S,T)$, where~$y$ is a permutation of length~$k$ and~$S,T$ are subsets of~$\lbrace 0,1,\dots,k\rbrace$. An occurrence of the bivincular pattern~$(y,S,T)$ in a permutation~$x=x_1\cdots x_n$ is then a classical occurrence~$x_{i_1}\cdots x_{i_k}$ of~$y$ such that:
\begin{itemize}
\item $i_{s+1}=i_s+1$, for each~$s\in S$;
\item $j_{t+1}=j_{t}+1$, for each~$t\in T$,
\end{itemize}
where~$\lbrace x_{i_1},\dots,x_{i_k}\rbrace=\lbrace j_1,\dots,j_k\rbrace$, with $j_1<\cdots<j_k$; by conventions, $i_0=j_0=0$ and $i_{k+1}=j_{k+1}=n+1$. The set~$S$ identifies contraints of adjacency on the positions of the elements of~$y$, while the set~$T$, symmetrically, identifies constraints on their values. An example of bivincular pattern is depicted in Figure~\ref{figure_bivincular_pattern_132}.

Mesh patterns generalize both classical and bivincular patterns. A \textit{mesh pattern}~\cite{BC} of length~$k$ is a pair~$(y,A)$, where~$y$ is a permutation of length~$k$ and~$A\subseteq\lbrace 0,1,\dots,k\rbrace\times\lbrace 0,1,\dots,k\rbrace$ is a set of pairs of integers. The elements of~$A$ identify the lower left corners of shaded squares in the plot of~$y$. An occurrence of the mesh pattern~$(y,A)$ in the permutation~$x$ is then an occurrence of the classical pattern~$y$ in~$x$ such that no elements of~$x$ are placed into a shaded square of~$A$. An example of mesh pattern is depicted in Figure~\ref{figure_mesh_pattern_132}.

\begin{figure}
\centering
\begin{DrawPerm}
\fillPerm{2,5,3,4,1}{5.999}{5.999}
\end{DrawPerm}
\hspace{25pt}
\begin{DrawPerm}
\meshBox{(0,2)}{(4,3)}
\meshBox{(2,0)}{(3,4)}
\fillPerm{1,3,2}{3.99}{3.99}
\end{DrawPerm}
\hspace{25pt}
\begin{DrawPerm}
\meshBox{(0,2)}{(1,3)}
\meshBox{(2,0)}{(3,1)}
\meshBox{(2,1)}{(3,2)}
\fillPerm{1,3,2}{3.99}{3.99}
\end{DrawPerm}
\caption[The plot of a permutation and two examples of generalized patterns.]{The plot of the permutation~$x=25341$, on the left. The bivincular pattern~$\sigma=(132,\lbrace 2\rbrace,\lbrace 2\rbrace)$, in the center. And the mesh pattern~$\mu=(132,\left\lbrace(0,2),(2,0),(2,1)\right\rbrace)$, on the right. Note that, for example, $253$ is an occurrence of~$\mu$ in~$x$, but it is not an occurrence of~$\sigma$, since the element~$4$ breaks the adjacency constraint between~$5$ and~$3$. On the other hand, $254$ is not an occurrence of~$\sigma$ in~$\pi$ due to the element~$3$ falling in the shaded square~$(2,1)$.}\label{figure_plot_perm}\label{figure_bivincular_pattern_132}\label{figure_mesh_pattern_132}
\end{figure}

In analogy with classical pattern avoidance, we use the same notation~$\Perm(\sigma)$ (respectively $\Perm_n(\sigma)$) to denote the set of permutations (respectively permutations of length~$n$) that avoid $\sigma$, where~$\sigma$ is either a barred, bivincular or mesh pattern. Notice that~$\Perm(\sigma)$ is not necessarily a permutation class when~$\sigma$ is a non-classical pattern.

\section{Statistics and decompositions}\label{section_stats_and_decomp}

Let~$x=x_1\cdots x_n$ be a permutation of length~$n$.

The \textit{trivial bijections} on~$\Perm_n$ are inverse, reverse and complement. The \textit{inverse} of~$x$ is its usual group theoretic inverse~$\inverse(x)$. In one-line notation, $\inverse(x)=y_1\cdots y_n$ is defined by~$y_i=j$, if~$x_j=i$, for each~$i=1,\dots,n$. The \textit{reverse} of~$x$ is~$\reverse(x)=x_n\cdots x_1$. The \textit{complement} of~$x$ is~$\complement(x)=(n+1-x_1)\cdots(n+1-x_n)$. The three trivial bijections behave well with respect to pattern involvement, that is~$x$ contains~$y$ if and only if~$\mathcal{O}(x)$ contains~$\mathcal{O}(y)$, for~$\mathcal{O}\in\lbrace\inverse,\reverse,\complement\rbrace$. As shown in~\cite{Sm2}, the trivial bijections generate the full automorphism group of the permutation pattern poset.

An \textit{inversion} is a pair of indices~$(i,j)$ such that~$i<j$ and~$x_i>x_j$. Equivalently, it is an occurrence~$x_ix_j$ of the pattern~$21$.

An \textit{ascent} is an index~$i\in\lbrace 1,\dots,n-1\rbrace$ such that~$x_i<x_{i+1}$. If~$i$ is an ascent, we will sometimes abuse notation and say that~$x_ix_{i+1}$ is an ascent. If~$x_{i+1}=x_i+1$, the ascent is said to be \textit{consecutive}. Denote by~$\Asc(x)$ the set of ascents of~$x$ and let~$\asc(x)=|\Asc(x)|$. \textit{Descents}, \textit{consecutive descents}, the set~$\Des(x)$ and the statistic~$\des(x)$ are defined symmetrically.

An entry~$x_i$ is a \textit{left-to-right minimum} (briefly \textit{ltr-minimum}) of~$x$ if~$x_i<x_j$, for each~$j<i$. The set of ltr-minima of~$x$ is denoted with~$\ltrmin(x)$ and~$\ltrminsize(x)=|\ltrmin(x)|$. The \textit{left-to-right minima decomposition} (briefly \textit{ltr-min decomposition}) of~$x$ is~$x=m_1B_1m_2B_2\cdots m_tB_t$, where~$t=\ltrminsize(x)$, $\ltrmin(x)=\lbrace m_1,\dots,m_t\rbrace$ and the \textit{block}~$B_i$ contains the entries of~$x$ that are placed strictly between~$m_{i}$ and~$m_{i+1}$, for~$i=1,\dots,n-1$. The last block~$B_t$ contains the entries that follow~$m_t$. Note that~$m_t=1$. The notion of \textit{left-to-right maximum} (briefly \textit{ltr-maximum}), the set~$\ltrmax(x)$, its size~$\ltrmaxsize(x)$ and the \textit{ltr-max decomposition}~$x=M_1B_1M_2B_2\cdots M_tB_t$ are defined analogously. In this case, we have~$M_t=n$.

\begin{example}
Let~$x=471823769$. Then~$\Asc(x)=\lbrace 1,3,5,6,8\rbrace$, where the only consecutive ascent is~$5$. The ltr-minima of~$x$ are~$\ltrmin(x)=\lbrace 4,1\rbrace$, thus its ltr-min decomposition is~$x=m_1B_1m_2B_2$, where $m_1=1$, $B_1=7$, $m_2=1$ and $B_2=823769$.
\end{example}

Let~$k\ge 1$. The~\textit{$k$-inflation of~$x$ at~$x_i$} is the permutation of length~$n+k-1$ obtained from~$x$ by replacing~$x_i$ with the consecutive increasing sequence~$x_i(x_{i}+1)\cdots(x_{i}+k-1)$ and suitably rescaling the remaining elements. For instance, the~$3$-inflation of~$451\mathbf{3}2$ at~$3$ is~$671\mathbf{345}2$.

Let~$x=x_1\cdots x_n$ and~$y=y_1\cdots y_k$. The \textit{direct sum} of~$x$ and~$y$ is the permutation~$x\oplus y=xy'$, where~$y'$ is obtained from~$y$ by adding~$n$ to each of its entries. In other words, $x\oplus y$ is the only permutation~$w_1\cdots w_{n+k}$ of length~$n+k$ such that~$w_i<w_j$ for each~$i\le n$ and~$j\ge n+1$, $w_1\cdots w_n$ is an occurrence of~$x$ and~$w_{n+1}\cdots w_{n+k}$ is an occurrence of~$y$. The \textit{skew sum}~$x\ominus y$ is obtained analogously by requiring~$w_i>w_j$ for each~$i\le n$ and~$j\ge n+1$. A permutation is \textit{layered} if it is the direct sum of decreasing permutations. As it is well known~\cite{AN}, a permutation~$w$ is layered if and only~$w\in\Perm(231,312)$ and there are~$2^{n-1}$ layered permutations of length~$n$, for each~$n\ge 1$. Similarly, a permutation~$w$ is \textit{co-layered} if~$w\in\Perm(132,213)$, or, equivalently, if~$w$ is the skew sum of increasing permutations.

\section{Sets of integer sequences}\label{section_sequences_integers}

Let~$x=x_1\cdots x_n$ be a word of length~$n$ on~$\nat$.

The word~$x$ is a \textit{Cayley permutation} if~$\lbrace x_1,\dots,x_n\rbrace=[\max(x)]$. Equivalently, if~$x$ contains at least one copy of each integer from one to its maximum value. Denote by~$\Cay$ the set of Cayley permutations. Cayley permutations are sometimes called \textit{normalized words}~\cite{DK}, but also surjective words\footnote{Cayley permutations encode surjective endofunctions~$[n]\mapsto [k]$.}, Fubini words or packed words. A Cayley permutation~$x=x_1\cdots x_n$ with maximum equal to~$k$ encodes an ordered set partition (\textit{ballot}) of~$[n]$ with~$k$ blocks~$B_1B_2\dots B_k$, where~$i\in B_{x_i}$ for each~$i$. Cayley permutations are enumerated, with respect to their length, by the Fubini numbers (sequence~A000670 in~\cite{Sl}). For example, the only Cayley permutation of length one is~$1$, there are three Cayley permutations of length two, namely~$11,12$ and~$21$, and thirteen Cayley permutations of length three, which are~$111$, $112$, $121$, $122$, $123$, $132$, $211$, $212$, $213$, $221$, $231$, $312$, $321$. Given a word~$x=x_1\cdots x_k$ on~$\nat$, there is exactly one Cayley permutation~$\std(x)$ that is order isomorphic to~$x$. We call~$\std(x)$ the \textit{standardization} of~$x$. The sequence~$\std(x)$ is obtained by replacing each occurrence of the smallest integer of~$x$ with~$1$, each occurrence of the second smallest integer with~$2$ and so on. For instance, we have~$\std(1381365)=1251243$. Notice that, if~$x$ is a word on~$\nat$ and~$y$ is a Cayley permutation, then~$x_{i_1}\cdots x_{i_k}$ is an occurrence of~$y$ in~$x$ if and only if~$\std(x_{i_1}\cdots x_{i_k})=y$. In other words, the set~$\Cay$ is the set of standardized sequences, and thus it is is the set where patterns live naturally in.

The word~$x$ is a \textit{restricted growth function} (briefly {\rgf}) if~$x_1=1$ and~$x_{i+1}\le 1+\max\lbrace x_1,\dots, x_i\rbrace$, for each~$i\le n-1$. Similarly to Cayley permutations, a {\rgf}~$x=x_1\cdots x_n$~naturally encodes the partition of~$\oneton$ with blocks~$B_1B_2\dots B_k$, where~$x_i$ is the index of the block that contains~$i$. Thus {\rgfs} are enumerated by the Bell numbers (sequence~A000110 in~\cite{Sl}). The set of {\rgfs} is denoted by~$\RGF$. Note that~$\RGF$ is a subset of~$\Cay$. Pattern avoidance on {\rgfs} was discussed in~\cite{CDDGGPS,JM,Sa}.

The word~$x$ is an \textit{ascent sequence} if~$x_1=1$ and~$x_{i+1}\le 2+\asc\left(x_1\cdots x_i\right)$, for each~$i\le n-1$. Ascent sequences were introduced\footnote{In the original work ascent sequences are~$0$-based, that is~$x_1=0$ and~$x_{i+1}\le 1+\asc\left(x_1\cdots x_i\right)$.} in~\cite{BMCDK} as an auxiliary class of objects that embodies the structure of~$(2+2)$-free posets, certain chord diagrams and Fishburn permutations. Fishburn permutations are those avoiding the bivincular pattern~$\fishpattern=(231,\lbrace 1\rbrace,\lbrace 1\rbrace)$ (see Figure~\ref{figure_fishpat_diagram_ascseq}). Roughly speaking, ascent sequences encode bijectively the so called active sites of Fishburn permutations. All these objects are enumerated by the Fishburn numbers (sequence~A022493 in~\cite{Sl}). Denote by~$\Ascseq$ the set of ascent sequences. Due to the intrinsically complicated structure of ascent sequences, pattern avoidance on~$\Ascseq$ seems to be rather more complicated than its analogue on permutations~\cite{BP,CeCl,DS,MS}. \textit{Modified ascent sequences}~\cite{BMCDK} are a slightly more manageable version of~$\Ascseq$. Let~$x=x_1\cdots x_n$ be an ascent sequence and let~$\Asc(x)=\lbrace i_1,\dots,i_k\rbrace$. The modified ascent sequence of~$x$ is obtained as follows. For~$j=1,\dots,k$, increase by one each entry that precedes position~$i_j$ and is greater than or equal to~$x_{i_j+1}$. The resulting sequence~$\mod(x)$ is the modified sequence of~$x$. Let~$\Modasc$ be the set of modified ascent sequences. This procedure can be easily inverted, thus the map~$\mod:\Ascseq\to\Modasc$ defined this way is a size-preserving bijection between~$\Ascseq$ and~$\Modasc$. The maps~$\phi$ and~$\mod$, as well as the bijection~$\phi':\Modasc\to\Fish$ obtained by composition, are depicted in Figure~\ref{figure_fishpat_diagram_ascseq}. A recursive construction of~$\Modasc$ can be found in~\cite{CeCl}: There is exactly one modified ascent sequence of length one, namely the single letter word~$1$. If~$n\geq 2$, then every~$x\in\Modasc_n$ is of one of two forms, depending on whether its last letter~$a$ forms an ascent with the penultimate letter:

\begin{itemize}
\item $x=ya$, with~$1\le a\le b$, or
\item $x=y'a$, with~$b<a\le 2+\asc(y)$,
\end{itemize}

where~$y\in\Modasc_{n-1}$, $b$ is the last letter of~$y$, and~$y'$ is obtained from~$y$ by increasing by one each entry that is less than or equal to~$a$. One of the main advantages of working with modified sequences is that~$\Modasc$ is a subset of~$\Cay$, contrary to~$\Ascseq$. For example, the sequence~$x=12124$ is an ascent sequence, but not a Cayley permutation. Its modified sequence is~$\mod(x)=13124$, which is indeed an element of~$\Cay$. Some more tools and notions on (modified) ascent sequences, as well as a characterization of~$\Modasc$ as a subset of~$\Cay$ (in terms of patterns), will be provided in Section~\ref{section_asc_seq_stack}.

In all the above sets, we either include or rule out the empty word (of length zero), according to the situation. Moreover, when possible, we extend the notions introduced in Section~\ref{section_pattern_avoidance} and Section~\ref{section_stats_and_decomp} accordingly.

\begin{figure}
\begin{minipage}{7cm}
\centering
$
\fishpattern=
\begin{DrawPerm}
\meshBox{(1,0)}{(2,4)}
\meshBox{(0,1)}{(4,2)}
\fillPerm{2,3,1}{3.99}{3.99}
\end{DrawPerm}
$
\end{minipage}
\begin{minipage}{7cm}
\centering
$
\xymatrix{
\Ascseq \ar@{->}[rr]^{\phi} \ar@{->}[d]_{\mod}& & \Fish\\
\Modasc\ar@{->}[urr]_{\phi'} & &\\
}
$
\end{minipage}
\caption[The Fishburn pattern and the bijections linking Fishburn permutations, ascent sequences and modified ascent sequences.]{The pattern~$\fishpattern$, on the left. How the bijections~$\phi$, $\mod$ and~$\phi'$ are related, on the right.}\label{figure_fishpat_diagram_ascseq}
\end{figure}

\section{Lattice paths}\label{section_lattice_paths}

Lattice paths are amongst the most studied combinatorial objects due to the huge amount of combinatorial issues they can model. They are well suited to be studied with the elegant symbolic approach, which often leads directly to enumerative results. In this thesis, we always consider lattice paths in the first quadrant of the discrete plane~$\mathbb{Z}\times\mathbb{Z}$, starting at the origin and ending on the~$x$-axis. The \textit{length} of a lattice path is its final abscissa. A lattice path is encoded by the word that records its steps, going from left to right. A \textit{labeled} lattice path is a path where each step has a label. According to what kind of steps are allowed, we obtain several families of lattice paths.

A \textit{Dyck path} is a lattice path that uses two kinds of steps (of length one), namely up steps~$\U=(+1,+1)$ and down steps~$\D=(+1,-1)$. The \textit{height} of a step is its final ordinate. The \textit{height} of a Dyck path is the maximum height of its steps. Observe that in a Dyck path the number of~$\U$ steps matches the number of~$\D$ steps, since the path ends on the~$x$-axis. Moreover, in each prefix the number of~$\U$ steps is at least equal to the number of~$\D$ steps (due to the path never falling below the~$x$-axis). These two properties characterize the set of words on~$\lbrace\U,\D\rbrace$ that encode Dyck paths. For each up step~$\U$, there is a unique \textit{matching} down step~$\D$ defined as the first~$\D$ step after~$\U$ that has height one less than~$\U$. Since each step has length one, the length of a Dyck path is equal to the total number of its steps, which is two times the number of~$\D$ (or equivalently~$\U$) steps. Given a Dyck path~$P$ of semilength~$n$, the \textit{reverse} path of~$P$ is the path~$\reverse(P)$ obtained from~$P$ by taking the symmetric path with respect to the vertical line~$x=n$. Equivalently, the path~$\reverse(P)$ is encoded by the word obtained by taking the reverse of the word that encodes~$P$ and then transforming each~$\U$ in~$\D$ and viceversa. It is well known that Dyck paths, according to their semilength, are enumerated by the Catalan numbers. An example of Dyck path is depicted in Figure~\ref{figure_dyck_path}. The set of Dyck paths of semilength~$n$ is denoted by~$\Dyck_n$, while~$\Dyck$ denotes the set of all Dyck paths.

\begin{remark}\label{remark_dyck_first_return}
It is well known that any non-empty Dyck path~$P$ has a unique decomposition~$P=\U Q_1\D Q_2$, where~$Q_1$ and~$Q_2$ are two (possibly empty) Dyck paths. Since the~$\D$ step that follows~$Q_1$ is the first return on the~$x$-axis, this is called the \textit{first-return decomposition} of~$P$.
\end{remark}

A \textit{Motzkin path} is defined exactly like a Dyck path, except that one additional kind of step is allowed: the horizontal step~$\H=(1,0)$. Again the length of a Motzkin path is equal to the total number of its steps. Motzkin paths, according to their length, are enumerated by the Motzkin numbers (sequence~A001006 in~\cite{Sl}). The sets of Motzkin paths and Motzkin paths of length~$n$ are denoted by~$\Motzkin$ and~$\Motzkin_n$, respectively. If we allow horizontal steps~$\H_2=(2,0)$ of length two, instead of~$\H$, we obtain \textit{Schr\"oder paths}. The length of a Schr\"oder path is then the sum of the number of its up steps and down steps plus twice the number of its double horizontal steps. Schr\"oder paths are enumerated by the large Schr\"oder numbers (sequence~A006318 in~\cite{Sl}).

Pattern containment can be naturally extended to lattice paths by considering the words that encode them. The arising notion of classical pattern is rather dull on words on small alphabets (such as~$\lbrace\U,\D\rbrace$). Thus we always consider \textit{consecutive} patterns on paths, that is where occurrences of a pattern must be realized by consecutive letters. From here on, we omit the word consecutive in this context. A \textit{valley} is an occurrence of the pattern~$\D\U$. A \textit{peak} is an occurrence of~$\U\D$. In a Dyck path, the number of peaks is equal to one plus the number of valleys. A \textit{double rise} is an occurrence of~$\U\U$. An \textit{ascending run} of length~$k$ is a maximal occurrence of~$\U^k$. A \textit{descending run} of length~$k$ is a maximal occurrence of~$\D^k$. Given a pattern~$q$ and a set of paths~$X$, we use the notation~$X(q)$ to denote the set of paths in~$X$ that avoid the (consecutive) pattern~$q$. Pattern avoidance on lattice paths was studied, for instance, in~\cite{BBFGPW} and~\cite{CiF}.

\begin{example}\label{example_dyck_213_bij}
The following is a well known bijection between~$213$-avoiding permutations (of length~$n$) and Dyck paths (of semilength~$n$). Given a Dyck path~$P$ of semilength~$n$, label its down steps from right to left with the integers~$[n]$ in increasing way. Then assign to each up step the label of its matching down step. Finally, read the labels of the up steps from left to right. The resulting sequence is a~$213$-avoiding permutation and this correspondence is bijective. For instance, the Dyck path~$\U\U\D\U\U\D\D\D\U\D$ is mapped to the~$213$-avoiding permutation~$25341$ (see Figure~\ref{figure_dyck_path}). An equivalent version of the above bijection, but using~$132$-avoiding permutations, was given by Krattenthaler in~\cite{Kr}.
\end{example}

\begin{figure}
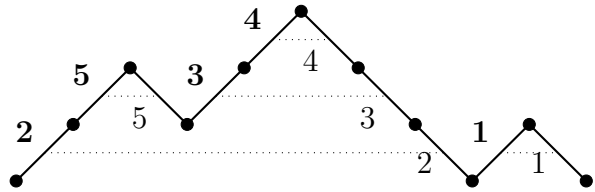

\centering
\begin{DrawPath}
\fillPath{1,1,-1,1,1,-1,-1,-1,1,-1}{0}{0}
\node [above left] at (0.5,0.5) {$\mathbf{2}$};
\node [above left] at (1.5,1.5) {$\mathbf{5}$};
\node [above left] at (3.5,1.5) {$\mathbf{3}$};
\node [above left] at (4.5,2.5) {$\mathbf{4}$};
\node [above left] at (8.5,0.5) {$\mathbf{1}$};
\node [below left] at (2.5,1.5) {$5$};
\node [below left] at (5.5,2.5) {$4$};
\node [below left] at (6.5,1.5) {$3$};
\node [below left] at (7.5,0.7) {$2$};
\node [below left] at (9.5,0.7) {$1$};
\draw [dotted] (0.5,0.5) -- (7.5,0.5);
\draw [dotted] (1.5,1.5) -- (2.5,1.5);
\draw [dotted] (3.5,1.5) -- (6.5,1.5);
\draw [dotted] (4.5,2.5) -- (5.5,2.5);
\draw [dotted] (8.5,0.5) -- (9.5,0.5);
\end{DrawPath}
\caption[A Dyck path and the corresponding~$213$-avoiding permutation.]{The Dyck path~$\U\U\D\U\U\D\D\D\U\D$, of semilength~$5$ and with~$3$ peaks. The corresponding~$213$-avoiding permutation~$25341$ is obtained by reading the bold labels from left to right. Dotted lines connect matching steps.}\label{figure_dyck_path}
\end{figure}

\section{Generating trees and succession rules}

Generating trees and succession rules are very powerful tools in enumerative and bijective combinatorics. Roughly speaking, a generating tree describes a combinatorial construction for a family of discrete objects. Each node of the tree produces a certain set of children, and each child has a unique father, that is it is uniquely obtained from a node situated at the previous level. Typically, starting from the root, each level of the tree contains all the objects of the family that have a given size, which is equal to one more than the size of the objects contained in the previous level. Generating trees are well encoded by succession rules. A succession rule naturally translates the combinatorial construction into something algebraic, which can be exploited to effectively enumerate the objects it represents. Succession rules were introduced by West in~\cite{We3} and~\cite{We4}. Below we recall some basic definitions, referring the interested reader to~\cite{BDLPP} and~\cite{Fe} for a more detailed discussion.

A \textit{generating tree} is a rooted, labeled tree with the property that the label of each node determines the labels of its children. A generating tree is recursively encoded by a \textit{succession rule} consisting in:
\begin{itemize}
\item the label of the root, and
\item a set of rules (\textit{productions}) that explain how to derive, given any node of the tree, the number of its children and their labels.
\end{itemize}

\begin{example}\label{example_dyck_paths_new_peak}
To illustrate the above constructions, we show a very well known succession rule for Dyck paths (see for instance~\cite{BDLPP}). Given a Dyck path~$P$, let~$k$ be the number of points with integer coordinates contained in the last descending run of~$P$. Then we obtain~$k$ Dyck paths of semilength one more by inserting a new peak~$\U\D$ either before a~$\D$ step of the last descending run or at the end of the path. The resulting paths are the children of~$P$. Note that every Dyck path~$P$ of semilength at least two is uniquely constructed this way. Since this construction depends solely on the parameter~$k$, we shall use that integer as label of a given Dyck path. This generation of Dyck paths is then encoded via the following succession rule:
$$
\Omega:\begin{cases}
(2)\\
(k)\longrightarrow (2)(3)\cdots(k)(k+1)
\end{cases}
$$
The root is the path~$\U\D$, which has label~$(2)$ and children~$\U\D\U\D$, with label~$(2)$, and~$\U\U\D\D$, with label~$(3)$. Note that Dyck paths of semilength~$n$ are in one-to-one correspondence with nodes at level~$n$, supposing the root is at level one. Therefore~$\Omega$ is a generating rule for Dyck paths according to their semilength.
\end{example}

\begin{example}\label{example_motzkin_paths_rule}
The following succession rule generates Motzkin paths according to their length.
$$
\Omega:\begin{cases}
(1)\\
(1)\longrightarrow (2)\\
(k)\longrightarrow (1)(2)\cdots(k-1)(k+1),\quad k\ge 2.
\end{cases}
$$
\end{example}

Observe that different generating rules may encode the same family of objects and the same generating rule may encode different families. Notice also that a bijection between two combinatorial families is immediately obtained by showing that both families are generated by the same succession rule, a fact that will be used in the rest of this thesis.

\chapter{Stack sorting and pattern-avoiding machines}

\section{Classical stack sorting}

A stack is a data structure equipped with two operations: \textit{push}, which adds an element to the stack, at the top; and \textit{pop}, which extracts from the stack the most recently pushed element. The problem of sorting permutations using a stack, together with its many variants, has been widely studied in the literature. The reader is referred to~\cite{Bo} for an extensive survey on stack-sorting disciplines. The original version was proposed by Knuth in~\cite{Kn}: given an input permutation~$\pi$, scan its elements from left to right. Every time an element of~$\pi$ is scanned, either push it into the stack or pop the top element of the stack, placing it into the output. The goal is to describe and enumerate sortable permutations. To sort a permutation means to produce a sorted output, that is the identity permutation. As stated in the next lemma, which will be repeatedly used throughout the rest of this thesis, an elegant solution to the original problem can be given in terms of pattern avoidance.

\begin{lemma}\cite{Kn}\label{lemma_classical_stacksort}
Let~$\pi$ be a permutation. Then~$\pi$ is sortable using a classical stack if and only if~$\pi$ avoids the pattern~$231$.
\end{lemma}

More in general, one can sort permutations using a connected network of data structures~\cite{Ta}. The input permutation is scanned and its entries flow through the network, going from one device to another, according to how the network is built (or equivalently to which sorting procedure is used). In this framework, some of the most interesting questions that arise are the following.

\begin{itemize}
\item How to characterize those permutations that can be sorted by a given network?
\item How to enumerate sortable permutations?
\item What is the behavior of specific sorting algorithms?
\end{itemize}

A reasonable way to pick a specific procedure consists in imposing static constraints on a stack, for instance by restricting the set of sequences it is allowed to contain. It is also interesting to search for the most efficient sorting algorithm. An algorithm is \textit{optimal} if it sorts every sortable permutations. In other words, if it has the same sorting power as the device used in its full generality (with a non-deterministic sorting procedure). Finding optimal strategies is often a hard task. Regarding classical stack sorting, it is very well known that there is an optimal algorithm, known as \textit{stacksort}, which is defined by the two following key properties (see Listing~\ref{listing_stacksort} in Appendix~\ref{appendix_listings}):

\begin{enumerate}
\item the elements in the stack are maintained in increasing order, reading from top to bottom; in other words, the stack tries to prevent an occurrence of~$21$ to be output. This can be expressed by saying that a classical stack is \textit{increasing}.
\item the algorithm is \textit{right-greedy}, meaning that it always performs a push operation, unless this violates the previous condition. The expression ``right-greedy" refers to the usual (and most natural) representation of this problem, depicted in Figure~\ref{figure_stacksort_machine}.
\end{enumerate}

Although the classical problem is rather simple, it becomes much harder as soon as one allows several stacks connected in series. Quite recently, Pierrot and Rossin~\cite{PR} proved that the problem of deciding whether a given permutation is sortable by two stacks in series is polynomial. Nevertheless, almost every other related question remains unsolved. For example, it is known that sortable permutations form a class, but its basis is infinite~\cite{Mu}, and still unknown. The enumeration of sortable permutations is still unknown too. In the attempt of gaining a better understanding of two stacks in series, some (simpler) variants of the problem have been considered. In his PhD thesis~\cite{We}, West considered two passes through a classical (i.e. increasing) stack, which is equivalent to perform a \textit{right-greedy} algorithm on two stacks in series. In~\cite{Sm}, Smith considered a decreasing stack followed by an increasing stack. Smith's approach was pushed further in~\cite{CeCiF}, where the authors consider many decreasing stacks, followed by an increasing one. More recently, Claesson, Ferrari and the current author~\cite{CeClF} introduced an even more general device consisting of two stacks in series with a right-greedy procedure, where a restriction on the first stack is given in terms of pattern avoidance. The present work of thesis is dedicated to the analysis of these devices, which we call \textit{pattern-avoiding machines}. Pattern-avoiding machines are defined formally in Section~\ref{section_pat_av_mac}, and discussed extensively in the following chapters.

Other than imposing restrictions on devices and sorting algorithms, one can also allow a larger set of input sequences (see~\cite{AAAHH,ALW,DK}). This line of research, on pattern-avoiding machines, is investigated in Chapter~\ref{chapter_sort_words_various_kind}.

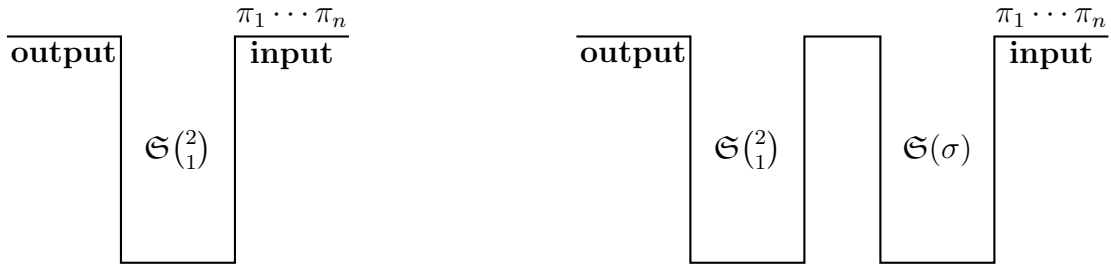
\begin{figure}
\centering
\begin{tikzpicture}[scale=1, baseline=20pt]
\draw[thick] (0,3)--(1.5,3)--(1.5,0)--(3,0)--(3,3)--(4.5,3);
\node at (3.75,2.75){\textbf{input}};
\node at (0.75,2.75){\textbf{output}};
\node at (3.75,3.25){$\pi_1\cdots\pi_n$};
\node at (2.25,1.5){$\Perm\binom{2}{1}$};
\end{tikzpicture}
\hfill
\begin{tikzpicture}[scale=1, baseline=20pt]
\draw[thick] (0,3)--(1.5,3)--(1.5,0)--(3,0)--(3,3)--(4,3)--(4,0)--(5.5,0)--(5.5,3)--(7,3);
\node at (6.25,2.75){\textbf{input}};
\node at (0.75,2.75){\textbf{output}};
\node at (6.25,3.25){$\pi_1\cdots\pi_n$};
\node at (2.25,1.5){$\Perm\binom{2}{1}$};
\node at (4.75,1.5){$\Perm(\sigma)$};
\end{tikzpicture}
\caption[Classical stack sorting and the~$\sigma$-machine.]{The usual representation of sorting with one stack, on the left. The $\sigma$-machine, on the right.}\label{figure_stacksort_machine}
\end{figure}

\section{Pattern-avoiding machines}\label{section_pat_av_mac}

Let~$\Sigma$ be a set of permutations.

\begin{definition}
A~$\Sigma$\textit{-avoiding stack} (or simply~$\Sigma$-stack) is a stack that is not allowed to contain an occurrence of the pattern~$\sigma$, reading from top to bottom, for each~$\sigma\in\Sigma$.
\end{definition}

\begin{definition}
The term~$\Sigma$\textit{-machine} refers to performing a right-greedy algorithm on two stacks in series: a~$\Sigma$-stack, followed by a~$21$-stack.
\end{definition}

Recall that a~$21$-stack is simply a stack as normally used in classical stack-sorting. Thus the~$\Sigma$-machine consists in a pass through a~$\Sigma$-avoiding stack, followed by a pass of the resulting output through a classical stack. From now on, we will assume that~$\Sigma$ does not contain the unit length permutation, since otherwise no element could be pushed in the~$\Sigma$-stack. For singletons and pairs of patterns, we omit the brackets to ease notation. For example, if~$\Sigma=\lbrace\sigma\rbrace$, we write~$\sigma$\textit{-stack} and~$\sigma$\textit{-machine}. An illustration of the~$\sigma$-machine is depicted in Figure~\ref{figure_stacksort_machine}, while the corresponding algorithm is described formally in Listing~\ref{listing_sigma_avoiding} of Appendix~\ref{appendix_listings}. The sequence of operations performed by the~$231$-stack on input~$2413$ is represented in Figure~\ref{figure_sorting_operations}.

\begin{figure}
\centering
\def\arraystretch{6}
\begin{tabular}{|c|c|c|}
\hline
\begin{tikzpicture}[baseline=20pt]
\draw[thick] (0,2)--(1.5,2)--(1.5,0)--(2.5,0)--(2.5,2)--(4,2);
\node at (3.25,1.75){\textbf{input}};
\node at (0.75,1.75){\textbf{output}};
\node at (3.25,2.25){$2413$};
\node at (0.5,0.5){\textbf{Step 1}};
\node at (2.75,0.25){\scriptsize{$\left\lfloor\begin{smallmatrix}2\\3\\1\end{smallmatrix}\right\rfloor$}};
\end{tikzpicture}
&
\begin{tikzpicture}[scale=1, baseline=20pt]
\draw[thick] (0,2)--(1.5,2)--(1.5,0)--(2.5,0)--(2.5,2)--(4,2);
\node at (3.25,1.75){\textbf{input}};
\node at (0.75,1.75){\textbf{output}};
\node at (3.25,2.25){$413$};
\node at (2,0.5){$2$};
\node at (0.5,0.5){\textbf{Step 2}};
\node at (2.75,0.25){\scriptsize{$\left\lfloor\begin{smallmatrix}2\\3\\1\end{smallmatrix}\right\rfloor$}};
\end{tikzpicture}
&
\begin{tikzpicture}[scale=1, baseline=20pt]
\draw[thick] (0,2)--(1.5,2)--(1.5,0)--(2.5,0)--(2.5,2)--(4,2);
\node at (3.25,1.75){\textbf{input}};
\node at (0.75,1.75){\textbf{output}};
\node at (3.25,2.25){$13$};
\node at (2,0.5){$2$};
\node at (2,1){$4$};
\node at (0.5,0.5){\textbf{Step 3}};
\node at (2.75,0.25){\scriptsize{$\left\lfloor\begin{smallmatrix}2\\3\\1\end{smallmatrix}\right\rfloor$}};
\end{tikzpicture}\\
\hline
\begin{tikzpicture}[scale=1, baseline=20pt]
\draw[thick] (0,2)--(1.5,2)--(1.5,0)--(2.5,0)--(2.5,2)--(4,2);
\node at (3.25,1.75){\textbf{input}};
\node at (0.75,1.75){\textbf{output}};
\node at (3.25,2.25){$3$};
\node at (2,0.5){$2$};
\node at (2,1){$4$};
\node at (2,1.5){$1$};
\node at (0.5,0.5){\textbf{Step 4}};
\node at (2.75,0.25){\scriptsize{$\left\lfloor\begin{smallmatrix}2\\3\\1\end{smallmatrix}\right\rfloor$}};
\end{tikzpicture}
&
\begin{tikzpicture}[scale=1, baseline=20pt]
\draw[thick] (0,2)--(1.5,2)--(1.5,0)--(2.5,0)--(2.5,2)--(4,2);
\node at (3.25,1.75){\textbf{input}};
\node at (0.75,1.75){\textbf{output}};
\node at (3.25,2.25){$3$};
\node at (2,0.5){$2$};
\node at (2,1){$4$};
\node at (0.75,2.25){$1$};
\node at (0.5,0.5){\textbf{Step 5}};
\node at (2.75,0.25){\scriptsize{$\left\lfloor\begin{smallmatrix}2\\3\\1\end{smallmatrix}\right\rfloor$}};
\end{tikzpicture}
&
\begin{tikzpicture}[scale=1, baseline=20pt]
\draw[thick] (0,2)--(1.5,2)--(1.5,0)--(2.5,0)--(2.5,2)--(4,2);
\node at (3.25,1.75){\textbf{input}};
\node at (0.75,1.75){\textbf{output}};
\node at (3.25,2.25){$3$};
\node at (2,0.5){$2$};
\node at (0.75,2.25){$14$};
\node at (0.5,0.5){\textbf{Step 6}};
\node at (2.75,0.25){\scriptsize{$\left\lfloor\begin{smallmatrix}2\\3\\1\end{smallmatrix}\right\rfloor$}};
\end{tikzpicture}\\
\hline
\begin{tikzpicture}[scale=1, baseline=20pt]
\draw[thick] (0,2)--(1.5,2)--(1.5,0)--(2.5,0)--(2.5,2)--(4,2);
\node at (3.25,1.75){\textbf{input}};
\node at (0.75,1.75){\textbf{output}};
\node at (2,0.5){$2$};
\node at (2,1){$3$};
\node at (0.75,2.25){$14$};
\node at (0.5,0.5){\textbf{Step 7}};
\node at (2.75,0.25){\scriptsize{$\left\lfloor\begin{smallmatrix}2\\3\\1\end{smallmatrix}\right\rfloor$}};
\end{tikzpicture}
&
\begin{tikzpicture}[scale=1, baseline=20pt]
\draw[thick] (0,2)--(1.5,2)--(1.5,0)--(2.5,0)--(2.5,2)--(4,2);
\node at (3.25,1.75){\textbf{input}};
\node at (0.75,1.75){\textbf{output}};
\node at (2,0.5){$2$};
\node at (0.75,2.25){$143$};
\node at (0.5,0.5){\textbf{Step 8}};
\node at (2.75,0.25){\scriptsize{$\left\lfloor\begin{smallmatrix}2\\3\\1\end{smallmatrix}\right\rfloor$}};
\end{tikzpicture}
&
\begin{tikzpicture}[scale=1, baseline=20pt]
\draw[thick] (0,2)--(1.5,2)--(1.5,0)--(2.5,0)--(2.5,2)--(4,2);
\node at (3.25,1.75){\textbf{input}};
\node at (0.75,1.75){\textbf{output}};
\node at (0.75,2.25){$1432$};
\node at (0.5,0.5){\textbf{Step 9}};
\node at (2.75,0.25){\scriptsize{$\left\lfloor\begin{smallmatrix}2\\3\\1\end{smallmatrix}\right\rfloor$}};
\end{tikzpicture}\\
\hline
\end{tabular}
\caption{The action of the~$231$-stack on input~$2413$.}\label{figure_sorting_operations}
\end{figure}
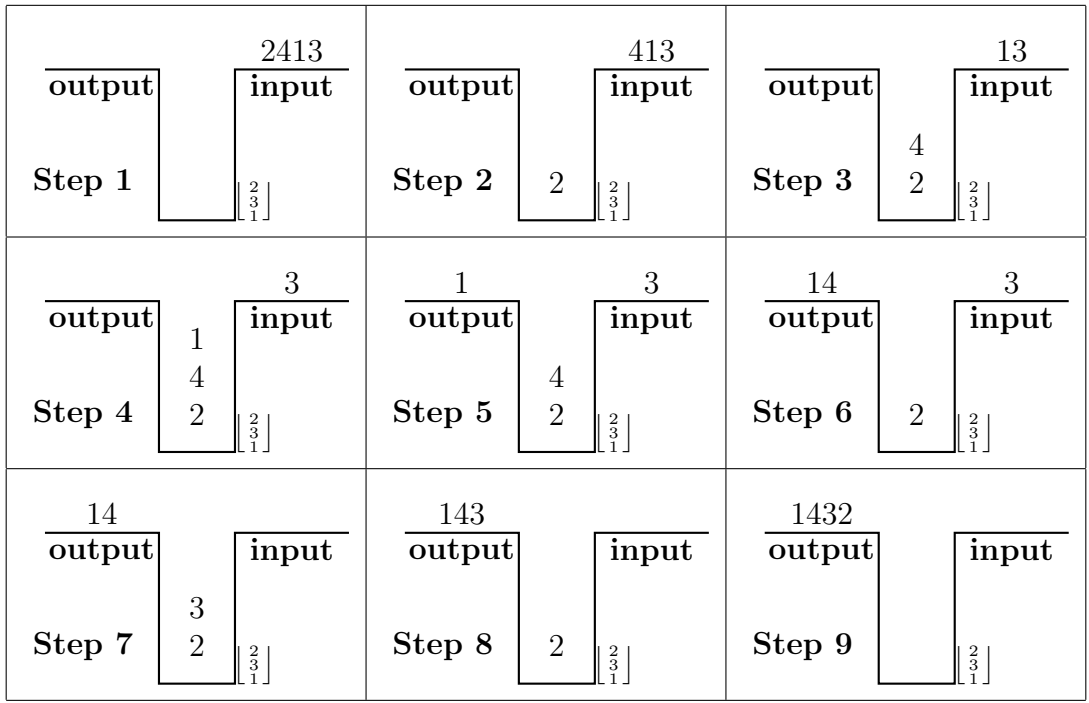

Next we introduce some tools and notations and prove some basic results regarding~$\Sigma$-machines. A permutation~$\pi$ is~\textit{$\Sigma$-sortable} if the~$\Sigma$-machine on input~$\pi$ yields the identity permutation. The set of~$\Sigma$-sortable permutations is denoted by~$\Sort(\Sigma)$. For~$n\ge 1$, denote by~$\fsigma{\Sigma}_n$ the cardinality of~$\Sort_n(\Sigma)$, that is the number of~$\Sigma$-sortable permutations of length~$n$. Let~$\Fsigma{\Sigma}(t)=\displaystyle{\sum_{n\ge 1}\fsigma{\Sigma}_n t^n }$ be the ordinary generating function of~$\Sort(\Sigma)$. Given a permutation~$\pi$, denote by~$\out{\Sigma}(\pi)$ the output of the~$\Sigma$-stack on input~$\pi$. Due to Lemma~\ref{lemma_classical_stacksort}, since~$\out{\Sigma}(\pi)$ is the input of the (final) classical stack in the~$\Sigma$-machine, a permutation~$\pi$ is~$\Sigma$-sortable if and only if~$\out{\Sigma}(\pi)$ avoids the pattern~$231$. This basic fact allows us to determine the~$\Sigma$-sortability of an input permutation~$\pi$ by simply checking whether~$\out{\Sigma}(\pi)$ avoids~$231$ or not, ignoring the final stack. We highlight this remark in the next lemma, which will be used repeatedly from now on.

\begin{lemma}\label{lemma_out_231}
Let~$\pi$ be an input permutation for the~$\Sigma$-machine. Then~$\pi$ is~$\Sigma$-sortable if and only if~$\out{\Sigma}(\pi)$ avoids~$231$.
\end{lemma}

The next lemma shows that prefixes of~$\Sigma$-sortable permutations are~$\Sigma$-sortable.

\begin{lemma}\label{lemma_prefix_sortable}
Let~$\pi=\pi_1\cdots\pi_n$ be a permutation of length~$n$. Suppose that the prefix~$\pi_1\cdots\pi_k$ is an occurrence of the pattern~$\gamma$, for some~$1\le k\le n$ and~$\gamma\in\Perm_k$. If~$\pi$ is~$\Sigma$-sortable, then~$\gamma$ is~$\Sigma$-sortable.
\end{lemma}
\begin{proof}
Observe that the behavior of the~$\Sigma$-stack on the prefix~$\pi_1\cdots\pi_k$ does not depend on the remaining entries~$\pi_{k+1}\cdots\pi_{n}$ of~$\pi$. Moreover, it is the same as the behavior on~$\gamma$, since the operations performed by the~$\Sigma$-stack depend solely on the relative order of the elements processed. Therefore~$\out{\Sigma}(\gamma)$ avoids~$231$, or else~$\out{\Sigma}(\pi)$ would contain~$231$ too, contradicting the hypothesis that~$\pi$ is~$\Sigma$-sortable.
\end{proof}

Lemma~\ref{lemma_prefix_sortable} suggests a recursive construction for~$\Sort_n(\Sigma)$. Indeed every~$\Sigma$-sortable permutation is obtained by appending a new rightmost element to a~$\Sigma$-sortable permutation of length one less, and suitably rescaling the other elements.

As anticipated before, almost the entirety of this thesis is devoted to the analysis of~$\Sigma$-machines. The combinatorics underlying these devices turns out to be unexpectedly rich and deep, offering links with other discrete objects such as lattice paths, set partitions and various families of integer sequences. Some of the questions that motivate our research are reported below (in random order).

\begin{itemize}
\item Given a set of patterns~$\Sigma$, how to characterize the set~$\Sort(\Sigma)$ of~$\Sigma$-sortable permutations? Ideally, we wish to find geometric descriptions, recursive generations and characterizations in terms of pattern avoidance.

\item Given a set of pattern~$\Sigma$, what is the number of~$\Sigma$-sortable permutations of length~$n$? To answer this question, we shall exploit an eventual structural description of~$\Sort(\Sigma)$, find a bijection with other discrete objects or provide a generating tree for~$\Sort(\Sigma)$.

\item Are there properties of~$\Sigma$ that allow us to determine structural information on~$\Sort(\Sigma)$? In this framework, our main result is a characterization of the patterns~$\sigma$ such that~$\Sort(\sigma)$ is a permutation class.

\item Given~$n\ge 1$, what is the number of Wilf-classes of~$\sigma$-sortable permutations of length~$n$? In other words, how many different counting sequences arise by considering the sets~$\Sort(\sigma)$, for each permutation~$\sigma\in\Perm_n$?
\end{itemize}

\section{\texorpdfstring{The~$\sigma$-stack}{The sigma-stack}}

Before moving on to the study of~$\Sigma$-machines, we analyze~$\sigma$-stacks separately.

\begin{lemma}\label{lemma_first_k-1_elements}
Let~$\sigma$ be a permutation of length~$k\ge 2$. Let~$\pi$ be a permutation of length~$n\ge k-2$ and suppose that~$\out{\sigma}(\pi)$ is the increasing permutation. Then~$\pi_1\pi_2\cdots\pi_{k-2}=n(n-1)\cdots (n-k+3)$.
\end{lemma}
\begin{proof}
Since~$\sigma$ has length~$k$, the elements~$\pi_1\cdots\pi_{k-2}$ are pushed directly into the~$\sigma$-stack. Then they remain at the bottom of the~$\sigma$-stack until the end of the sorting process and thus they are the rightmost elements of~$\out{\sigma}(\pi)$.
\end{proof} 

\begin{theorem}\label{theorem_single_stack}
Let~$\pi=\pi_1\cdots\pi_n$ be a permutation of length~$n$ and let~$\sigma\in\Perm_k$, with~$k\ge 2$.
\begin{enumerate}
\item If~$\sigma=\identity_k$, then~$\out{\sigma}(\pi)$ is the increasing permutation if and only if~$\pi=\antiid_n$ and~$n\le k-1$.
\item If~$\sigma=21\oplus\identity_t$, then~$\out{\sigma}(\pi)$ is the increasing permutation if and only if~$\pi=\antiid_t\ominus\alpha$, for some~$\alpha\in\Perm(231)$.
\item In all the remaining cases, $\out{\sigma}(\pi)$ is the increasing permutation if and only if~$\pi=\antiid$.
\end{enumerate}
\end{theorem}
\begin{proof}
\begin{enumerate}
\item Let~$\sigma=\identity_k$. If~$\pi=\antiid_n$, for some~$n\le k-1$, then~$\out{\sigma}(\pi)=\reverse(\pi)=\identity_n$. Conversely, suppose that~$\out{\sigma}(\pi)=\identity$. If~$n\le k-1$, then it must be~$\pi=\antiid_n$ due to Lemma~\ref{lemma_first_k-1_elements}. Otherwise, suppose, for a contradiction, that~$n\ge k$ and write~$\pi=\pi_1\cdots\pi_{k-1}\pi_k\cdots\pi_n$. Initially, the elements~$\pi_1,\dots,\pi_{k-1}$ are pushed into the~$\sigma$-stack. If~$\pi_k\pi_{k-1}\cdots\pi_1\simeq\sigma$, then~$\pi_k<\pi_{k-1}$ and~$\pi_{k-1}$ is extracted, which is a contradiction with the hypothesis that~$\out{\sigma}(\pi)$ is increasing. On the other hand, if~$\pi_k\cdots\pi_1$ is not an occurrence of~$\sigma$, then~$\pi_k$ is pushed into the~$\sigma$-stack and thus~$\out{\sigma}(\pi)$ contains the substring~$\pi_k\cdots\pi_1$. But, since~$\pi_k\cdots\pi_1$ is not an occurrence of~$\sigma=\identity_k$, the output~$\out{\sigma}(\pi)$ is not increasing, which contradicts the hypothesis.

\item Let~$\sigma=21\oplus\identity_t$. Suppose that~$\out{\sigma}(\pi)$ is the increasing permutation. By Lemma~\ref{lemma_first_k-1_elements}, we have~$\pi_1\pi_2\cdots\pi_t=n(n-1)\cdots (n-t+1)$. Thus~$\pi=\antiid_t\ominus\alpha$. We wish to show that~$\alpha$ avoids~$231$. Observe that the elements~$\pi_1\cdots\pi_t$ are pushed into the~$\sigma$-stack at the beginning of the sorting process. Then, since~$\sigma=21\oplus\identity_t$, the behavior of the~$\sigma$-stack with~$n,n-1,\dots,n-t+1$ at the bottom is equivalent to the behavior of an empty~$21$-stack. Indeed the elements~$n,n-1,\dots,n-t+1$ play the role of~$\identity_t$ in any potential occurrence of~$\sigma=21\oplus\identity_t$. In other words, the~$\sigma$-stack with~$n,n-1,\dots,n-t+1$ at the bottom performs a pop operation if and only if an empty~$21$-stack performs a pop operation. Therefore~$\alpha$ is~$21$-sortable, which in turn is equivalent to~$\alpha$ avoiding~$231$ by Lemma~\ref{lemma_classical_stacksort}. Similarly, it is easy to prove that if~$\pi=\antiid_t\ominus\alpha$, with~$\alpha$ a $231$-avoiding permutation, then~$\out{\sigma}(\pi)$ is increasing.

\item Finally, suppose that~$\sigma$ is not increasing and~$\sigma$ is not the direct sum of~$21$ and the identity permutation. Since~$\sigma$ is not increasing, we have~$\out{\sigma}(\antiid)=\reverse(\antiid)=\identity$. Conversely, suppose that~$\pi\neq\antiid$. Write~$\pi=\pi_1\cdots\pi_i\pi_{i+1}\cdots\pi_n$, where~$i$ is the leftmost ascent of~$\pi$. We show that~$\out{\sigma}(\pi)$ is not increasing. Since~$\sigma\neq\identity$ and~$\pi_1>\pi_2>\cdots>\pi_i$, the elements~$\pi_1,\dots,\pi_i$ are pushed into the~$\sigma$-stack. Now, if~$\pi_{i+1}$ enters the~$\sigma$-stack above~$\pi_i$, then~$\pi_{i+1}$ precedes~$\pi_i$ in~$\out{\sigma}(\pi)$, with~$\pi_{i+1}>\pi_i$, thus~$\out{\sigma}(\pi)$ is not increasing. Otherwise, suppose that~$\pi_i$ is extracted from the~$\sigma$-stack before~$\pi_{i+1}$ enters. That is, the~$\sigma$-stack contains~$k-1$ elements, say~$\pi_{j_2},\dots,\pi_{j_k}$, with~$j_2<\cdots<j_k$, such that~$\pi_{i+1}\pi_{j_k}\cdots\pi_{j_2}$ is an occurrence of~$\sigma$. Notice that~$\pi_{j_k}<\cdots<\pi_{j_2}$ due to our assumptions. Without losing generality, choose the minimal indices~$j_2,\dots,j_k$ such that~$\pi_{i+1}\pi_{j_k}\cdots\pi_{j_2}$ is an occurrence of~$\sigma$, so that~$\pi_{i+1}$ enters the~$\sigma$-stack above~$\pi_{j_{k-1}}$ (and thus~$\pi_{i+1}$ precedes~$\pi_{j_{k-1}}$ in~$\out{\sigma}(\pi)$). Now, if~$\pi_{i+1}<\pi_{j_{k-1}}$, then~$\sigma=21\oplus\identity_{k-2}$, which is a contradiction. Otherwise, if~$\pi_{i+1}>\pi_{j_{k-1}}$, then~$\out{\sigma}(\pi)$ is not increasing, as desired.
\end{enumerate}
\end{proof}

\chapter{\texorpdfstring{The~$\sigma$-machine}{The sigma-machine}}\label{chapter_single_pattern}

This chapter is devoted to the analysis of~$\sigma$-machines. Most\footnote{With the exception of Section~\ref{section_bivincular_result}.} of the results presented here are contained in~\cite{CeClF}. Patterns~$\sigma$ of length two are discussed in Section~\ref{section_patterns_length_two}. In Section~\ref{section_classes_vs_nonclasses} we prove the main result of this chapter, which is a characterization of those patterns~$\sigma$ where the set of~$\sigma$-sortable permutations is a class. We prove the surprising fact that there are~$\catalan_n$ patterns~$\sigma$ of length~$n$ such that~$\Sort(\sigma)$ is not a class. If instead~$\Sort(\sigma)$ is a class, we explicitly determine its basis, which is either the singleton~$\lbrace 132\rbrace$ or the pair~$\lbrace 132,\reverse(\sigma)\rbrace$. In Section~\ref{section_bivincular_result} we define a bivincular pattern~$\xi$ and show that $\sigma$-sortable permutations avoid~$\xi$, unless~$\sigma$ is the skew-sum of~$12$ minus a~$231$-avoiding permutation. Permutations avoiding~$\xi$ display a rather regular geometric structure. This suggests that the cases where~$\Sort(\sigma)$ is not contained in~$\Perm(\xi)$ could be the most challenging to be solved. In Section~\ref{section_decreasing_pattern} we investigate the decreasing pattern~$\sigma=\antiid$ and show that~$\Sort(\antiid_k)$ is in bijection with Dyck paths of height at most~$k-1$. Finally, in Section~\ref{section_open_problems} we suggest some open problems and lines of research.

\section{Patterns of length two}\label{section_patterns_length_two}

We start by analyzing the~$12$- and~$21$-machines.

Let~$\sigma=12$. Recall that the~$12$-machine consists in a pass through a~$12$-stack, followed by a pass through a classical stack. Notice that this device is substantially different from the one considered in~\cite{Sm}, which is constituted by the same stacks, but allows a non-deterministic (and thus more powerful) sorting procedure.

\begin{theorem}\label{theorem_12machine}
Let~$\pi$ be a permutation. If~$\pi$ is~$12$-sortable, then~$\out{12}(\pi)=\antiid$. Moreover, we have:
$$
\Sort(12)=\Perm(213).
$$
Therefore~$\fsigma{12}_n=\catalan_n$, the~$n$-th Catalan number.
\end{theorem}
\begin{proof}
Suppose that~$\pi$ is~$12$-sortable. We show that~$\out{12}(\pi)=\antiid$ and~$\pi$ avoids~$213$ by induction on the length of~$\pi$. This is trivial for the unit length permutation. Let~$\pi$ be a permutation of length two or more. Write~$\pi$ as~$\pi=L1R$, where~$L$ is the prefix of~$\pi$ preceding~$1$ and~$R$ is the suffix of~$\pi$ following~$1$. Since~$1x$ is an occurrence of~$12$ for each~$x\in L$, $1$ enters the~$12$-stack only when the~$12$-stack is empty. Similarly, $1$ is extracted from the~$12$-stack only at the end, since~$y1$ is not an occurrence of~$12$ for each~$y\in R$. Therefore~$\out{12}(\pi)=\out{12}(L)\out{12}(R)1$. By the inductive hypothesis, $\out{12}(L)$ and~$\out{12}(R)$ are decreasing. Moreover, it must be~$x>y$ for each~$x\in\out{12}(L)$ and~$y\in\out{12}(R)$, otherwise~$xy1$ would be an occurrence of~$231$ in~$\out{12}(\pi)$, contradicting the hypothesis that~$\pi$ is~$12$-sortable. Therefore~$\out{12}(\pi)=\antiid$, as wanted. Finally, suppose, for a contradiction, that~$\pi$ contains an occurrence~$bac$ of~$213$. If~$bac$ is contained in~$L$, then~$\out{12}(L)$ contains~$231$ by the inductive hypothesis and thus~$\out{12}(\pi)$ contains~$231$, a contradiction with~$\pi$ being~$12$-sortable. The same happens if~$bac$ is contained in~$R$. On the other hand, if~$b\in L$ and~$c\in R$, then~$bc1$ is an occurrence of~$231$ in~$\out{12}(\pi)$, which is impossible.

Conversely, suppose that~$\pi$ is not~$12$-sortable, or, equivalently, that~$\out{12}(\pi)$ contains an occurrence~$bca$ of~$231$. We wish to show that~$\pi$ contains~$213$. Note that necessarily~$b$ comes before~$c$ in~$\pi$. Indeed a non-inversion in the output necessarily comes from a non-inversion in the input, since the~$12$-stack cannot repair inversions\footnote{If~$\pi_i>\pi_j$, with~$i<j$, then~$\pi_i$ is extracted from the~$12$-stack before~$\pi_j$ enters.}. Moreover, $b$ is extracted from the~$12$-stack before~$c$ enters. This must be due to the presence of an element~$x$, located between~$b$ and~$c$ in~$\pi$, which is smaller than~$b$. More precisely, $x$ is the next element of the input when~$b$ is extracted. The three elements~$b$, $x$ and~$c$ are thus an occurrence of~$213$ in~$\pi$, as desired.
\end{proof} 

Next we consider the pattern~$21$. The~$21$-machine is equivalent to the device considered by West in~\cite{We2}, where the following result is stated.

\begin{theorem}\cite{We2}\label{theorem_west_2stack}
We have:
$$
\Sort(21)=\Perm(3241,3\bar{5}241).
$$
\end{theorem}

Due to the presence of the barred pattern~$3\bar{5}241$, the set~$\Sort(21)$ is not a permutation class. For example, the~$21$-sortable permutation~$35241$ contains the pattern~$3241$, which is not~$21$-sortable. On the other hand, $\Sort(12)$ is a class due to Theorem~\ref{theorem_12machine}. By looking at some data for permutations of length three or more, it seems that the number of permutations~$\sigma$ such that~$\sigma$-sortable permutations are not a class is equal to the~$n$-th Catalan number. We will prove this rather striking fact in Section~\ref{section_classes_vs_nonclasses}.

\section{\texorpdfstring{Classes and non-classes of~$\sigma$-sortable permutations}{Classes and non-classes of sigma-sortable permutations}}\label{section_classes_vs_nonclasses}

Let~$\sigma$ be a permutation of length two or more. As one would expect, the~$\sigma$-sortability of a permutation~$\pi$ is strongly affected by how~$\pi$ is related to the pattern~$\sigma$ defining the constraint of the stack. The permutation~$\hat{\sigma}$, defined below, proves to be crucial.

\begin{definition}\label{definition_hat_sigma}
Let~$\sigma=\sigma_1\cdots\sigma_k$, with~$k\ge 2$. Then define~$\hat{\sigma}$ as the permutation:
$$
\hat{\sigma}=\sigma_2\sigma_1\sigma_3\cdots\sigma_k.
$$
In other words, $\hat{\sigma}$ is the permutation obtained from~$\sigma$ by interchanging the first two elements~$\sigma_1$ and~$\sigma_2$.
\end{definition}

\begin{lemma}\label{lemma_reverse_outputs_hat}
Let~$\pi$ be a permutation. If~$\pi$ contains~$\reverse(\sigma)$, then~$\out{\sigma}(\pi)$ contains~$\hat{\sigma}$.
\end{lemma}
\begin{proof}
Let~$s_ks_{k-1}\cdots s_1$ be the (lexicographically) leftmost occurrence of~$\reverse(\sigma)$ in~$\pi$, where~$k$ is the length of~$\sigma$. Consider the action of the~$\sigma$-stack on~$\pi$. Initially, every element of~$\pi$ is pushed into the~$\sigma$-stack, until~$s_1$ is the next element of the input. Now, before pushing~$s_1$ into the~$\sigma$-stack, the element~$s_2$ has to be extracted, since~$s_1s_2\cdots s_k\simeq\sigma$, with~$s_2\cdots s_k$ inside the~$\sigma$-stack. On the other hand, $s_3$ is still in the~$\sigma$-stack when~$s_1$ enters, otherwise~$\pi$ would contain another occurrence of~$\reverse(\sigma)$ strictly to the left of~$s_k s_{k-1}\cdots s_1$, which is a contradiction. Thus~$s_1$ is pushed into the~$\sigma$-stack above~$s_3$ and~$\out{\sigma}(\pi)$ contains~$s_2s_1s_3\cdots s_k$, which is an occurrence of~$\hat{\sigma}$.
\end{proof}

\begin{lemma}\label{lemma_hat_sigma}
Let~$\pi$ be an input permutation for the~$\sigma$-machine.
\begin{enumerate}
\item If~$\pi$ avoids~$\reverse(\sigma)$, then~$\out{\sigma}(\pi)=\reverse(\pi)$. In this case, $\pi$ is~$\sigma$-sortable if and only if~$\pi$ avoids~$132$.
\item If~$\pi$ contains~$\reverse(\sigma)$, then~$\out{\sigma}(\pi)$ contains~$\hat{\sigma}$. In this case, if~$\hat{\sigma}$ contains~$231$, then~$\pi$ is not~$\sigma$-sortable.
\end{enumerate}
\end{lemma}
\begin{proof}
\begin{enumerate}
\item If~$\pi$ avoids~$\reverse(\sigma)$, then the restriction of the~$\sigma$-stack is never triggered. Thus every element of~$\pi$ is pushed directly into the~$\sigma$-stack and~$\out{\sigma}(\pi)=\reverse(\pi)$. In particular, $\pi$ is~$\sigma$-sortable if and only if~$\reverse(\pi)$ avoids~$231$, or, equivalently, $\pi$ avoids~$132$.
\item Suppose that~$\pi$ contains~$\reverse(\sigma)$. By Lemma~\ref{lemma_reverse_outputs_hat}, $\out{\sigma}(\pi)$ contains~$\hat{\sigma}$. Therefore, if~$\hat{\sigma}$ contains~$231$ then~$\pi$ is not~$\sigma$-sortable.
\end{enumerate}
\end{proof}

\begin{corollary}\label{corollary_av_132_sigmarev}
For each permutation~$\sigma$, we have:
$$
\Perm\left(132,\reverse(\sigma)\right)\subseteq\Sort(\sigma).
$$
\end{corollary}
\begin{proof}
It is an immediate consequence of the first item of Lemma~\ref{lemma_hat_sigma}.
\end{proof}

\begin{theorem}\label{theorem_suff_cond_class}
Let~$\sigma$ be a permutation. If~$\hat{\sigma}$ contains~$231$, then:
$$
\Sort(\sigma)=\Perm(132,\reverse(\sigma)).
$$
Therefore~$\Sort(\sigma)$ is a class with basis either~$\lbrace132,\reverse(\sigma)\rbrace$, if~$\reverse(\sigma)$ avoids~$132$, or~$\lbrace132\rbrace$, otherwise.
\end{theorem}
\begin{proof}
Following Corollary~\ref{corollary_av_132_sigmarev}, all we need to prove is that~$\Sort(\sigma)\subseteq\Perm(132,\reverse(\sigma))$. Suppose that~$\pi$ is~$\sigma$-sortable. If~$\pi$ contains~$\reverse(\sigma)$, then~$\out{\sigma}(\pi)$ contains~$\hat{\sigma}$ by Lemma~\ref{lemma_hat_sigma}. But then~$\hat{\sigma}$ contains~$231$ by hypothesis, contradicting the fact that~$\pi$ is~$\sigma$-sortable. Otherwise, suppose that~$\pi$ avoids~$\reverse(\sigma)$, but contains~$132$. Due to the same Lemma~\ref{lemma_hat_sigma}, we have~$\out{\sigma}(\pi)=\reverse(\pi)$, which contains~$231$, a contradiction with~$\pi$ being~$\sigma$-sortable.
\end{proof}

\begin{corollary}\label{corollary_decr_pattern_class}
Let~$k\ge 3$. Then:
$$
\Sort(\antiid_k)=\Perm(132,\identity_k).
$$
In particular, the set of~$321$-sortable permutations is a class with basis~$\lbrace 132,123\rbrace$.
\end{corollary}
\begin{proof}
It follows immediately from Theorem~\ref{theorem_suff_cond_class}, since~$\hat{\antiid_k}=(k-1)k(k-2)\cdots21$ contains an occurrence~$(k-1)k1$ of~$231$.
\end{proof}

Theorem~\ref{theorem_suff_cond_class}, which is a rather straightforward consequence of Lemma\ref{lemma_hat_sigma}, gives a sufficient condition for the set~$\Sort(\sigma)$ to be a class. Next we show that this condition is also necessary.

\begin{theorem}\label{theorem_necess_cond_class}
If~$\hat{\sigma}$ avoids the pattern~$231$, then~$\Sort(\sigma)$ is not a permutation class.
\end{theorem}
\begin{proof}
The case by case analysis of Table~\ref{table_patterns_length_three} and Corollary~\ref{corollary_decr_pattern_class} show that the theorem holds for patterns~$\sigma$ of length three. Now suppose that~$\sigma$ has length at least four. It is not hard to realize that the permutation~$132$ is not~$\sigma$-sortable, since~$\out{\sigma}(132)=231$. Next we show that, under the hypothesis that~$\hat{\sigma}$ avoids~$231$, it is always possible to construct a permutation~$\alpha$ such that~$\alpha$ contains~$132$, but~$\alpha$ is~$\sigma$-sortable. This proves that~$\Sort(\sigma)$ is not closed downwards, as desired. Let~$\sigma=\sigma_1\sigma_2\cdots\sigma_k$. We distinguish two cases, according to whether~$\sigma_1<\sigma_2$ or~$\sigma_1>\sigma_2$.

\begin{enumerate}
\item If~$\sigma_1<\sigma_2$, define~$\alpha=\sigma'_k\sigma'_{k-1}\cdots\sigma'_3z\sigma'_2\sigma'_1$, where:

\begin{itemize}
\item $z=\sigma_1$;
\item $\sigma'_i=
\begin{cases}
\sigma_i, & \mbox{ if }\sigma_i<\sigma_1;\\
\sigma_i+1, & \mbox{ otherwise.}
\end{cases}$
\end{itemize}

Notice that~$z\sigma'_2\sigma'_1$ is an occurrence of~$132$. We show that~$\alpha$ is~$\sigma$-sortable by providing a detailed analysis of the behavior of the~$\sigma$-machine on input~$\alpha$. Initially, the elements of~$\alpha$ are pushed into the~$\sigma$-stack until~$\sigma'_1$ is the next element of the input. In particular, both the additional element~$z$ and~$\sigma'_2$ can be safely pushed: indeed~$\sigma'_2z\cdots\sigma'_{k-1}\sigma'_k$ is not an occurrence of~$\sigma$, since~$\sigma_1<\sigma_2$, whereas~$\sigma'_2>z$. Now, before~$\sigma'_1$ enters the~$\sigma$-stack, the element~$\sigma'_2$ is extracted. At this point, $\sigma'_1$ can enter without violating the restriction, again because~$\sigma_2>\sigma_1$, whereas~$z<\sigma'_1$, and so~$\sigma'_1z\sigma'_3\cdots\sigma'_k$ is not an occurrence of~$\sigma$. The output of the~$\sigma$-stack is then~$\out{\sigma}(\alpha)=\sigma'_2\sigma'_1z\sigma'_3\cdots\sigma'_k$. We wish to show that~$\out{\sigma}(\alpha)$ avoids~$231$. Since~$\hat{\sigma}$ avoids~$231$ by hypothesis, and~$\sigma'_2\sigma'_1\sigma'_3\cdots\sigma'_k$ is an occurrence of~$\hat{\sigma}$, any potential occurrence of~$231$ necessarily involves the additional element~$z$. In particular, it is easy to observe that~$z$ can be neither the smallest nor the biggest element of such a pattern, because~$z<\sigma'_1<\sigma'_2$ and~$z$ is the third element of~$\out{\sigma}(\alpha)$. Finally, if~$z$ were the first element of an occurrence~$z\sigma'_j\sigma'_l$ of~$231$ in~$\out{\sigma}(\alpha)$, then~$\sigma_1\sigma_j\sigma_l$ would be an occurrence of~$231$ in~$\hat{\sigma}$, contradicting the hypothesis.

\item If~$\sigma_1>\sigma_2$, define~$\alpha=\sigma'_k\sigma'_{k-1}\cdots\sigma'_3\sigma'_2\sigma'_1z$, where:

\begin{itemize}
\item $z=\sigma_2+1$;
\item $\sigma'_i=
\begin{cases}
\sigma_i , & \mbox{ if }\sigma_i\le\sigma_2;\\
\sigma_i+1 , & \mbox{ otherwise.}
\end{cases}$
\end{itemize}

Observe that~$\sigma'_2\sigma'_1z$ is an occurrence of~$132$. As for the previous case, we now describe what happens when~$\alpha$ is processed by the~$\sigma$-machine. The first element that cannot be pushed into the~$\sigma$-stack is~$\sigma'_1$, which forces~$\sigma'_2$ to be extracted. Successively both~$\sigma'_1$ and~$z$ can enter the~$\sigma$-stack, since~$z\sigma'_1\sigma'_3\cdots\sigma'_k$ is not an occurrence of~$\sigma$: indeed~$\sigma_1>\sigma_2$, whereas~$z<\sigma'_1$. Therefore the output of the~$\sigma$-stack is~$\out{\sigma}(\alpha)=\sigma'_2z\sigma'_1\sigma'_3\cdots\sigma'_k$. Again any potential occurrence of~$231$ in~$\out{\sigma}(\alpha)$ must involve the additional element~$z$. However~$z$ cannot be the smallest element of a pattern~$231$, because it is the second element of~$\out{\sigma}(\alpha)$. Moreover, if~$z$ were the first element of a~$231$, then~$\sigma_2$ would be the first element of an occurrence of~$231$ in~$\hat{\sigma}$, which is forbidden. Finally, if~$z$ were the largest element of a~$231$, then~$\sigma'_2$ would be the first element of such an occurrence, so also~$\sigma'_1$, which is greater than both~$\sigma'_2$ and~$z$, would be the largest element of an occurrence of~$231$ which does not involve~$z$, giving again a contradiction.
\end{enumerate}

In each of the two cases considered, we proved that~$\out{\sigma}(\alpha)$ avoids~$231$, thus~$\alpha$ is a~$\sigma$-sortable permutation that contains the non~$\sigma$-sortable pattern~$132$, as desired.
\end{proof}

\begin{table}
\centering
\def\arraystretch{1.1}
\begin{tabular}{lcc}
\toprule
$\sigma$ & $\sigma$\textbf{-sortable permutation} & \textbf{Non-}$\sigma$\textbf{-sortable pattern}\\
\midrule
123 & 4132 & 132\\
132 & 2413 & 132\\
213 & 4132 & 132\\
231 & 361425 & 1324\\
312 & 3142 & 132\\
321 & \text{class} &\\
\bottomrule
\end{tabular}
\caption[Classes and non-classes of~$\sigma$-sortable permutations, for patterns~$\sigma$ of length three.]{Classes and non-classes of~$\sigma$-sortable permutations, for patterns~$\sigma$ of length three.}\label{table_patterns_length_three}
\end{table}

\begin{corollary}\label{corollary_class_nonclass_char}
For every permutation~$\sigma$ of length three or more, the set~$\Sort(\sigma)$ is a permutation class if and only if~$\hat{\sigma}$ contains the pattern~$231$.
\end{corollary}

\begin{corollary}\label{corollary_class_nonclass_enum}
The permutations~$\sigma$ for which~$\Sort(\sigma)$ is not a permutation class are enumerated by the Catalan numbers.
\end{corollary}
\begin{proof}
Such permutations are in bijection with~$\Perm(231)$, which is known to be enumerated by the Catalan numbers.
\end{proof}

What we have proved so far assures that~$\Sort(\sigma)$ is a permutation class if and only if~$\hat{\sigma}$ contains the pattern~$231$. In this case, $\Sort(\sigma)=\Perm\left(132,\reverse(\sigma)\right)$, hence the basis of~$\Sort(\sigma)$ has exactly two elements if and only if~$\reverse(\sigma)$ avoids~$132$, or, equivalently, if~$\sigma$ avoids~$231$. Next we enumerate those patterns~$\sigma$ such that the basis of~$\Sort(\sigma)$ has two elements.

\begin{lemma}
Let~$\sigma=\sigma_1\cdots\sigma_k$, with~$k\ge 3$, and suppose that~$\reverse(\sigma)$ avoids~$132$. Then~$\hat{\sigma}$ contains the pattern~$231$ if and only if~$\sigma_1\sigma_2\sigma_3$ is an occurrence of~$321$.
\end{lemma}
\begin{proof}
Observe that, since~$\sigma$ avoids~$231$ by hypothesis, an occurrence of~$231$ in~$\hat{\sigma}=\sigma_2\sigma_1\sigma_3\cdots\sigma_k$ must involve both~$\sigma_1$ and~$\sigma_2$, respectively as the first and the second element of the pattern, with~$\sigma_2<\sigma_1$.

Suppose that~$\hat{\sigma}$ contains an occurrence~$\sigma_2\sigma_1\sigma_i$ of~$231$, for some~$i\ge 3$. If~$\sigma_3>\sigma_2$, then~$i>4$ and thus~$\sigma_2\sigma_3\sigma_i$ is an occurrence of~$231$ in~$\sigma$, which is a contradiction. Therefore we have~$\sigma_3<\sigma_2$ and~$\sigma_1\sigma_2\sigma_3$ is an occurrence of~$321$, as desired.

Conversely, if~$\sigma_1\sigma_2\sigma_3$ is an occurrence of the pattern~$321$, then clearly~$\sigma_2\sigma_1\sigma_3$ is an occurrence of~$231$ in~$\hat{\sigma}$.
\end{proof}

\begin{proposition}\label{proposition_fourth_convolution}
Let~$n\ge 1$. Define~$\mathcal{A}_n=\left\lbrace\pi\in\Perm_n(231):\pi_1\pi_2\pi_3\simeq 321\right\rbrace$ and let~$a_n=|\mathcal{A}_n|$. Then, for each~$n\ge 2$, we have~$a_n=\catalan_n-2\catalan_{n-1}$. In particular, the generating function of the sequence~$(a_n)_{n\ge 0}$ is:
$$
A(x)=\frac{1-4x+2x^2 -(1-2x)\sqrt{1-4x}}{2x}.
$$
\end{proposition}
\begin{proof} Suppose that~$n\ge 2$. Define the sets:
$$
\mathcal{F}_n=\left\lbrace\pi\in\Perm_n(231):\pi_1<\pi_2\right\rbrace
$$
and
$$
\mathcal{G}_n=\left\lbrace\pi\in\Perm_n(231):\pi_1>\pi_2,\pi_2<\pi_3\right\rbrace,
$$
so that:
$$
\Perm_n(231)=\mathcal{A}_n\dot{\cup}\mathcal{F}_n\dot{\cup}\mathcal{G}_n.
$$
Let~$f_n=|\mathcal{F}_n|$ and~$g_n=|\mathcal{G}_n|$. Since~$|\Perm_n(231)|=\catalan_n$, we have~$a_n=\catalan_n -(f_n +g_n)$. We now show that~$f_n=g_n=\catalan_{n-1}$ by providing bijections between~$\mathcal{F}_n$ and~$\Perm_{n-1}(231)$, as well as between~$\mathcal{G}_n$ and~$\Perm_{n-1}(231)$. The desired enumeration follows.

\begin{itemize}
\item If~$\pi\in\mathcal{F}_n$, then it must be~$\pi_1=1$, otherwise~$\pi_1\pi_21$ would be an occurrence of~$231$ in~$\pi$. Define the map~$f:\mathcal{F}_n\rightarrow\Perm_{n-1}(231)$, where~$f(\pi)$ is obtained from~$\pi$ by removing~$\pi_1=1$ and subtracting one to the remaining entries. It is easy to realize that~$f(\pi)\in\Sort_n(231)$ and that~$f$ is an injection. Moreover, if~$\tau\in\Perm_{n-1}(231)$, then adding a new minimum at the beginning (and rescaling the other elements) cannot create any occurrence of~$231$, so~$f$ is also surjective.

\item If~$\pi\in\mathcal{G}_n$, then it must be~$\pi_2=1$, otherwise it would be~$\pi_2\pi_31\simeq 231$ in~$\pi$, a contradiction. We thus define~$g:\mathcal{G}_n\rightarrow\Perm_{n-1}(231)$ such that~$g(\pi)$ is obtained from~$\pi$ by removing~$\pi_2=1$ and rescaling the remaining elements. Again it is clear that~$g(\pi)\in\Perm_n(231)$ and that~$g$ is an injection. Finally, if~$\tau\in\Perm_{n-1}(231)$, then the permutation~$\pi$ obtained from~$\tau$ by adding a new minimum in the second position avoids~$231$. Indeed a potential occurrence of~$231$ in~$\pi$ should involve the added element~$\pi_2$, and so~$\pi_2$ would be either the first or the second element of such an occurrence. But this is impossible since~$\pi_2=1$. Therefore~$g$ is a bijection between~$\mathcal{G}_n$ and~$\Perm_{n-1}(231)$, as desired.
\end{itemize}

Let us now compute the generating function~$A(t)=\sum_{n\ge 1}a_nt^n$. Let~$\CatalanFun(t)=(1-\sqrt{1-4t})/(2t)$ be the generating function for the Catalan numbers. We have:

\begin{equation*}
\begin{split}
A(t)=\sum_{n\ge0}a_{n+2}t^{n+2}=\\
\sum_{n\ge0}\catalan_{n+2}t^{n+2}-2t\sum_{n\ge0}a_{n+1}t^{n+1}\\
=\CatalanFun(t)-t-1-2t(\CatalanFun(t)-1)=\\
\CatalanFun(t)(1-2t)+t-1,
\end{split}
\end{equation*}
from which
$$
A(t)=\frac{1-4t+2t^2 -(1-2t)\sqrt{1-4t}}{2t},
$$
as desired.
\end{proof}

The sequence~$(a_n)_{n\ge 0}$ is recorded (with offset two) as sequence~A002057 in~\cite{Sl}. The first terms are~$0,0,1,4,14,48,165,572,2002$. An alternative expression for its generating function is given by~$A(t)=t^2C(t)^4$, although we are not able to provide a combinatorial explanation of this fact.

We end this section by collecting some enumerative results concerning classes of~$\sigma$-sortable permutation with basis of cardinality two (see Appendix~\ref{appendix_basis_size_two}). A direct combinatorial argument can be used in order to prove each of these results, as we show in the following example. In fact, due to Theorem~\ref{theorem_suff_cond_class}, each of these classes is a subclass of~$\Perm(132)$ of the form~$\Perm(132,\reverse(\sigma))$, and thus its generating function is rational. A constructive proof of this fact can be found in~\cite{MV}, which provides an algorithm to compute the generating function in all such cases. A clear and succint description of the algorithm (in the context of representing catalan structures as arch systems) is given in~\cite{AB}.

\begin{example}
Let~$\sigma=421356$. Then~$\Sort(\sigma)=\Perm(132,653124)$ due to Theorem~\ref{theorem_suff_cond_class}. Given~$\pi\in\Perm_n(132,653124)$, write~$\pi=LnR$, where~$L$ is the prefix of~$\pi$ that precedes~$n$ and~$R$ is the suffix of~$\pi$ that follows~$n$. Notice that, since~$\pi$ avoids~$132$, we have~$L>R$, i.e.~$x>y$ for each~$x\in L$ and~$y\in R$ (otherwise it would be~$xny\simeq 132$). Now, we can partition~$\Perm(132,653124)$ according to whether~$L$ is increasing or not in the above decomposition of~$\pi$. If~$L$ is increasing, then~$\pi\in\Perm(132,653124)$ if and only if~$R\in\Perm(132,53124)$. Indeed any occurrence of~$53124$ in~$R$ would realize an occurrence of~$653124$ together with~$n$. Conversely, if~$L$ is increasing, then the elements corresponding to~$53124$ in any occurrence of~$653124$ in~$\pi$ must belong to~$R$. Similarly, if~$L$ contains at least one descent, then~$\pi\in\Perm(132,653124)$ if and only if~$R\in\Perm(132,3124)$. Let~$G(t)=\sum_{k\ge 0}g_kt^k$ be the generating function of~$\Perm(132,53124)$ and let~$H(t)=\sum_{k\ge 0}h_kt^k$ be the generating function of~$\Perm(132,3124)$. Let~$F(t)=\Fsigma{421356}(t)$ and~$f_n=\fsigma{421356}$, for~$n\ge 0$. Then, due to the above discussion:
$$
f_{n+1}=\sum_{k=0}^n1\cdot g_{n-k}+\sum_{k=0}^n (f_k-1)h_{n-k}.
$$
By summing over~$n$, we get:
$$
\frac{1}{t}\left(F(t)-1)\right)=\frac{1}{1-t}G(t)+F(t)H(t)-\frac{1}{1-t}H(t).
$$
Now, it is easy to compute the generating functions~$G(t)=\frac{t^2-3t+1}{3t^2-4t+1}$ and~$H(t)=\frac{1-2t}{1-3t+t^2}$. Then, solving the above equation yields:
$$
F(t)=\frac{2t^5-16t^4+29t^3-23t^2+8t-1}{9t^5-33t^4+46t^3-30t^2+9t-1}
$$
The resulting sequence starts~$1,2,5,14,42,131,416,1329,4247,13544,\dots$ and does not appear in~\cite{Sl}.
\end{example}

\section{A set of challenging patterns}\label{section_bivincular_result}

For the rest of this chapter, denote by~$\xi=(132,\lbrace 0,2\rbrace,\emptyset)$ the bivincular pattern depicted in Figure~\ref{figure_bivincular_pattern_132_0_2}. The main result of this section is a proof that~$\Sort(\sigma)$ is always a subset of~$\Perm(\xi)$, unless~$\sigma$ is the skew sum of~$12$ with a non-empty~$231$-avoiding permutation~$\beta$. The geometric structure of permutations avoiding~$\xi$ can be described precisely, as we show in what follows. This suggests that the family of~$\sigma$-machines, when~$\sigma=12\ominus\beta$, could contain the more challenging~$\sigma$-machines to be studied. The shortest such pattern is~$231$. In fact, as suggested by some data, the~$231$-machine seems to be the~$\sigma$-machine that can sort the largest amount of permutations. For example, it is the only one that can sort every permutation of length three.

We start by providing a geometric description of~$\Perm(\xi)$, from which its enumeration follows easily. Let~$\pi=\pi_1\cdots\pi_n\in\Perm(\xi)$ and let~$\pi_1=t+1$, for some~$t\ge 0$. Let~$\lbrace \pi_{i_1},\dots,\pi_{i_t}\rbrace$ be the set of elements of~$\pi$ that are smaller than~$\pi_1$, with~$i_1<i_2<\cdots<i_t$. For~$j=1,\dots,t$, let~$b_j=\pi_{i_j}$. Finally, write:
$$
\pi=\pi_1 B_0 b_1 B_1 b_2 B_2\cdots b_t B_t,
$$
where~$B_j=\pi_{i_j+1}\cdots\pi_{i_{j+1}-1}$, for~$j=0,1,\dots,t$. We refer to this as the \textit{first-element decomposition} of~$\pi$; for~$j=0,1,\dots,t$, $B_j$ is said to be the~$j$-th \textit{block} of~$\pi$ in its first-element decomposition.

\begin{figure}
\centering
$\xi=$
\begin{DrawPerm}
\meshBox{(0,0)}{(1,4)}
\meshBox{(2,0)}{(3,4)}
\fillPerm{1,3,2}{3.99}{3.99}
\end{DrawPerm}
\qquad
$\reverse(\xi)=$
\begin{DrawPerm}
\meshBox{(3,0)}{(4,4)}
\meshBox{(1,0)}{(2,4)}
\fillPerm{2,3,1}{3.99}{3.99}
\end{DrawPerm}
\caption[Bivincular patterns~$\xi$ and~$\reverse(\xi)$.]{Bivincular patterns~$\xi=(132,\lbrace 0,2\rbrace,\emptyset)$ and~$\reverse(\xi)$.}\label{figure_bivincular_pattern_132_0_2}
\end{figure}

\begin{lemma}\label{lemma_first_el_dec_increasing_blocks}
Let~$\pi=\pi_1 B_0 b_1 B_1 b_2 B_2\cdots b_t B_t$ be the first-element decomposition of~$\pi$. Then~$\pi$ avoids~$\xi$ if and only if~$B_j$ is increasing for each~$j$.
\end{lemma}
\begin{proof}
Suppose that~$\pi$ avoids~$\xi$ and let~$j\ge0$. By definition of first-block decomposition, all the elements contained in~$B_j$ are greater than~$\pi_1$. Therefore~$B_j$ is increasing, since otherwise a descent in~$B_j$ would result in an occurrence of~$\xi$. On the other hand, suppose that~$\pi_u\pi_v\pi_{v+1}$ is an occurrence of~$\xi$ in~$\pi$. Note that~$u=1$ and~$\pi_v>\pi_{v+1}$, with~$v\ge 2$ and~$\pi_{v+1}>\pi_1$. Thus~$\pi_v$ and~$\pi_{v+1}$ are in the same block~$B_j$, for some~$j$, and~$B_j$ is not increasing.
\end{proof}

\begin{corollary}\label{lemma_first_el_dec_identity}
If~$\pi$ avoids~$\xi$ and~$\pi_1=1$, then~$\pi$ is the increasing permutation.
\end{corollary}

\begin{theorem}\label{theorem_first_el_dec_enumer}
For~$n\ge 0$ and~$t=0,1,\dots,n-1$, define~$\Perm_n^t(\xi)$ by
$$
\Perm_n^t(\xi)=\lbrace\pi\in\Perm_n(\xi):\pi_1=t+1\rbrace.
$$
Let~$f_{n,t}$ be the cardinality of~$\Perm_n^t$. Then:
$$
f_{n,t}=t!(t+1)^{n-t-1}.
$$
In particular, $|\Perm_n(\xi)|=\displaystyle{\sum_{t=0}^{n-1}t!(t+1)^{n-t-1}}$ (sequence~A129591 in~\cite{Sl}).
\end{theorem}
\begin{proof}
Any permutation~$\pi\in\Perm_n^t(\xi)$ can be constructed as follows. The~$t$ elements of~$\pi$ that are smaller than~$\pi_1=t+1$ can be chosen freely, since they cannot contribute to an occurrence of~$\xi$. This can be done in~$t!$ distinct ways. On the other hand, by Lemma~\ref{lemma_first_el_dec_increasing_blocks}, elements greater than~$\pi_1$ must be arranged in increasing blocks. In other words, for each of them it is sufficient to choose the index of the (increasing) block it belongs to. So there are~$t+1$ possibilities for each of the remaining~$n-t-1$ elements. Therefore~$f_{n,t}=t!\cdot(t+1)^{n-t-1}$, as desired.
\end{proof}

\begin{theorem}\label{theorem_bivincular_pattern_gen_result}
Let~$\sigma$ be a permutation of length at least three. The following three conditions are equivalent:
\begin{enumerate}
\item[$(1)$]~$\Sort(\sigma)\not\subseteq\Perm(\xi)$.
\item[$(2)$]~$\sigma= 12\ominus\beta$, for some~$\beta\in\Perm(231)$.
\item[$(3)$]~$\hat{\sigma}\in\Perm(231)$ and~$\reverse(\sigma)\notin\Perm(\xi)$.
\end{enumerate}
\end{theorem}
\begin{proof}
Let~$\sigma=\sigma_1\cdots\sigma_k$, with~$k\ge 3$.
\begin{itemize}
\item We start by proving that~$(2)$ and~$(3)$ are equivalent. Suppose that~$\sigma=12\ominus\beta$, for some~$\beta=\beta_1\cdots\beta_s\in\Perm(231)$, where~$s=k-2$. Observe that~$\hat{\sigma}=(s+2)(s+1)\beta$ avoids~$231$, since~$\beta$ does so. Finally, we have~$\reverse(\sigma)=\beta_s\cdots\beta_1(s+2)(s+1)$, thus~$\beta_s(s+2)(s+1)$ is an occurrence of~$\xi$ in~$\reverse(\sigma)$, as wanted.

Conversely, suppose that~$\hat{\sigma}$ avoids~$231$ and~$\reverse(\sigma)$ contains~$\xi$, or, equivalently, $\sigma$ contains~$\reverse(\xi)$. The pattern~$\reverse(\xi)$ is depicted in Figure~\ref{figure_bivincular_pattern_132_0_2}. Let~$\sigma_i\sigma_{i+1}\sigma_k$ be an occurrence of~$\reverse(\xi)$ in~$\sigma$. Note that the classical pattern underlying~$\reverse(\xi)$ is~$231$, but~$\hat{\sigma}$ avoids~$231$ by hypothesis. Therefore it has to be~$i=1$, otherwise~$\sigma_i\sigma_{i+1}\sigma_k$ would still be an occurrence of~$231$ in~$\hat{\sigma}$, which is impossible. Thus~$\sigma_k<\sigma_1<\sigma_2$. Now, observe that~$\sigma_u<\sigma_1$ for each~$u>2$. Otherwise, if~$\sigma_u>\sigma_1$ for some~$2<u<k$, then~$\sigma_1\sigma_u\sigma_k$ would be an occurrence of~$231$ in~$\hat{\sigma}$, which is again impossible. Therefore~$\sigma=\sigma_1\sigma_2\ominus\beta=12\ominus\beta$, where~$\beta=\sigma_3\cdots\sigma_k$. Finally, $\beta$ avoids~$231$ because~$\hat{\sigma}$ does so, as wanted.

\item Next we wish to prove that~$(3)$ implies~$(1)$. Suppose that~$\hat{\sigma}$ avoids~$231$ and~$\reverse(\sigma)$ contains~$\xi$. We show that~$\reverse(\sigma)$ is~$\sigma$-sortable (and contains~$\xi$), thus~$\reverse(\sigma)\in\Sort(\sigma)\setminus\Perm(\xi)$. Due to Lemma~\ref{lemma_hat_sigma}, we have~$\out{\sigma}(\reverse(\sigma))=\hat{\sigma}$. Finally, $\hat{\sigma}$ avoids~$231$, so~$\reverse(\sigma)$ is~$\sigma$-sortable, as desired.

\item Finally, we show that~$(1)$ implies~$(2)$, which completes the proof. Suppose that there is a permutation~$\pi=\pi_1\cdots\pi_n$ such that~$\pi$ is~$\sigma$-sortable and~$\pi$ contains~$\xi$. Let~$\pi_1\pi_j\pi_{j+1}$ be an occurrence of~$\xi$ in~$\pi$. Let~$\beta=\sigma_3\cdots\sigma_k$. We show that~$\sigma_2>\sigma_1>\sigma_u$ for each~$u\ge 3$ and~$\beta$ is a~$231$-avoiding permutation. Observe that~$\hat{\sigma}$ avoids~$231$. Otherwise it would be~$\Sort(\sigma)=\Perm(132,\reverse(\sigma))$ due to Theorem~\ref{theorem_suff_cond_class} and thus~$\Sort(\sigma)\subseteq\Perm(\xi)$, contradicting the hypothesis. In particular, $\beta$ avoids~$231$, as wanted. Now, since~$\pi_1$ is the last element that exits the~$\sigma$-stack, $\pi_j$ must be extracted before~$\pi_{j+1}$ enters, else~$\pi_{j+1}\pi_j\pi_1$ would be an occurrence of~$231$ in~$\out{\sigma}(\pi)$, contradicting the fact that~$\pi$ is~$\sigma$-sortable. Let us consider the instant when~$\pi_j$ is extracted (and~$\pi_{j+1}$ is the next element of the input). Since a pop operation is performed by the~$\sigma$-stack, the~$\sigma$-stack must contain~$k-1$ elements~$\alpha_2\alpha_3\cdots\alpha_k$ (reading from top to bottom) such that~$\pi_{j+1}\alpha_2\cdots\alpha_k$ is an occurrence of~$\sigma$. Without losing generality, we can suppose that~$\alpha_3$ is still in the~$\sigma$-stack when~$\pi_{j+1}$ enters: this can be achieved, for instance, by taking the ``deepest" such sequence of elements in the~$\sigma$-stack. Note that~$\out{\sigma}(\pi)$ contains the occurrence~$\alpha_2\pi_{j+1}\alpha_3\cdots\alpha_k$ of~$\hat{\sigma}$. Now, if~$\alpha_v>\pi_{j+1}$ for some~$v\ge 3$, then~$\alpha_v\neq\pi_1$ (because~$\pi_1<\pi_{j+1}$) and~$\pi_{j+1}\alpha_v\pi_1$ is an occurrence of~$231$ in~$\out{\sigma}(\pi)$, a contradiction with~$\pi$ being~$\sigma$-sortable. Therefore, since~$\pi_{j+1}\alpha_2\cdots\alpha_k\simeq\sigma$, we have~$\sigma_u<\sigma_1$ for each~$u\ge 3$. To conclude the proof, we have to show that~$\sigma_1<\sigma_2$. Suppose, for a contradiction, that~$\sigma_1>\sigma_2$. Then~$\sigma_1=k$ is the maximum element of~$\sigma$, since~$\sigma_1>\sigma_u$ for each~$u\ge 3$. Now, consider the instant immediately after~$\pi_j$ is pushed into the~$\sigma$-stack (and~$\pi_{j+1}$ is the next element of the input). Note that~$\pi_{j+1}>\alpha_2$, because we are assuming~$\sigma_1>\sigma_2$ and~$\pi_{j+1}\alpha_2\simeq\sigma_1\sigma_2$. But~$\pi_j>\pi_{j+1}$, thus~$\pi_j\alpha_2\cdots\alpha_s$ is an occurrence of~$\sigma$ contained in the~$\sigma$-stack, which is impossible.
\end{itemize}
\end{proof}

\begin{corollary}\label{corollary_bivinc_patt_identity}
Let~$\sigma=12\ominus\beta$, for some non-empty and~$231$-avoiding permutation~$\beta$. Let~$\pi=\pi_1\cdots\pi_n$ be a~$\sigma$-sortable permutation with~$\pi_1=1$. Then~$\pi$ is the identity permutation.
\end{corollary}
\begin{proof}
It follows from Lemma~\ref{lemma_first_el_dec_identity} and Theorem~\ref{theorem_bivincular_pattern_gen_result}.
\end{proof}

Corollary~\ref{corollary_bivinc_patt_identity} fails if~$\Sort(\sigma)\not\subseteq\Perm(\xi)$. For example, the permutations~$12354$, $12453$, $12534$ and $12543$ are~$3421$-sortable.

\section{The decreasing pattern}\label{section_decreasing_pattern}

In this section we provide some enumerative results for the sets~$\Sort(\antiid_k)$, highlighting a link with a class of pattern-avoiding lattice paths. The results of the previous section allow us to directly characterize~$\sigma$-sortable permutations when~$\sigma$ is the decreasing pattern. Indeed, by Theorem~\ref{theorem_suff_cond_class} and for each~$k\ge 1$, we have~$\Sort(\antiid_k)=\Perm(\identity_k,132)$. The sequences that enumerate these sets, for~$k\le 7$, are reported in Table~\ref{table_decr_pattern}.

\begin{table}
\centering
\def\arraystretch{1.1}
\begin{tabular}{llr}
\toprule
$k$ & \textbf{Sequence}~$\lbrace\fsigma{\antiid_k}_n\rbrace_n$ & \textbf{OEIS}\\
\midrule
3 & 1, 2, 4, 8, 16, 32, 64, 128, 256, 512, 1024 & A011782\\
4 & 1, 2, 5, 13, 34, 89, 233, 610, 1597, 4181, 10946 & A001519\\
5 & 1, 2, 5, 14, 41, 122, 365, 1094, 3281, 9842, 29525 & A124302\\
6 & 1, 2, 5, 14, 42, 131, 417, 1341, 4334, 14041, 45542 & A080937\\
7 & 1, 2, 5, 14, 42, 132, 428, 1416, 4744, 16016, 54320 & A024175\\
\bottomrule
\end{tabular}
\caption[Enumerative results for the~$\antiid_k$-machine.]{Enumerative results for~$\antiid_k$-sortable permutations, with~$k=3,4,5,6,7$, starting from permutations of length one.}\label{table_decr_pattern}
\end{table}

If~$n<k$, then obviously~$\Sort_n(\antiid_k)=\Perm_n(132)$ and thus~$\fsigma{\antiid_k}_n=\catalan_n$. Therefore the rows of Table~\ref{table_decr_pattern} tend to the sequence of Catalan numbers. By looking at the reference in~\cite{Sl} for small values of~$k$, we notice that~$\lbrace\fsigma{\antiid_k}_n\rbrace_n$ counts the number of Dyck paths of height at most~$k-1$. A formal proof can be obtained by using the bijection between Dyck paths and~$132$-avoiding permutations mentioned in Example~\ref{example_dyck_213_bij}. Indeed, if~$\pi$ is a~$132$-avoiding permutation and~$P$ is the Dyck path associated to~$\pi$, then the maximum length of an increasing sequence in~$\pi$ is equal to to the height of~$P$. Finally, a permutation~$\pi$ avoids~$\identity_k$ if and only if the maximum length of an increasing sequence in~$\pi$ is at most~$k-1$. Dyck paths of bounded height are rather well studied objects (see for example~\cite{BM,GX}).

We now compute the generating function of~$\lbrace\fsigma{\antiid_k}_n\rbrace_n$ by exploiting this connection with Dyck paths of bounded height. Let~$F_k(t)=\Fsigma{\antiid_k}(t)$. Given a Dyck path~$P$, consider its first-return decomposition~$P=\U Q_1\D Q_2$, for some (possibly empty) Dyck paths~$Q_1,Q_2$ (see Remark~\ref{remark_dyck_first_return}). If~$P$ has height at most~$k$, then~$Q_2$ has height at most~$k$, whereas~$Q_1$ has height at most~$k-1$. This provides a recursive description of~$F_k(t)$ with respect to the semilength:
$$
\begin{cases}
F_0(t)=1;\\
F_k(t)=1+tF_{k-1}(t)F_k(t),\ k\ge 1.
\end{cases}
$$
A consequence of the above recurrence is that~$F_k(t)$ is rational, for all~$k$; indeed we have
$$
F_k(t)=\frac{G_k(t)}{G_{k+1}(t)},
$$
where~$G_0(t)=G_1(t)=1$ and~$G_k(t)$ satisfies the recurrence
$$
G_{k+1}(t)=G_k(t)-tG_{k-1}(t).
$$
Solving this recurrence yields
$$
G_k(t)=\sum_{i\ge 0}\binom{n-1}{i}(-t)^i.
$$
The polynomials~$G_k(t)$ are sometimes called \textit{Catalan polynomials} (see for instance~\cite{CLF}); the table of their coefficients is sequence~A115139 in~\cite{Sl}.

\section{Open problems}\label{section_open_problems}

In this chapter we provided some general results regarding~$\sigma$-machines and sets of~$\sigma$-sortable permutations. As a consequence of Corollary~\ref{corollary_class_nonclass_char}, we are able to tell when~$\Sort(\sigma)$ is a permutation class by simply checking whether~$\hat{\sigma}$ contains~$231$ or not. If~$\Sort(\sigma)$ is a class, Theorem~\ref{theorem_suff_cond_class} states that~$\Sort(\sigma)=\Perm(132,\reverse(\sigma))$, thus the set of~$\sigma$-sortable permutations is completely determined (and enumerated). On the other hand, Theorem~\ref{theorem_bivincular_pattern_gen_result} is currently the only known general result when~$\Sort(\sigma)$ is not a permutation class. It would be interesting to provide more results in order to find structural information on the sets~$\Sort(\sigma)$, when they are not permutation classes.

\begin{openproblem}
Find geometric properties of the set~$\Sort(\sigma)$, when~$\Sort(\sigma)$ is not a permutation class.
\end{openproblem}

More specifically, the only non-class for patterns~$\sigma$ of length two is~$\Sort(21)$, which is the classical case of West's~$2$-stack sortable permutations. Moving on to patterns of length three, the only permutation class is~$\Sort(321)=\Perm(132,123)$. We provide a characterization of the sets~$\Sort(123)$, in Chapter~\ref{chapter_pattern123}, and~$\Sort(132)$, in Chapter~\ref{chapter_pattern132}. The remaining three patterns are yet to be solved. Some related data are reported in Table~\ref{table_unsolved_patterns}. A potentially interesting link with ascent sequences is the following: in Chapter~\ref{chapter_pattern132} we prove that~$\Sort(132)$ is Wilf-equivalent to the set~$\Ascseq(312,321)$ of ascent sequences avoiding~$312$ and~$321$ (see~\cite{BP}), while~$\Sort(321)$ seems to be Wilf-equivalent to~$\Ascseq(312)$.

\begin{openproblem}
Characterize and enumerate the sets~$\Sort(213)$, $\Sort(231)$ and~$\Sort(312)$.
\end{openproblem}

\begin{table}
\centering
\def\arraystretch{1.1}
\begin{tabular}{llr}
\toprule
$\sigma$ & \textbf{Sequence}~$\lbrace\fsigma{\sigma}_n\rbrace_n$ & \textbf{OEIS}\\
\midrule
213 & 1, 2, 5, 16, 62, 273, 1307, 6626, 35010, 190862 & \\
231 & 1, 2, 6, 23, 102, 496, 2569, 13934, 78295, 452439 & \\
312 & 1, 2, 5, 15, 52, 201, 843, 3764, 17659, 86245 & A202062\\
\bottomrule
\end{tabular}
\caption[Unsolved patterns of length three.]{Enumerative data for unsolved patterns of length three, starting from~$\sigma$-sortable permutations of length one.}
\label{table_unsolved_patterns}
\end{table}

If we consider the family of~$\sigma$-machines from the enumerative perspective, it would be nice to investigate deeper the notion of Wilf-equivalence that naturally arises by looking at how many different sequences of~$\sigma$-sortable permutations can be obtained for patterns~$\sigma$ of a fixed length. Formally, we say that two patterns~$\sigma$ and~$\tau$ of length~$k$ are \textit{PAM-Wilf-equivalent} (where PAM stands for pattern-avoiding machine) if the sets~$\Sort(\sigma)$ and~$\Sort(\tau)$ are Wilf-equivalent in the usual sense. Denote by~$w_k$ the number of PAM-Wilf classes of length~$k$.

\begin{openproblem}
Compute the number of PAM-Wilf classes, that is the sequence~$\lbrace w_k\rbrace_{k\ge 1}$.
\end{openproblem}

A slightly easier version of the above open problem can be obtained by considering the sets~$\Sort(\sigma)$ which are permutation classes only. Some data (for which the author is extremely grateful to Christian Bean and Anders Claesson) indicate that the first terms of the resulting sequence, starting from length two, are~$1,1,2,5,11,25,55,126,283$ (not in~\cite{Sl}). For example, there are~$11$ such Wilf-classes for patterns~$\sigma$ of length six:~$10$ of them are reported in Appendix~\ref{appendix_basis_size_two} and the last one consists of those patterns~$\sigma$ such that~$\Sort(\sigma)$ is a class and~$\sigma\ge 132$, that is where~$\Sort(\sigma)=\Perm(132)$ and the counting sequence is the sequence of Catalan numbers.

\chapter{\texorpdfstring{The~$123$-machine}{The 123-machine}}\label{chapter_pattern123}

This chapter is devoted to the analysis of the~$123$-machine. The paper~\cite{CeClFS} contains most of the results presented in this part of the thesis. Since, as a consequence of Corollary~\ref{corollary_class_nonclass_char}, the set~$\Sort(123)$ is not a permutation class, this pattern is considerably more challenging that the decreasing pattern of the same length. If we compute the first terms of the sequence~$\lbrace\fsigma{123}_n\rbrace_{n\ge 1}$, we get~$1,2,5,13,35,99,...$, which suggests a match with~A294790 in~\cite{Sl}. This sequence enumerates, for example, Schr\"oder paths avoiding the (consecutive) path~$\U\H_2\D$ (see~\cite{CiF}). Our goal is to provide a length-preserving bijection between~$123$-sortable permutations and this family of pattern-avoiding paths. To do that, we follow a step-by-step procedure, aiming to progressively reduce the problem of characterizing~$123$-sortable permutations to more manageable subsets of~$\Sort(123)$.

\section{\texorpdfstring{Structural description of~$\Sort(123)$}{Structural description of Sort(123)}}\label{section_123_struct}

We start by dealing with~$123$-sortable permutations that start with an ascent.

\begin{lemma}\label{lemma_123_init_asc}
Let~$\pi\in\Perm_n$. If~$\pi$ is~$123$-sortable, then~$\pi_2\le\pi_1+1$.
\end{lemma}
\begin{proof}
Suppose, for a contradiction, that~$\pi_2>\pi_1+1$. Then there exists an index~$i\ge 3$ such that~$\pi_2>\pi_i>\pi_1$. Note that the first two elements~$\pi_1$ and~$\pi_2$ are extracted from the~$123$-stack only when the~$123$-stack is emptied at the end of the sorting process. Indeed, since~$\pi_1<\pi_2$ (and~$\pi_2$ enters above~$\pi_1$), they cannot be both part of an occurrence of~$123$. Thus~$\pi_i\pi_2\pi_1$ is an occurrence of~$231$ in~$\out{123}(\pi)$, contradicting the hypothesis that~$\pi$ is~$123$-sortable.
\end{proof}

Let us now partition~$\Sort(123)$ according to Lemma~\ref{lemma_123_init_asc}: permutations starting with a consecutive ascent~$\pi_2=\pi_1+1$, and permutations starting with a descent. The next step consists in showing that inflating the first element of a~$123$-sortable permutation does not affect its~$123$-sortability.

\begin{lemma}\label{lemma_123_infl_lemma}
Let~$\pi$ be a permutation of length~$n$ and let~$\pi'$ be the permutation (of length~$n+1$) obtained from~$\pi$ by~$2$-inflating~$\pi_1$. Then~$\pi$ is~$123$-sortable if and only if~$\pi'$ is~$123$-sortable.
\end{lemma}
\begin{proof}
Observe that, by hypothesis, the first two elements of~$\pi'$ are consecutive in value ($a$ and~$a+1$, say) and the first one is smaller than the second one. Therefore, during the sorting process, such two elements remain at the bottom of the $123$-stack (with~$a+1$ above~$a$) until all the other elements of the input permutations have exited it. Moreover, since~$a+1$ is above~$a$, the behavior of the~$123$-stack is not affected by the presence of~$a+1$, meaning that~$a$ and~$a+1$ can be considered as a single element. As a consequence, the last two elements of~$\out{123}(\pi')$ are~$a+1$ and~$a$. Finally, it is easy to realize that~$\out{123}(\pi)$ contains~$231$ if and only~$\out{123}(\pi')$ contains~$231$.
\end{proof}

\begin{corollary}\label{corollary_123_infl_cor}
Let~$\pi$ be a permutation of length~$n$ and let~$\pi'$ be the permutation (of length~$n+k-1$) obtained from~$\pi$ by~$k$-inflating~$\pi_1$, for some~$k\ge 1$. Then~$\pi$ is~$123$-sortable if and only if~$\pi'$ is~$123$-sortable.
\end{corollary}
\begin{proof} This is a direct consequence of the previous corollary, by just iterating the same argument.
\end{proof}

Due to Lemma~\ref{lemma_123_init_asc} and Corollary~\ref{corollary_123_infl_cor}, in order to describe~$\Sort(123)$ we just need to investigate the sortability of permutations starting with a descent. Denote by~$\DSort(123)$ the set:
$$
\DSort(123)=\lbrace\pi\in\Sort(123):\pi_1>\pi_2\rbrace.
$$
By first characterizing and enumerating~$\DSort_n(123)$, we can easily recollect the analogous results for~$\Sort(123)$. Indeed by deflating the prefix of consecutive ascents (if there is one), we can always trace back the~$123$-sortability of a permutation to another permutation in~$\DSort(123)$.

\begin{lemma}\label{lemma_123_output}
Let~$\pi\in\DSort_n(123)$, with~$\pi_1=k$. Then:
$$
\out{123}(\pi)=n(n-1)\cdots(k+1)(k-1)\cdots 21k.
$$
\end{lemma}
\begin{proof}
Let~$\out{123}(\pi)=\gamma_1\gamma_2\cdots\gamma_n$. Clearly~$\gamma_n=k$. Now suppose, for a contradiction, that the two elements~$u$ and~$v$ constitute an ascent in~$\out{123}(\pi)$, with~$u<v$ and~$v\neq k$. We first show that~$v$ precedes~$u$ in~$\pi$. Suppose in fact that this is not the case, and focus on the instant when~$u$ is extracted from the~$123$-stack. Let~$a$ be the next element of the input when this happens. Then there are two elements~$b,c$ in the~$123$-stack, with~$b<c$ and~$b$ above~$c$, such that~$abc\simeq 123$. We distinguish two cases.

\begin{itemize}
\item $u=b$. In this case, we have~$a\neq v$, and so~$v$ follows~$a$ in~$\pi$. Therefore~$\out{123}(\pi)$ contains either the subword~$uav$, which is impossible since~$u$ and~$v$ are supposed to be consecutive in~$\out{123}(\pi)$, or the subword~$uva$, which is impossible too since otherwise~$\out{123}(\pi)$ would contain the pattern~$231$, contradicting the fact that~$\pi$ is~$123$-sortable.

\item $u\neq b$. In this case, $\out{123}(\pi)$ would contain the subword~$ubv$, which is impossible, again because~$u$ and~$v$ would not be consecutive.
\end{itemize}

Thus we can write~$\pi$ as~$\pi=k\pi_2\cdots v\cdots u\cdots$. Since~$u$ and~$v$ are consecutive in~$\out{123}(\pi)$, $u$ must enter the~$123$-stack just above~$v$. This implies, in particular, that~$v>\pi_1$, otherwise~$u$, $v$ and~$\pi_1$ would constitute a forbidden~$123$ inside the~$123$-stack. We also notice that, when~$u$ enters the~$123$-stack, at the bottom of the~$123$-stack there is at least one element~$w<\pi_1$ just above~$\pi_1$. Indeed, either~$\pi_2$ is still in the~$123$-stack (and in this case~$w=\pi_2$) or~$\pi_2$ has been forced to exit by some~$\tilde{w}<\pi_2<\pi_1$; in this case, $\tilde{w}$ replaces~$\pi_2$ just above~$\pi_1$. Iterating this argument, we get the desired property. Summing up, when~$u$ enters the~$123$-stack, the~$123$-stack itself contains the elements (from top to bottom)~$u,v,w,\pi_1$. Now, it must be~$u>w$, otherwise~$uw\pi_1$ would be a forbidden~$123$ in the~$123$-stack. Hence~$\out{123}(\pi)$ contains the subword~$uvw\simeq 231$ and~$\pi$ is not~$123$-sortable, a contradiction.
\end{proof}

\begin{corollary}\label{corollary_123_before_maximum}
Let~$\pi\in\DSort_n(123)$ and suppose that~$\pi_1<n$. Also, suppose that~$\pi_i=n$, for some~$i\ge 2$. Then either~$\pi_{i-1}=n-1$, if~$\pi_1\neq n-1$, or~$\pi_{i-1}=n-2$, if~$\pi_1=n-1$.
\end{corollary}
\begin{proof}
Notice that~$i\ge 3$: indeed~$i\neq 1$ by hypothesis and~$i\neq 2$ since~$\pi$ starts with a descent. The element~$\pi_i=n$ enters the~$123$-stack immediately above~$\pi_{i-1}$, since pushing the maximum~$n$ into the~$123$-stack can never generate a forbidden pattern~$123$. Moreover, $n$ and~$\pi_{i-1}$ are extracted from the~$123$-stack together, since~$n$ cannot play the role of the second element in a forbidden pattern inside the~$123$-stack. Therefore, $\out{123}(\pi)$ contains the factor~$n\pi_{i-1}$. The desired result follows then from Lemma~\ref{lemma_123_output}.
\end{proof}

\begin{corollary}\label{corollary_123_starting_maximum}
Let~$n\ge 2$. Then the set of permutations of~$\DSort_n(123)$ starting with~$n$ is the set of~$213$-avoiding permutations of length~$n$ that start with~$n$.
\end{corollary}
\begin{proof}
Let~$\pi\in\DSort_n(123)$, and suppose that~$\pi_1=n$. As soon as~$n$ enters the~$123$-stack, it makes the~$123$-stack act as a~$12$-stack for the rest of the permutation. This means, formally, that since~$n$ is the maximum of~$\pi$, from now on the restriction of the~$123$-stack is triggered if and only if the restriction of a~$12$-stack that ignores~$n$ is triggered. Therefore, by Theorem~\ref{theorem_12machine}, $\pi$ is~$123$-sortable if and only if the permutation obtained from~$\pi$ by removing the first element avoids~$213$, which is in turn equivalent to the fact that~$\pi$ avoids~$213$.
\end{proof}

A straightforward consequence of Corollary~\ref{corollary_123_starting_maximum} is that there are~$\catalan_{n-1}$ permutations in~$\DSort_n(123)$ that start with the maximum~$n$. The remaining permutations of~$\DSort_n(123)$ are precisely those having at least two ltr-maxima. Denote this set by~$\DSort_n({\ge}2; 123)$. Similarly, denote by~$\DSort_n(i;123)$ the set of permutations of~$\DSort_n(123)$ having exactly~$i$ ltr-maxima.

\begin{theorem}\label{theorem_123_bijection}
Let~$n\ge 3$. There exists a bijection:
$$
\varphi:\DSort_{n-1}(123)\rightarrow\DSort_n({\ge}2,123).
$$
Moreover, the restriction of~$\varphi$ to~$\DSort_{n-1}(i;123)$ is a bijection between~$\DSort_{n-1}(i;123)$ and~$\DSort_n(i+1;123)$.
\end{theorem}
\begin{proof}
Let~$\pi=\pi_1\cdots\pi_{n-1}\in\DSort_{n-1}(123)$. Let~$\varphi(\pi)$ be obtained from~$\pi$ by inserting~$n$:

\begin{itemize}
\item either immediately after~$n-1$, if~$\pi_1\neq n-1$, or
\item immediately after~$n-2$, if~$\pi_1=n-1$.
\end{itemize}

First we show that~$\varphi$ is well defined, that is~$\varphi(\pi)\in\DSort_n({\ge} 2;123)$. We analyze the two cases in the definition of~$\varphi$ separately.

\begin{itemize}
\item If~$\pi\in\DSort_{n-1}(1;123)$ (that is~$\pi_1=n-1$), then by Lemma~\ref{lemma_123_output} we have:
$$
\out{123}(\pi)=(n-2)(n-3)\cdots 1(n-1).
$$
Now we analyze what happens on input~$\varphi(\pi)$ after the first pass through the~$123$-stack. Remember that the first element of~$\varphi(\pi)$ is~$n-1$ and that~$n$ immediately follows~$n-2$; moreover, suppose that~$n-2$ is the~$i$-th element of~$\varphi(\pi)$. Therefore, the first~$i$ elements of~$\pi$ and~$\varphi(\pi)$ are equal, and so they are processed exactly in the same way by the~$123$-stack. In particular, since~$n-2$ is the first element of~$\out{123}(\pi)$, when~$n-2$ enters the~$123$-stack, all the previous elements of~$\varphi(\pi)$ are still inside the~$123$-stack. Immediately after~$n-2$ enters the~$123$-stack, $n$ enters the~$123$-stack as well, since it cannot produce a forbidden pattern. Now we claim that~$n$ and~$n-2$ exit the~$123$-stack together. This is trivial if~$n$ is the last element of the input. If instead~$n$ is not the last element of~$\varphi(\pi)$, consider the next element~$\pi_{i+1}$. Such element cannot enter the~$123$-stack, otherwise~$\pi_{i+1}$, $n-2$ and~$n-1$ (which is at the bottom of the~$123$-stack) would constitute a forbidden pattern~$123$. Thus~$n-2$ must exit the~$123$-stack before~$\pi_{i-1}$ enters it, and this forces~$n$ to exit as well. As a consequence of this fact, we have that:
$$
\out{123}(\varphi(\pi))=n(n-2)(n-3)\cdots 1(n-1),
$$
which avoids~$231$. Hence~$\varphi(\pi)$ is sortable.

\item If~$\pi\in\DSort_{n-1}({\ge} 2;123)$ (that is~$\pi_1=k\neq n-1$), then by Lemma~\ref{lemma_123_output} we have:
$$
\out{123}(\pi)=(n-1)(n-2)\cdots (k+1)(k-1)\cdots 21k.
$$
Finally, an analogous argument can be used to prove that:
$$
\out{123}(\varphi(\pi))=n(n-1)(n-2)\cdots (k+1)(k-1)\cdots 21k,
$$
and so that~$\varphi(\pi)$ is $123$-sortable.
\end{itemize}

To complete the proof we now have to show that~$\varphi$ is a bijection. The fact that~$\varphi$ is injective is trivial. To show that~$\varphi$ is surjective, consider the map~$\psi:\DSort_n({\ge} 2;123)\rightarrow\DSort_{n-1}(123)$ which removes~$n$ from~$\alpha\in\DSort_n({\ge} 2;123)$. Let~$\alpha=\alpha_1\cdots\alpha_n$ and let~$i\in\lbrace3,4,\dots n\rbrace$ such that~$\alpha_i=n$. From Corollary~\ref{corollary_123_before_maximum}, we have that either~$\alpha_{i-1}=n-1$ (if~$\alpha_1\neq n-1$) or~$\alpha_{i-1}=n-2$ (if~$\alpha_1 =n-1$). Moreover, Lemma~\ref{lemma_123_output} implies that:
$$
\out{123}(\pi)=n(n-1)\cdots (k+1)(k-1)\cdots 21k,
$$
with~$k=\alpha_1\ge 2$. Therefore, when~$n$ enters the~$123$-stack, all the previous elements are still inside the~$123$-stack. In particular, at the top of the~$123$-stack there are~$n$ and~$\alpha_{i-1}$. Now notice that, if~$n$ is forced to exit the~$123$-stack, this is due to the fact that there exist~$j,h,l$, with~$j<h\le i$ and~$l>i$, such that~$\alpha_l$, $\alpha_h$ and~$\alpha_j$ form an occurrence of~$123$. However, it cannot be~$h=i$, since~$n$ cannot play the role of the~$2$ in a~$123$. Similarly, it cannot be~$h=i-1$: in fact, if~$\alpha_{i-1}=n-1$, then~$n$ and~$n-1$ are consecutive in the~$123$-stack and so they play the same role in any pattern; if instead~$\alpha_{i-1}=n-2$, then~$\alpha_1 =n-1$ is at the bottom of the~$123$-stack, and so~$n$ and~$n-2$ play the same role in any forbidden pattern. As a consequence, $h<i-1$, and so~$n$ and~$\alpha_{i-1}$ are forced to leave the~$123$-stack together. This means that basically~$n$ does not modify the behavior of the machine, and so:
$$
\out{123}(\psi(\alpha))=(n-1)(n-2)\cdots (k+1)(k-1)\cdots 21k,
$$
that is~$\psi(\alpha)$ is~$123$-sortable, as desired.
\end{proof}

\begin{corollary}
For all~$n\ge 3$, $|\DSort_n({\ge} 2;123)|=|\DSort_{n-1}(123)|$.
\end{corollary}

\section{\texorpdfstring{Enumeration of~$\Sort(123)$}{Enumeration of Sort(123)}}\label{section_enumeration_123}

We now use the results proved in Section~\ref{section_123_struct} to enumerate~$\Sort_n(123)$. Due to Corollary~\ref{corollary_123_infl_cor}, Corollary~\ref{corollary_123_starting_maximum} and Theorem~\ref{theorem_123_bijection}, any~$123$-sortable permutation~$\pi$ which is not the identity permutation can be uniquely constructed as follows:

\begin{enumerate}
\item choose~$\alpha=\alpha_1\alpha_2\cdots\alpha_k\in\Perm_k(213)$, with~$\alpha_1=k\ge 2$;
\item add~$h$ new maxima, $k+1,\dots,k+h$, one at a time, using the bijection~$\varphi$ of Theorem~\ref{theorem_123_bijection};
\item add~$n-k-h$ consecutive ascents at the beginning, by inflating the first element of the permutation, according to Corollary~\ref{corollary_123_infl_cor}.
\end{enumerate}

As an example to illustrate the given construction, let~$\pi=567148923$. By deflating the prefix of initial consecutive ascents, we get the permutation~$\pi'=5146723$; due to Corollary~\ref{corollary_123_infl_cor}, $\pi$ is~$123$-sortable if and only if~$\pi'$ is~$123$-sortable. Now, $\pi'$ is (uniquely) obtained by adding two new maxima to the permutation~$\pi''=51423$, whose first element is its maximum, according to the bijection of Theorem~\ref{theorem_123_bijection}. Since~$\pi''$ avoids~$213$, we can finally conclude that~$\pi$ is~$123$-sortable.

\begin{theorem}\label{theorem_123_enum}
For all~$n\ge 1$, we have:
$$
\fsigma{123}_n=1+\sum_{h=1}^{n-1}(n-h)\catalan_h.
$$
\end{theorem}
\begin{proof}
A permutation~$\pi\in\Sort_n(123)$ is either the identity or it is obtained by choosing a permutation~$\alpha$ in~$\Perm_k (213)$ starting with its maximum~$k$ (with~$k\ge 2$) and then (possibly) adding the remaining~$n-k$ elements according to the above construction, i.e. adding new maxima and/or consecutive ascents at the beginning. Concerning~$\alpha$, there are~$\catalan_{k-1}$ possible choices, thanks to the observation following Corollary~\ref{corollary_123_starting_maximum}. For the remaining elements, one has to choose, for instance, the number of new maxima to add, which runs from~$0$ to~$n-k$, so that the total number of choices is~$n-k+1$. Summing on all possible values of~$k$, we get:
$$
\fsigma{123}_n=1+\sum_{k=2}^{n}\catalan_{k-1}\cdot(n-k+1)=1+\sum_{h=1}^{n-1}(n-h)\catalan_h,
$$
as desired.
\end{proof}

We end this section by computing the generating function~$\Fsigma{123}(t)$ of~$\Sort(123)$. As anticipated, we shall exploit the link with pattern-avoiding Schr\"oder paths.

\begin{theorem}\label{theorem_genfun_123}
We have:
$$
\Fsigma{123}(t)=\frac{1-2t+t\CatalanFun(t)}{(1-t)^2}.
$$
\end{theorem}
\begin{proof}
We start by providing a bijection~$f$ between~$123$-sortable permutations of length~$n$ and~$\U\H_2\D$-avoiding Schr\"oder paths of semilength~$n-1$. Given~$\pi\in\Sort_n(123)$, decompose it as~$\pi=L\beta$, where~$L$ is the (possibly empty) initial sequence of consecutive ascents of~$\pi$, deprived of the last element, and~$\beta$ is the remaining suffix of~$\pi$. Suppose that~$L$ has length~$r$. Now repeatedly remove the maximum from~$\beta$ until the remaining word~$\beta'$ starts with its maximum. Denote with~$s$ the number of elements removed this way. Then~$\beta'$ is order isomorphic to a~$213$-avoiding permutation~$\alpha$ of length~$k+1=n-r-s$, that starts with its maximum. Removing the maximum from~$\alpha$ results in another~$213$-avoiding permutation~$\rho$ of length~$k$. We can now describe the Schr\"oder path~$f(\pi)$ associated with~$\pi$: it starts with~$r$ double horizontal steps and ends with~$s$ double horizontal steps; in the middle, there is the Dyck path of semilength~$k$ associated to the~$213$-avoiding permutation~$\rho$ through the bijection described in Example~\ref{example_dyck_213_bij}.

Next, as announced, we express the generating function of~$\Sort(123)$ by exploiting the bijection~$f$. In fact, the generic Schr\"oder path avoiding~$\U\H_2\U$ either consists of double horizontal steps only (so the generating function is~$(1-t)^{-1}$), or can be obtained by concatenating an initial sequence of double horizontal steps (having generating function~$(1-t)^{-1}$) with a non-empty Dyck path (whose generating function is~$(\CatalanFun(t)-1)\cdot t$, where~$\CatalanFun(t)$ is the generating function of the Catalan numbers and the additional factor~$t$ takes into account the removal of the starting maximum from the permutation~$\alpha$ above), finally adding a sequence of double horizontal steps (again with generating function~$(1-t)^{-1}$). Summing up, we get:
$$
\Fsigma{123}(t)
=\frac{1}{1-t}+\frac{1}{1-t}\big(t(\CatalanFun(t)-1)\big)\frac{1}{1-t}
=\frac{1-2t+t\CatalanFun(t)}{(1-t)^2}.
$$
\end{proof}

\begin{figure}
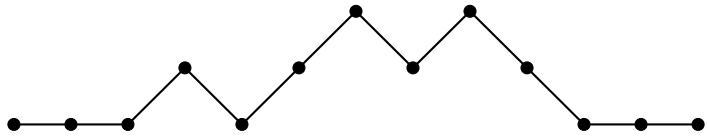

\centering
\begin{DrawPath}
\fillPath{0,0,1,-1,1,1,-1,1,-1,-1,0,0}{0}{0}
\end{DrawPath}
\caption[An~$\U\H_2\D$-avoiding Schr\"oder path and the corresponding~$123$-sortable permutation.]{The~$\U\H_2\D$-avoiding Schr\"oder path associated to the~$123$-sortable permutation~$567489132$. Referring to the notations of Theorem~\ref{theorem_genfun_123}, we have~$L=56$, $\beta=7489132$, and so~$r=s=2$. Moreover, $\alpha=4132$ and 
the associated Dyck path is~$\U\D\U\U\D\U\D\D$.}\label{figure_123_schroder_bij}
\end{figure}

\chapter{\texorpdfstring{The~$132$-machine}{The 132-machine}}\label{chapter_pattern132}

This chapter, whose paper version is~\cite{CeClFS}, is devoted to the study of the~$132$-machine. We prove that~$132$-sortable permutations are enumerated by the binomial transform of Catalan numbers (sequence~A007317 in~\cite{Sl}) by first characterizing~$\Sort(132)$ in terms of avoidance of a classical pattern and a mesh pattern. Then we exploit this result to determine some geometric properties of~$\Sort(132)$. These ultimately lead to a bijection with the set of $12231$-avoiding \rgfs, whose enumeration is a corollary of a much more general mechanism proposed by Jel\'inek and Mansour in~\cite{JM}. We then exhibit direct combinatorial proofs for the enumeration of some patterns in the same Wilf-equivalence class as $12231$, highlighting connections with lattice paths and pattern-avoiding permutations. We enumerate two of these patterns via a bijection with a family of labeled Motzkin paths, which provide a combinatorial interpretation of a beautiful continued fraction for the related counting sequence. Finally, by putting all these pieces together, we obtain an independent proof of the enumeration of~$\Sort(132)$.

\section{\texorpdfstring{Characterization of~$\Sort(132)$}{Characterization of Sort(132)}}\label{section_mesh_132}

We start by showing a useful decomposition lemma for the~$132$-stack.

\begin{lemma}\label{lemma_ltr_min_dec_132}
Let~$\pi$ be a permutation and let~$\pi=m_1B_1m_2B_2\cdots m_tB_t$ be the ltr-min decomposition of~$\pi$. Then:
\begin{enumerate}
\item $\out{132}(\pi)=\widetilde{B_1}\widetilde{B_2}\cdots\widetilde{B_t} m_t m_{t-1}\cdots m_2 m_1$, where each~$\widetilde{B_i}$ is a suitable rearrangement of the elements of~$B_i$.
\item If~$\pi$ is~$132$-sortable, then~$x>y$ for each~$x\in B_i$, $y\in B_j$, with~$i<j$.
\end{enumerate}
\end{lemma}
\begin{proof}
\begin{enumerate}
\item For each~$x\in B_1$, $m_1xm_2\simeq 231$, thus every element of~$B_1$ has to be popped from the~$132$-stack before~$m_2$ enters. After that, we have~$m_1$ and~$m_2$ on the~$132$-stack, with~$m_1>m_2$ and~$m_2$ above~$m_1$. Note that they cannot both be part of a~$132$, therefore~$m_2$ remains in the~$132$-stack until the end of the sorting process. Similarly, each element of~$B_2$ has to be extracted before~$m_3$ enters, since~$m_3xm_2\simeq 132$ for each~$x\in B_2$. The same argument holds for every~$m_j$ with~$j\ge2$.
\item Suppose there are two elements~$x,y$ such that~$x<y$, $x\in B_i$ and~$y\in B_j$, with~$i<j$. Then, as a consequence of the previous item,~$xym_t$ is an occurrence of~$231$ in~$\out{132}(\pi)$, which is impossible since~$\pi$ is~$132$-sortable.
\end{enumerate}
\end{proof}

\begin{lemma}\label{lemma_unimodal_stack_132}
Let~$\pi\in\Sort_n(132)$ and let~$\pi=m_1B_1m_2B_2\cdots m_tB_t$ be its ltr-min decomposition. Suppose that the next element of the input is~$x\in B_i$, for some~$i$. Then the content of the~$132$-stack when read from bottom to top is:
$$
m_1m_2\cdots m_ix_1x_2\cdots x_s,
$$
where~$\{x_1,\dots,xs\}$ is a (possibly empty) subset of~$B_i$ such that~$x_1<x_2<\cdots<x_s$.
\end{lemma}
\begin{proof}
The first~$i$ ltr-minima~$m_1,\dots,m_i$ of~$\pi$ lie at the bottom of the~$132$-stack, by Lemma~\ref{lemma_ltr_min_dec_132}. Then the remaining elements~$x_1,\dots,x_s$ of~$B_i$ in the~$132$-stack must be in increasing order from bottom to top, for otherwise, if~$x_h>x_{\ell}$ for some~$h<\ell$, then~$\out{132}(\pi)$ would contain~$x_{\ell}x_hm_i\simeq 231$, contradicting the~$132$-sortability of~$\pi$.
\end{proof}

\begin{figure}
\centering
\begin{DrawPerm}
\meshBox{(0,2)}{(1,3)}
\meshBox{(2,0)}{(3,1)}
\meshBox{(2,1)}{(3,2)}
\fillPerm{1,3,2}{3.99}{3.99}
\end{DrawPerm}
\caption[The mesh pattern~$\mu$.]{The mesh pattern~$\mu=(132,\left\lbrace(0,2),(2,0),(2,1)\right\rbrace)$}\label{figure_mesh_pattern_132_bis}
\end{figure}

Next we provide a characterization of~$\Sort(132)$ in terms of pattern avoidance. For the rest of this section, denote by~$\mu$ the mesh pattern~$\mu=(132,\left\lbrace(0,2),(2,0),(2,1)\right\rbrace)$ depicted in Figure~\ref{figure_mesh_pattern_132_bis}. A permutation~$\pi$ thus contains an occurrence of~$\mu$ if~$\pi$ contains an occurrence~$acb$ of the classical pattern~$132$ such that:

\begin{itemize}
\item every element that precedes~$a$ in~$\pi$ is either smaller than~$b$ or greater than~$c$;
\item every element between~$c$ and~$b$ in~$\pi$ is greater than~$b$.
\end{itemize}

\begin{theorem}\label{theorem_mesh_patterns_necessary_132}\label{mesh_patterns_sufficient}\label{corollary_mesh_pattern_char_132}
We have:
$$
\Sort(132)=\Perm(2314,\mu).
$$
\end{theorem}
\begin{proof}
We start by showing that~$\Sort(132)\subseteq\Perm(2314,\mu)$. Let~$\pi=m_1 B_1 m_2 B_2\cdots m_t B_t$ be the ltr-min decomposition of~$\pi$. Suppose, for a contradiction, that~$\pi$ contains an occurrence~$bcad$ of~$2314$. When~$a$ enters the~$132$-stack, at least one of~$b$ and~$c$, call it~$x$, has already been popped from the~$132$-stack, otherwise we would get the forbidden pattern~$acb\simeq 132$ inside the~$132$-stack. Hence, by Lemma~\ref{lemma_ltr_min_dec_132}, $\out{132}(\pi)$ contains~$x d m_t\simeq 231$, violating the hypothesis that~$\pi$ is~$132$-sortable. Next suppose that~$acb$ is an occurrence of~$132$ in~$\pi$. We wish to show that~$acb$ is part of an occurrence of either~$\mathbf{3}142$, $24\mathbf{1}3$ or~$14\mathbf{2}3$, thus proving that~$\pi$ avoids~$\mu$. Let~$m(a)$ be the ltr-minimum immediately preceding the block that contains~$a$, or~$a$ itself if~$a$ is an ltr-minimum. Then~$m(a)\le a$ and~$m(a)$ exits the~$132$-stack after~$b$ and~$c$ (by Lemma~\ref{lemma_ltr_min_dec_132}), so~$c$ has to be popped before~$b$ enters, otherwise~$b c m(a)$ would be an occurrence of~$231$ inside~$\out{132}(\pi)$. We consider the following two cases. Note that~$a<b<c$, so~$b,c$ are not ltr-minima in~$\pi$.

\begin{itemize}
\item $c\in B_i$ and~$b\in B_j$, with~$i < j$. In this case, $m_j < m(a)\le a$, hence~$a c m_j b\simeq 2413$, which is one of the desired patterns.

\item $c$ and~$b$ are in the same block~$B_i$. First suppose there is an ltr-minimum~$m=m_\ell$, with~$\ell<i$, such that~$b<m<c$; then~$m>m(a)$, so~$m$ precedes~$m(a)$ in~$\pi$ and~$macb\simeq 3142$, again one of the listed patterns. Otherwise, suppose that, for every ltr-minimum~$m$, either~$m<b$ or~$m>c$ and consider the element~$w$ that immediately precedes~$b$ in~$\pi$. We wish to show that~$w<b$, which will conclude the proof. Suppose, for a contradiction, that~$w>b$ and let~$x_1,x_2,\dots,x_s=w$ be the elements on the~$132$-stack, after~$w$ has been pushed, that are not ltr-minima when we read from bottom to top. By Lemma~\ref{lemma_unimodal_stack_132}, we have~$x_1<x_2<\cdots<x_s$; moreover~$x_s=w > b$, so there is a minimum index~$u$ such that~$x_u>b$. Now observe that, for~$\ell>u$, all the elements~$x_\ell$ are popped from the~$132$-stack before~$b$ enters, because~$b x_\ell x_u\simeq 132$. We also observe that necessarily~$x_u\le c$, otherwise~$c$ would already have been popped and~$\out{132}(\pi)$ would contain the pattern~$c x_u m(a)\simeq 231$. We can now assert that~$b$ is pushed onto the~$132$-stack immediately above~$x_u$. In fact, $x_\ell < b$ for every~$\ell < u$; moreover, our hypothesis implies that either~$m<b$ or~$m>c$ for every ltr-minimum~$m$ inside the~$132$-stack, therefore~$b$ cannot be the first element of an occurrence of~$231$ (read from top to bottom) that involves elements inside the~$132$-stack. However this results in an occurrence~$b x_u m(a)$ of~$231$ in~$\out{132}{(\pi)}$, which again contradicts the hypothesis that~$\pi$ is~$132$-sortable.
\end{itemize}
We have thus proved that~$\Sort(132)\subseteq\Perm(2314,\mu)$. Next we show the opposite inclusion~$\Perm(2314,\mu)\subseteq\Sort(132)$. Let~$\pi\in\Perm(2314,\mu)$. Suppose, for a contradiction, that~$\pi$ is not~$132$-sortable, that is, $\out{132}(\pi)$ contains an occurrence~$bca$ of~$231$. Let again~$\pi=m_1B_1m_2B_2\cdots m_tB_t$ be the ltr-min decomposition of~$\pi$. By Lemma~\ref{lemma_ltr_min_dec_132}, we have
$$
\out{132}(\pi)=\widetilde{B_1}\widetilde{B_2}\cdots\widetilde{B_t} m_t m_{t-1}\cdots m_2 m_1.
$$
Since the ltr-minima are popped from the~$132$-stack in increasing order, neither~$b$ nor~$c$ can be an ltr-minimum. Suppose that~$b\in B_i$ and~$c\in B_j$, for some~$i\le j$. If~$i<j$, then~$m_i b m_j c\simeq 2314$, which is forbidden. Suppose instead that~$i=j$ and consider the leftmost ascent~$x<y$ in~$\widetilde{B_i}$ (indeed there is at least one ascent in~$\widetilde{B_i}$, since the elements~$b,c$ constitute a noninversion in~$\widetilde{B_i}$). There are two possibilities.

\begin{itemize}
\item If~$y$ comes after~$x$ in~$\pi$ then~$x$ has to be popped before~$y$ is pushed onto the~$132$-stack. Therefore, when~$x$ is popped, there are two elements~$u,v$ in the~$132$-stack, with~$v$ above~$u$, such that~$u v w\simeq 231$, where~$w$ is the next element of the input. If~$v\neq x$, then also~$v$ is popped after~$x$ (for the same reason), but this is a contradiction with the fact that~$x$ and~$y$ constitute an ascent in~$\widetilde{B_i}$. Thus we have~$v=x$ and~$u x w\simeq 231$, which implies that~$w\neq y$ and~$u x w y\simeq 2314$ in~$\pi$, contradicting the assumption that~$\pi$ avoids~$2314$.

\item Suppose instead that~$y$ precedes~$x$ in~$\pi$. Observe that~$y$ has to be on the~$132$-stack when~$x$ enters, because~$\out{132}(\pi)$ contains the ascent~$(x,y)$ (this fact will be frequently used in the sequel). In this situation, $m_i y x$ is an occurrence of~$132$ in~$\pi$. We now show that either~$m_i y x$ is an occurrence of~$\mu$ or~$\pi$ contains~$2314$. If there is an element~$z$ that precedes~$m_i$ in~$\pi$ such that~$x<z<y$ (so that~$z m_i y x\simeq 3142$), then~$z$ cannot be an ltr-minimum. In such a case, in fact, by Lemma~\ref{lemma_ltr_min_dec_132}, $z$ would be in the~$132$-stack below~$y$ when~$x$ is pushed, but~$xyz\simeq 132$, which is impossible due to the restriction of the~$132$-stack. Instead, if~$z\in B_\ell$ for some~$\ell<i$, then~$m_\ell z m_i y\simeq 2314$. Therefore we can assume that every element that precedes~$m_i$ in~$\pi$ is either smaller than~$x$ or greater than~$y$. Finally, suppose that there is an element~$z$ between~$y$ and~$x$ in~$\pi$ such that~$z<x$, which gives an occurrence~$m_i y z x$ of either~$2413$ or~$1423$. Then, since~$y$ is still in the~$132$-stack when~$x$ is pushed and~$z$ precedes~$x$ in~$\pi$, $z$ enters the~$132$-stack above~$y$, and so~$\widetilde{B_I}$ contains either~$x\dots z\dots y$ or~$z\dots x\dots y$, with~$z<x$. However, both cases give a contradiction, because~$(x,y)$ is the first ascent in~$\out{132}(\pi)$.
\end{itemize}
\end{proof}

Due to the presence of the mesh pattern~$\mu$ (and in accordance with Theorem~\ref{theorem_necess_cond_class}), the set~$\Perm( 2314,\mu)$ is not a permutation class. For instance, the~$132$-sortable permutation~$2413$ contains the non~$132$-sortable pattern~$132$.

\section{\texorpdfstring{A grid decomposition for~$\Sort(132)$}{A grid decomposition for Sort(132)}}\label{section_grid_132}

In the previous section we have proved that~$\Sort(132)=\Perm(2314,\mu)$, obtaining a precise description (in terms of generalized pattern avoidance) of~$132$-sortable permutations. However, we are still not able to enumerate~$\Sort(132)$ directly. In this section we will thus investigate its geometric structure by refining the ltr-minima decomposition as follows. Let~$\pi$ be a permutation of length~$n$ with~$t$ ltr-minima and let~$\pi=m_1B_1m_2B_2\cdots m_tB_t$ its ltr-min decomposition. Then:

\begin{itemize}
\item for~$j\ge 1$, the~$j$-th \textit{vertical strip} of~$\pi$ is~$B_j$;
\item for~$i\ge 1$, the~$i$-th \textit{horizontal strip} of~$\pi$ is~$H_i=\left\lbrace x\in\pi:m_i<x< m_{i-1}\right\rbrace$, where~$m_0=+\infty$.
\item for any two indices~$i,j$, the \textit{cell} of indices~$i,j$ of~$\pi$ is~$C_{i,j}=H_i\cap B_j$ (note that~$C_{i,j}$ is empty when~$i>j$).
\item the \textit{core} of~$\pi$ is~$\core(\pi)=B_1B_2\cdots B_k$, obtained from~$\pi$ by removing the ltr-minima.
\end{itemize}

From now on, the content of each~$B_j,H_i,C_{i,j}$ will be regarded as a permutation. As an example, consider the permutation~$\pi=13\,14\,15\,10\,12\,6\,7\,8\,11\,9\,3\,1\,4\,5\ 2$. Then (see Figure~\ref{figure_grid_dec_132}):

\begin{itemize}
\item $\pi$ has six ltr-minima, namely~$13,10,6,3,1$;
\item the vertical strips of~$\pi$ are~$B_1=14\,15\simeq 1\,2$, $B_2=12\simeq 1$, $B_3=7\,8\,11\,9\simeq 1\,2\,4\,3$, $B_4=\emptyset$ and~$B_5=4\,5\,2\simeq 2\,3\,1$;
\item the horizontal strips of~$\pi$ are~$H_1=14\,15\simeq 1\,2$, $H_2= 12\,11\simeq 2\,1$, $H_3= 7\,8\,9\simeq 1\,2\,3$, $H_4 = 4\,5\simeq 1\,2$ and~$H_5 = 2\simeq 1$;
\item the nonempty cells of~$\pi$ are~$C_{1,1}= 14\,15\simeq 1\,2$, $C_{2,2}=12\simeq 1$, $C_{2,3}=11\simeq 1$, $C_{3,3}=7\,8\,9\simeq 1\,2\,3$, $C_{4,5}=4\,5\simeq 1\,2$ and~$C_{5,5}= 2\simeq 1$;
\item the core of~$\pi$ is~$\core(\pi)=14\,15\,12\,7\,8\,11\,9\,4\,5\,2\simeq 9\ 10\,8\,4\,5\,7\,6\,2\,3\,1$.
\end{itemize}

The terminology introduced above refers to the graphical representation of~$\pi$, as illustrated in Figure~\ref{figure_grid_dec_132}.

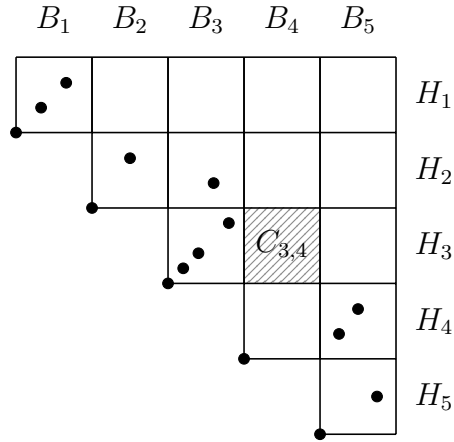
\begin{figure}
\centering
\begin{tikzpicture}[scale=1, baseline=20.5pt]
\fill[NE-lines] (3,2) rectangle (4,3);
\node at (3.5,2.5) {$C_{3,4}$};
\draw [semithick] (4,0) grid (5,5);
\draw [semithick] (3,1) grid (4,5);
\draw [semithick] (2,2) grid (3,5);
\draw [semithick] (1,3) grid (2,5);
\draw [semithick] (0,4) grid (1,5);
\filldraw (0,4) circle (2pt);
\filldraw (0.33,4.33) circle (2pt);
\filldraw (0.66,4.66) circle (2pt);
\filldraw (1,3) circle (2pt);
\filldraw (1.5,3.66) circle (2pt);
\filldraw (2,2) circle (2pt);
\filldraw (2.2,2.2) circle (2pt);
\filldraw (2.4,2.4) circle (2pt);
\filldraw (2.6,3.33) circle (2pt);
\filldraw (2.8,2.8) circle (2pt);
\filldraw (3,1) circle (2pt);
\filldraw (4,0) circle (2pt);
\filldraw (4.25,1.33) circle (2pt);
\filldraw (4.5,1.66) circle (2pt);
\filldraw (4.75,0.5) circle (2pt);
\node[scale=1] at (0.5,5.5) {$B_1$};
\node[scale=1] at (1.5,5.5) {$B_2$};
\node[scale=1] at (2.5,5.5) {$B_3$};
\node[scale=1] at (3.5,5.5) {$B_4$};
\node[scale=1] at (4.5,5.5) {$B_5$};
\node[scale=1] at (5.5,4.5) {$H_1$};
\node[scale=1] at (5.5,3.5) {$H_2$};
\node[scale=1] at (5.5,2.5) {$H_3$};
\node[scale=1] at (5.5,1.5) {$H_4$};
\node[scale=1] at (5.5,0.5) {$H_5$};
\end{tikzpicture}
\caption[The grid decomposition of~$132$-sortable permutations.]{The grid decomposition of~$\pi=13\,14\,15\,10\,12\,6\,7\,8\,11\,9\,3\,1\,4\,5\,2$. The image of~$\pi$ under the bijection of Theorem~\ref{theorem_bij_12231} is the {\rgf}~$\eta(\pi)=111223332345445$.}\label{figure_grid_dec_132}
\end{figure}

In what follows we prove that the requirement of being~$132$-sortable imposes precise constraints on the grid structure of a permutation: both the content of strips and cells and the relative position of non-empty cells are affected.

\begin{lemma}\label{lemma_no_switch_component_132}
Let~$\pi$ be a~$132$-sortable permutation and suppose that the cell~$C_{i,j}$ is nonempty, for some~$i,j$. Then the cell~$C_{u,v}$ is empty for each pair of indices~$(u,v)$ such that~$u<i$ and~$v>j$\footnote{That is when~$C_{u,v}$ is strictly northeast of~$C_{i,j}$ (see Figure~\ref{figure_cells_132}).}.
\end{lemma}
\begin{proof}
Suppose there are two elements~$x\in C_{i,j}$ and~$y\in C_{u,v}$ such that~$u<i$ and~$v>j$. Then~$m_ixm_vy\simeq 2314$, which is impossible due to Theorem~\ref{theorem_mesh_patterns_necessary_132}.
\end{proof}

\begin{lemma}\label{lemma_inversion_in_a_cell_132}
Let~$\pi$ be a~$132$-sortable permutation and suppose that the cell~$C_{i,j}$ contains an inversion~$x>y$, where~$x$ precedes~$y$ in~$C_{i,j}$. Then there is an element~$z$ between~$x$ and~$y$ in~$\pi$ such that~$z<m_i$.
\end{lemma}
\begin{proof}
We refer to Figure~\ref{figure_cells_132} for a description of the statement of the lemma. For~$x$ and~$y$ as above, we have~$m_ixy\simeq 132$. In particular, $x$ and~$y$ are in the same cell~$C_{i,j}$ and~$m_i$ is the corresponding ltr-minimum, hence every element~$w$ preceding~$m_i$ in~$\pi$ is greater than~$x$ (because~$w>m_{i-1}$ and~$x<m_{i-1}$). Therefore, as a consequence of Theorem~\ref{theorem_mesh_patterns_necessary_132}, there exists an element~$z$ between~$x$ and~$y$ in~$\pi$ such that~$z<y$. If~$z<m_i$, then we are done. Otherwise, if~$z>m_i$, we can repeat the same argument using the occurrence~$m_ixz$ of~$132$, in which we have replaced~$y$ with the element~$z$ that comes strictly before~$y$ in~$\pi$; continuing in this way we eventually find an element of~$\pi$ with the desired property.
\end{proof}

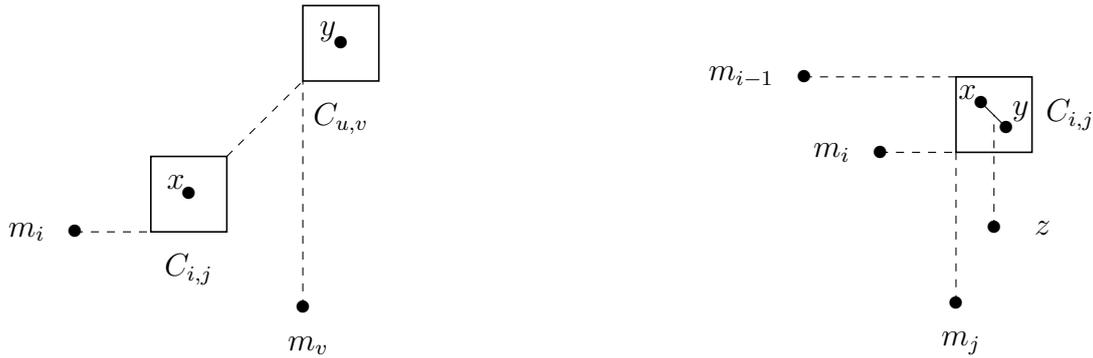
\begin{figure}
\centering
\begin{tikzpicture}
\draw [semithick] (0,0) rectangle (1,1);
\draw [semithick] (2,2) rectangle (3,3);
\draw [dashed] (-1,0) -- (0,0);
\draw [dashed] (2,-1) -- (2,2);
\draw [dashed] (1,1) -- (2,2);
\node [label=left:~$m_i$] at (-1,0) {$\bullet$};
\node [label=below:~$m_v$] at (2,-1) {$\bullet$};
\node [] at (0.5,-0.5) {$C_{i,j}$};
\node [] at (2.5,1.5) {$C_{u,v}$};
\node [] at (0.33,0.66) {$x$};
\node [] at (2.33,2.66) {$y$};
\node [] at (0.5,0.5) {$\bullet$};
\node [] at (2.5,2.5) {$\bullet$};
\end{tikzpicture}
\hfill
\begin{tikzpicture}
\draw [semithick] (1,1) rectangle (2,2);
\draw [dashed] (1,-1) -- (1,1);
\draw [dashed] (0,1) -- (1,1);
\draw [dashed] (-1,2) -- (1,2);
\draw [dashed] (1.5,0) -- (1.5,1.5);
\draw [thin] (1.33,1.66) -- (1.66,1.33);
\node [label=below:~$m_j$] at (1,-1) {$\bullet$};
\node [label=left:~$m_i$] at (0,1) {$\bullet$};
\node [label=left:~$m_{i-1}$] at (-1,2) {$\bullet$};
\node [] at (2.5,1.5) {$C_{i,j}$};
\node [] at (1.33,1.66) {$\bullet$};
\node [] at (1.66,1.33) {$\bullet$};
\node [] at (1.15,1.8) {$x$};
\node [] at (1.85,1.55) {$y$};
\node [label=right:~$z$] at (1.5,0) {$\bullet$};
\end{tikzpicture}
\caption[The geometric constructions of Lemma~\ref{lemma_no_switch_component_132} and~\ref{lemma_inversion_in_a_cell_132}.]{The constructions of Lemma~\ref{lemma_no_switch_component_132}, on the left, and of Lemma~\ref{lemma_inversion_in_a_cell_132}, on the right.}\label{figure_cells_132}
\end{figure}

\begin{proposition}\label{proposition_layered_cells_132}
If~$\pi$ is~$132$-sortable, then~$C_{i,j}\in\Perm(132,213)$, for every~$i,j$.
\end{proposition}
\begin{proof}
Suppose that~$C_{i,j}$ contains an occurrence~$acb$ of~$132$. By Lemma~\ref{lemma_inversion_in_a_cell_132}, there exists an element~$z$ between~$c$ and~$b$ in~$\pi$ such that~$z<m_i$. In particular, $m_iazb\simeq 2314$, which is a contradiction since~$\pi$ is~$132$-sortable (by Theorem~\ref{theorem_mesh_patterns_necessary_132}). On the other hand, if~$C_{i,j}$ contains an occurrence~$bac$ of~$213$, then~$(b,a)$ is an inversion in the cell~$C_{i,j}$ and therefore, again by Lemma~\ref{lemma_inversion_in_a_cell_132}, there is an element~$z$ between~$b$ and~$a$ in~$\pi$ with~$z<m_i$ and~$m_ibzc\simeq 2314$, a contradiction.
\end{proof}

\begin{proposition}\label{proposition_layered_H_strips_132}
If~$\pi$ is a~$132$-sortable permutation, then~$H_i\in\Perm(132,213)$, for every~$i$.
\end{proposition}
\begin{proof}
This is a consequence of Lemma~\ref{lemma_ltr_min_dec_132} and Proposition~\ref{proposition_layered_cells_132}.
\end{proof}

\begin{proposition}\label{proposition_core_avoids_213_132}
If~$\pi$ is~$132$-sortable, then~$\core(\pi)\in\Perm(213)$.
\end{proposition}
\begin{proof}
Suppose that~$\pi$ contains an occurrence~$bac$ of~$213$ that does not involve any ltr-minimum and suppose that~$b\in C_{i,j}$ for some~$i,j$. Note that~$b<c$, so, by Lemma~\ref{lemma_ltr_min_dec_132}, $b$ and~$c$ must belong to the same vertical strip~$B_j$. Now, if~$a\in C_{\ell,j}$, with~$\ell>i$, then~$m_ibac\simeq 2314$, which is a contradiction, since~$\pi$ is~$132$-sortable. Therefore we must have~$a\in C_{i,j}$. This results in an occurrence~$m_iba$ of~$132$, with~$b$ and~$a$ both in the cell~$C_{i,j}$; thus, by Lemma~\ref{lemma_inversion_in_a_cell_132}, there is an element~$z$ between~$b$ and~$a$ in~$\pi$ such that~$z<m_i$ and~$m_ibzc\simeq 2314$, which is again a contradiction.
\end{proof}

The results proved so far in this section provide necessary conditions that a permutation has to satisfy in order to be~$132$-sortable. Now, since the prefix of a~$\sigma$-sortable permutation is always~$\sigma$-sortable (see Lemma~\ref{lemma_prefix_sortable}), if we remove the last element from a~$132$-sortable permutation we get another~$132$-sortable permutation of length one less. Equivalently, every~$132$-sortable permutation is obtained from a~$132$-sortable permutation (of length one less) by inserting a new rightmost element, and suitably rescaling the remaining ones. Our next goal is to understand which integers are allowed for such an insertion, so to obtain a recursive construction for the set~$\Sort(132)$. For example, since the insertion of a new minimum can never create either~$2314$ or~$\mu$, by Theorem~\ref{corollary_mesh_pattern_char_132}, such an insertion is always allowed. In all the other cases, we need to satisfy the requirements of Lemma~\ref{lemma_no_switch_component_132} and Propositions~\ref{proposition_layered_H_strips_132} and~\ref{proposition_core_avoids_213_132}.

Let~$\pi$ be a~$132$-sortable permutation with~$t$ ltr-minima. Suppose we insert a new rightmost element in a cell~$C_{i,t}$ of the last vertical strip. By Proposition~\ref{proposition_layered_H_strips_132}, any horizontal strip~$H_i$ in a~$132$-sortable permutation avoids both~$132$ and~$213$, that is~$H_i$ is co-layered. Therefore, if we wish to obtain a new co-layered permutation by inserting a new rightmost element, there are exactly two possibilities:

\begin{enumerate}
\item $\Do$: to insert a new minimum in~$C_{i,t}$ (which is also a new minimum of the horizontal strip~$H_i$);
\item $\Co$: to create a consecutive ascent\footnote{recall that an ascent~$\pi_i<\pi_{i+1}$ is consecutive if~$\pi_{i+1}=\pi_i+1$.} in the two final positions of~$C_{i,t}$.
\end{enumerate}

We formalize this construction by introducing the notion of active cell. Let~$\pi$ be a~$132$-sortable permutation with~$t$ ltr-minima. For~$i\ge 1$, the cell~$C_{i,t}$ is said to be \textit{active} if both the following conditions are satisfied:
\begin{itemize}
\item[(i)]~$C_{u,v}$ is empty for each~$u,v$ such that~$u>i$ and~$v<t$;
\item[(ii)] inserting a new rightmost element according to~$\Do$ does not create an occurrence of~$213$ in~$\core(\pi)$. 
\end{itemize}

Thanks to condition (i), we can equivalently express condition (ii) by saying that the permutation~$\bigcup_{j\ge i+1} C_{j,t}$ is increasing. Moreover, as a consequence of Lemma~\ref{lemma_no_switch_component_132} and
Proposition~\ref{proposition_core_avoids_213_132}, if we insert a new rightmost element in a cell~$C_{i,t}$ that is not active, then we get a non~$132$-sortable permutation. Othwerise, if~$C_{i,t}$ is active, we wish to show that exactly one of the operations~$\Do$ and~$\Co$ returns a~$132$-sortable permutation. Let us consider two cases, according to whether~$C_{i,t}$ is empty or not.

\begin{proposition}\label{proposition_cell_legal_insertion_132}
Let~$\pi=\pi_1\cdots\pi_n$ be a~$132$-sortable permutation with~$t$ ltr-minima and let~$C_{i,t}=\gamma_1\cdots\gamma_k$ be a nonempty active cell of~$\pi$. Let~$x=\pi_n$ and suppose~$x\in C_{\ell,t}$. Then:
\begin{enumerate}
\item by performing~$\Do$ on~$C_{i,t}$ we get a~$132$-sortable permutation~$\pi'$ if and only if $\ell>i$;
\item by performing~$\Co$ on~$C_{i,t}$ we get a~$132$-sortable permutation~$\pi'$ if and only if~$\ell\le i$.
\end{enumerate}
\end{proposition}
\begin{proof}
\begin{enumerate}
\item Suppose that~$\ell<i$ and we want to insert a new rightmost element~$\gamma_{k+1}$ into~$C_{i,t}$ according to~$\Do$. Assume, for a contradiction, that the resulting permutation~$\pi'$ is~$132$-sortable. The elements~$\gamma_k$ and~$\gamma_{k+1}$ form an inversion in~$C_{i,t}$, so by Lemma~\ref{lemma_inversion_in_a_cell_132} there exists an element~$z$ between~$\gamma_k$ and~$\gamma_{k+1}$ in~$\pi$ such that~$z<m_i$. Hence~$m_i\gamma_kzx\simeq 2314$, which contradicts the assumption that~$\pi$ is~$132$-sortable. Instead, if~$\ell=i$, that is, $\gamma_k=x=\pi_n$, then~$\gamma_k\gamma_{k+1}$ is an inversion inside~$C_{i,t}$ such that~$\gamma_k$ and~$\gamma_{k+1}$ are adjacent in~$\pi$. This implies that~$\pi$ is not~$132$-sortable (again as a consequence of Lemma~\ref{lemma_inversion_in_a_cell_132}).

Conversely, suppose that~$\ell>i$ and~$\gamma_{k+1}$ is inserted into~$C_{i,t}$ according to~$\Do$. By Theorem~\ref{theorem_mesh_patterns_necessary_132}, $\pi\in\Perm(2314,\mu)$, so we just have to show that the permutation~$\pi'$ obtained after the insertion still avoids the two forbidden patterns. If~$\gamma_{k+1}$ plays the role of the~$2$ in an occurrence of~$132$, say~$ac\gamma_{k+1}$, then we have either~$acx\gamma_{k+1}\simeq 1423$ or~$acx\gamma_{k+1}\simeq 2413$, which means that the selected occurrence of~$132$ is not an occurrence of the mesh pattern~$\mu$. Otherwise, suppose there is an occurrence~$bca\gamma_{k+1}$ of~$2314$ in~$\pi'$. If~$m_t=1$ precedes~$c$ in~$\pi$, then~$ca\gamma_k\simeq 213$ in~$\core(\pi)$, contradicting Proposition~\ref{proposition_core_avoids_213_132}. On the other hand, if~$m_t$ follows~$c$ in~$\pi$, then~$c\in B_j$, for some~$j<t$, and~$\gamma_k\in B_t$, with~$c<\gamma_k$, contradicting Lemma~\ref{lemma_ltr_min_dec_132}.

\item Suppose we insert~$\gamma_{k+1}$ into~$C_{i,t}$ according to~$\Co$ and~$\ell>i$. Then~$\gamma_kx\gamma_{k+1}$ is an occurrence of~$213$ in~$\core(\pi')$, hence~$\pi'$ is not~$132$-sortable, due to Proposition~\ref{proposition_core_avoids_213_132}, as desired.

Conversely, suppose that~$\ell<i$ and we insert~$\gamma_{k+1}$ into~$C_{i,t}$ according to~$\Co$; this means that~$\gamma_{k+1}=\gamma_k+1$. The resulting permutation~$\pi'$ does not contain an occurrence~$bcad$ of~$2314$ with~$\gamma_{k+1}=d$, for otherwise~$bcax$ would be an occurrence of~$2314$ in~$\pi$, contradicting the hypothesis that~$\pi$ is~$132$-sortable. On the other hand, suppose there are two elements~$a,c$ in~$\pi$ such that~$ac\gamma_{k+1}$ is an occurrence of~$132$. We now prove that~$ac\gamma_{k+1}$ is not an occurrence of the mesh pattern~$\mu$ by distinguishing two cases. 

If~$c>m_{i-1}$ (note that~$i>\ell$, so~$m_{i-1}$ exists), then~$a<\gamma_{k+1}<m_{i-1}$, so~$m_{i-1}$ precedes~$a$ in~$\pi$ (because~$a<m_{i-1}$ and~$m_{i-1}$ is an ltr-minimum) and~$m_{i-1}ac\gamma_{k+1}$ would be an occurrence of~$3142$.

Instead, if~$c<m_{i-1}$, then~$c$ is not an ltr-minimum, because~$a<c$ precedes~$c$; moreover, $c$ is in~$C_{i,t}$, since~$c<m_{i-1}$ and~$c>\gamma_{k+1}$, hence~$c\gamma_kx$ is an occurrence of~$213$ in~$\core(\pi)$, which is impossible due to Proposition~\ref{proposition_core_avoids_213_132}.

Finally, if~$\ell=i$, then~$x=\gamma_k$, $\gamma_{k+1}=\gamma_k+1$ and they are adjacent in~$\pi'$, so~$\gamma_{k+1}$ is neither part of an occurrence of~$2314$ nor of~$\mu$, since otherwise~$\gamma_k$ would be as well, contradicting the hypothesis that~$\pi$ is~$132$-sortable.
\end{enumerate}
\end{proof}

When~$C_{i,t}$ is empty, the only possibility is to try to perform~$\Do$ (since~$\Co$ does not make sense). Next we show that this is always allowed.

\begin{proposition}\label{proposition_empty_cell_132}
Let~$\pi=\pi_1\cdots\pi_n$ be a~$132$-sortable permutation with~$t$ ltr-minima and let~$C_{i,t}$ be an empty active cell of~$\pi$. Let~$\pi'$ be the permutation obtained from~$\pi$ by inserting a new rightmost element~$y$ in~$C_{i,t}$ according to~$\Do$. Then~$\pi'$ is~$132$-sortable.
\end{proposition}
\begin{proof}
By Theorem~\ref{theorem_mesh_patterns_necessary_132} we have that~$\pi\in\Perm(2314,\mu)$ and we want to prove that~$\pi'\in\Perm(2314,\mu)$ as well. Suppose there are three elements~$b,c,a$ in~$\pi$ such that~$bcay\simeq 2314$. Since~$c>b$, the element~$c$ is not an ltr-minimum of~$\pi$. Suppose that~$c\in C_{u,v}$, for some~$u,v$. If~$a$ is an ltr-minimum, then of course~$v<t$, and we have also~$u>i$, because~$y$ is the minimum of its horizontal strip and~$y>c$. This would imply that~$C_{u,v}$ is a nonempty cell, with~$u>i$ and~$v<t$, which is impossible since~$C_{i,t}$ is active. Otherwise, if~$a$ is not an ltr-minimum, then~$cay\simeq 213$ in~$\core(\pi')$, which again contradicts the assumption that~$C_{i,t}$ is active.

Next, in order to prove that~$\pi'$ does not contain the mesh pattern~$\mu$, suppose there are two elements~$a,c$ in~$\pi$ such that~$acy\simeq 132$ and suppose~$c\in B_j$, for some~$j\le t$. If~$j<t$, then~$acm_ty$ is an occurrence of~$2413$, as desired. Otherwise, if~$j=t$, we have that~$c\in C_{\ell,t}$, for some~$\ell<t$, because~$C_{i,t}$ is empty before we insert~$y$; moreover, $m_\ell$ precedes~$a$ in~$\pi$, because~$m_\ell>y$ and~$a<y$. Thus~$m_\ell acy\simeq 3142$, as desired.
\end{proof}

\begin{corollary}\label{corollary_recursive_construction_132}
Let~$\pi$ be a~$132$-sortable permutation. Then, for every active cell of~$\pi$, exactly one of~$\Do$ and~$\Co$ generates a~$132$-sortable permutation.
\end{corollary}

As a consequence of Propositions~\ref{proposition_cell_legal_insertion_132} and~\ref{proposition_empty_cell_132}, every~$132$-sortable permutation can be constructed inductively by repeatedly inserting a new rightmost element either as a new minimum or by performing~$\Do$ and~$\Co$, according to the rules of Propositions~\ref{proposition_cell_legal_insertion_132}. In particular, given a~$132$-sortable permutation~$\pi$ with~$k$ active cells, then~$k+1$ $132$-sortable permutations are produced this way (one for each active cell and one when the new minimum is inserted). Using the generating tree terminology, these are the \textit{children} of~$\pi$. Understanding the distribution of active cells of~$132$-sortable permutations would lead to a generating tree for~$\Sort(132)$, which could be used directly to find its enumeration. So far we were not able to fulfill this task, which is left as an open problem.

\begin{openproblem}
Given~$n\ge 1$ and~$k\ge 0$, compute the number of~$132$-sortable permutations with~$k$ active cells. Moreover, given a~$132$-sortable permutation~$\pi$ with~$k$ active cells, compute the number of active cells of each child of~$\pi$.
\end{openproblem}

Instead of using the generating tree approach, we wish to exploit the grid structure of~$132$-sortable permutations in order to determine a bijection with a class of pattern-avoiding {\rgfs}, ultimately obtaining the desired enumeration of~$\Sort(132)$.

Let~$\pi=\pi_1\cdots\pi_n$ be a permutation with~$t$ ltr-minima~$m_1,\dots,m_t$ and set~$m_0=+\infty$. Define the map~$\eta$ by setting~$\eta(\pi)=r_1\cdots r_n$, where~$r_i=j$ if~$m_{j}\le\pi_i < m_{j-1}$. An alternative description of~$\eta(\pi)$ is the following: scan the permutation~$\pi$ from left to right and record the index of the horizontal strip that contains the current element, including the ltr-minima in the corresponding strips. An example of this construction is illustrated in Figure~\ref{figure_grid_dec_132}. It is easy to realize that~$\eta$ is defined for any permutation and that~$\eta(\pi)$ is a {\rgf}. The next theorem asserts that if we restrict to~$132$-sortable permutations, then~$\eta$ is a bijection between~$\Sort_n(132)$ and~$\RGF_n(12231)$. First a useful lemma concerning pattern avoidance on {\rgfs}.

\begin{lemma}\label{lemma_RGF_prop}
Let~$w=w_1w_2\cdots w_k$ be a sequence of positive integers. Let~$w'=\std(w)$ be the standardization\footnote{Recall that~$w'$ is obtained by replacing all the occurrences of the smallest integer of~$w$ with~$1$, all the occurrences of the second smallest integer with~$2$ and so on.} of~$w$ and suppose that~$w'_1=k$, for some~$k\ge 1$. Let~$R$ be a {\rgf}. Then~$w'\le R$ if and only if~$12\dots(k-1)w'\le R$.
\end{lemma}

\begin{theorem}\label{theorem_bij_12231}
The map~$\eta$ defined above is injective and the image of~$\Sort_n(132)$ through~$\eta$ is~$\RGF(12231)$.
\end{theorem}
\begin{proof}
By Lemma~\ref{lemma_RGF_prop}, we have~$\RGF(12231)=\RGF(2231)$. We start by proving that, for each~$132$-sortable permutation~$\pi$, $\eta(\pi)$ avoids~$2231$. Suppose, on the contrary, that~$\eta(\pi)$ contains an occurrence~$r_{i_1}r_{i_2}r_{i_3}r_{i_4}$ of~$2231$. Consider the leftmost occurrence~$r_j$ of the integer~$r_{i_1}$ in~$\pi$ (note that~$j\le i_1$). Then~$r_j$ corresponds through~$\eta$ to the ltr-minimum of the horizontal strip of index~$r_{i_1}$ in~$\pi$. Hence the elements~$\pi_j\pi_{i_2}\pi_{i_3}\pi_{i_4}$ form an occurrence of~$2314$ in~$\pi$\footnote{Note that the value order between elements of~$\pi$ coded by distinct values in~$\eta(\pi)$ is the reverse of their order in~$\eta(\pi)$.}, which contradicts Theorem~\ref{theorem_mesh_patterns_necessary_132}.

That~$\eta$ is injective on~$\Sort_n(132)$ is a consequence of Corollary~\ref{corollary_recursive_construction_132}. Moreover, using the construction of Proposition~\ref{proposition_cell_legal_insertion_132}, we will show that~$\eta(\Sort_n(132))=\RGF(2231)$. Given a {\rgf}~$R=r_1r_2\cdots r_n$, construct the permutation~$\pi_R$ by scanning~$R$ from left to right and, when the current element is~$r_\ell$, insert a new rightmost element~$\pi_\ell$ in the following way (suitably rescaling the previous elements when necessary):

\begin{itemize}
\item when~$r_\ell$ is the first occurrence of an integer in~$R$ then~$\pi_\ell =1$;
\item otherwise, $\pi_\ell$ is inserted in the horizontal strip~$H_{r_\ell}$, according to the rules of Proposition~\ref{proposition_cell_legal_insertion_132}.
\end{itemize}

We now wish to prove that, if the {\rgf}~$R$ avoids~$2231$, then~$\pi_R$ is a~$132$-sortable permutation such that~$\eta(\pi_R)=R$. It is easy to see that~$\eta(\pi_R)=R$, as a direct consequence of the definition of~$\eta$. Since insertions inside active cells are always allowed, what remains to be shown is that each element is in fact inserted into an active cell. We now argue by contradiction, and suppose that~$y$ is the first element that is inserted into a nonactive cell~$C_{i,j}$. According to the definition of an active cell, there are two cases to consider.

\begin{enumerate}
\item If there exists a nonempty cell~$C_{u,v}$, with~$u>i$ and~$v<j$, then, given any~$x\in C_{u,v}$, the elements of~$R$ corresponding to~$m_uxm_jy$ form an occurrence of~$2231$, which is forbidden.
\item Suppose that inserting a new rightmost element according to~$\Do$ creates an occurrence~$bay$ of~$213$ that does not involve any ltr-minima. Let~$H_u$ be the horizontal strip that contains~$b$ and let~$H_v$ be the horizontal strip that contains~$a$. Note that~$v\ge u>i$. If~$v>u$, then the elements corresponding to~$m_ubay$ in~$R$ form an occurrence of~$2231$, which is again a contradiction. On the other hand, if~$v=u$, then~$a$ belongs to the same horizontal strip of~$b$, so, since~$a<b$, $a$ was inserted according to~$\Do$. Therefore, by Proposition~\ref{proposition_cell_legal_insertion_132} and our choice of~$y$, the element~$a'$ that precedes~$a$ in~$\core(\pi)$ belongs to~$H_w$, for some~$w>u$. As a consequence, the elements~$m_uba'c$ correspond to an occurrence of~$2231$ in~$R$, which is impossible.
\end{enumerate}
\end{proof}

\begin{corollary}\label{corollary_enumeration_bij_132}
For every natural number~$n$, we have:
$$
|\Sort_n(132)| = |\RGF_n(12231)|.
$$
\end{corollary}

The enumeration of~$\RGF(12231)$ is an immediate consequence of the results proved in~\cite{JM}, where the authors determine the Wilf-equivalence class of~$12231$ (see Table~\ref{table_wilf_class_12231}). Amongst the Wilf-equivalent patterns, $12332$ can be easily enumerated. Indeed~$1221$-avoiding {\rgfs} are enumerated by the Catalan numbers (see again~\cite{JM}). Moreover, as a result of Theorem 31 in~\cite{JM}, we immediately obtain that:
$$
|\Sort_n(132)|=\sum_{k=0}^{n-1}\binom{n-1}{k}\mathfrak{c}_k,
$$
that is sequence A007317 in~\cite{Sl}.

\begin{table}
\centering
\begin{tabular}{lcr}
\toprule
\textbf{Pattern}~$p$ & \textbf{Formula} & \textbf{OEIS}\\
\midrule
12123, 12132, 12134, 12213 &  & \\
12231, 12234, 12312, 12321 & $\displaystyle{|\RGF_n(p)|=\sum_{k=0}^{n-1}\binom{n-1}{k}\mathfrak{c}_k}$ & A007317\\
12323, 12331, 12332 & & \\
\bottomrule
\end{tabular}
\caption[A Wilf-class of pattern-avoiding {\rgf}s enumerated by the binomial transform of Catalan numbers.]{The eleven patterns of the Wilf-class containing~$12231$.}\label{table_wilf_class_12231}
\end{table}

\section{Combinatorial proofs for pattern-avoiding restricted growth functions}\label{section_enum_132}

The problem of enumerating~$132$-sortable permutations has been solved in the previous section by means of a bijection~$\eta$ between~$\Sort(132)$ and~$\RGF(12231)$. The enumeration of~$\RGF(12231)$ is a corollary of the (much more general) theory developed by Jel\'{\i}nek and Mansour in~\cite{JM}. However, although~$\eta$ has a neat description ($\eta(\pi)$ records the index of the horizontal strip that contains each element of~$\pi$, from left to right), it is not enough to have a clear understanding of why~$\Sort(132)$ is enumerated by the binomial transform of Catalan numbers.

We choose to devote this section to a deeper investigation on the combinatorics underlying some related sets of pattern-avoiding {\rgfs}, aiming to find a more transparent connection with~$132$-sortable permutations. Ideally, we would like to provide a link between~$\Sort(132)$ and some combinatorial objects that immediately reveals why this counting sequence arises. We start by showing a (presumably) new bijection between~$\RGF_n(1221)$ and the set~$\Dyck_n$ of Dyck paths of semilength~$n$. Then we define new bijections between~$\RGF(y)$, with~$y$ pattern in the Wilf-equivalence class of~$12231$, and other families of combinatorial objects, such as labeled Motzkin paths and pattern-avoiding permutations. Finally, we obtain a bijective argument that clearly justifies the enumeration of~$\Sort(132)$ by showing a bijection between~$\RGF(12231)$ and~$\RGF(12321)$.

\subsection{\texorpdfstring{Pattern~$1221$}{Pattern 1221}}

The following lemma can be found in \cite{CDDGGPS}.

\begin{lemma}[\cite{CDDGGPS}, Lemma 6.2]\label{lemma_char_1221}
Let~$R$ be a \rgf. Then~$R\in\RGF(1221)$ if and only if the subword~$w(R)$ obtained by removing the first occurrence of each letter in~$R$ is weakly increasing.
\end{lemma}

An immediate consequence of Lemma~\ref{lemma_char_1221} is the following.

\begin{corollary}\label{corollary_active_sites_1221}
Let~$R=r_1\cdots r_n\in\RGF(1221)$ and~$M=\max(R)$. If~$R$ has no repeated elements let~$t=1$; otherwise let~$t$ be the maximum among repeated elements of~$R$. Then~$r_1\cdots r_nj\in\RGF(1221)$ if and only if~$t\le j\le M+1$.
\end{corollary}

Using again the language of generating trees, we say that an integer~$j$ is an \textit{active site} of the {\rgf}~$R\in\RGF(1221)$ if by appending~$j$ at the end of~$R$ we get another {\rgf} in~$\RGF(1221)$, which is said to be a \textit{child} of~$R$. The set of active sites of~$R$ is the interval~$\lbrace t,t+1,\dots,M,M+1\rbrace$ due to Corollary~\ref{corollary_active_sites_1221}. Thus~$R$ has~$M+1-t+1$ active sites, where~$M$ and~$t$ are defined as in the corollary.

Now, recall from section~\ref{section_lattice_paths} that a \textit{double rise} in a Dyck path is an occurrence of the consecutive pattern~$\U\U$.

\begin{theorem}\label{theorem_enum_1221}
There is a bijection~$\psi:\RGF_n(1221)\to\Dyck_n$, such that the maximum of~$R\in\RGF_n(1221)$ equals one plus the number of double rises in the path~$\psi(R)$. As a consequence, denoting by~$f_{n,k}$ the number of elements in~$\RGF_n(1221)$ whose maximum is~$k$, we get that~$f_{n,k}=\narayana_{n,k}$, where~$\narayana_{n,k}$ is the~$(n,k)$-th Narayana number.
\end{theorem}
\begin{proof}
Recall from Example~\ref{example_dyck_paths_new_peak} that every Dyck path~$\tilde{P}$ of semilength~$n+1$ is obtained (in a unique way) from a Dyck path~$P$ of semilength~$n$ by inserting a peak~$\U\D$ either before a~$\D$-step in the last descending run of~$P$ or after the last~$\D$-step. This construction gives rise to a well known generating tree for Dyck paths, such that the number of active sites of a path~$P$ is~$k+1$, where~$k$ is the length of the last descending run of~$P$. The path~$\tilde{P}$ is therefore a child of~$P$ in the associated generating tree. Our goal is to define (in a recursive fashion) a bijection~$\alpha$ between the generating tree of~$\RGF(1221)$ and the generating tree of Dyck paths. In other words, we wish to show that~$\alpha$ is a bijection preserving both the size (that is, a {\rgf} of length~$n$ is mapped to a Dyck path of semilength~$n$) and the number of active sites.

We start by setting~$\alpha(1)=\U\D$. Note that~$1$ has two active sites, since the children of~$1$ are~$11$ and~$12$. The path~$\U\D$ has two active sites as well, since its children are~$\U\U\D\D$ and~$\U\D\U\D$. Now let~$R=r_1\cdots r_n$ and~$\alpha(R)=p_1\cdots p_{2n}$, for some~$n\ge 1$. Suppose that the number of active sites of both~$R$ and~$\alpha(R)$ is~$k$. Let~$M=\max(R)$ and let~$t$ be the maximum element of~$R$ that is not an ltr-maximum of~$R$. By Corollary~\ref{corollary_active_sites_1221}, the active sites of~$R$ form the interval~$\lbrace t,t+1,\dots,M,M+1\rbrace$, with~$M+1-t+1=k$ by hypothesis. Moreover, the length of the last descending run of~$\alpha(R)$ is~$k-1$. We shall define~$\alpha$ on the children of both~$R$ and~$\alpha(R)$, and show that the number of active sites is still preserved.

\begin{itemize}
\item The child of~$R$ corresponding to the active site~$M$ is mapped to the path obtained from~$\alpha(R)$ by inserting a new peak~$\U\D$ immediately after the last~$\D$-step of~$\alpha(R)$. Here the active sites of the resulting sequence are~$M+1-M+1=2$. The same holds for the resulting Dyck path, since the length of its last descending run is~$1$.

\item For~$i=1,\dots,M-t$, the child of~$R$ corresponding to the active site~$t+i-1$ is mapped to the path obtained from~$\alpha(R)$ by inserting a new peak~$\U\D$ immediately after the~$i$-th~$\D$ step of the last descending run. Then the number of active sites of the resulting {\rgf} is then~$(M+1)-(t+i-1)+1=M-t-i3$, which is equal to one plus the length of the last descending run of the resulting path, that is~$(M+1-t)-i+1$.

\item Finally, the child of~$R$ corresponding to the active site~$M+1$ is mapped to the path obtained from~$\alpha(R)$ by inserting a new peak~$\U\D$ immediately before the first~$\D$-step of the last descending run of~$\alpha(R)$. In this case the number of active sites of the resulting {\rgf} is~$M+2-t+1=k+1$. Moreover, the number of active sites of the resulting path is also~$k+1$, since the length of its maximal suffix of~$\D$-steps is increased by one with respect to~$\alpha(R)$.
\end{itemize}

Therefore~$\alpha$ is a bijection between the two generating trees, as desired. To conclude, observe that the number of double rises in~$\alpha(R)$ is equal to~$\max(R)-1$. Indeed, by definition of~$\alpha$, each double rise in~$\alpha(R)$ corresponds to the first occurrence of an integer in~$R$, except for the first occurrence of~$1$ (which does not create a double rise). It is well known (see for example~\cite{Deu}) that the number of Dyck paths of semilength~$n$ with~$k-1$ double rises is equal to~$\narayana_{n,k}$, which gives the desired equality~$f_{n,k}=\narayana_{n,k}$.
\end{proof}

\begin{corollary}\label{corollary_enumeration_12332}
Let~$n\ge 0$. Denote by~$g(n,k)$ the number of elements in~$\RGF_n(12332)$ whose maximum is~$k$, for~$1\le k\le n$. Then:
$$
g(n+1,k+1)=\sum_{j=k}^{n}\binom{n}{j}\narayana_{j,k}.
$$
\end{corollary}
\begin{proof}
As observed in~\cite{JM}, every~$12332$-avoiding {\rgf} of length~$n+1$ can be obtained by choosing~$n-j$ positions for the~$1$s (except for the first~$1$, which is fixed) and then choosing a {\rgf}~$R\in\RGF_j(1221)$ for the remaining~$j$ spots (where the elements of~$R$ incremented by one will be inserted). In particular, if the maximum of~$R$ is~$k$, then the resulting {\rgf} has maximum~$k+1$. So, as a consequence of Theorem~\ref{theorem_enum_1221}, we have~$g(n+1,k+1)=\sum_{j=k}^{n}\binom{n}{j}\narayana_{j,k}$.
\end{proof}

In the following sections (Proposition~\ref{proposition_narayana_321} and Theorem~\ref{theorem_bij_12321_12231}), we provide a bijection between~$12231$- and~$12321$-avoiding {\rgfs} in order to prove that~$132$-sortable permutations, according to the number of their ltr-minima, are enumerated by the formula in Corollary~\ref{corollary_enumeration_12332}. A direct proof of this fact is still to be found.

\begin{openproblem}\label{open_prob_distribution_132}
Prove directly (that is, without using a bijection involving different objects) that the number of~$132$-sortable permutations of length~$n+1$ with~$k+1$ left-to-right minima is equal to~$\displaystyle{\sum_{j=k}^{n}\binom{n}{j}\narayana_{j,k}}$.
\end{openproblem}

\subsection{\texorpdfstring{Patterns~$12323$ and~$12332$}{Patterns 12323 and 12332}}\label{section_cont_frac}

Consider the ordinary generating function of~$132$-sortable permutations:
$$
\Fsigma{132}(t)=\sum_{n\ge 0}\left(\sum_{k=0}^{n-1}\binom{n-1}{k}\mathfrak{c}_k\right) t^n
$$
Then~$\Fsigma{132}(t)$ can be expressed using the following continued fraction (see, for example, \cite{Ba,Fl}):
$$
\Fsigma{132}(t)=\cfrac{1}{1-2t-\cfrac{t^2}{1-3t-\cfrac{t^2}{1-3t-\cfrac{t^2}{1-3t-\dots}}}}
$$
Labeled Motzkin paths provide a neat combinatorial interpretation for the above continued fraction, via Flajolet's general correspondence~\cite{Fl}. The~$n$-th term of the sequence~$\lbrace|\Sort_{n+1}(132)|\rbrace_n$ is equal to the number of Motzkin paths of length~$n$ such that:

\begin{itemize}
\item each horizontal step at height zero has two types of labels~$\ell_0$, $\ell_1$;
\item each horizontal step at height at least one has three types of labels~$\ell_0,\ell_1,\ell_2$.
\end{itemize}

Denote by~$\Motzkin^{lab}_n$ the set of such labeled Motzkin paths of length~$n$. We shall define a bijection~$\beta$ from~$\Motzkin^{lab}_n$ to~$\RGF_{n+1}(12323)$ by scanning a Motzkin path from left to right and suitably intepreting each labeled step. We use an auxiliary stack~$\Delta$, which is initialized as the empty stack. Let~$P\in\Motzkin^{lab}_n$. Start by setting~$R=1$. Then, if~$L$ is the label of the currently scanned step, append a new rightmost element to~$R$ according to the following rules:

\begin{itemize}
\item if~$L=\U$, then append a new strict maximum~$M$ and push~$M$ onto~$\Delta$;
\item if~$L=\D$, then append~$\top(\Delta)$ and pop it from~$\Delta$;
\item if~$L=\ell_0$, then append a new strict maximum (without pushing it onto~$\Delta$);
\item if~$L=\ell_1$, then append~$1$;
\item if~$L=\ell_2$, then append~$\top(\Delta)$ (without popping it from~$\Delta$).
\end{itemize}

Equivalently, $\U$ corresponds to the first occurrence of a letter~$x$ that appears at least twice in~$R$, $\D$ to the last occurrence of such a letter, and~${\ell}_2$ to an occurrence of such an~$x$ that is neither the first nor the last. Moreover, the label~$\ell_0$ corresponds to an element~$x\neq 1$ appearing only once and the label~$\ell_1$ corresponds to the element~$1$. An example of this construction is illustrated in Figure~\ref{figure_Motzkin_RGF_bij}.

We can also express the correspondence between the labels of~$P$ and~$R=\beta(P)$ in terms of properties of the set partition associated to~$R$. If~$B$ is a block of cardinality at least two in such a partition and~$B$ does not contain~$1$, then~$\U$, $\D$ and~$\ell_2$ correspond, respectively, to the least, the largest and any of the remaining elements of the block. Moreover, $\ell_0$ corresponds to a singleton block not containing~$1$ and~$\ell_1$ corresponds to the elements of the block containing~$1$. At each step of the construction of~$R$, the auxiliary stack~$\Delta$ keeps track of the currently open blocks in the corresponding partition (that is those blocks that have not yet received all their elements).

\begin{figure}
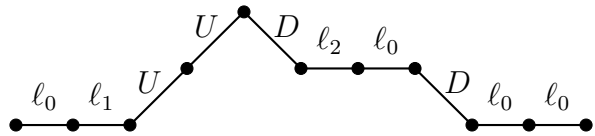

\centering
\begin{DrawPath}
\fillPath{0,0,1,1,-1,0,0,-1,0,0}{0}{0}
\node[] at (0.5,0.5) {$\ell_0$};
\node[] at (1.5,0.5) {$\ell_1$};
\node[above,left] at (2.75,0.75) {$U$};
\node[above,left] at (3.75,1.75) {$U$};
\node[] at (4.75,1.75) {$D$};
\node[] at (5.5,1.5) {$\ell_2$};
\node[] at (6.5,1.5) {$\ell_0$};
\node[] at (7.75,0.75) {$D$};
\node[] at (8.5,0.5) {$\ell_0$};
\node[] at (9.5,0.5) {$\ell_0$};
\end{DrawPath}
\caption[A labeled Motzkin path and the corresponding~$12323$-avoiding {\rgf}.]{The labeled Motzkin path corresponding to the {\rgf}~$R=12134435367$ via the bijection~$\beta$ of Theorem~\ref{theorem_bij_motzkin_12323}. The set partition associated to~$R$ is~$13|2|479|56|8|10|11$.}\label{figure_Motzkin_RGF_bij}
\end{figure}

\begin{theorem}\label{theorem_bij_motzkin_12323}
The map~$\beta$ is a bijection between~$\Motzkin^{lab}_n$ and~$\RGF_{n+1}(12323)$.
\end{theorem}
\begin{proof}
It is straightforward to see that~$\beta$ is injective and that~$\beta(P)$ is a {\rgf} for every~$P\in\Motzkin^{lab}_n$. Since~$|\Motzkin^{lab}_n|=|\RGF_n(12323)|$, we only need to show that~$\beta(P)$ avoids~$12323$, for each~$P\in\Motzkin^{lab}_n$. Suppose, for a contradiction, that~$abcb'c'$ is an occurrence of~$12323$ in~$\beta(P)$. This implies, of course, that~$b,c\neq 1$. Without loss of generality, we may assume that~$b$ and~$c$ are the first occurrences of the corresponding integers in~$\beta(P)$; then both~$b$ and~$c$ correspond to~$\U$-steps in~$P$ and are pushed onto~$\Delta$. Moreover, since~$b'=b$ and~$b'$ follows~$c$ in~$\beta(P)$, when~$c$ enters~$\Delta$, $b$ is still in, and so~$c$ lies above~$b$ in~$\Delta$. Now observe that the element~$b'$ must correspond to either a~$\D$-step or a horizontal step labeled~$\ell_2$ of~$P$. However, in both cases, when~$b'$ is inserted into~$\beta(P)$, $b$ has to be at the top of the stack, hence~$c$ should have been popped. This would imply that there are no more occurrences of~$c$ in~$\beta (P)$ after~$b'$, which is not the case, since~$c'=c$.
\end{proof}

\begin{remark}\label{remark_12332_motzkin}
If we replace the stack~$\Delta$ with a queue~$\Xi$, then the same map gives a bijection with {\rgfs} avoiding~$12332$. The proof is analogous to the previous one, and thus omitted.
\end{remark}

\begin{remark}\label{remark_catalan_cfrac}
If we restrict the previous bijections to Motzkin paths with no horizontal steps labeled~$\ell_1$, then we get bijections with {\rgfs} that avoid~$1212$ (if we use a stack~$\Delta$) or~$1221$ (if we use a queue~$\Xi$), provided that we remove the~$1$ at the beginning and decrease all the other elements by one. This follows again from the characterization of~$\RGF(12323)$ and~$\RGF(12332)$ given in~\cite{JM} (and mentioned in the proof of Corollary~\ref{corollary_enumeration_12332}). The corresponding continued fraction is then:
$$
G(t)=\cfrac{1}{1-t-\cfrac{t^2}{1-2t-\cfrac{t^2}{1-2t-\cfrac{t^2}{1-2t-\cdots}}}}
$$
This gives an alternative proof of the fact that {\rgfs} avoiding either~$1221$ or~$1212$ are enumerated by the Catalan numbers, whose generating function is known to be given by the above continued fraction.
\end{remark}

\begin{remark}\label{remark_sum_distribution_motzkin}
As a consequence of the bijections in Theorem~\ref{theorem_bij_motzkin_12323} and Remark~\ref{remark_12332_motzkin}, the statistic ``sum of the numbers of~$\U$ and~$\ell_0$ steps'' in~$\Motzkin^{lab}_n$ is equidistributed with the statistic ``(value of the) maximum minus one'' both in~$\RGF_{n+1}(12332)$ and in~$\RGF_{n+1}(12323)$. The same holds for the statistics ``number of labels~$\ell_0$'' and ``number of singletons~$\neq\{ 1\}$'', as well as for the statistics ``number of labels~$\ell_1$'' and ``number of occurrences of~$1$ minus one''. Some computations seem to suggest that the distribution of the maximum is the same for several other patterns of the same Wilf-class, namely~$12123$, $12132$, $12213$, $12231$, $12312$, $12321$, $12331$, so we suspect that the same approach should lead to straightforward bijections, by suitably modifying the interpretation of the steps. For example, define~$r_i$ to be a \textit{repeated} ltr-maximum of a {\rgf}~$r_1r_2\cdots r_n$ if~$r_i=\max\left\lbrace r_1,\dots,r_{i-1}\right\rbrace$. Then steps having label~$\ell_1$ seem to have the same distribution as the repeated ltr-maxima in~$\RGF(12321)$ and~$\RGF(12312)$, so in order to define a bijection with~$\Motzkin^{lab}$ it could be enough to find the ``correct'' interpretations for steps having labels~$\D$ and~$\ell_2$.
\begin{openproblem}
Find suitable interpretations of the steps of labeled Motzkin paths in~$\Motzkin^{lab}$ to obtain bijections with the remaining sets of pattern-avoiding {\rgfs} in the same Wilf-equivalence class.
\end{openproblem}
\end{remark}

\subsection{\texorpdfstring{Patterns~$12321$ and~$12312$}{Patterns 12321 and 12312}}

In this section we show a connection between {\rgfs} avoiding the patterns~$12321$ and~$12312$ and permutations avoiding the patterns~$321$ and~$312$, respectively. We initially provide a bijection between~$\RGF(12321)$ and~$\Perm(321)$ by showing that these two sets share the same combinatorial structure: elements of both sets can be written as shuffle of two weakly increasing sequences, one of those being the sequence of ltr-maxima. An analogous property links~$\RGF(12231)$ and~$\Perm(231)$.

Let~$R=r_1\cdots r_n$ be a \rgf. Recall from Remark~\ref{remark_sum_distribution_motzkin} that~$r_i$ is said to be a repeated ltr-maximum when~$r_i=\max\left\lbrace r_1,\dots,r_{i-1}\right\rbrace$. Let~$\RGF^{n.r.}$ be the set of {\rgfs} that have no repeated ltr-maxima. Define~$\RGF^{n.r.}_n$ and~$\RGF^{n.r.}(Q)$, for a pattern~$Q$, as usual. Given~$R=r_1\cdots r_n\in\RGF^{n.r.}$, define~$\tilde{R}$ as the subsequence obtained by deleting the ltr-maxima of~$R$. It is easy to realize that~$\tilde{R}$ is not necessarily a \rgf. For example, if~$R=121311245246$, then~$\tilde{R}=111224$.

\begin{lemma}\label{lemma_321_RGF}
Let~$R\in\RGF^{n.r.}$. Then~$R$ avoids~$12321$ if and only~$\tilde{R}$ is weakly increasing.
\end{lemma}
\begin{proof}
Suppose that~$\tilde{R}=\cdots ba\cdots$, with~$b>a$. Note that~$b$ is not a repeated ltr-maximum of~$R$, so there has to be an element~$c$ in~$R$ such that~$c>b$ and~$c$ comes before~$b$. Then~$R$ contains an occurrence~$cba$ of~$321$ and therefore it also contains~$12321$, by Lemma~\ref{lemma_RGF_prop}.

Conversely, if~$R$ contains an occurrence~$abcb'a'$ of~$12321$, then~$b'$ precedes~$a'$ in~$\tilde{R}$ and~$b'>a'$, so~$\tilde{R}$ is not weakly increasing.
\end{proof}

We now wish to describe the anticipated bijection between~$\RGF_n(12321)$ and~$\Perm_n(321)$. Let~$R=r_1\cdots r_n\in\RGF^{n.r.}(12321)$ and let~$\tilde{R}=r_{i_1}\cdots r_{i_k}$, with~$k\ge0$. Construct a permutation~$\pi(R)$ of length~$n$ by keeping the same positions for the ltr-maxima and mapping~$\tilde{R}$ to a strictly increasing sequence~$S=s_1\cdots s_k$ as follows:

\begin{itemize}
\item $s_1=r_{i_1}$;
\item $s_j=s_{j-1}+(r_{i_j}-r_{i_{j-1}})+1$, for~$j\ge 2$.
\end{itemize}

Finally, insert the remaining elements in increasing order in order to get a permutation that avoids~$321$: elements inserted at this point will be the ltr-maxima of the resulting permutation~$\pi(R)$. For instance, let~$R=121314234$. Then the string obtained by removing the ltr-maxima from~$R$ is~$\tilde{R}=11234$. We thus get the increasing sequence~$S=12468$ and finally~$\pi(R)=\mathbf{3}\mathbf{5}1\mathbf{7}2\mathbf{9}468$ (where ltr-maxima of~$\pi(R)$ are bolded). Observe that the number of ltr-maxima of~$\pi(R)$ is equal to the number of ltr-maxima of the starting {\rgf}~$R$. Moreover, it is easy to realize (by construction) that~$\pi(R)$ is a~$321$-avoiding permutation. The map defined this way is also injective. Indeed positions and values of the ltr-maxima uniquely determine a~$321$-avoiding permutations, thus the strictly increasing sequence~$S$ is enough to identify one such permutation. The construction proposed can be inverted in a similar fashion. It follows that the map~$R\mapsto\pi(R)$ is a size-preserving bijection between~$\RGF^{n.r.}(12321)$ and~$\Perm(321)$. The next corollary is an immediate consequence of what discussed so far.

\begin{proposition}\label{proposition_narayana_321} 
The number of {\rgfs} in~$\RGF_n ^{n.r.}(12321)$ is~$\catalan_n$. Moreover, the number of {\rgfs} in~$\RGF_n ^{n.r.}(12321)$ having maximum~$k$ is given by~$\narayana_{n,k}$.
\end{proposition}

Next, in order to enumerate~$\RGF(12321)$, it is sufficient to show that any {\rgf} avoiding~$12321$ can be uniquely obtained by  inserting some repeated ltr-maxima in a sequence in~$\RGF^{n.r.}(12321)$.

\begin{theorem}\label{theorem_321}
Let~$R$ be a {\rgf} and let~$\alpha(R)$ be the sequence obtained from~$R$ by removing all the repeated ltr-maxima. Then~$\alpha(R)$ is a {\rgf}. Moreover, $R$ avoids~$12321$ if and only~$\alpha(R)$ avoids~$12321$.
\end{theorem}
\begin{proof}
It is easy to check that~$\alpha(R)$ is still a {\rgf} and clearly~$\alpha(R)$ avoids~$12321$ if~$R$ does. On the other hand, suppose that~$R$ contains an occurrence~$abcb'a'$ of~$12321$. Note that~$b'$ and~$a'$ are not repeated ltr-maxima, so they are elements of~$\alpha(R)$ and they follow~$c$ in~$R$. Let~$c'$ be the first occurrence of the integer~$c$ in~$R$. Then~$c'\in\alpha(R)$ and~$c'$ precedes~$b'$ in~$\alpha(R)$, so~$\alpha(R)$ contains an occurrence~$c'b'a'$ of~$321$, which is equivalent to containing~$12321$.
\end{proof}

\begin{corollary}\label{corollary_enumer_12321}
For each~$n\ge 1$, we have:
$$
|\RGF_{n+1}(12321)|=\sum_{k=0}^n\binom{n}{k}\mathfrak{c}_k.
$$
Moreover, there are~$\sum_{j=k}^{n}\binom{n}{j}\narayana_{j,k}$ {\rgfs} in~$\RGF_{n+1}(12321)$ with maximum~$k$.
\end{corollary}
\begin{proof} This is a direct consequence of the results proved in this section, together with the fact that the first element of a {\rgf} cannot be a repeated ltr-maximum.
\end{proof}

\begin{remark}
The same approach can be used to find a bijection between $\RGF^{n.r.}(12312)$ and~$\Perm(312)$. In fact, $312$-avoiding permutations are also uniquely determined by the positions and values of their ltr-maxima, and a completely analogous argument can be applied. As a consequence, we also have:
$$
|\RGF_{n+1}(12312)|=\sum_{k=0}^n\binom{n}{k}\mathfrak{c}_k.
$$
\end{remark}

\subsection{\texorpdfstring{A bijection between~$\RGF(12321)$ and~$\RGF(12231)$}{A bijection between RGF(12321) and RGF(12231)}}

In Section~\ref{section_grid_132} we showed a bijection between~$\Sort(132)$ and~$\RGF(12231)$. A direct combinatorial enumeration of~$\RGF(12231)$ could be obtained by using the labeled Motzkin path approach described in Section~\ref{section_cont_frac}, but so far the pattern~$12231$ proved to be rather more complicated than some other patterns in the same equivalence class. The main goal of this section is to obtain an independent\footnote{Without relying on the Wilf-equivalence showed in~\cite{JM}.} proof of the enumeration of~$\Sort(132)$ by means of a bijection~$\delta$ between~$\RGF(12231)$ and~$\RGF(12321)$.

For the rest of this chapter, we say that the {\rgf}~$R$ contains an occurrence of the pattern~$\tilde{2}31$ if~$R$ contains three elements~$r_{i_1}r_{i_2}r_{i_3}$ such that~$r_{i_1}r_{i_2}r_{i_3}\simeq 231$ and~$r_{i_1}$ is not an ltr-maximum of~$R$ (equivalently, $r_{i_1}$ is not the leftmost occurrence of the corresponding integer in~$R$). Due to Lemma~\ref{lemma_RGF_prop}, we have~$\RGF(12231)=\RGF(\tilde{2}31)$ and also~$\RGF(12321)=\RGF(321)$. This allows us to focus on~$\tilde{2}31$ and~$321$ instead of~$12231$ and~$12321$, respectively. More precisely, we shall define the promised map~$\delta$ from~$\RGF(12231)$ to~$\RGF(12321)$ by repeatedly transforming the rightmost occurrence of~$321$ into an occurrence of~$\tilde{2}31$, until the resulting {\rgf} avoids~$321$. This is formalized in what follows.

Let~$R=r_1\cdots r_n$ be a {\rgf}. Define~$\rmost(R,321)=i_1i_2i_3$, where~$r_{i_i}r_{i_2}r_{i_3}$ are the indices of the lexicographically rightmost occurrence of~$321$ in~$R$. More extensively, this can be expressed by saying that for any other occurrence~$r_{j_i}r_{j_2}r_{j_3}$ of~$321$ in~$R$, it must be either~$j_1<i_1$, or~$j_1=i_1$ and~$j_2<i_2$, or~$j_1=i_1$, $j_2=i_2$ and~$j_3<i_3$. If~$R$ avoids~$321$, we assume~$\rmost(321)= 000$ by convention. Similarly, denote by~$\lmost(R,\tilde{2}31)=i_1i_2i_3$ the indices of the lexicographically leftmost occurrence of~$\tilde{2}31$ in~$R$. If~$R$ avoids~$\tilde{2}31$, we assume~$\lmost(R,\tilde{2}31)=(n+1)(n+1)(n+1)$.

Let us now define the map~$\delta$. Suppose that~$R=r_1\cdots r_n\in\RGF(\tilde{2}31)$. Define recursively a {\rgf}~$\delta(R)$ as follows.

\begin{enumerate}
\item $R^{(0)}=R$.
\item For~$t\ge 0$, if~$R^{(t)}$ contains~$321$, then~$R^{(t+1)}$ is obtained from~$R^{(t)}$ by exchanging the elements~$r_{i_1}$ and~$r_{i_2}$, where~$i_1i_2i_3=\rmost(R^{(t)},321)$.
\item Finally, define~$\delta(R)=R^{(k)}$, where~$k$ is the minimum index such that~$R^{(k)}$ avoids~$321$.
\end{enumerate}

Observe that~$R^{(t)}$ is a {\rgf} for each~$t$ and that~$R^{(k)}$ avoids~$321$ by construction. Moreover, as a consequence of the next lemma, the integer~$k$ exists and thus the map~$\delta$ is well defined.

\begin{lemma}\label{lemma_lex_321}
For every~$t\ge 0$, we have~$\rmost(R^{(t+1)},321)<_\ell\rmost(R^{(t)},321)$, where~$<_\ell$ denotes the lexicographical order.
\end{lemma}
\begin{proof}
Let~$R^{(t)}=r^{(t)}_1\cdots r^{(t)}_n$ and, similarly, $R^{(t+1)}=r^{(t+1)}_1\cdots r^{(t+1)}_n$. Moreover, let~$\rmost(R^{(t)},321)=i_1i_2i_3$ and~$\rmost(R^{(t+1)},321)=j_1j_2j_3$. Note that, as illustrated in Figure~\ref{figure_321_231}, our hypothesis imposes some constraints on the elements of~$R^{(t)}$. More precisely, $r^{(t)}_j\le r^{(t)}_{i_2}$, for each~$j=i_1+1,\dots,i_2-1$. Also, for each~$j=i_2+1,\dots,i_3-1$, either~$r^{(t)}_j\le r^{(t)}_{i_3}$ or~$r^{(t)}_j\ge r^{(t)}_{i_1}$. Finally, $r^{(t)}_j\ge r^{(t)}_{i_2}$ for each~$j>i_3$. We will repeatedly use these inequalities throughout this proof. Our goal is now to show that~$j_1j_2j_3<_\ell i_1i_2i_3$. Suppose, by contradiction, that~$j_1j_2j_3\ge_\ell i_1i_2i_3$. Consider the following case analysis.

\begin{itemize}
\item Suppose~$j_1>i_1$. If~$j_1<i_2$, then necessarily~$r^{(t+1)}_{j_1}=r^{(t)}_{j_1}\le r^{(t)}_{i_2}$, due to the above constraints. Hence we must have~$j_2,j_3\neq i_2$, since otherwise the indices~$j_1,j_2,j_3$ would not correspond to an occurrence of~$321$ in~$R^{(t+1)}$. This implies that~$r^{(t+1)}_{j_1}r^{(t+1)}_{j_2}r^{(t+1)}_{j_3}=r^{(t)}_{j_1}r^{(t)}_{j_2}r^{(t)}_{j_3}$ is an occurrence of~$321$ in~$R^{(t)}$ as well, with~$j_1j_2j_3>_\ell i_1i_2i_3$: this is a contradiction, since we are assuming that~$\rmost(R^{(t)},321)=i_1i_2i_3$. Next suppose that~$j_1=i_2$ (and so~$j_2>i_2$). Note that~$r^{(t)}_{i_1}=r^{(t+1)}_{i_2}=r^{(t+1)}_{j_1}$, hence~$r^{(t)}_{i_1}r^{(t)}_{j_2}r^{(t)}_{j_3}$ is an occurrence of~$321$ in~$R^{(t)}$ with~$i_1j_2j_3>_\ell i_1i_2i_3$, which is impossible. Finally, suppose that~$j_1>i_2$. Then obviously~$r^{(t)}_{j_1}r^{(t)}_{j_2}r^{(t)}_{j_3}=r^{(t+1)}_{j_1}r^{(t+1)}_{j_2}r^{(t+1)}_{j_3}$ is an occurrence of~$321$ in~$R^{(t)}$, with~$j_1j_2j_3>_\ell i_1i_2i_3$, again a contradiction.

\item Suppose instead that~$j_1=i_1$ and~$j_2>i_2$. Then~$r^{(t+1)}_{i_1}=r^{(t)}_{i_2}$ and~$j_2>i_2$, so~$r^{(t)}_{i_2}r^{(t)}_{j_2}r^{(t)}_{j_3}$ is an occurrence of~$321$ in~$R^{(t)}$, with~$i_2j_2j_3>_\ell i_1i_2i_3$, which is impossible.

\item Finally, the case~$j_1=i_1$ and~$j_2=i_2$ is clearly impossible, since we have~$r^{(t+1)}_{i_1}=r^{(t)}_{i_2}<r^{(t)}_{i_1}=r^{(t+1)}_{i_2}$.
\end{itemize}
\end{proof}

\begin{figure}
\centering
\begin{tikzpicture}[scale=0.5, baseline=19pt]
\fill[NE-lines] (1.15,2.15) rectangle (1.85,4);
\fill[NE-lines] (2.15,1.15) rectangle (2.85,2.85);
\fill[NE-lines] (3.15,0) rectangle (3.85,1.85);
\draw [semithick] (0,0.85) -- (4,0.85);
\draw [semithick] (0,1.15) -- (4,1.15);
\draw [semithick] (0,1.85) -- (4,1.85);
\draw [semithick] (0,2.15) -- (4,2.15);
\draw [semithick] (0,2.85) -- (4,2.85);
\draw [semithick] (0,3.15) -- (4,3.15);
\draw [semithick] (0.85,0) -- (0.85,4);
\draw [semithick] (1.15,0) -- (1.15,4);
\draw [semithick] (1.85,0) -- (1.85,4);
\draw [semithick] (2.15,0) -- (2.15,4);
\draw [semithick] (2.85,0) -- (2.85,4);
\draw [semithick] (3.15,0) -- (3.15,4);
\filldraw (1,3) circle(5pt);
\filldraw (2,2) circle (5pt);
\filldraw (3,1) circle (5pt);
\node[] at (1,4.5){$i_1$};
\node[] at (2,4.5){$i_2$};
\node[] at (3,4.5){$i_3$};
\end{tikzpicture}
\hspace{50pt}
\begin{tikzpicture}[scale=0.5, baseline=19pt]
\fill[NE-lines] (1.15,2.15) rectangle (1.85,4);
\fill[NE-lines] (2.15,1.15) rectangle (2.85,2.85);
\fill[NE-lines] (3.15,0) rectangle (3.85,1.85);
\draw [semithick] (0,0.85) -- (4,0.85);
\draw [semithick] (0,1.15) -- (4,1.15);
\draw [semithick] (0,1.85) -- (4,1.85);
\draw [semithick] (0,2.15) -- (4,2.15);
\draw [semithick] (0,2.85) -- (4,2.85);
\draw [semithick] (0,3.15) -- (4,3.15);
\draw [semithick] (0.85,0) -- (0.85,4);
\draw [semithick] (1.15,0) -- (1.15,4);
\draw [semithick] (1.85,0) -- (1.85,4);
\draw [semithick] (2.15,0) -- (2.15,4);
\draw [semithick] (2.85,0) -- (2.85,4);
\draw [semithick] (3.15,0) -- (3.15,4);
\filldraw (1,2) circle (5pt);
\filldraw (2,3) circle (5pt);
\filldraw (3,1) circle (5pt);
\node[] at (1,4.5){$i_1$};
\node[] at (2,4.5){$i_2$};
\node[] at (3,4.5){$i_3$};
\end{tikzpicture}
\caption[The rightmost occurrence of~$321$ as a Cayley-mesh pattern.]{On the left, the rightmost occurrence of the pattern~$321$ in~$R^{(t)}$, with indices~$i_1i_2i_3$, represented as a Cayley-mesh pattern. On the right, the resulting (Cayley-mesh) pattern in~$R^{(t+1)}$, obtained by exchanging the elements in positions~$i_1$ and~$i_2$.}\label{figure_321_231}
\end{figure}
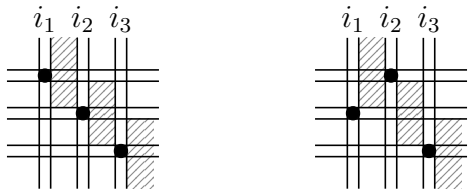

We wish to show that~$\delta$ is a bijection by proving that the recursive construction proposed above can be reversed (in the expected way!). Indeed, in order to obtain~$R^{(t)}$ from~$R^{(t+1)}$, it is sufficient to transform the leftmost occurrence of~$\tilde{2}31$ into an occurrence of~$321$ (see Figure~\ref{figure_diagram_gamma}). A formal proof is given in the next lemma.

\begin{figure}
\centering
$
\xymatrix{
 & \rmost(321)\dashrightarrow\tilde{2}31\ar@{<-}[dl]\ar@{->}[dr] & \\
R^{(t)}\ar@{<-}[dr] & & R^{(t+1)}\ar@{->}[dl]\\
 & 321\dashleftarrow\lmost(\tilde{2}31) &
}
$
\caption{The diagram of Lemma~\ref{lemma_inverse_gamma}.}\label{figure_diagram_gamma}
\end{figure}
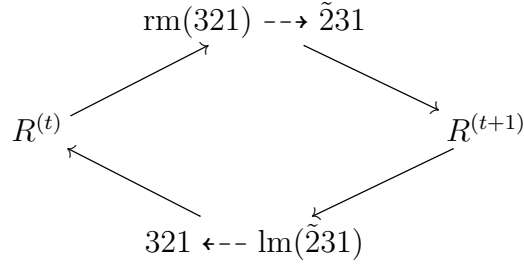

\begin{lemma}\label{lemma_inverse_gamma}
Let~$t\ge0$. Let~$\rmost(R^{(t)},321)=i_1i_2i_3$ and~$\lmost(R^{(t+1)},\tilde{2}31)=j_1j_2j_3$. Then~$i_1=j_1$ and~$i_2=j_2$.
\end{lemma}
\begin{proof}
We again refer to Figure~\ref{figure_321_231} for an illustration of the constraints imposed on the elements of~$R^{(t)}$ by the position of the rightmost occurrence of~$321$ inside~$R^{(t)}$. We proceed by induction on~$t$.

Suppose first that~$t=0$, that is, $R^{(0)}=r^{(0)}_1\cdots r^{(0)}_n$ avoids~$\tilde{2}31$, but contains~$321$. Set~$R^{(1)}=r^{(1)}_1\cdots r^{(1)}_n$, $\rmost(R^{(0)},321)=i_1i_2i_3$ and~$\lmost(R^{(1)},\tilde{2}31)=j_1 j_2 j_3$. Note that~$r^{(1)}_{i_1}r^{(1)}_{i_2}r^{(1)}_{i_3}$ is an occurrence of~$\tilde{2}31$ in~$R^{(1)}$. Indeed, by Lemma~\ref{lemma_RGF_prop}, the first occurrence of the integer~$r^{(0)}_{i_2}$ in~$R^{(0)}$ precedes~$r^{(0)}_{i_1}$, since~$r^{(0)}_{i_1}>r^{(0)}_{i_2}$. Therefore~$j_1j_2j_3\le_{\ell} i_1i_2i_3$. We have to show that~$i_1=j_1$ and~$i_2=j_2$. Suppose, to the contrary, that~$j_1<i_1$. If either~$j_2=i_1$ or~$j_2=i_2$, then~$r^{(0)}_{j_1}r^{(0)}_{i_1}r^{(0)}_{j_3}$ would be an occurrence of~$\tilde{2}31$ in~$R^{(0)}$, which is impossible since~$R^{(0)}\in\RGF(\tilde{2}31)$. Thus we must have~$j_2\neq i_1$ and~$j_2\neq i_2$. In particular, since~$j_2\neq i_2$, we must have either~$j_3=i_1$ or~$j_3=i_2$, otherwise~$r^{(0)}_{j_1}r^{(0)}_{j_2}r^{(0)}_{j_3}=r^{(1)}_{j_1}r^{(1)}_{j_2}r^{(1)}_{j_3}$ would be an occurrence of~$\tilde{2}31$ in~$R^{(0)}$ as well. However, if either~$j_3=i_1$ or~$j_3=i_2$, then~$r^{(0)}_{j_1}r^{(0)}_{j_2}r^{(0)}_{i_2}$ would be an occurrence of~$\tilde{2}31$ in~$R^{(0)}$, which is again a contradiction. Therefore it has to be~$i_1=j_1$. Finally, the case~$j_1=i_1$ and~$j_2<i_2$ is forbidden, due to the restrictions depicted in Figure~\ref{figure_321_231}. Summing up, we must have~$i_1=j_1$ and~$i_2=j_2$, as desired.

Now suppose that~$t\ge 1$. Let~$R^{(t)}=r^{(t)}_1\cdots r^{(t)}_n$. For the rest of this proof, we fix the following notation:

\begin{itemize}
\item[-]~$\rmost(R^{(t-1)},321)=t_1t_2t_3$;
\item[-]~$\lmost(R^{(t)},\tilde{2}31)=s_1s_2s_3$;
\item[-]~$\rmost(R^{(t)},321)=i_1i_2i_3$;
\item[-]~$\lmost(R^{(t+1)},\tilde{2}31)=j_1j_2j_3$.
\end{itemize}

By the inductive hypothesis we have~$s_1=t_1$ and~$s_2=t_2$. Moreover, Lemma~\ref{lemma_lex_321} implies that~$t_1t_2t_3 >_{\ell} i_1i_2i_3$, hence~$t_1t_2\ge_{\ell} i_1i_2$ and~$s_1s_2\ge_{\ell} i_1i_2$. Note that~$r^{(t+1)}_{i_1}r^{(t+1)}_{i_2}r^{(t+1)}_{i_3}$ is an occurrence of~$\tilde{2}31$ in~$R^{(t+1)}$, so we must have~$j_1j_2j_3\le_{\ell} i_1 i_2 i_3$. Our goal is to show that~$i_1=j_1$ and~$i_2=j_2$. We shall proceed by contradiction, so we assume that~$j_1<i_1$ or~$j_2<i_2$. Our strategy consists in finding an occurrence of~$\tilde{2}31$ in~$R^{(t)}$ such that the indices of its first two elements strictly precede~$i_1i_2$ (in the lexicographical order). Indeed, this would imply that~$s_1s_2<_{\ell}i_1i_2$, since~$s_1s_2s_3=\lmost(R^{(t)},\tilde{2}31)$, which is impossible since we know that~$s_1s_2\ge_{\ell}i_1 i_2$.

Suppose first that~$j_1<i_1$. If~$\{j_2,j_3\}\cap\{i_1,i_2\}=\emptyset$, then~$r^{(t)}_{j_1}r^{(t)}_{j_2}r^{(t)}_{j_3}$ is the desired occurrence of~$\tilde{2}31$ in~$R^{(t)}$, since in this case~$j_1,j_2,j_3$ are not involved in the transition from~$R^{(t)}$ to~$R^{(t+1)}$ and we are assuming that~$j_1<i_1$. Therefore at least one of~$j_2$ and~$j_3$ must coincide with either~$i_1$ or~$i_2$. We will now show that, in each case, we are able to find an occurrence of~$\tilde{2}31$ in~$R^{(t)}$ with the desired property.

\begin{itemize}
\item If~$j_2=i_1$, then~$r^{(t+1)}_{j_2}=r^{(t+1)}_{i_1}<r^{(t)}_{i_1}$, hence~$r^{(t)}_{j_1}r^{(t)}_{j_2}r^{(t)}_{j_3}$ is an occurrence of~$\tilde{2}31$ in~$R^{(t)}$, and~$j_1j_2<_{\ell}i_1i_2$.
\item If~$j_2=i_2$, then~$r^{(t)}_{j_1}r^{(t)}_{i_1}r^{(t)}_{j_3}$ is an occurrence~$\tilde{2}31$ in~$R^{(t)}$, and~$j_1i_1 <_{\ell}i_1i_2$.
\item If~$j_3=i_1$, then~$r^{(t)}_{j_1}r^{(t)}_{j_2}r^{(t)}_{i_2}$ is an occurrence of~$\tilde{2}31$ in~$R^{(t)}$, and~$j_1j_2<_{\ell}i_1i_2$.
\item If~$j_3=i_2$, then~$r^{(t)}_{i_2}<r^{(t)}_{i_1}=r^{(t+1)}_{i_2}$, hence~$r^{(t)}_{j_1}r^{(t)}_{j_2}r^{(t)}_{i_2}$ is an occurrence of~$\tilde{2}31$ in~$R^{(t)}$, and~$j_1j_2< i_1i_2$.
\end{itemize}

The above case-by-case analysis shows that~$i_1=j_1$. Moreover, we cannot have~$j_2<i_2$; this is again due to the restrictions illustrated in Figure~\ref{figure_321_231}.
\end{proof}

\begin{theorem}\label{theorem_bij_12321_12231}
The map~$\delta$ is a size-preserving bijection between~$\RGF(12321)$ and~$\RGF(12231)$. Moreover, $\delta$ preserves the maximum value of a {\rgf}.
\end{theorem}

As a consequence of Theorem~\ref{theorem_bij_12321_12231} and Corollary~\ref{corollary_enumer_12321}, the distribution of the maximum letter in {\rgfs} over~$\RGF_n(12231)$ is given by~$\sum_{i=k}^{n}\binom{n}{i}\narayana_{i,k}$. This provides a combinatorial (even if not direct) proof of the formula stated in Open Problem~\ref{open_prob_distribution_132} for the distribution of ltr-minima of~$\Sort(132)$.

\chapter{\texorpdfstring{The~$(\sigma,\tau)$-machine}{The (sigma,tau)-machine}}\label{chapter_sigma_tau_machine}

In this chapter we consider~$\Sigma$-machines where~$\Sigma=\lbrace\sigma,\tau\rbrace$ is a pair of patterns of length three. For specific pairs of patterns, the set of~$(\sigma,\tau)$-sortable permutations is enumerated by either the (binomial transform of) Catalan numbers or the Schr\"oder numbers. We also determine an infinite family of pairs of patterns, namely the pairs~$(\sigma,\hat{\sigma})$, whose enumeration is given by the Catalan numbers. Some of the mentioned cases were discussed in~\cite{BCKV}. Here we will sometimes follow an alternative approach. We also provide enumerative results for other pairs of patterns, which are currently unpublished.

\section{Two decomposition lemmas}

The avoidance of specific pairs of patterns deeply influences the geometric structure of the output of a~$(\sigma,\tau)$-stack. The following two decomposition lemmas are particularly useful.

\begin{lemma}\label{theorem_ltr-max_pairs_of_pat_312}
Consider the~$(312,\sigma)$-machine, where~$\sigma=\sigma_1\cdots\sigma_{k-1}\sigma_k\in\Perm_k$, with~$\sigma_{k-1}<\sigma_k$ and~$k\ge 3$. Given a permutation~$\pi$ of length~$n$, let~$\pi=M_1B_1\cdots M_tB_t$ be its ltr-max decomposition. Then:
\begin{enumerate}
\item Every time an ltr-maximum~$M_i$ is pushed into the~$(312,\sigma)$-stack, the~$(312,\sigma)$-stack contains the elements~$M_i,M_{i-1},\dots,M_2,M_1$, reading from top to bottom. Moreover, we have:
$$
\out{312,\sigma}(\pi)=\tilde{B}_1\cdots\tilde{B}_tM_t\cdots M_1,
$$
where~$\tilde{B}_i$ is a suitable rearrangement of~$B_i$.
\item If~$\pi$ is~$(312,\sigma)$-sortable, then~$M_j=n-t+j$, for each~$j=1,\dots,t$. Therefore~$\lbrace M_1,\dots,M_t\rbrace=\lbrace n-t+1,\dots,n\rbrace$.
\end{enumerate}
\end{lemma}
\begin{proof}
\begin{enumerate}
\item Let us consider the action of the~$(312,\sigma)$-stack on input~$\pi$. Note that, since~$k\ge 3$, the element~$M_1$ remains at the bottom of the ~$(312,\sigma)$-stack until the end of the sorting procedure. Now, for each~$x\in B_1$, the elements~$M_2xM_1$ form an occurrence of~$312$. Therefore the block~$B_1$ is extracted before~$M_2$ enters the~$(312,\sigma)$-stack. As soon as~$M_2$ enters, the stack contains~$M_2M_1$, reading from top to bottom. Since~$M_2>M_1$, but~$\sigma_{k-1}<\sigma_k$ by hypothesis, $M_2$ cannot play the role of either~$\sigma_{k-1}$ in an occurrence of~$\sigma$ or~$1$ in an occurrence of~$312$. Thus~$M_2$ remains at the bottom of the stack (above~$M_1$) until the end of the sorting procedure. The thesis follows by iterating the same argument on each block~$B_i$, for~$i\ge 2$.
\item Suppose, for a contradiction, that there is an element~$j\in\lbrace n-t+1,\dots,n\rbrace$ which is not an ltr-maximum. Note that~$j\neq\pi_1=M_1$ and~$j\neq n=M_t$. Then, by what proved above, $\out{312,\sigma}(\pi)$ contains an occurrence~$jnM_1$ of~$231$, which contradicts the hypothesis that~$\pi$ is~$(312,\sigma)$-sortable.
\end{enumerate}
\end{proof}

\begin{lemma}\label{theorem_ltr-min_pairs_of_pat_132}
Consider the~$(132,\sigma)$-machine, where~$\sigma=\sigma_1\cdots\sigma_{k-1}\sigma_k\in\Perm_k$, with~$\sigma_{k-1}>\sigma_k$ and~$k\ge 3$. Given a permutation~$\pi$ of length~$n$, let~$\pi=m_1 B_1\cdots m_t B_t$ be its ltr-min decomposition. Then:
\begin{enumerate}
\item Every time an ltr-minimum~$m_i$ is pushed into the~$(132,\sigma)$-stack, the~$(132,\sigma)$-stack contains the elements~$m_i,m_{i-1},\dots,m_2,m_1$, reading from top to bottom. Moreover, we have:
$$
\out{132,\sigma}(\pi)=\tilde{B}_1\cdots\tilde{B}_tm_t\cdots m_1,
$$
where~$\tilde{B}_i$ is a suitable rearrangement of~$B_i$.
\item If~$\pi$ is~$(132,\sigma)$-sortable, then~$\tilde{B}_i$ is increasing for each~$i$. Moreover, for each~$i\le t-1$, we have~$B_i>B_{i+1}$ (i.e.~$x>y$ for each~$x\in B_i, y\in B_{i+1}$).
\end{enumerate}
\end{lemma}
\begin{proof}
\begin{enumerate}
\item The proof is analogous to the first part of Lemma~\ref{theorem_ltr-max_pairs_of_pat_312}.
\item Suppose that~$\pi$ is~$(132,\sigma)$-sortable. Suppose, for a contradiction, that~$\tilde{B}_i$ is not increasing, for some~$i$, and let~$xy$ be a descent in~$B_i$. Then, by what proved above, $\out{132,\sigma}(\pi)$ contains an occurrence~$xym_t$ of~$231$, which is impossible. Finally, suppose that~$x>y$, for~$x\in B_i$ and~$y\in B_{i+1}$. Then again~$xym_t$ is an occurrence of~$231$ in~$\out{132,\sigma}(\pi)$, a contradiction.
\end{enumerate}
\end{proof}

Some enumerative data for~$(\sigma,\tau)$-sortable permutations are reported in Table~\ref{table_pairs_of_pat}.

\begin{table}
\centering
\def\arraystretch{1.25}
\begin{tabular}{p{15mm}lcr}
\toprule
$(\sigma,\tau)$ & \textbf{Sequence}~$\lbrace\fsigma{\sigma,\tau}_n\rbrace_n$ & \textbf{Reference} & \textbf{OEIS}\\
\midrule
123,132\newline 123,213\newline 132,312 & 1, 2, 5, 14, 42, 132, 429, 1430 & Section~\ref{section_123_132} & A000108\\
231,321 & 1, 2, 5, 14, 42, 132, 429, 1430 & Remark~\ref{remark_R_circ_sigma_bij} & A000108\\
123,231 & 1, 2, 6, 21, 79, 310, 1252, 5168 &  & \\
123,312 & 1, 2, 5, 15, 51, 188, 731, 2950 & Section~\ref{section_123_312} & A007317\\
123,321 & 1, 2, 4, 7, 14, 28, 56, 112 & Section~\ref{section_123_321_and_132_321} & \\
132,213 & 1, 2, 5, 16, 61, 261, 1206, 5882 &  & \\
132,231 & 1, 2, 6, 22, 90, 394, 1806, 8558 & Section~\ref{section_132_231} & A006318\\
132,321 & 1, 2, 4, 10, 26, 72, 206, 606 & Section~\ref{section_123_321_and_132_321} & A102407\\
213,231 & 1, 2, 6, 23, 101, 483, 2450, 12978 &  & \\
213,312 & 1, 2, 5, 16, 61, 261, 1206, 5882 &  & \\
213,321 & 1, 2, 4, 12, 46, 200, 941, 4677 &  & \\
231,312 & 1, 2, 6, 23, 101, 484, 2471, 13254 &  & \\
312,321 & 1, 2, 4, 10, 28, 85, 274, 925 &  & \\
\bottomrule
\end{tabular}
\caption[Enumerative results for the~$(\sigma,\tau)$-machine.]{Enumerative results for~$(\sigma,\tau)$-sortable permutations, with~$\sigma$ and~$\tau$ patterns of length three and starting from~$(\sigma,\tau)$ sortable permutations of length one.}\label{table_pairs_of_pat}
\end{table}

\section{\texorpdfstring{The~$(132,231)$-machine}{The (132,231)-machine}}\label{section_132_231}

In this section we analyze the~$(132,231)$-machine. We use Lemma~\ref{theorem_ltr-min_pairs_of_pat_132} to provide a geometric description of the~$(132,231)$-sortable permutations, which we exploit to show that~$\Sort(132,231)=\Perm(1324,2314)$. As it is well known, the set~$\Perm(1324,2314)$ is enumerated by the large Schr\"oder numbers.

\begin{lemma}\label{lemma_132-231_equivalence}
Let~$\pi=m_1B_1\cdots m_tB_t$ be the ltr-min decomposition of a permutation~$\pi$. Write~$\out{132,231}(\pi)=\tilde{B}_1\cdots\tilde{B}_tm_t\cdots m_1$ as in Lemma~\ref{theorem_ltr-min_pairs_of_pat_132}. Then~$\tilde{B}_i=\out{12}(B_i)$, for each~$i$.
\end{lemma}
\begin{proof}
Let~$i\ge 1$. Consider the instant when~$m_i$ enters the~$(132,231)$-stack, that is as soon as~$B_i$ is the next block of~$\pi$ to be processed. Here, by Lemma~\ref{theorem_ltr-min_pairs_of_pat_132}, the stack contains~$m_i,m_{i-1},\dots,m_1$, reading from top to bottom. We shall prove that the behavior of the~$(132,231)$-stack on~$B_i$ is equivalent to the behavior of an empty~$12$-stack on input~$B_i$. In other words, the~$(132,231)$-stack performs a pop operation if and only if a~$12$-stack that ignores the current content of the~$(132,231)$-stack does the same. If either the next element of the input is~$m_{i+1}$ or the input is empty, then both the~$(132,231)$-stack and the~$12$-stack perform a pop operation. Otherwise, suppose that the next element of the input is~$y$, for some~$y$ in~$B_i$. Suppose that the~$(132,231)$-stack pops the element~$x\in B_i$. Thus the~$(132,231)$-stack contains two elements~$z,w$, with~$z$ above~$w$, such that~$yzw$ is an occurrence of either~$132$ or~$231$. Note that, since~$z>w$, $z$ is not an ltr-minimum. Therefore~$yz$ is an occurrence of~$12$ and the~$12$-stack performs a pop operation too, as desired. Conversely, suppose that the~$12$-stack pops the element~$x$. Thus the~$12$-stack contains an element~$z$ such that~$z>y$. Therefore~$yzm_i$ is an occurrence of~$231$ and the~$(132,231)$-stack performs a pop operation as well, as desired.
\end{proof}

\begin{corollary}
Let~$\pi=m_1B_1\cdots m_tB_t$ be the ltr-min decomposition of a permutation~$\pi$. The following are equivalent:
\begin{enumerate}
\item $B_i$ avoids~$213$ and~$B_i>B_{i+1}$, for each~$i$.
\item $\pi$ is~$(132,231)$-sortable.
\item $\pi\in\Perm(1324,2314)$.
\end{enumerate}
In particular, we have~$\Sort(132,231)=\Perm(1324,2314)$.
\end{corollary}
\begin{proof}
By Lemma~\ref{theorem_ltr-min_pairs_of_pat_132} and Lemma~\ref{lemma_132-231_equivalence}, we have:
$$
\out{132,231}(\pi)=\out{12}(B_1)\cdots\out{12}(B_t)m_t\cdots m_1.
$$
We will use this decomposition for the rest of the proof.

\begin{itemize}
\item $\left[1\Rightarrow 2\right]$ Suppose, for a contradiction, that~$\out{132,231}(\pi)$ contains an occurrence~$bca$ of~$231$. Note that, since~$c>a$, while~$m_t<\cdots<m_1$, $c$ is not an ltr-minimum of~$\pi$ (and thus neither~$b$ is). Moreover, since we assumed that~$B_i>B_{i+1}$ for each~$i$, $b$ and~$c$ must belong to the same block~$B_j$. Therefore~$\out{12}(B_j)$ is not decreasing and, by Theorem~\ref{theorem_12machine}, $B_j$ contains~$213$, which contradicts the hypothesis.

\item $\left[2\Rightarrow 3\right]$ Suppose, for a contradiction, that~$\pi$ contains~$1324$ or~$2314$. Initially, suppose that~$\pi$ contains an occurrence~$acbd$ of~$1324$. Note that~$b$, $c$ and~$d$ are not ltr-minima of~$\pi$. Let~$b\in B_j$ and~$c\in B_k$, for some~$j\le k$. If~$j=k$, then~$B_j$ contains an occurrence~$bac$ of~$213$. Therefore~$\out{12}(B_j)$ contains an occurrence of~$231$ due to Theorem~\ref{theorem_12machine}, which contradicts the hypothesis. Otherwise, if~$j<k$, then~$bcm_k$ is an occurrence of~$231$ in~$\out{132,231}(\pi)$, which is again impossible. The pattern~$2314$ can be addressed similarly, and it is left to the reader.

\item $\left[3\Rightarrow 1\right]$ Again we argue by contradiction. If~$B_i$ contains an occurrence~$bac$ of~$213$, then~$\pi$ contains an occurrence~$m_ibac$ of~$1324$, which is impossible. Otherwise, if~$\pi$ contains two elements~$x\in B_j$, $y\in B_k$, with~$x<y$ and~$j<k$, then~$m_jxm_ky$ is an occurrence of~$2314$, which is impossible too.
\end{itemize}
\end{proof}

\section{\texorpdfstring{The~$(123,321)$- and~$(132,321)$-machines}{The (123,321)- and (132,321)-machines}}\label{section_123_321_and_132_321}

The pairs of patterns~$(123,321)$ and~$(132,321)$ can be studied with similar tools. We show that, in both cases, sortable permutations avoid~$123$. Therefore the restriction involving the pattern~$321$ is never triggered when processing~$(123,321)$- and~$(132,321)$-sortable permutations. Equivalently, the~$(123,321)$-stack acts as a~$123$-stack when the input is a~$(123,321)$-sortable permutation and thus we can use the results of Chapter~\ref{chapter_pattern123} to describe~$\Sort(123,321)$. In a similar fashion, the~$(132,321)$-stack acts as a~$132$-stack when the input is a~$(123,321)$-sortable permutations, which allows us to use what we proved in Chapter~\ref{chapter_pattern123}.

\begin{lemma}\label{lemma_pairs_avoid_123}
Let~$\sigma\in\lbrace 123,132\rbrace$. If~$\pi$ is~$(\sigma,321)$-sortable, then~$\pi$ avoids~$123$.
\end{lemma}
\begin{proof}
We prove the contrapositive statement by showing that if~$\pi$ contains~$123$, then~$\out{\sigma,321}(\pi)$ contains~$231$. Let~$abc$ be the leftmost\footnote{with respect to the lexicographical order of the indices.} occurrence of~$123$ in~$\pi$. Let us consider the instant when~$c$ is pushed into the~$(\sigma,321)$-stack. If~$a$ is still in the $(\sigma,321)$-stack when~$c$ enters, then~$b$ is not inside the $(\sigma,321)$-stack, since otherwise~$cba$ would be an occurrence of~$321$, which is forbidden. Thus~$\out{\sigma,321}(\pi)$ contains~$bca\simeq 231$ and we are done. Therefore we can assume that~$a$ is extracted before~$c$ enters the~$(\sigma,321)$-stack. Let us consider the instant when~$a$ is extracted from the~$(\sigma,321)$-stack. Let~$z$ be the next element of the input. Then there are two elements~$x,y$ into the stack, with~$x$ below~$y$, such that~$zyx$ is an occurrence of either~$\sigma$ or~$321$. Note that necessarily~$x$ comes before~$a$ in~$\pi$ (it could be~$y=a$). If~$zyx\simeq 321$, then~$xyz$ is an occurrence of~$123$ which precedes~$abc$, which is impossible due to our choice of~$abc$. Therefore we have~$zyx\simeq\sigma$. Next we show that~$y=a$. Indeed, suppose, for a contradiction, that~$y\neq a$. Note that both~$x$ and~$y$ are greater than~$a$, since otherwise we would have an occurrence~$xac$ or~$yac$ of~$123$ which precedes~$abc$, a contradiction. But then the~$(\sigma,321)$-stack would contain an occurrence~$xya$ of~$\sigma$, which is impossible. We can now assume~$y=a$. Note that~$x>a$ due to our choice of~$abc$. Thus~$\sigma\neq 132$, since~$zax$ is an occurrence of~$\sigma$ and~$x>a$. This completes the proof for the case~$\sigma=132$. Let us now assume that~$\sigma=123$ for the remaining part of the proof. Now, $zax\simeq 132$, thus~$c\neq z$. If~$z$ is still in the~$(123,321)$-stack when~$c$ enters, then~$\out{123,321}(\pi)$ contains an occurrence~$acz$ of~$231$, as desired. Otherwise, consider the instant when~$z$ is extracted from the~$(123,321)$-stack, with~$c$ still in the input. Let~$w$ be the next element of the input. Then there are two elements~$u,v$ into the stack, with~$u$ below~$v$, such that~$wvu$ is an occurrence of either~$123$ or~$321$. Observe that, since~$z$ is the next element of the input when~$a$ is extracted, the elements~$u$ and~$v$ precede~$a$ in~$\pi$ (otherwise they would have been extracted from the~$(123,321)$-stack before~$a$). Therefore it cannot be~$wvu\simeq 321$, due to our choice of~$abc$ as leftmost occurrence of~$123$ in~$\pi$. Finally, if~$wvu\simeq 123$, then we can repeat the same argument on the triple~$wvu$, in place of~$zyx$. Since~$wvu$ is strictly to the left of~$zyx$, iterating this process will sooner or later result in either a contradiction or in finding an occurrence of~$231$ in the output of the~$(123,321)$-stack, as desired.
\end{proof}

\begin{corollary}\label{corollary_sigma_321_char}
We have:
$$
\Sort(123,321)=\Sort(123)\cap\Perm(123)
$$
and
$$
\Sort(132,321)=\Sort(132)\cap\Perm(123).
$$
\end{corollary}
\begin{proof}
Let~$\sigma\in\lbrace 123,132\rbrace$. Due to Lemma~\ref{lemma_pairs_avoid_123}, any~$(\sigma,321)$-sortable permutation avoids~$123$. Therefore the behavior of a~$(\sigma,321)$-stack on a~$(\sigma,321)$-sortable permutation is equivalent to the behavior of a~$\sigma$-stack, since the restriction involving the pattern~$321$ is never triggered.
\end{proof}

Let us now focus on the pair~$(123,321)$. Since~$\Sort(123,321)=\Sort(123)\cap\Perm(123)$, we will describe this set by exploiting the characterization of~$\Sort(123)$ provided in Chapter~\ref{chapter_pattern123}. 

\begin{theorem}
Let~$f_n=\fsigma{123,321}_n$ be the number of~$(123,321)$-sortable permutations of length~$n$. Then:
\begin{equation*}
\begin{cases}
f_1=1;\\
f_2=2;\\
f_3=4;\\
f_n=7\cdot 2^{n-4},\ n\ge 4.
\end{cases}
\end{equation*}
\end{theorem}
\begin{proof}
Each permutation of length one and two is~$(123,321)$-sortable, while~$132$ and~$123$ are the only permutations of length three that are not~$(123,321)$-sortable. Let~$n\ge 4$ and let~$\pi\in\Sort(123,321)$. Due to Corollary~\ref{corollary_sigma_321_char}, we have~$\Sort(123,321)=\Sort(123)\cap\Perm(123)$. Therefore~$\pi$ can be uniquely constructed according to the procedure described in Section\ref{section_enumeration_123}, which we recall below, as long as occurrences of~$123$ are not created:
\begin{enumerate}
\item Choose~$\alpha=\alpha_1\alpha_2\cdots\alpha_k\in\Perm_k (213)$, with~$\alpha_1=k\ge 2$;
\item add~$h$ new maxima, $k+1,\dots k+h$, one at a time, using the bijection~$\varphi$ of Theorem~\ref{theorem_123_bijection};
\item add~$n-k-h$ consecutive ascents at the beginning, by inflating the first element of the permutation, according to Corollary~\ref{corollary_123_infl_cor}.
\end{enumerate}

Observe that it must be~$k\ge n-1$. Otherwise, if~$k<n-1$, then it would be~$k+h\ge 2$, which would necessarily result in an occurrence of~$123$ in~$\pi$. Thus we have either~$k=n$ or~$k=n-1$. We distinguish the following cases.

\begin{itemize}
\item If~$k=n$, then~$\pi=\alpha$ is a permutation of~$\Perm_n(213,123)$ with~$\pi_1=n$. Notice that~$\pi$ is uniquely obtained by adding the initial maximum to a permutation in~$\Perm_{n-1}(213,123)$. In other words, this operation yields a bijection between the set of~$(123,321)$-sortable permutations of length~$n$ that start with their maximum value and~$\Perm_{n-1}(213,123)$. It is well known that~$|\Perm_{n-1}(213,123)|=2^{n-2}$.

\item Suppose that~$k=n-1$ and~$h=1$, that is~$\pi$ is obtained by adding a new maximum, immediately after~$n-2$, to some~$\alpha\in\Perm_{n-1}(213,123)$, with~$\alpha_1=n-1$. Then it must be necessarily~$\alpha_2=n-2$, otherwise~$\alpha_2n-2n$ would be an occurrence of~$123$ in~$\pi$. Therefore~$\alpha\in\Perm_{n-1}(213,123)$, with~$\alpha_1=n-1$ and~$\alpha_2=n-2$. Similarly to the previous item, removing the first two elements of~$\alpha$ yields a bijection with~$\Perm_{n-3}(213,123)$, and~$|\Perm_{n-3}(213,123)|=2^{n-4}$.

\item Finally, suppose that~$k=n-1$ and~$h=0$, that is~$\pi$ is obtained from some~$\alpha\in\Perm_{n-1}(213,123)$, with~$\alpha_1=n-1$, by inflating the first element of~$\alpha$ by one. Since this operation cannot create an occurrence of~$123$, each~$\alpha\in\Perm_{n-1}(213,123)$ such that~$\alpha_1=n-1$ is a suitable choice, and~$|\Perm_{n-1}(213,123)|=2^{n-3}$.
\end{itemize}

At the end we obtain:
$$
f_n=2^{n-2}+2^{n-4}+2^{n-3}=7\cdot 2^{n-4},
$$
as desired.
\end{proof}

Next we consider the set~$\Sort(132,321)$. We shall define a bijection between~$(132,321)$-sortable permutations of length~$n$ and Dyck paths of semilength~$n$ that avoid the (consecutive) pattern~$\D\U\D\U$. More precisely, we will describe a bijection between two generating trees for these families. First we refine the geometric description of~$\Sort(132,321)$ by polishing the characterization of~$\Sort(132)$ provided in Chapter~\ref{chapter_pattern132}. Recall from Section~\ref{section_mesh_132} that~$\mu=(132,\left\lbrace(0,2),(2,0),(2,1)\right\rbrace)$ is the mesh pattern depicted in Figure~\ref{figure_mesh_pattern_132}.

\begin{theorem}\label{theorem_132_321_char}
We have:
$$
\Sort(132,321)=\Perm(\mu,123).
$$
\end{theorem}
\begin{proof}
It follows from Theorem~\ref{corollary_mesh_pattern_char_132} and Corollary~\ref{corollary_sigma_321_char}, since~$123\le 2314$.
\end{proof}

The grid decomposition of a permutation~$\pi$ was introduced in Section~\ref{section_grid_132}. Let~$\pi=m_1B_1m_2B_2\dots m_tB_t$ be the ltr-min decomposition of~$\pi$. Recall that:

\begin{itemize}
\item for~$i\ge 1$, the~$j-th$ \textit{vertical strip} of~$\pi$ is~$B_j$;
\item for~$i\ge 1$, the~$i-th$ \textit{horizontal strip} of~$\pi$ is~$H_i=\left\lbrace x\in\pi: m_{i-1}<x< m_i\right\rbrace$, where~$m_0=+\infty$.
\item for any pair of indices~$i,j$, the \textit{cell} of indices~$i,j$ of~$\pi$ is~$C_{i,j}=H_i\bigcap B_j$ (note that~$C_{i,j}$ is empty for each~$i>j$).
\end{itemize}

\begin{proposition}\label{proposition_cell_at_most_one_element_132_321}
Let~$\pi$ be a~$(132,321)$-sortable permutation. Then each cell~$C_{i,j}$ contains at most one element.
\end{proposition}
\begin{proof} Suppose, for a contradiction, that the cell~$C_{i,j}$ contains at least two elements and let~$C_{i,j}=xy\cdots$. If~$x<y$, then~$m_ixy$ is an occurrence of~$123$, which contradicts Theorem~\ref{theorem_132_321_char}. Otherwise, if~$x>y$, then due to Lemma~\ref{lemma_inversion_in_a_cell_132} there has to be some element~$z$ between~$x$ and~$y$ in~$\pi$ such that~$z<m_i$. Let~$m_j$ be the ltr-minimum of the horizontal strip that contains~$z$, for some~$j\ge i$. Notice that~$m_j\neq z$, since~$x$ and~$y$ are in the same cell and~$z$ is placed between~$x$ and~$y$ in~$\pi$. Then~$m_jzy$ is an occurrence of~$123$ in~$\pi$, which is again impossible.
\end{proof}

\begin{proposition}\label{proposition_active_cell_132_321}
Let~$\pi$ be a~$(132,321)$-sortable permutation. Suppose that the cell~$C_{i,j}$ is not empty. Then the cell~$C_{u,v}$ is empty for each pair of indices~$u,v$ such that~$i<u$ and~$j\le v$.
\end{proposition}
\begin{proof}
Suppose that~$x\in C_{i,j}$ and~$y\in C_{u,v}$, with~$i<u$ and~$j\le v$. Then~$m_jxy$ is an occurrence of~$123$ in~$\pi$, which is impossible due to Theorem~\ref{theorem_132_321_char}.
\end{proof}

\begin{proposition}\label{proposition_decr_blocks_132_321}
Let~$\pi$ be a~$(132,321)$-sortable permutation. Then~$B_i>B_{i+1}$ for each pair of consecutive vertical strips~$B_i,B_{i+1}$.
\end{proposition}
\begin{proof}
It follows from Lemma~\ref{lemma_ltr_min_dec_132} and Theorem~\ref{theorem_132_321_char}.
\end{proof}

Now, any prefix of a~$(132,321)$-sortable permutation is~$(132,321)$-sortable by Lemma~\ref{lemma_prefix_sortable}. Therefore every permutation of~$\Sort_n(132,321)$ can be uniquely constructed by inserting a new rightmost integer~$a\in\lbrace 1,\dots,n\rbrace$ in a permutation~$\pi\in\Sort_{n-1}(132,321)$, and suitably rescaling the other elements\footnote{i.e. adding~$1$ to each integer~$b$ such that~$b\ge a$.}. Every permutation obtained this way from~$\pi$ is said to be a \textit{child} of~$\pi$. Now, due to Propositions~\ref{proposition_cell_at_most_one_element_132_321} and \ref{proposition_decr_blocks_132_321}, there is at most one way to insert a new rightmost element in each cell of the last vertical strip of a~$(132,321)$-sortable permutation. Indeed each cell can contain at most one element due to Proposition~\ref{proposition_cell_at_most_one_element_132_321} and the value of this element is completely determined due to Proposition~\ref{proposition_decr_blocks_132_321}: it has to be less than all the other elements in the same horizontal strip. Finally, some choices will be forbidden due to Theorem~\ref{theorem_132_321_char}. Given a~$(132,321)$-sortable permutation~$\pi$, where~$\pi=m_1B_1\cdots m_tB_t$ is the usual ltr-min decomposition, a cell~$C_{i,t}$ is said to be \textit{active} if inserting a new rightmost element in the cell~$C_{i,t}$ yields a~$(132,321)$-sortable permutation. Next we characterize precisely which cells are active. First we introduce a new parameter.

Let~$\pi=m_1B_1\cdots m_tB_t$. Let~$y$ be the rightmost element of~$\pi$ which is not an ltr-minimum and suppose that~$y\in C_{i,j}$, for some~$i,j$. Then the \textit{depth} of~$\pi$ is~$\depth(\pi)=t-i+1$. If~$\pi=m_1\cdots m_t$ is the decreasing permutation, we assume~$\depth(\pi)=t$.

\begin{theorem}\label{theorem_gener_permutation_132_321}
Let~$\pi=m_1 B_1\cdots m_tB_t$ a~$(132,321)$-sortable permutation and let~$d=\depth(\pi)$.
\begin{enumerate}
\item If~$B_t$ is not empty, then the cell~$C_{i,t}$ is active if and only if~$i<d$. In this case~$\pi$ has~$d$ children.
\item If~$B_t$ is empty, then the cell~$C_{i,t}$ is active if and only if~$i\le d$. In this case~$\pi$ has~$d+1$ children.
\end{enumerate}
\end{theorem}
\begin{proof}
If~$\pi$ is the decreasing permutation, then~$\depth(\pi)=t$ by definition. Moreover~$B_t$ is empty and the cell~$C_{i,t}$ is active for each~$i=1,\dots,d-1$ due to Theorem~\ref{theorem_132_321_char}. Since inserting a new rightmost minimum is always allowed, in this case~$\pi$ has~$d$ children.

Otherwise, let~$y$ be the rightmost element of~$\pi$ which is not an ltr-minimum and suppose that~$y\in C_{u,v}$, for some~$u,v$, with~$d=t-u+1$. Note that every cell~$C_{i,t}$, with~$i>d$, is not active. Indeed inserting a new rightmost element~$a\in C_{i,t}$, with~$i>d$, creates an occurrence~$m_vya$ of~$123$, and this is forbidden due to Theorem~\ref{theorem_132_321_char}. On the other hand, we shall prove that each cell~$C_{i,t}$, with~$i<d$, is active. Due to the same Theorem~\ref{theorem_132_321_char}, it is sufficient to show that inserting a new rightmost element~$a$ in the cell~$C_{i,t}$, with~$i<d$, cannot create an occurrence of either~$123$ or~$\mu$. Suppose that~$\pi_{j_1}\pi_{j_2}a\simeq 123$, for some indices~$j_1<j_2$. Note that~$\pi_{j_2}$ is not an ltr-minimum of~$\pi$, therefore it precedes~$y$ in~$\pi$. Moreover we have~$\pi_{j_2}<y$, since~$y>a>\pi_{j_2}$. Thus~$\pi_{j_1}\pi_{j_2}y$ is an occurrence of~$123$ in~$\pi$, which contradicts the fact that~$\pi$ is~$(132,321)$-sortable. The pattern~$\mu$ can be treated analogously, so we leave the details to the reader.

Finally, we wish to prove that the cell~$C_{t,d}$ is active if and only if~$B_t$ is empty. If~$B_t$ is not empty, then~$C_{t,d}$ is not active due to Proposition~\ref{proposition_active_cell_132_321}. If instead~$B_t$ is empty, then inserting~$a$ in~$C_{t,h}$ cannot create an occurrence of~$123$. This can be proved by repeating the same argument used above. Instead, suppose that~$a$ creates an occurrence~$\pi_{j_1}\pi_{j_2}a$ of~$132$. Then, since~$B_t$ is empty, $\pi_{j_1}\pi_{j_2}m_t a$ is an occurrence of~$2143$, and thus~$\pi_{j_1}\pi_{j_2}a$ is not an occurrence of~$\mu$, as desired.
\end{proof}

Theorem~\ref{theorem_gener_permutation_132_321} allows us to define a generating tree for~$\Sort(132,321)$. The node corresponding to the~$(132,321)$-sortable permutation~$\pi$, with~$\pi=m_1B_1\cdots m_tB_t$, is equipped with two labels~$(d,b)$, where:
\begin{itemize}
\item $d=\depth(\pi)$ is the depth of~$\pi$.
\item $b$ is a boolean counter with value~$b=0$, if~$B_t$ is empty, and~$b=1$, otherwise.
\end{itemize}

According to Theorem~\ref{theorem_gener_permutation_132_321}, the following rule provides a generating tree for~$\Sort(132,321)$:

\begin{equation*}\label{equation_gen_tree_132_321}
\Omega_1:
\begin{cases}
(1,0)\\
(d,0)\longrightarrow (d+1,0)(1,1)(2,1)\cdots(d-1,1)(d,1)\\
(d,1)\longrightarrow (d+1,0)(1,1)(2,1)\cdots(d-1,1)
\end{cases}
\end{equation*}

We shall prove that the above tree is a generating tree for Dyck paths avoiding~$\D\U\D\U$ as well. Let us consider the generating tree for Dyck paths described in Example~\ref{example_dyck_paths_new_peak}. In this tree, the children of a given Dyck path~$P$ are obtained by inserting a new peak~$\U\D$ either before a~$\D$ step of the last descending run or at the end of~$P$. To obtain a generating tree for~$\D\U\D\U$-avoiding paths, we have to remove all the children where this operation creates an occurrence of the forbidden pattern~$\D\U\D\U$. More precisely, let~$k$ be the length of the last descending run of~$P$ and write~$P=p_1\cdots p_{i-1}p_i\D^k$, where~$p_i$ is the rightmost~$\U$ step of~$P$. If~$p_{i-1}=\U$, then inserting a new peak is always allowed. Otherwise, if~$p_{i-1}=\D$, and thus~$p_{i-1}p_i=\D\U$, then we cannot insert a new peak immediately before~$p_{i+2}$, since~$p_{i-1}p_ip_{i+1}\U$ would be an occurrence of~$\D\U\D\U$. We then assign to each Dyck path~$P$ two labels~$(k,b)$, where:
\begin{itemize}
\item $k$ is the length of the last descending run of~$P$;
\item $b$ is a boolean counter with value~$0$, if the step that precedes the last~$\U$ step of~$P$ is~$\U$, and~$1$ otherwise. We assume~$b=0$ for the path~$\U\D$.
\end{itemize}

According to what observed above, the following is a generating tree for Dyck paths avoiding~$\D\U\D\U$:

\begin{equation*}\label{equation_gen_tree_DUDU}
\Omega_2:
\begin{cases}
(1,0)\\
(k,0)\longrightarrow (k+1,0)(1,1)(2,1)\cdots(k-1,1)(k,1)\\
(k,1)\longrightarrow (k+1,0)(1,1)(2,1)\cdots(k-1,1)
\end{cases}
\end{equation*}

\begin{corollary}
The number of~$(132,321)$-sortable permutation of length~$n$ is equal to the number of Dyck paths of semilength~$n$ avoiding~$\D\U\D\U$.
\end{corollary}
\begin{proof}
The rules~$\Omega_1$ and~$\Omega_2$ are identical.
\end{proof}

An example of the bijection between~$\Sort(132,321)$ and the set of~$\D\U\D\U$-avoiding Dyck paths induced by the rules~$\Omega_1$ and~$\Omega_2$ is illustrated in Figure~\ref{figure_gen_trees_DUDU_132_321}.

\begin{figure}
\centering
\begin{tikzpicture}[scale=0.75]
\draw [help lines] (3,1) grid (4,5);
\draw [help lines] (2,2) grid (3,5);
\draw [help lines] (1,3) grid (2,5);
\draw [help lines] (0,4) grid (1,5);
\node[scale=1] at (0,4) {$\bullet$};
\node[scale=1] at (1,3) {$\bullet$};
\node[scale=1] at (2,2) {$\bullet$};
\node[scale=1] at (3,1) {$\bullet$};
\node[scale=1] at (0.5,4.66) {$\bullet$};
\node[scale=1] at (1.5,4.33) {$\bullet$};
\node[scale=1] at (3.5,2.5) {$\bullet$};
\node[scale=1] at (0.5,5.5) {$B_1$};
\node[scale=1] at (1.5,5.5) {$B_2$};
\node[scale=1] at (2.5,5.5) {$B_3$};
\node[scale=1] at (3.5,5.5) {$B_4$};
\node[scale=1] at (-0.5,4.5) {$H_1$};
\node[scale=1] at (-0.5,3.5) {$H_2$};
\node[scale=1] at (-0.5,2.5) {$H_3$};
\node[scale=1] at (-0.5,1.5) {$H_4$};
\node[scale=1] at (4.5,4.5) {$\times$};
\node[scale=1] at (4.5,3.5) {{$\times$}};
\node[scale=1] at (4.5,2.5) {{$\checkmark$}};
\node[scale=1] at (4.5,1.5) {$\checkmark$};
\node[scale=1] at (4.5,0.5) {$\checkmark$};
\end{tikzpicture}
\hfill
\begin{tikzpicture}[scale=0.5]
\fillPath{1,-1,1,1,-1,1,1,1,-1,-1,1,-1,-1,-1}{0}{0}
\node[scale=1] at (11,3.5) {$\checkmark$};
\node[scale=1] at (12,2.5) {$\times$};
\node[scale=1] at (13,1.5) {$\checkmark$};
\node[scale=1] at (14,0.5) {$\checkmark$};
\end{tikzpicture}
\caption[The grid decomposition of a~$(132,321)$-sortable permutations and the corresponding Dyck path.]{On the left, the grid decomposition of the~$(132,321)$-sortable permutation~$\pi=6837214$, where~$\depth(\pi)=3$ and the last block~$B_4$ is not empty. On the right, the corresponding Dyck path~$P$, with~$k=3$ and~$b=1$. Active sites are marked with~``$\checkmark$", while the symbol~``$\times$" marks those sites that are not active.}\label{figure_gen_trees_DUDU_132_321}
\end{figure}
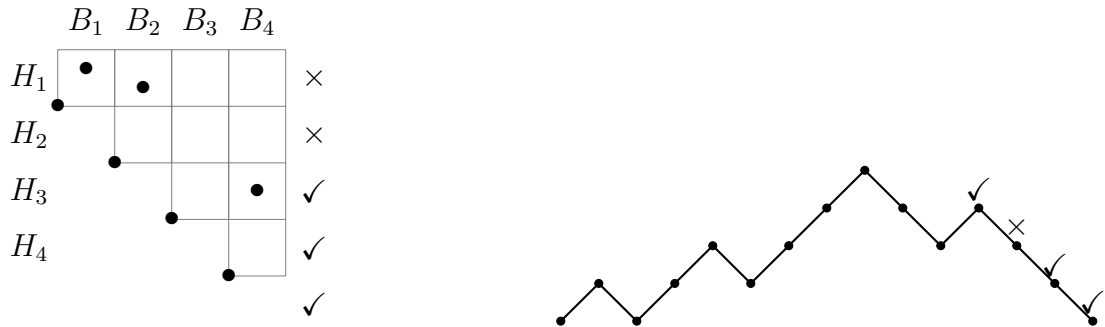

\section{\texorpdfstring{The~$(123,132)$-machine}{The (123,132)-machine}}\label{section_123_132}

In this section we discuss the~$(123,132)$-machine. In~\cite{BCKV}, the authors show that permutations in~$\Sort(123,132)$ are characterized by the avoidance of four (generalized) patterns of length four. Then they prove that the distribution of the first element in~$\Sort(123,132)$ is given by the well known Catalan triangle (sequence~A009766 in~\cite{Sl}). An immediate consequence is that~$(123,132)$-sortable permutations are enumerated by the Catalan numbers. In this thesis we follow an alternative path. We first provide a decomposition lemma for~$(123,132)$-sortable permutations, then we collect several geometric properties of~$\Sort(123,132)$ that lead, towards a step by step procedure, to the same enumerative result.

Although Lemma~\ref{theorem_ltr-min_pairs_of_pat_132} does not apply to~$\Sort(123,132)$, it is still useful to look at the ltr-min decomposition of~$(123,132)$-sortable permutations.

\begin{lemma}\label{lemma_123_132_decomp}
Let~$\pi=m_1B_1\cdots m_tB_t$ be the ltr-min decomposition of the permutation~$\pi$. Then:
\begin{enumerate}
\item $\out{123,132}(\pi)=\tilde{B}_1\tilde{B}_2m_2\tilde{B}_3m_3\cdots\tilde{B}_km_km_1$, where~$\tilde{B}_i$ is a rearrangement of~$B_i$, for each~$i$.
\item If~$\pi$ is~$(123,132)$-sortable, then~$B_i>B_{i+1}$\footnote{That is~$x>y$ for each~$x\in B_i, y\in B_{i+1}$.} for each~$i=1,\dots,t-1$.
\item If~$\pi$ is~$(123,132)$-sortable, then~$B_1$ is increasing and~$\tilde{B}_1=\reverse(B_1)$ is the reverse of~$B_1$.
\item If~$\pi$ is~$(123,132)$-sortable, then~$x<m_{i-1}$ for each~$x\in B_i$ and~$i\ge 3$. Moreover, we have~$\tilde{B}_i=\out{12}(B_i)$.
\end{enumerate}
\end{lemma}
\begin{proof}
\begin{enumerate}
\item Consider the action of the~$(123,132)$-stack on the input permutation~$\pi$. Note that the element~$m_1$ remains at the bottom of the $(123,132)$-stack until the end of the sorting procedure. Now, since~$m_2xm_1$ is an occurrence of~$132$ for each~$x$ in~$B_1$, all the elements in~$B_1$ are extracted from the~$(123,132)$-stack before~$m_2$ enters. Then the element~$m_2$ can never play the role of~$2$ in either~$123$ or~$132$, since~$m_2<m_1$ and~$m_1$ is the only element below~$m_2$ in the~$(123,132)$-stack. Therefore~$m_2$ is never involved in any occurrence of either~$123$ or~$132$ with~$m_1$ and some elements of~$B_2$. By repeating the same argument, we deduce that the block~$B_2$ has to be extracted from the~$(123,132)$-stack before~$m_3$ enters ($m_3xm_2\simeq 132$ for each~$x\in B_2$). Then~$m_2$ is extracted too, since~$m_3m_2m_1\simeq 123$. Next, $m_3$ is pushed above~$m_1$. The thesis follows by iterating the same argument on the remaining blocks.

\item Suppose, for a contradiction, that there are two elements~$x\in B_i$ and~$y\in B_{i+1}$ such that~$x<y$. Then, due to what proved in the previous item, $\out{123,132}(\pi)$ contains an occurrence~$xym_t$ of~$231$, contradicting the fact that~$\pi$ is~$(123,132)$-sortable.

\item Suppose, for a contradiction, that~$B_1$ is not increasing. Write~$B_1=x_1\cdots x_kx_{k+1}\cdots$, where~$x_k>x_{k+1}$ is the leftmost descent of~$B_1$. Then the elements~$x_1,\dots,x_k,x_{k+1}$ are pushed into the~$(123,132)$-stack and~$S^{123,132}(\pi)$ contains an occurrence~$x_{k+1}x_km_1$ of~$231$, contradicting the fact that~$\pi$ is~$(123,132)$-sortable. Thus~$B_1$ is increasing. In particular, each element of~$B_1$ can be pushed into the~$(123,132)$-stack, so that~$\tilde{B}_1=\reverse(B_1)$.

\item Let~$i\ge 3$. If~$B_i$ contains an element~$x>m_{i-1}$, then~$\out{123,132}(\pi)$ contains an occurrence~$m_{i-1}xm_i$ of~$231$, a contradiction with~$\pi$~$(123,132)$-sortable. Finally, we show that~$\tilde{B}_i=\out{12}(B_i)$. Consider the instant when~$m_i$ is pushed into the~$(123,132)$-stack, that is as soon as the first element of~$B_i$ is the next one to be processed. As a consequence of what said in the first item of this lemma, at this step the~$(123,132)$-stack contains~$m_i m_1$, with~$m_i$ above~$m_1$. We show that, on~$B_i$, the behavior of the~$(123,132)$-stack is equivalent to the behavior of a~$12$-stack that ignores~$m_i m_1$. Suppose that the~$12$-stack performs a pop operation. In other words, the restriction of the~$12$-stack is triggered by an occurrence~$y_2y_1$ of~$12$, where~$y_2$ is the next element of the input and~$y_1$ is in the~$12$-stack. Notice that~$m_1>x$ for each~$x\in B_i$, since we have just proved that~$x<m_{i-1}$ and obviously~$m_{i-1}<m_1$. Thus~$y_2y_1m_1\simeq 123$ and the~$(123,132)$-stack performs a pop operation too. Conversely, suppose that the~$(123,132)$-stack performs a pop operation, that is the restriction of the~$(123,132)$-stack is triggered by an occurrence~$y_3y_2y_1$ of either~$123$ or~$231$, where~$y_3$ is the next element of the input. If~$y_3y_2y_1\simeq 123$, then~$y_2\neq m_i$, since~$m_1>m_2$ and~$y_3<y_2$. Therefore~$y_2$ and~$y_3$ are elements of~$B_i$ and the restriction of the~$12$-stack is triggered by~$y_3y_2\simeq 12$, as desired. Otherwise, suppose that~$y_3y_2y_1\simeq 132$. Note that both~$m_1$ and~$m_i$ cannot be part of this occurrence. Indeed~$m_1\neq y_1$, since~$m_1>x$ for each~$x\in B_i$, and~$m_i\neq y_1,y_2$ since~$m_i<x$ for each~$x\in B_i$. Therefore~$y_3<y_1$ are elements of~$B_i$ that trigger the restriction of the~$12$-stack. This completes the proof.
\end{enumerate}
\end{proof}

What proved so far determines the structure of a~$(123,132)$-sortable permutation~$\pi=m_1B_1\cdots m_tB_t$, except for the second block~$B_2$. Indeed, since~$B_1$ is increasing and~$B_i>B_{i+1}$ for each~$i$, then~$B_1$ contains the biggest elements of~$\pi$ in increasing order. Then, for each~$i\ge 3$, the block~$B_i$ is a~$213$-avoiding permutation due to Theorem~\ref{theorem_12machine}. Moreover, since~$m_i<x< m_{i-1}$ for each~$x\in B_i$, elements in~$m_iB_i$ are strictly greater than elements in~$m_{i+1}B_{i+1}$. This guarantees that occurrences of~$213$ are not created if we juxtapose~$m_iB_i$ and~$m_{i+1}B_{i+1}$. As a consequence, we can equivalently say that~$m_3B_3\cdots m_tB_t$ is a~$213$-avoiding permutation. Next we describe the structure of the remaining block~$B_2$. For the rest of this section, let~$\xi=(132,\lbrace 0,2\rbrace,\emptyset)$ be the bivincular pattern depicted in Figure~\ref{figure_bivincular_pattern_132_0_2}. A geometric description of the set~$\Perm(\xi)$ was provided in Section~\ref{section_bivincular_result}. Recall that an occurrence of~$\xi$ in a permutation~$\pi=\pi_1\cdots\pi_n$ is a descent~$\pi_i>\pi_{i+1}$ such that~$\pi_{i+1}>\pi_1$. Finally, denote by~$*213$ an occurrence of either~$4213$, $3214$, $2314$, or~$1324$. Note that for any permutation~$\pi=\pi_1\cdots\pi_n$, we have~$\pi\in\Perm(*213)$ if and only if~$\pi_2\cdots\pi_n$ avoids~$213$.

\begin{lemma}\label{lemma_second_block_123_132}
Let~$\pi=m_1m_2B_2$ be a permutation with two ltr-minima and such that the first block in the ltr-min decomposition is empty. Then~$\pi$ is~$(123,132)$-sortable if and only if~$\pi\in\Perm(\xi,*213)$.
\end{lemma}
\begin{proof}
By Lemma~\ref{lemma_123_132_decomp} we have~$\out{123,132}(\pi)=\tilde{B_2}m_2m_1$. Suppose initially that~$\pi$ is~$(123,132)$-sortable. We wish to prove that~$\pi$ avoids~$\xi$ and~$m_2B_2$ avoids~$213$. Suppose, for a contradiction, that~$\pi$ contains an occurrence~$m_1\pi_i\pi_{i+1}$ of~$\xi$. If~$\pi_{i+1}$ enters the~$(123,132)$-stack before~$\pi_i$ is extracted, then~$\out{123,132}(\pi)$ contains an occurrence~$\pi_{i+1}\pi_i m_2$ of~$231$, a contradiction with~$\pi$ being~$(123,132)$-sortable. Therefore~$\pi_i$ must be extracted when~$\pi_{i+1}$ is the next element of the input. Let~$m_1m_2x_1\cdots x_l\pi_i$ be the elements contained in the~$(123,132)$-stack at this point. Notice that it has to be~$x_1<x_2<\cdots<x_l<c$, otherwise~$\out{123,132}(\pi)$ would contain an occurrence of~$231$ (with~$m_2$ playing the role of~$1$), which is impossible. Now, starting from the top of the stack (which is~$\pi_i>\pi_{i+1}$) and going down, consider the last element~$y$ such that~$y>\pi_{i+1}$. Since~$m_2<\pi_{i+1}$, it must be either~$y=\pi_i$ or~$y=x_j$, for some~$j$. In any case, $\pi_{i+1}$ enters above~$y$ and~$\out{123,132}(\pi)$ contains an occurrence~$\pi_{i+1}ym_2$ of~$231$, which is again a contradiction.\\
Otherwise, suppose, for a contradiction, that~$m_2B_2$ contains an occurrence~$\pi_i\pi_j\pi_k$ of~$213$. Notice that~$i>2$, since~$m_2=1$. If~$\pi_j<m_1$, then~$\pi_j\pi_im_1$ is an occurrence of either~$321$ or~$231$. Thus~$\pi_i$ has to be extracted before~$\pi_j$ enters the~$(123,132)$-stack. But this results in an occurrence~$pi_i\pi_jm_2$ of~$231$ in~$\out{123,132}(\pi)$, which contradicts the fact that~$\pi$ is~$(123,132)$-sortable. Therefore we can suppose~$\pi_j>m_1$. If~$j=i+1$, then~$\pi_1\pi_i\pi_{i+1}$ is an occurrence of~$\xi$ and we are back to the previous case. Instead, if~$j>i+1$, consider the element~$\pi_{i+1}$. If~$\pi_{i+1}<\pi_j$, then we repeat the same argument replacing~$\pi_j$ with~$\pi_{i+1}$. Finally, if~$\pi_{i+1}>\pi_j$, we repeat the same argument replacing~$\pi_i$ with~$\pi_{i+1}$. Sooner or later this will lead to a contradiction.

Conversely, suppose that~$\pi$ is not~$(123,132)$-sortable. Equivalently, let~$bca$ be an occurrence of~$231$ in~$\out{123,132}(\pi)=\tilde{B}_2m_2m_1$. Note that~$m_2\neq b,c$, since~$m_2<m_1$. Therefore~$b$ and~$c$ are elements of~$\tilde{B}_i$. We distinguish two cases. Initially, suppose that~$b$ precedes~$c$ in~$\pi$ (and in~$\out{123,132}(\pi)$ as well). Therefore~$b$ is extracted from the~$(123,132)$-stack before~$c$ enters. Let~$y$ be the next element of the input when~$b$ is extracted. If~$y<b$, then~$byc$ is an occurrence of~$213$ in~$m_2B_2$, as desired. Otherwise, let~$y>b$. Since the restriction of the~$(123,132)$-stack is triggered by~$y$, there are two elements into the stack~$u,v$, with~$u$ below~$v$, such that~$yvu$ is an occurrence of either~$123$ or~$132$. Since we assumed~$y>b$, we have~$bvu\simeq yvu$ and thus~$bvu$ is an occurrence of either~$123$ or~$132$ inside the~$(123,132)$-stack, which is impossible. Finally, suppose that~$c$ precedes~$b$ in~$\pi$. Since~$b$ is extracted before~$c$, $c$ is still inside the~$(123,132)$-stack when~$b$ enters. This implies that~$b>m_1$, otherwise~$bcm_1$ would be an occurrence of either~$123$ (if~$m_1>c$) or~$132$ (if~$m_1<c$), which is impossible. Similarly, for each element~$x$ between~$c$ and~$b$ in~$\pi$ we have~$x>m_1$. Now, if~$c$ and~$b$ are consecutive in position, then~$m_1cb$ is an occurrence of~$\xi$, as wanted. Otherwise, let~$x$ be the element immediately after~$c$ in~$\pi$. If~$x<c$, then~$m_1cx$ is the desired occurrence of~$\xi$. Otherwise, if~$c<x$ (and thus~$x>b$), we can repeat the same argument using~$x$ instead of~$c$.
\end{proof}

\begin{lemma}\label{lemma_first_block_123_132}
Let~$\pi$ be a permutation of length~$n$. Let~$\pi=m_1m_2B_2m_3B_3\cdots m_tB_t$ be the ltr-min decomposition of~$\pi$ and suppose that~$B_1$ is empty. For~$k\ge 1$, define:
$$
\bar{\pi}=m_1(n+1)(n+2)\dots(n+k)m_2B_2\cdots m_tB_t.
$$
Then~$\pi$ is~$(123,132)$-sortable if and only if~$\bar{\pi}$ is~$(123,132)$-sortable.
\end{lemma}
\begin{proof}
Let~$k\ge 1$ and suppose that~$\pi$ is~$(123,132)$-sortable. Consider the action of the~$(123,132)$-stack on input~$\bar{\pi}$. Since~$m_1<n+1<n+2<\cdots<n+k$, the prefix of~$\bar{\pi}$ up to~$n+k$ is pushed into the~$(123,132)$-stack. Then, since~$m_2(n+1)m_1\simeq 132$, the elements~$n+k,n+k-1,\dots,n+1$ are extracted from the~$(123,132)$-stack. Therefore we have:
$$
\out{123,132}(\bar{\pi})=(n+k)(n+k-1)\dots (n+1)\out{123,132}(\pi).
$$
It is easy to realize that~$\out{123,132}(\pi)$ contains~$231$ if and only if~$\out{123,132}(\bar{\pi})$ contains~$231$, thus the thesis follows.
\end{proof}

\begin{corollary}\label{corollary_123_132_char}
Let~$\pi=m_1B_1\cdots m_tB_t$ be the ltr-min decomposition of the permutation~$\pi$. Then~$\pi$ is~$(123,132)$-sortable if and only if the following four conditions are satisified:
\begin{enumerate}
\item $B_i>B_{i+1}$ for each~$i=1,\dots,t-1$;
\item $B_1$ is increasing;
\item $m_1m_2B_2\in\Perm(\xi,*213)$;
\item $m_3B_3\cdots m_tB_t\in\Perm(213)$.
\end{enumerate}
\end{corollary}
\begin{proof}
It is an immediate consequence of lemmas~\ref{lemma_123_132_decomp}, \ref{lemma_second_block_123_132} and~\ref{lemma_first_block_123_132}.
\end{proof}

The structural characterization of Corollary~\ref{corollary_123_132_char} can be exploited in order to compute the number of~$(123,132)$-sortable permutations. For~$n\ge 1$ and~$1\le k\le n$, let~$b_{n,k}$ be the~$(n,k)$-th ballot number. The triangle of ballot numbers is also known as the Catalan triangle (sequence~A009766 in~\cite{Sl}, see Figure~\ref{figure_ballot}). Let~$\tilde{b}_{n,k}=b_{n,n+1-k}$ be the triangle obtained by reversing its rows (sequence~A033184 in~\cite{Sl}).

\begin{lemma}\label{lemma_ballot_computation}
\begin{enumerate}
\item $\displaystyle{\sum_{i=1}^s\tilde{b}_{s,i}\binom{n-1-s+i}{i}=b_{n,s+1}}$. 
\item $\displaystyle{1+\sum_{s=1}^{n-1} b_{n,s+1}}=\catalan_n$.
\end{enumerate}
\end{lemma}
\begin{proof}
\begin{enumerate}
\item We have:
\begin{equation*}
\begin{split}
\sum_{i=1}^{s}\tilde{b}_{s,i}\binom{n-1-s+i}{i}=\\
\sum_{i=1}^{s}b_{s,s+1-i}\binom{n-1-s+i}{i}=\\
\sum_{j=1}^{s}b_{s,j}\binom{n-j}{s-j+1}.
\end{split}
\end{equation*}
We shall prove that
$$
b_{n,s+1}=\displaystyle{\sum_{j=1}^{s} b_{s,j}\binom{n-j}{s-j+1}}
$$
by showing that each coefficient~$b_{s,j}$ contributes~$\binom{n-j}{s-j+1}$ times to~$b_{n,s+1}$ (see Figure~\ref{figure_ballot}). It is well known that ballot numbers are defined, for example, by the recurrence:
$$
b_{n,s+1}=\sum_{i=1}^{s+1} b_{n-1,i}.
$$
In other words, in the triangle~$b_{n,k}$ each coefficient~$b_{n,k}$ is obtained by summing the coefficients of indices~$1,2,\dots,k$ of the previous row. Let us consider the coefficient~$b_{s,j}$, for some~$j$. The contribution of~$b_{s,j}$ to~$b_{n,s+1}$ is obtained by choosing (see again Figure~\ref{figure_ballot}):

\begin{itemize}
\item a coefficient~$b_{s+1,t_1}$ in the~$(s+1)$-th row, with~$t_1\in\lbrace j,j+1,\dots,s+1\rbrace$;
\item a coefficient~$b_{s+2,t_2}$ in the~$(s+2)$-th row, with~$t_1\le t_2\le s+1$;
\item a coefficient~$b_{s+3,t_3}$ in the~$(s+3)$-th row, with~$t_2\le t_3\le s+1$;
\item[$\vdots$]
\item a coefficient~$b_{n-1,t_{n-s-1}}$ in the~$(n-1)$-th row, with~$t_{n-s-2}\le t_{n-s-1}\le s+1$.
\end{itemize}

Finally, there are~$\binom{(n-s-1)+(s+1-j+1)-1}{(s+1-j+1)-1}=\binom{n-j}{s-j+1}$ sequences~$t_1\le t_2\le\dots\le t_{n-s-1}$, with values in~$\lbrace j,j+1,\dots,s+1\rbrace$, therefore the desired equality follows.

\item Since~$b_{n,1}=1$, we have
$$
1+\sum_{s=1}^{n-1} b_{n,s+1}=b_{n,1}+\sum_{t=2}^{n} b_{n,t}=\sum_{t=1}^{n} b_{n,t}=\catalan_n.
$$
\end{enumerate}
\end{proof}

\begin{figure}
\begin{minipage}{7cm}
\centering
\def\arraystretch{1.1}
\begin{tabular}{c|ccccccc}
$n\backslash k$  & 1 & 2 & 3 & 4 & 5 & 6 & 7\\
\hline
1 & 1 & & & & & & \\
2 & 1 & 1 & & & & & \\
3 & 1 & 2 & 2 & & & & \\
4 & 1 & 3 & 5 & 5 & & & \\
5 & 1 & 4 & 9 & 14 & 14 & & \\
6 & 1 & 5 & 14 & 28 & 42 & 42 & \\
7 & 1 & 6 & 20 & 48 & 90 & 132 & 132\\
\end{tabular}
\end{minipage}
\hspace{25pt}
\begin{minipage}{6cm}
\centering
\begin{tikzpicture}[baseline=20pt, scale=0.5]
\draw[ultra thin] (0,0) grid (8,8);
\draw[dotted] (0,8)--(2,7)--(2,6)--(4,5)--(5,4)--(5,3)--(5,2)--(7,1)--(8,0);
\node[scale=1] at (0,8){$\bullet$};
\node[scale=1] at (0,8){$\bullet$};
\node[scale=1] at (2,7){$\bullet$};
\node[scale=1] at (2,6){$\bullet$};
\node[scale=1] at (4,5){$\bullet$};
\node[scale=1] at (5,4){$\bullet$};
\node[scale=1] at (5,3){$\bullet$};
\node[scale=1] at (5,2){$\bullet$};
\node[scale=1] at (7,1){$\bullet$};
\node[scale=1] at (8,0){$\bullet$};
\node[above,scale=1] at (0,8){$(s,j)$};
\node[below,scale=1] at (8,0){$(n,s+1)$};
\node[right,scale=1] at (2,7){$t_{1}$};
\node[right,scale=1] at (2,6){$t_{2}$};
\node[above,scale=1] at (5,2){$t_{n-s-2}$};
\node[above,scale=1] at (7,1){$t_{n-s-1}$};
\end{tikzpicture}
\end{minipage}
\caption[The Catalan triangle.]{The triangle of the ballot numbers, on the left, and the construction described in Lemma~\ref{lemma_ballot_computation}, on the right.}\label{figure_ballot}
\end{figure}
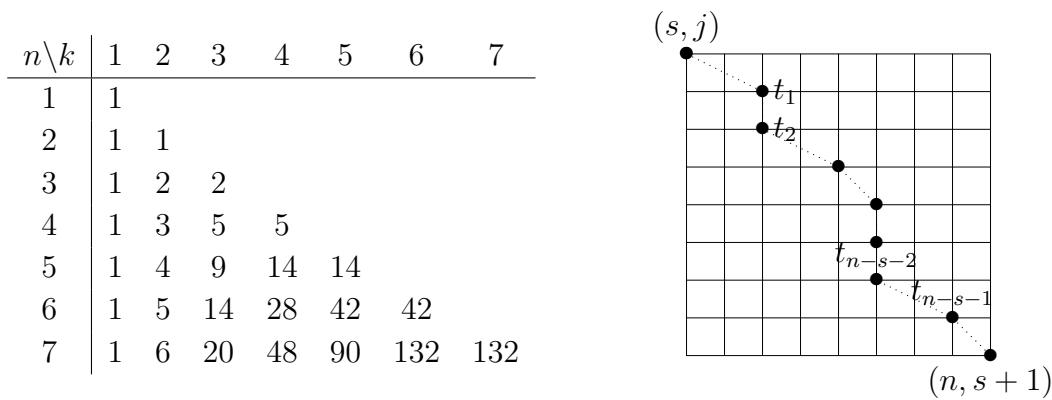

\begin{theorem}
Let~$f_n=\fsigma{123,132}_n$ be the number of~$(123,132)$-sortable permutations of length~$n$. Let~$g_n$ be the number of~$(123,132)$-sortable permutation of length~$n$ with at least two ltr-minima and where the first block~$B_1$ in the ltr-min decomposition is empty. Denote by~$h_n$ the number of permutations~$\pi$ in~$\Perm_n(\xi,*213)$ such that~$\pi_2=1$ (i.e. where~$\pi_1$ and~$\pi_2$ are the only two ltr-minima of~$\pi$.). Then:
\begin{enumerate}
\item $h_1=0$ and ~$h_{n+1}=\catalan_{n}$, for each~$n\ge 1$.
\item $g_1=0$ and~$g_{n+1}=\catalan_{n+1}-\catalan_n$, for each~$n\ge 1$.
\item $f_{n}=\catalan_{n}$, for each~$n\ge 1$.
\end{enumerate}
\end{theorem}
\begin{proof}
\begin{enumerate}
\item Let~$\pi\in\Perm(\xi,*213)$ be a permutation of length at least two and suppose that~$\pi_2=1$. Let~$\pi_1=k$ and write:
$$
\pi=k1A_1\alpha_1A_2\alpha_2\cdots A_s\alpha_sA_{s+1},
$$
where~$\alpha_i<k$ for each~$i$ and the elements of the blocks~$A_1,\dots,A_{s+1}$ are greater than~$k$. Let~$i$ be the minimum index such that~$\alpha_i>\alpha_{i+1}$. Note that~$A_j$ is empty for each~$j\ge i+2$. Otherwise, if~$x\in A_j$, then~$\pi_1\alpha_i\alpha_{i+1}x$ would be an occurrence of~$*213$ in~$\pi$, which contradicts the hypothesis. Analogously, it must be~$x>y$ for each~$x\in B_j$ and~$y\in B_{j+1}$, with~$j=1,\dots,i$, or else~$\pi_1x\alpha_j y$ would be an occurrence of~$*213$. Moreover, the block~$B_j$ is increasing for each~$j=1,\dots,i+1$. Otherwise, a descent~$x>y$ in~$B_j$ would result in an occurrence~$\pi_1xy$ of~$\xi$, which is impossible. What observed so far is enough to completely characterize a permutation~$\pi\in\Perm_{n+1}(\xi,*213)$, with~$\pi_2=1$. Such permutation is either the increasing permutation or it can be constructed as follows (see Figure~\ref{figure_123_132_grid}):
\begin{itemize}
\item Choose a permutation~$\alpha\in\Perm_s(213)$, for some~$1\le s\le n-1$. Let~$i$ be the index of the leftmost descent in~$\alpha$ (if~$\alpha$ is the increasing permutation, let~$i=n-1$).
\item Choose how to distribute~$n-1-s$ elements in the blocks~$A_1,\dots,A_i$. As a consequence of what observed above, the blocks are increasing and the relative order of the blocks is uniquely determined, therefore there are~$\binom{n-1-s+i}{i}$ different ways to distribute such elements.
\item Finally, add the two initial elements~$\pi_1\pi_2=k1$.
\end{itemize}

Now, it is well known that the the number of~$213$-avoiding permutations of length~$s$, where~$i$ is the index of the leftmost descent, is given by~$\tilde{b}_{s,i}$. Therefore we have:
$$
h_{n+1}=1+\sum_{s=1}^{n-1}\left(\sum_{i=1}^s\tilde{b}_{s,i}\binom{n-1-s+i}{i}\right),
$$
and the thesis follows from Lemma~\ref{lemma_ballot_computation}.

\item Due to Corollary~\ref{corollary_123_132_char}, each~$(123,132)$-sortable permutation~$\pi=m_1m_2B_2\cdots m_tB_t$, where~$B_1$ is empty, is obtained by juxtaposing a~$213$-avoiding permutation~$m_3B_3\cdots m_tB_t$ to a permutation~$m_1m_2B_2\in\Perm(\xi,*213)$. By summing over the length~$j$ of the~$213$-avoider (note that~$0\le j\le n-1$), we obtain:
$$
g_{n+1}=\sum_{j=0}^{n-1}h_{n+1-j}\catalan_j =\sum_{j=0}^{n-1}\catalan_{n-j}\catalan_j.
$$
Finally, it is well known that:
$$
\catalan_{n+1}=\sum_{j=0}^n\catalan_{n-j}\catalan_j=\catalan_{n}+\sum_{j=0}^{n-1}\catalan_{n-j}\catalan_j,
$$
thus the thesis follows.

\item Let~$\pi$ be a~$(123,132)$-sortable permutation of length~$n$. If~$\pi$ has one ltr-minimum, then~$\pi$ is the increasing permutation due to Lemma~\ref{lemma_123_132_decomp}. Otherwise, according to Lemma~\ref{lemma_first_block_123_132}, $\pi$ is obtained from a~$(123,132)$-sortable permutation, with~$B_1$ empty, by inserting~$k\ge 0$ consecutive ascents~$k+1,k+2,\dots,n-1,n$ immediately after~$\pi_1$. Therefore, summing over~$k$, we have:
$$
f_n=1+\sum_{k=2}^n g_n =\catalan_1+\sum_{k=2}^n (\catalan_n -\catalan_{n-1}) =\catalan_n.
$$
\end{enumerate}
\end{proof}

\begin{figure}
\centering
\begin{tikzpicture}[baseline=20pt,scale=0.5]
\draw[thick] (0,7) -- (12,7);
\draw[thin] (1,7) grid (3,13);
\draw[thin] (5,7) grid (7,13);
\draw[thick] (1,12) -- (2,13);
\draw[thick] (2,11) -- (3,12);
\draw[thick] (5,8) -- (6,9);
\draw[thick] (6,7) -- (7,8);
\node[scale=1] at (1.25,13.5){$A_{1}$};
\node[scale=1] at (2.75,13.5){$A_{2}$};
\node[scale=1] at (4,10){$\dots$};
\node[scale=1] at (5.25,13.5){$A_{i}$};
\node[scale=1] at (6.75,13.5){$A_{i+1}$};
\node[scale=1] at (0,7){$\bullet$};
\node[scale=1] at (0,6.5){$\pi_1$};
\node[scale=1] at (1,0){$\bullet$};
\node[scale=1] at (2.5,0){$\pi_2=1$};
\draw[thin] (1,0) -- (1,12);
\node[scale=1] at (2,2){$\bullet$};
\draw[thin] (2,2) -- (2,12);
\node[scale=1] at (3,3){$\bullet$};
\draw[thin] (3,3) -- (3,12);
\node[scale=1] at (5,5){$\bullet$};
\node[scale=1] at (4,4){$\iddots$};
\draw[thin] (5,5) -- (5,12);
\node[scale=1] at (6,6){$\bullet$};
\draw[thin] (6,6) -- (6,12);
\node[scale=1] at (7,4){$\bullet$};
\draw[thin] (7,4) -- (7,12);
\node[scale=1] at (2.5,1.5){$\alpha_{1}$};
\node[scale=1] at (3.5,2.5){$\alpha_{2}$};
\node[scale=1] at (5.5,4.5){$\alpha_{i-1}$};
\node[scale=1] at (6.5,5.5){$\alpha_{i}$};
\node[scale=1] at (7.5,3.5){$\alpha_{i+1}$};
\node[scale=1] at (10,4){$\alpha_{i+2}\cdots\alpha_s$};
\node[scale=1] at (10,10){$\emptyset$};
\end{tikzpicture}
\caption[The geometric structure of~$\Perm(\xi,*213)$.]{The geometric structure of a permutation~$\pi\in\Perm(\xi,*213)$, with~$\pi_2=1$.}\label{figure_123_132_grid}
\end{figure}
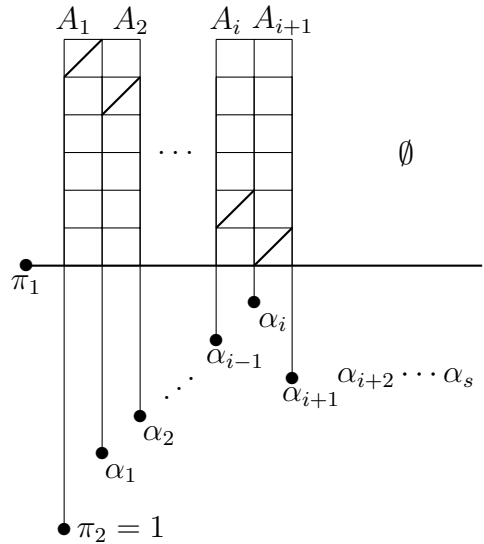

\section{\texorpdfstring{The~$(123,312)$-machine}{The (123,312)-machine}}\label{section_123_312}

In this section we provide a structural description of~$(123,312)$-sortable permutations and define a generating tree for the set~$\Sort(123,312)$. The enumeration of~$\Sort(123,312)$, which is given by the binomial transform of the Catalan numbers, was proved in~\cite{BCKV} by means of a bijection with a family of pattern-avoiding partial permutations.

Let~$\pi=M_1B_1\cdots M_tB_t$ the ltr-max decomposition of a permutation~$\pi$. By Lemma~\ref{theorem_ltr-max_pairs_of_pat_312}, we have~$\out{123,312}(\pi)=\tilde{B}_1\cdots\tilde{B}_tM_t\cdots M_1$, where~$\tilde{B}_i$ is a suitable rearrangement of~$B_i$. Due to the same theorem, if~$\pi$ is~$(123,312)$-sortable and has length~$n$, then~$M_j=n-t+j$, for each $j=1,\dots,t$.

\begin{theorem}\label{123_312_avoids_213}
Let~$\pi=M_1B_1\cdots M_tB_t$ be a~$(123,312)$-sortable permutation. Then:
\begin{enumerate}
\item $B_i$ avoids~$213$ for each~$i$.
\item $\tilde{B}_i=\out{12}(B_i)$, for each~$i$.
\end{enumerate}
\end{theorem}
\begin{proof}
Let~$i\ge 2$. Notice that, as a consequence of Lemma~\ref{theorem_ltr-max_pairs_of_pat_312}, immediately after the push of~$M_i$ into the~$(123,312)$-stack, the~$(123,312)$-stack contains the elements~$M_i\cdots M_2M_1$, reading from top to bottom. Moreover, these elements remain at the bottom of the~$(123,312)$-stack until the end of the sorting procedure, since they are the last elements of~$\out{123,312}(\pi)$. This fact will be used for the rest of the proof.

\begin{enumerate}
\item Suppose, for a contradiction, that~$B_i$ contains an occurrence~$bac$ of~$213$, for some~$i$.  Therefore, since~$abM_i$ is an occurrence of~$123$, $b$ is extracted from the~$(123,312)$-stack before~$a$ enters. Since~$\pi$ is~$(123,312)$-sortable, this implies that~$a$ is then extracted before~$c$ enters, otherwise~$bca$ would be an occurrence of~$231$ in~$\out{123,312}(\pi)$. Consider the instant when~$a$ is extracted from the~$(123,312)$-stack. Let~$x$ be the next element of the input. Since a pop operation is performed, there must be two elements~$y,z$ in the~$(123,312)$-stack, with~$y$ above~$z$, such that~$xyz$ is an occurrence of either~$123$ or~$312$. Notice that~$z=M_j$, for some~$j\le i$. Otherwise, if~$z\in B_i$, then~$yzM_i$ would be an occurrence of~$123$ in the~$(123,312)$-stack, a contradiction. Since~$y<z$, we necessarily have~$y\in B_i$ ($y$ cannot be an ltr-maximum). Now, $xyz$ is not an occurrence of~$312$. Otherwise it would be~$x>z=M_j$ and thus~$x$ would be an ltr-maximum. But since~$x$ precedes~$c\in B_i$ (it could be~$x=c$), this is impossible. Therefore~$xyz\simeq 123$. If~$x>a$, then it would be~$y>a$ as well and thus~$M_iya$ would be an occurrence of~$123$ in the~$(123,312)$-stack, which is impossible. So we have~$x<a$ and~$bxc$ is an occurrence of~$213$ in~$B_i$ (note that~$x$ is strictly to the right of~$a$), thus we can repeat the same argument using~$bxc$ instead of~$bac$, until we eventually find a contradiction.

\item Let us consider the action of the~$(123,312)$-stack on the block~$B_i$. We wish to show that the behavior of the~$(123,312)$-stack when processing~$B_i$ is equivalent to the behavior of an empty~$12$-stack on input~$B_i$. In other words, we prove that the restriction of the~$(123,312)$-stack is triggered if and only if the next element of the input forms an occurrence of~$12$ together with some other element in the~$(123,312)$-stack. As soon as~$M_i$ enters the~$(123,312)$-stack (and the first element of~$B_i$ is the next one to be processed), the~$(123,312)$-stack contains the elements~$M_i\cdots M_2M_1$, reading from top to bottom. Observe that~$B_i$ avoids~$213$ by what proved above, therefore the~$(123,312)$-stack cannot be triggered by an occurrence of~$312$ when processing~$B_i$. Suppose that the next element of the input~$x$ forms an occurrence~$xy$ of~$12$ with some~$y\in B_i$. Then~$xyM_i$ is an occurrence of~$123$ in the~$(123,312)$-stack, as desired. Conversely, suppose that the~$(123,312)$-stack is triggered by an occurrence of~$xyz$ of~$123$, where~$x$ is the next element of the input. Since~$M_i>M_{i-1}>\cdots>M_1$, it must be~$y\in B_i$. Thus~$xy$ is an occurrence of~$12$ that triggers the~$12$-stack, as wanted.
\end{enumerate}
\end{proof}

As a consequence of what proved so far in this section, for any~$(123,312)$-sortable permutation~$\pi=M_1B_1\cdots M_tB_t$ of length~$n$, we have~$B_i\in\Sort(213)$ and~$M_1,\dots,M_t=n-t+1,\dots,n$. Moreover, $\out{123,312}(\pi)=\tilde{B}_1\cdots\tilde{B}_tM_t\cdots M_1$, where~$\tilde{B}_i$ is decreasing. Therefore, for any three elements~$x,y,z$, with~$x\in B_i$, $y\in B_j$ and~$z\in B_k$, with~$i<j\le k$, $xyz$ cannot be an occurrence of~$231$. Otherwise~$xyz$ would still be an occurrence of~$231$ in~$\out{123,312}(\pi)$, contradicting the fact that~$\pi$ is~$(123,312)$-sortable. From now on, we say that~$xyz$ is an occurrence of~$2-3-1$ if~$xyz\simeq 231$, with~$x\in B_i$, $y\in B_j$, $z\in B_k$ and~$i<j<k$. In the analogous case, but when~$j=k$, we say that~$xyz$ is an occurrence of~$2-31$.

\begin{theorem}\label{theorem_4_conditions_123_312}
Let~$\pi=M_1B_1\cdots M_tB_t$ be the ltr-max decomposition of a permutation of length~$n$. Let~$\out{123,312}(\pi)=\tilde{B}_1\cdots\tilde{B}_tM_t\cdots M_1$. Then~$\pi$ is~$(123,312)$-sortable if and only if the following four conditions are satisfied:
\begin{enumerate}
\item $M_j=n-t+j$, for each $j=1,\dots,t$.
\item $B_i$ avoids~$213$ for each~$i$ (and thus~$\tilde{B}_i$ is decreasing for each~$i$).
\item $\pi$ avoids~$2-3-1$.
\item $\pi$ avoids~$2-31$.
\end{enumerate}
\end{theorem}
\begin{proof}
If~$\pi$ is~$(123,312)$-sortable, then~$\pi$ satisfies all the above conditions as a consequence of what proved before in this section. Conversely, it is easy to check that, if~$\pi$ satisfies the above conditions, then~$\out{123,312}(\pi)$ avoids~$231$. We leave this part of the proof to the reader.
\end{proof}

In~\cite{BCKV}, the structural description of Theorem~\ref{theorem_4_conditions_123_312} is reformulated in terms of avoidance of (generalized) patterns. This ultimately leads to a bijection between~$\Sort(123,312)$ and the set of partial permutations avoiding the pattern~$213$, whose enumeration is given by the binomial transform of the Catalan numbers (sequence~A007317 in~\cite{Sl}). We refer the reader to~\cite{BCKV} for a definition of partial permutations, as well as for a detailed proof of the results mentioned above. In what follows we provide a generating tree for~$\Sort(123,312)$. As usual, we wish to generate all the permutations in~$\Sort(123,312)$ by inserting a new rightmost element (and rescaling the others). Before doing that, we reformulate the third condition of Theorem~\ref{theorem_4_conditions_123_312} in the following lemma, whose easy proof is omitted.

\begin{lemma}\label{lemma_blocks_between_two_el_123_312}
Let~$\pi=M_1B_1\cdots M_tB_t$ be the ltr-max decomposition of the~$(123,312)$-sortable permutation~$\pi$. Let~$\out{123,312}(\pi)=\tilde{B}_1\cdots\tilde{B}_tM_t\cdots M_1$. Then~$\out{123,312}(\pi)$ avoids~$2-31$ if and only if for each~$x\in B_i$, $y\in B_j$, with~$i<j$, we have:
\begin{itemize}
\item if~$y>x$, then~$B_j>x$.
\item If~$y<x$, then~$B_j<x$.
\end{itemize}
\end{lemma}

In other words, due to Lemma~\ref{lemma_blocks_between_two_el_123_312}, each block~$B_j$ of a~$(123,312)$-sortable permutation~$\pi$ is bounded between two previous elements of~$\pi$. Now, let~$\pi=M_1B_1\cdots M_tB_t$ be a~$(123,312)$-sortable permutation of length~$n$. We distinguish three possible ways to insert a new rightmost element~$x$ in order to get a~$(123,312)$-sortable permutation:

\begin{enumerate}
\item[$(A)$] Insert a new ltr-maximum~$x=n+1$;
\item[$(B)$] if~$B_t$ is empty, insert the first element of~$B_t$;
\item[$(C)$] if~$B_t$ is not empty, insert an element in~$B_t$.
\end{enumerate}

In order to provide a generating tree for~$(123,312)$-sortable permutations, we assign to each element of~$\Sort(123,312)$ two labels~$(k,m)$. The label~$k$ denotes the number of active sites of the current block~$B_t$, thus it takes into account the insertion of~$x$ according to~$(C)$. Due to Lemma~\ref{lemma_blocks_between_two_el_123_312}, the relative position of the block~$B_t$ with respect to the previous blocks is uniquely determined by its first element. Therefore we just have to make sure that~$B_t$ avoids~$213$, as stated in Theorem~\ref{theorem_4_conditions_123_312}, and the active sites related to the label~$k$ will be Catalan-type:
$$
(k)\rightarrow (2)(3)\cdots (k)(k+1).
$$
The other label~$m$ denotes the number of active sites with respect to the relative order of the blocks and it takes into account the insertion of~$x$ according to~$(B)$. More precisely, since we have to avoid creating an occurrence of~$2-3-1$, the label~$m$ will be Catalan-type too:
$$
(m)\rightarrow (2)(3)\cdots (m)(m+1).
$$
Notice that inserting~$x$ according to~$(C)$ affects the label~$m$ as well. Indeed a new element in the block~$B_t$ creates one additional active site for the relative order of the blocks (so~$m$ is increased by one). Finally, inserting a new ltr-maximum~$x=n+1$ according to~$(A)$ always produces a permutation where the last block is empty, that is where~$k=1$. Note that this operation does not affect the label~$m$.

\begin{theorem}\label{theorem_gen_tree_123_312}
The following rule provides a generating tree for~$\Sort(123,312)$:
\begin{equation}
\Omega:
\begin{cases}
(1,0)\longrightarrow (1,0)(2,2)\\
(1,m)\longrightarrow (1,m)(2,2)(2,3)\cdots (2,m)(2,m+1),\ m\ge 2\\
(k,m)\longrightarrow (1,m)(2,m+1)(3,m+1)\cdots (k+1,m+1),\ k,m\ge 2\\
\end{cases}
\end{equation}\label{rule_123_312}
\end{theorem}
\begin{proof}
Let~$\pi=M_1 B_1\cdots M_t B_t$ be a~$(123,312)$-sortable permutation with label~$(k,m)$, for some integers~$k,m$. Suppose we insert a new rightmost element~$x$. This can be done according to either condition~$(A)$, $(B)$ or~$(C)$, as described below Lemma~\ref{lemma_blocks_between_two_el_123_312}. We discuss each case separately.

\begin{itemize}
\item If~$(k,m)=(1,0)$, then~$\pi=1\cdots n$ consists solely of ltr-maxima, since~$m=0$. If~$x$ is a new ltr-maximum, then the label of the resulting permutation is again~$(1,0)$. Otherwise, if~$x$ is the first (and only) element of~$B_t$, then the resulting label is~$(2,2)$.

\item If~$(k,m)=(1,m)$, for some~$m\ge 2$, then the last element of~$\pi$ is~$M_t$ and~$B_t$ is empty. If~$x$ is a new ltr-maximum, then the resulting label is~$(1,m)$, as noted above. Otherwise, $x$ is the first element of~$B_t$, according to~$(B)$. Then the behavior of the label~$m$ is Catalan-type, according to the relative order of the blocks. Moreover, the label~$k$ of any resulting permutation is always~$2$, since~$x$ is the only element of~$B_t$ in the resulting permutation.

\item Finally, let~$k,m\ge 2$. Then, since~$k\ge 2$, the block~$B_t$ is not empty (and it has~$k$ Catalan-type active sites). Again if~$x$ is a new ltr-maximum, then the resulting label is~$(1,m)$. Otherwise, $x$ is an element of~$B_t$. Therefore the behavior of the label~$k$ is Catalan-type, according to~$(C)$. The label~$m$, instead, is increased by one, as noted before.
\end{itemize}
\end{proof}

The problem of showing directly that the family of objects generated by the rule of~Theorem~\ref{theorem_gen_tree_123_312} is counted by the binomial transform of the Catalan numbers remains open.

\begin{openproblem}
Prove directly that objects generated by Rule~\ref{rule_123_312} are counted by the binomial transform of Catalan numbers.
\end{openproblem}

\section{\texorpdfstring{The family of~$(\sigma$,$\hat{\sigma})$-machines}{The family of (sigma,hat(sigma))-machines}}

We end this chapter by mentioning a result for a family of pairs of patterns. Let~$\sigma$ be a permutation of length~$k$. Recall that~$\hat{\sigma}=\sigma_2\sigma_1\sigma_3\cdots\sigma_k$ is the permutation obtained by interchanging the first two entries of~$\sigma$. Then the map~$\pi\mapsto\out{\sigma,\hat{\sigma}}(\pi)$ is bijective from~$\Sort(\sigma,\hat{\sigma})$ to~$\Perm(231)$. More precisely, each~$(\sigma,\hat{\sigma})$-sortable permutation~$\pi$ is obtained uniquely from a~$231$-avoiding permutation~$\alpha$ as:
$$
\pi=\reverse(\out{\sigma,\hat{\sigma}}(\reverse(\alpha)).
$$
The above equality gives a constructive description of~$\Sort(\sigma,\hat{\sigma})$. Indeed it follows immediately that:
$$
\Sort(\sigma,\hat{\sigma})=\reverse(\mapsigma{\sigma}(\reverse(\Perm(231)))
$$
A proof of this fact (which will be obtained as a corollary of a much more general result for Cayley permutations) is postponed to Remark~\ref{remark_R_circ_sigma_bij} in Section~\ref{section_operator_Cayley}. An immediate consequence is that~$\fsigma{\sigma,\hat{\sigma}}_n=\catalan_n$, for each permutation~$\sigma$ and~$n\ge 1$.

\chapter{\texorpdfstring{Dynamical aspects of the~$\sigma$-machine}{Dynamical aspects of the sigma-machine}}\label{chapter_sorted_perm_fertilities}

In this chapter we analyze some dynamical aspects of $\sigma$-machines by regarding a $\sigma$-stack as an operator $\pi\mapsto\out{\sigma}(\pi)$. This approach has been adopted recently for pattern-avoiding machines in~\cite{BCKV,Ber,Ce,DZ}. Suppose to perform a deterministic sorting procedure. Then it is natural to study the map~$\mapsigma{}$ that associates to an input string~$\pi$ the (uniquely determined) output of the sorting process. Some of the problems that arise, and are classically considered, are the following:
\begin{itemize}
\item Determine the fertility of a string, which is the number of its pre-images under the map~$\mapsigma{}$.
\item Determine the image of~$\mapsigma{}$, i.e. the strings with positive fertility. These are often called sorted permutations~\cite{BM2}.
\end{itemize}

In what follows, we define some properties of the operator~$\mapsigma{\sigma}$ associated to a~$\sigma$-stack, then we start to collect the first related results.

\section{Sorted permutations and fertility}\label{section_sorted_perm_fertilities}

Let~$\sigma$ be a permutation of length two or more. The map~$\mapsigma{\sigma}$ is defined by:
\begin{align*}
\mapsigma{\sigma}:\ &\Perm\to\Perm\\
&\pi\mapsto\out{\sigma}(\pi).
\end{align*}
Define~$\sorted(\sigma)=\mapsigma{\sigma}(\Sort(\sigma))$. Permutations in~$\sorted(\sigma)$ are thus images of~$\sigma$-sortable permutations through the map~$\mapsigma{\sigma}$. We call them the \textit{$\sigma$-sorted} permutations. Notice that~$\sorted(\sigma)=\mapsigma{\sigma}(\Perm)\cap\Perm(231)$.

\begin{remark}
In the adopted definition, $\sigma$-sorted permutations are those that are both output of the~$\sigma$-stack and~$12$-sortable. A different notion (of sorted permutations) can be obtained by considering the set~$\mapsigma{\sigma}(\Perm)$, thus including all the possible outputs of the~$\sigma$-stack. This framework will be adopted in Chapter~\ref{chapter_sort_words_various_kind} in the (more general) case of Cayley permutations.
\end{remark}

The \textit{$\sigma$-fertility} of a permutation~$\pi$ is
$$
\fert{\sigma}(\pi)=|\left(\mapsigma{\sigma}\right)^{-1}(\pi)|.
$$
Due to the adopted definition of~$\sigma$-sorted permutations, we have:
$$
|\Sort_n(\sigma)|=\sum_{\gamma\in\sorted(\sigma)}\fert{\sigma}(\gamma).
$$
A permutation~$\sigma$ is \textit{surjective} if~$\sorted(\sigma)=\Perm(231)$. A permutation~$\sigma$ is \textit{injective} if~$\fert{\sigma}(\pi)\le 1$ for each~$\pi\in\Perm(231)$. We say that~$\sigma$ is \textit{bijective} if~$\sigma$ is both surjective and injective. In this case, for each~$n\ge 1$, we have:
$$
|\Sort_n(\sigma)|=\sum_{\gamma\in\sorted(\sigma)}\fert{\sigma}(\gamma)=\sum_{\gamma\in\Perm_n(231)}1=\catalan_n,
$$
where~$\catalan_n$ is the~$n$-th Catalan number. Equivalently, $\sigma$ is said to be injective, surjective or bijective if the (restricted) map~$\mapsigma{\sigma}:\Sort(\sigma)\to\Perm(231)$ is respectively injective, surjective or bijective. Finally, $\sigma$ is said to be \textit{effective} if~$\sorted(\sigma)\subseteq\Perm(\sigma)$, that is the~$\sigma$-stack succesfully performs its task of preventing occurrences of~$\sigma$ to be produced. One of the main goals of this chapter is to characterize which patterns are effective. First we prove a simple lemma that leads to an equivalent definition of effectiveness.

\begin{lemma}\label{lemma_sorted_inclusions}
We have:
$$
\Perm(231,\sigma)\subseteq\sorted(\sigma)\subseteq\Perm(231).
$$
\end{lemma}
\begin{proof}
The inclusion~$\sorted(\sigma)\subseteq\Perm(231)$ has already been noted. For the other inclusion, observe that every~$\sigma$-avoiding permutation is equal to the output of~$\mapsigma{\sigma}$ on its reverse, and therefore, every such permutation that also avoids~$231$ belongs to~$\sorted(\sigma)$. Moreover, by Lemma~\ref{lemma_hat_sigma} we have:
$$
\mapsigma{\sigma}(\Perm(132,\reverse(\sigma)))=\reverse(\Perm(132,\reverse(\sigma)))=\Perm(231,\sigma),
$$
and thus~$\Perm(231,\sigma)\subseteq\sorted(\sigma)$.
\end{proof}

\begin{corollary}\label{corollary_effective_patterns}
A pattern~$\sigma$ is effective if and only if~$\sorted(\sigma)=\Perm(231,\sigma)$.
\end{corollary}
\begin{proof}
If~$\sigma$ is effective, that is~$\sorted(\sigma)\subseteq\Perm(\sigma)$, then by Lemma~\ref{lemma_sorted_inclusions} we have~$\Perm(231,\sigma)\subseteq\sorted(\sigma)\subseteq\Perm(231,\sigma)$ and thus~$\sorted(\sigma)=\Perm(231,\sigma)$. The other implication is trivial.
\end{proof}

\begin{theorem}\label{theorem_sorted_hat_231}
Let~$\sigma$ be a permutation of length two or more. If~$\hat{\sigma}\ge 231$, then~$\sorted(\sigma)=\Perm(231,\sigma)$ and~$\sigma$ is both injective and effective.
\end{theorem}
\begin{proof}
Suppose that~$\hat{\sigma}\ge 231$. By Theorem~\ref{theorem_suff_cond_class}, we have~$\Sort(\sigma)=\Perm(132,\reverse(\sigma))$. Therefore, if~$\pi\in\Sort(\sigma)$, then~$\pi$ avoids~$\reverse(\sigma)$ and thus~$\mapsigma{\sigma}(\pi)=\reverse(\pi)$. In other words, the operator~$\mapsigma{\sigma}$ acts as the reverse operator on~$\Sort(\sigma)$ and thus the~$\sigma$-machine is injective. Finally, we have:
$$
\sorted(\sigma)=\reverse(\Sort(\sigma))=\reverse(\Perm(132,\reverse(\sigma)))=\Perm(231,\sigma),
$$
hence the~$\sigma$-machine is effective due to Corollary~\ref{corollary_effective_patterns}.
\end{proof}

Due to Theorem~\ref{theorem_sorted_hat_231}, if both~$\hat{\sigma}$ and~$\sigma$ contain~$231$, then~$\sorted(\sigma)=\Perm(231)$ and the~$\sigma$-machine is also surjective. On the other hand, if~$\hat{\sigma}\ge 231$, but~$\sigma$ avoids~$231$, then~$\sorted(\sigma)$ is strictly contained in~$\Perm(231)$ and the~$\sigma$-machine is not surjective.

\section{Characterization of effective patterns}\label{section_effective_char}

In Lemma~\ref{lemma_sorted_inclusions} we have proved that~$\Perm(231,\sigma)\subseteq\sorted(\sigma)$. Our next goal is to provide a characterization of the effective patterns. Such patterns are precisely those where the equality holds, as stated in Corollary~\ref{corollary_effective_patterns}.

\begin{proposition}\label{proposition_sorted_necessary}
Let~$\sigma$ be a permutation of length two or more. If~$\hat{\sigma}=1\oplus\alpha$, for some~$\alpha\in\Perm(231)$, then~$\sigma$ is not effective.
\end{proposition}
\begin{proof}
Let~$\hat{\sigma}=1\oplus\alpha$, with~$\alpha\in\Perm(231)$. We show that there is a~$\sigma$-sortable permutation~$\pi$ such that~$\mapsigma{\sigma}(\pi)\ge\sigma$. Let~$\sigma=\sigma_1\sigma_2\cdots\sigma_k$, where~$\sigma_2=1$ since the first element of~$\hat{\sigma}$ is~$1$ by hypothesis. Suppose that~$\sigma_1=t$, for some~$2\le t\le k$. Define
$$
\pi=\reverse(\sigma)\ominus\reverse(\sigma_2\cdots\sigma_t)=\sigma'_k\sigma'_{k-1}\cdots\sigma'_2\sigma'_1\sigma_t\sigma_{t-1}\cdots\sigma_2,
$$
where~$\sigma'_i=\sigma_i+t-1$ for each~$i$. Notice that:
$$
\sigma'_2=\sigma_2+t-1=1+t-1=t=\sigma_1.
$$
We shall prove that:
$$
\mapsigma{\sigma}(\pi)=\sigma_1\sigma_2\cdots\sigma_t\sigma'_1\sigma'_3\cdots\sigma'_k.
$$
Due to our assumptions, we have that~$\sigma=t1\sigma_3\cdots\sigma_k$ and~$\sigma$ avoids~$231$. Thus it must be
$$
\lbrace\sigma_3,\dots,\sigma_t\rbrace=\lbrace 2,\dots,t-1\rbrace\text{ and }\lbrace\sigma_{t+1},\dots,\sigma_k\rbrace=\lbrace t+1,\dots,k\rbrace,
$$
otherwise there would be two indices~$i\in\lbrace2,\dots,t-1\rbrace$ and~$j\in\lbrace t+1,\dots,k\rbrace$ such that~$\sigma_i>t$ and~$\sigma_j<t$. But then~$\sigma$ would contain an occurrence~$\sigma_1\sigma_i\sigma_j$ of~$231$, which is a contradiction. An immediate consequence is that the string~$w=\sigma_1\cdots\sigma_t\sigma'_{t+1}\cdots\sigma'_k$ is order isomorphic to~$\sigma$. Indeed~$w$ is obtained from~$\sigma$ by adding~$t-1$ to the elements~$\sigma'_{t+1},\dots,\sigma'_k$. Moreover, the string~$z=\sigma_2\cdots\sigma_t\sigma'_1\sigma'_3\cdots\sigma'_k$ avoids~$\sigma$. Indeed we have~$\sigma_1=t$, whereas~$\sigma_i<t$ for each~$i\le t$. Thus no element amongst~$\sigma_2,\dots,\sigma_t$ can play the role of~$\sigma_1$ in an occurrence of~$\sigma$ in~$z$. Finally, the remaining suffix of~$z$ is too short to contain~$\sigma$ (it has length~$k-1$). Now, let us consider the action of the~$\sigma$-stack on input~$\pi$. Since~$\pi$ contains the prefix~$\sigma'_k\sigma'_{k-1}\cdots\sigma'_2\sigma'_1$, the first element that cannot be pushed into the~$\sigma$-stack is~$\sigma'_1$, which causes the pop of~$\sigma'_2$. Then~$\sigma'_1$ is pushed and, immediately after, the~$\sigma$-stack contains~$\sigma'_k\cdots\sigma'_3\sigma'_1$, reading from bottom to top. The remaining elements of the input are~$\sigma_t\cdots\sigma_2$. Notice that~$\sigma'_k\cdots\sigma'_3\sigma'_1\sigma_t\cdots\sigma_2=\reverse(z)$, which avoids~$\reverse(\sigma)$ due to what observed above. Therefore all the remaining elements of the input are pushed into the~$\sigma$-stack directly and the output is:
$$
\mapsigma{\sigma}(\pi)=\sigma'_2\sigma_2\cdots\sigma_t\sigma'_1\sigma'_3\cdots\sigma'_k.
$$
This is precisely what we wanted, since~$\sigma'_2=\sigma_1$. Now, $\mapsigma{\sigma}(\pi)$ contains the substring~$w$, which is an occurrence of~$\sigma$. Finally, it is easy to show that~$\mapsigma{\sigma}(\pi)$ avoids~$231$, since~$\sigma$ avoids~$231$ and~$\sigma'_i>\sigma_j$, for each~$i,j$. Thus~$\pi$ is~$\sigma$-sortable and~$\mapsigma{\sigma}(\pi)\ge\sigma$. This completes the proof.
\end{proof}

Next we show that if~$\hat{\sigma}$ is not the direct sum of~$1$ plus a~$231$-avoiding permutation, then~$\sigma$ is effective. If~$\hat{\sigma}\ge 231$, then the desired results follows immediately by Theorem~\ref{theorem_sorted_hat_231}. We just need to address the remaining case where~$\hat{\sigma}$ avoids~$231$ and~$\sigma_2\neq 1$.

\begin{proposition}\label{proposition_sorted_suff}
Let~$\sigma$ be a permutation of length two or more. If~$\hat{\sigma}$ avoids~$231$ and~$\sigma_2\neq 1$, then~$\sigma$ is effective.
\end{proposition}
\begin{proof}
We show that~$\sorted(\sigma)\subseteq\Perm(\sigma)$. Let~$\gamma\in\sorted(\sigma)$ and suppose, for a contradiction, that~$\gamma\ge\sigma$. Let~$\pi$ be a~$\sigma$-sortable permutation such that~$\mapsigma{\sigma}(\pi)=\gamma$. If~$\pi$ avoids~$\reverse(\sigma)$, then~$\gamma=\mapsigma{\sigma}(\pi)=\reverse(\pi)$ avoids~$\sigma$, which is a contradiction. Therefore we can assume that~$\pi\ge\reverse(\sigma)$. By Lemma~\ref{lemma_hat_sigma}, we have that~$\gamma\ge\hat{\sigma}$. Thus~$\gamma$ contains both~$\sigma$ and~$\hat{\sigma}$ and moreover~$\gamma$ avoids~$231$. Since~$\sigma_2\neq 1$, it must be~$\sigma_1=1$. Otherwise, if~$\sigma_i=1$, with~$i\ge 3$, then it would be either~$\sigma_1\sigma_2\sigma_i\simeq 231$, if~$\sigma_1<\sigma_2$, or~$\sigma_2\sigma_1\sigma_i\simeq 231$, if~$\sigma_1>\sigma_2$. In the first case, it would be~$\gamma\ge\sigma\ge 231$, which is impossible. In the second case, it would be~$\gamma\ge\hat{\sigma}\ge 231$, again a contradiction.

Now, let~$\tilde{\sigma}_1\cdots\tilde{\sigma}_k$ be the leftmost occurrence of~$\sigma$ in~$\gamma$. Let us consider the instant when~$\tilde{\sigma}_1$ is extracted from the~$\sigma$-stack. If the input is empty, then the~$\sigma$-stack must contain all the elements~$\tilde{\sigma}_1\cdots\tilde{\sigma}_k$, from top to bottom, which is impossible by definition of~$\sigma$-stack. Thus~$\tilde{\sigma}_1$ is extracted due to the fact that the next element of the input, say~$\sigma'_1$, triggers the restriction of the~$\sigma$-stack. More explicitly, $\sigma'_1$ realizes an occurrence of~$\sigma$ together with some elements~$\sigma'_2\cdots\sigma'_k$ contained in the~$\sigma$-stack (from top to bottom). Since~$\sigma_1=1$, it must be~$\tilde{\sigma}_1>\sigma'_1$, otherwise~$\tilde{\sigma}_1\sigma'_2\cdots\sigma'_k$ would be an occurrence of~$\sigma$ inside the~$\sigma$-stack. If~$\tilde{\sigma}_2$ precedes~$\sigma'_1$ in~$\gamma$, then~$\tilde{\sigma}_1\tilde{\sigma}_2\sigma'_1\simeq 231$ in~$\gamma$, which is impossible. Thus we can assume that~$\tilde{\sigma}_2$ follows~$\sigma'_1$ in~$\gamma$. We consider two cases, according to whether or not~$\tilde{\sigma}_2$ follows~$\sigma'_1$ in~$\pi$.

\begin{itemize}
\item Suppose initially that~$\tilde{\sigma}_2$ precedes~$\sigma'_1$ in~$\pi$, and thus~$\tilde{\sigma}_2$ is contained in the~$\sigma$-stack when~$\sigma'_1$ is the next element of the input. Consider the instant when~$\tilde{\sigma}_1$ is extracted from the~$\sigma$-stack. Recall that at this point~$\sigma'_1$ is the next element of the input and~$\sigma'_1\sigma'_2\cdots\sigma'_k\simeq\sigma$, for some elements~$\sigma'_2\cdots\sigma'_k$ contained in the~$\sigma$-stack. Moreover, the top element of the~$\sigma$-stack is~$\tilde{\sigma}_1$ and~$\tilde{\sigma}_2$ is still contained in the~$\sigma$-stack. We shall prove that~$\tilde{\sigma}_j$ is contained in the~$\sigma$-stack for each~$j\ge3$. Suppose, for a contradiction, that~$\tilde{\sigma}_j$ follows~$\sigma'_1$ in the input, for some~$j\ge 3$. Notice that~$\tilde{\sigma}_2$ is extracted from the~$\sigma$-stack before~$\tilde{\sigma}_j$ enters. Let~$\sigma'''_1\sigma'''_2\cdots\sigma'''_k$ be the occurrence of~$\sigma$ that causes the pop of~$\tilde{\sigma}_2$, with~$\sigma'''_1$ the next element of the input. Again we have~$\tilde{\sigma}_2>\sigma'''_1$, since~$\sigma_1=1$. Moreover~$\gamma$ contains~$\tilde{\sigma}_1\tilde{\sigma}_2\sigma'''_2$, thus it must be~$\sigma'''_1>\tilde{\sigma}_2$, or else~$\tilde{\sigma}_1\tilde{\sigma}_2\sigma'''_1$ would be an occurrence of~$231$ in~$\gamma$, which is impossible. But then, when~$\tilde{\sigma}_1$ is the top of the~$\sigma$-stack, the~$\sigma$-stack contains~$\tilde{\sigma}_1\sigma'''_2\cdots\sigma'''_k$, which is an occurrence of~$\sigma$ since~$\sigma'''_1>\tilde{\sigma}_1$ (and~$\sigma_1=1$), again a contradiction. We can thus assume that, when~$\tilde{\sigma}_1$ is extracted, $\tilde{\sigma}_j$ is contained in the~$\sigma$-stack for each~$j\ge 3$. But this is impossible, since the~$\sigma$-stack would contain an occurrence~$\tilde{\sigma}_1\tilde{\sigma}_2\cdots\tilde{\sigma}_k$ of~$\sigma$.

\item Suppose instead that~$\tilde{\sigma}_2$ follows~$\sigma'_1$ in~$\pi$, and thus~$\sigma'_1$ is extracted from the~$\sigma$-stack before~$\tilde{\sigma}_2$ enters. Let us consider the instant when~$\sigma'_1$ is extracted. At this point, the~$\sigma$-stack contains some elements~$\sigma''_2\cdots\sigma''_k$ that realize an occurrence of~$\sigma$ together with the next element~$\sigma''_1$ of the input. Again it must be~$\sigma'_1>\sigma''_1$, otherwise (since~$\sigma_1=1$) the~$\sigma$-stack would contain an occurrence~$\sigma'_1\sigma''_2\cdots\sigma''_k$ of~$\sigma$. Thus~$\tilde{\sigma}_1>\sigma'_1>\sigma''_2$ and~$\tilde{\sigma}_1\tilde{\sigma}_2\sigma''_1\simeq 231$. This means that~$\tilde{\sigma}_2$ must follow~$\sigma''_1$ in~$\gamma$. Now, if~$\tilde{\sigma}_2$ precedes~$\sigma''_1$ in~$\pi$, then we are back to the previous case. Otherwise, we can repeat the same argument on~$\tilde{\sigma}_1\tilde{\sigma}_2\sigma''_1$, with~$\sigma''_1$ in place of~$\sigma'_1$. Sooner or later, since~$\sigma''_1$ is strictly to the right of~$\sigma'_1$ in~$\pi$, this will result in a contradiction.
\end{itemize}
\end{proof}

What proved so far in this section, together with Lemma~\ref{lemma_sorted_inclusions}, leads to the following characterization of effective patterns.

\begin{corollary}\label{corollary_sorted_class}
Let~$\sigma$ be a permutation of length two or more. Then~$\sigma$ is not effective if and only if~$\hat{\sigma}=1\oplus\alpha$, for some~$\alpha\in\Perm(231)$.
\end{corollary}

An immediate consequence of Corollary~\ref{corollary_sorted_class} is that there are~$\catalan_{n-1}=|\Perm_{n-1}(231)|$ patterns of length~$n$ that are not effective. For instance, there is one such pattern of length two, namely~$21$, since~$\hat{21}=12=1\oplus 1$. Similarly, there are two such patterns of length three, namely~$213$ and~$312$ (see Table~\ref{table_sorted_perms}).

\begin{table}
\centering
\def\arraystretch{1.1}
\begin{tabular}{lclr}
\toprule
$\sigma$ &~$\sorted(\sigma)$ & \textbf{Sequence}~$\lbrace|\sorted_n(\sigma)|\rbrace_n$ & \textbf{OEIS}\\
\midrule
123 &~$\Perm(123,231)$ & 1, 2, 4, 7, 11, 16, 22, 29, 37,$\dots$ & A000124\\
132 &~$\Perm(132,231)$ & 1, 2, 4, 8, 16, 32, 64, 128, 256,$\dots$ & A000079\\
213 & ? & 1, 2, 4, 9, 22, 58, 161, 466, 1390 & \\
231 &~$\Perm(231)$ & 1, 2, 5, 14, 42, 132, 429, 1430, 4862,$\dots$ & A000108\\
312 & ? & 1, 2, 4, 8, 17, 40, 104, 291, 855 & \\
321 &~$\Perm(231,321)$ & 1, 2, 4, 8, 16, 32, 64, 128, 256,$\dots$ & A000079\\
\bottomrule
\end{tabular}
\caption[Enumerative results for~$\sigma$-sorted permutations.]{$\sigma$-sorted permutations for patterns~$\sigma$ of length three, starting from~$\sigma$-sorted permutations of length one.}\label{table_sorted_perms}
\end{table}

\section{\texorpdfstring{Fertility and sorted permutations of the~$123$-machine}{Fertility and sorted permutations of the 123-machine}
}

Fertility and sorted permutations for the~$123$-machine can be determined from the results proved in Chapter~\ref{chapter_pattern123}. Recall that any~$\pi\in\Sort_n(123)$ which is not the identity permutation can be uniquely constructed as follows:

\begin{itemize}
\item choose~$\alpha\in\Perm_k(213)$, with~$\alpha_1=k\ge 2$;
\item add~$h$ new maxima~$k+1,\ldots,k+h$, one at a time, using the bijection~$\varphi$ of Theorem~\ref{theorem_123_bijection};
\item add~$t=n-k-h$ consecutive ascents at the beginning, by inflating the first element of the permutation, according to Corollary~\ref{corollary_123_infl_cor}.
\end{itemize}

We wish to exploit the above construction to describe the set~$\sorted(123)$ and compute the fertility of~$123$-sorted permutations. Let~$\pi$ be a~$123$-sortable permutation. If~$\pi$ starts with a descent~$\pi_1>\pi_2$, with~$\pi_1=k$, then by Lemma~\ref{lemma_123_output}, we have~$\out{123}(\pi)=n(n-1)\cdots (k+1)(k-1)\cdots 21k$. Moreover, observe that inserting~$t$ consecutive ascents~$\pi_1(\pi_1+1)\cdots(\pi_1+t)$ at the beginning does not affect the behavior of the~$123$-stack. Indeed the elements~$\pi_1(\pi_1+1)\cdots(\pi_1+t)$ act as a single element at the bottom of the~$123$-stack during the sorting process. Therefore, if~$\pi'$ is obtained from~$\pi$ by~$t$-inflating~$\pi_1=k$, then we have:
$$
\out{123}(\pi')=\underbrace{(k+t+h)\cdots (k+t+1)}_{(I)}\underbrace{k-1 \cdots 21}_{(II)}\underbrace{(k+t)\cdots (k+1)k}_{(III)},
$$
where the segment~$(I)$ corresponds to the~$h$ new maxima added using~$\varphi$, $(II)$ corresponds to the elements of~$\alpha\in\Perm_k(213)$ and~$(III)$ contains the~$t$-inflation of~$\pi_1$. The fertility of~$\out{123}(\pi')$ is then~$\catalan_{k-1}$, since there are~$\catalan_{k-1}$ permutations~$\alpha$ in~$\Perm_k(213)$ whose first element is equal to~$k$.

\begin{corollary}\label{corollary_sorted_123}
We have:
$$
\sorted(123)_n=\lbrace\identity_h^{-1}\ominus\left(\identity_k^{-1}\oplus\identity_t^{-1}\right): k\ge 2,\ h,t\ge 0,\ k+h+t=n\rbrace\dot{\cup}\lbrace\identity_n^{-1}\rbrace.
$$
Moreover, the fertility of~$\identity_h^{-1}\ominus(\identity_k^{-1}\oplus\identity_t^{-1})$ is equal to~$\catalan_{k-1}$.
\end{corollary}
\begin{proof}
If~$\pi=\identity_n$, then~$\out{123}(\pi)=\identity_n^{-1}$. The rest follows from what discussed before.
\end{proof}

From the description obtained in Corollary~\ref{corollary_sorted_123}, and in accordance with Corollary~\ref{corollary_sorted_class}, it is easy to deduce that~$\sorted(123)=\Perm(231,123)$. Corollary~\ref{corollary_sorted_123} can be used to obtain an alternative proof of the enumeration of~$\Sort(123)$:
\begin{align*}
\Sort(123)=\sum_{\gamma\in\sorted(123)}\fert{123}(\gamma)=&\\
1+\sum_{k\ge 2}\sum_{h,t\ge 0, k+h+t=n}\catalan_{k-1}=&\\
1+\sum_{k=2}^n(n-k+1)\catalan_{k-1}=&\\
1+\sum_{k=1}^{n-1}(n-k)\catalan_k,
\end{align*}
which is the same as what we got in Theorem~\ref{theorem_123_enum}.

\chapter{Sorting words of various types}\label{chapter_sort_words_various_kind}

In this chapter we extend~$\Sigma$-machines to Cayley permutations, ascent sequences and modified ascent sequences, which has been defined in Section~\ref{section_sequences_integers}. Pattern-avoiding machines are built upon the notion of pattern, which is inherently more general, thus it is natural to allow different sets of strings as input sequences. The idea of analyzing sorting procedures on words is not new in the literature~\cite{AAAHH,ALW,DK}. For example, classical stacksort on~$\nat^*$ has been discussed in~\cite{DK}. Due to the presence of sequences with repeated elements, there are two possibilities when defining a stack sorting algorithm on~$\nat^*$. One can either allow a letter to sit on a copy of itself in the stack or force a pop operation if the next element of the input is equal to the top element of the stack. In this thesis we choose the former possibility, leaving the latter for future investigation. This is equivalent to regarding a classical stack as a~$21$-avoiding stack (instead of as a~$(11,21)$-stack). Moreover, we relax the condition for the output to be sorted by requiring that it is weakly increasing. The following theorem was proved in~\cite{DK}.
 
\begin{theorem}\cite{DK}\label{theorem_hare_stacksort}
Let~$\pi$ be a word on~$\nat$. Then~$\pi$ is sortable using a~$21$-stack if and only if~$\pi$ avoids~$231$.
\end{theorem}

Patterns live\footnote{As standardized sequences.} in the set~$\Cay$ of Cayley permutations, thus it seems appropriate to start our analysis by studying~$\sigma$-machines where both input sequences and the forbidden pattern that defines the constraint of the stack are elements of~$\Cay$. We then consider~$\sigma$-machines on ascent sequences and modified ascent sequences. Following a principle of uniformity, we always require forbidden patterns and input sequences to be chosen in the same set. Notice that the output of the~$\sigma$-stack in all these cases is a word on~$\nat$, therefore we can use Theorem~\ref{theorem_hare_stacksort} to determine whether an input sequence is sortable. In Chapter~\ref{chapter_single_pattern}, we characterized those patterns~$\sigma$ such that the set of~$\sigma$-sortable permutations is a class. The main goal of this chapter is to prove analagous results for the sets~$\Cay$, $\Ascseq$ and~$\Modasc$. In Section~\ref{section_operator_Cayley} we study the operator~$\mapsigma{\sigma}$ on Cayley permutations. Some enumerative and structural properties of ascent and modified ascent sequences are derived in Section~\ref{section_asc_seq_stack} and Section~\ref{section_modasc_seq_stack}.

\section{\texorpdfstring{The~$\sigma$-machine on Cayley permutations}{The sigma-machine on Cayley permutations}
}\label{section_sorting_Cayley_perms}

In this section we consider~$\sigma$-machines on Cayley permutations. Some of the results contained here can be found in~\cite{Ce}. We give a formal definition of these devices in the case of Cayley permutations. The corresponding machines, on~$\Ascseq$ and~$\Modasc$, are defined analogously. Let~$\sigma$ be a Cayley permutation of length at least two. A~$\sigma$\textit{-stack} is a stack that is not allowed to contain an occurrence of the pattern~$\sigma$ when reading the elements from top to bottom. The term~$\sigma$\textit{-machine} refers to performing a right-greedy algorithm on two stacks in series: a~$\sigma$-stack, followed by a~$21$-avoiding stack. A Cayley permutation~$\pi$ is~$\sigma$\textit{-sortable} if the output of the~$\sigma$-machine on input~$\pi$ is weakly increasing. All the definitions and notations regarding~$\sigma$-machines on Cayley permutations are inherited from the classical case. If necessary, we add an apex to avoid confusion: for instance, we denote by~$\Sort^{\Cay}(\sigma)$ the set of~$\sigma$-sortable Cayley permutations. Note that, being~$\out{\sigma}(\pi)$ the input of the~$21$-stack, Theorem~\ref{theorem_hare_stacksort} guarantees that~$\pi\in\Sort^{\Cay}(\sigma)$ if and only if~$\out{\sigma}(\pi)$ avoids~$231$. This fact will be used repeatedly for the rest of this chapter. In analogy with Definition~\ref{definition_hat_sigma}, let~$\hat{\sigma}=\sigma_2\sigma_1\sigma_3\cdots\sigma_k$ be the Cayley permutation obtained from~$\sigma$ by interchanging~$\sigma_1$ with~$\sigma_2$. Denote by~$\reverse:\Cay\to\Cay$ the reverse operator on Cayley permutations.

\begin{remark}\label{remark_hat_Cayley}
For any~$\sigma\in\Cay$, if the input Cayley permutation~$\pi$ avoids~$\reverse(\sigma)$, then the restriction of the~$\sigma$-stack is never triggered and thus~$\out{\sigma}(\pi)=\reverse(\pi)$. Otherwise, the leftmost occurrence of~$\sigma$ results necessarily in an occurrence of~$\hat{\sigma}$ in~$\out{\sigma}(\pi)$. The proof of this fact is identical to that of Lemma~\ref{lemma_reverse_outputs_hat}. An analogous result can be similarly obtained by replacing~$\Cay$ with either~$\Ascseq$ or~$\Modasc$.
\end{remark}

The next result is the analogue of Theorem~\ref{theorem_suff_cond_class} on Cayley permutations. The proof is identical, with Remark~\ref{remark_hat_Cayley} playing the role of Lemma~\ref{lemma_hat_sigma}. We report it anyway for completeness.

\begin{theorem}\label{theorem_suff_class_Cayley}
Let~$\sigma$ be a Cayley permutation. If~$\hat{\sigma}$ contains~$231$, then~$\Sort^{\Cay}(\sigma)=\Cay(132,\reverse(\sigma))$. In this case, $\Sort^{\Cay}(\sigma)$ is a class with basis either~$\{132,\reverse(\sigma)\}$, if~$\reverse(\sigma)$ avoids~$132$, or~$\{132\}$, otherwise.
\end{theorem}
\begin{proof}
We start by proving that~$\Sort^{\Cay}(\sigma)\subseteq\Cay(132,\reverse(\sigma))$. Let~$\pi\in\Sort^{\Cay}(\sigma)$. Note that~$\out{\sigma}(\pi)$ avoids~$231$. Suppose, for a contradiction, that~$\pi$ contains~$\reverse(\sigma)$. Then~$\out{\sigma}(\pi)$ contains~$\hat{\sigma}$ due to Remark~\ref{remark_hat_Cayley} and~$\hat{\sigma}$ contains~$231$ by hypothesis, which is impossible. Otherwise, if~$\pi$ avoids~$\reverse(\sigma)$, but contains~$132$, then~$\out{\sigma}(\pi)=\reverse(\pi)$ due to the same remark. Moreover~$\reverse(\pi)$ contains~$231$ by hypothesis, again a contradiction with~$\pi\in\Sort^{\Cay}(\sigma)$. This proves that~$\Sort^{\Cay}(\sigma)\subseteq\Cay(132,\reverse(\sigma))$.

Conversely, suppose that~$\pi$ avoids both~$132$ and~$\reverse(\sigma)$. Then, again by Remark~\ref{remark_hat_Cayley}, we have~$\out{\sigma}(\pi)=\reverse(\pi)$, which avoids~$\reverse(132)=231$ by hypothesis, therefore~$\pi$ is~$\sigma$-sortable.
\end{proof}

Next we show that the condition of Theorem~\ref{theorem_suff_class_Cayley} is also necessary for~$\Sort^{\Cay}(\sigma)$ in order to be a class. The only exception is given by the pattern~$\sigma=12$.

\begin{theorem}\label{theorem_12_stack_Cayley}
We have:
$$
\Sort^{\Cay}(12)=\Cay(213).
$$
\end{theorem}
\begin{proof}
Let~$\pi$ be a Cayley permutation. Suppose that~$\pi$ contains~$k$ occurrences of the minimum element~$1$ and write:
$$
\pi=A_11A_21\cdots A_k1A_{k+1}.
$$
It is easy to see that:
$$
\out{12}(\pi)=\out{12}(A_1)\out{12}(A_2)\cdots\out{12}(A_k)\out{12}(A_{k+1})1\cdots 1.
$$
Indeed any entry equal to~$1$ enters the~$12$-stack only if the~$12$-stack is either empty or contains other copies of~$1$ only. Moreover, any entry equal to~$1$ cannot play the role of~$2$ in an occurrence of the (forbidden) pattern~$12$. Therefore the presence of some copies of~$1$ at the bottom of the~$12$-stack does not affect the sorting process of the block~$A_i$, for each~$i$.

Now, suppose that~$\pi$ contains an occurrence~$bac$ of~$213$. We prove that~$\pi$ is not~$12$-sortable by showing that~$\out{12}(\pi)$ contains~$231$. We argue by induction on the length of~$\pi$. Let~$\pi=A_11A_21\cdots A_k1A_{k+1}$ and~$\out{12}(\pi)=\out{12}(A_1)\out{12}(A_2)\cdots\out{12}(A_k)\out{12}(A_{k+1})1\cdots 1$ as above. Suppose that~$b\in A_i$ and~$c\in A_j$, for some~$i\le j$ (note that~$b,c\neq 1$). If~$i=j$, then~$A_i$ contains an occurrence~$bac$ of~$213$. Thus~$\out{12}(A_i)$ contains~$231$ by the inductive hypothesis\footnote{Formally we apply the inductive hypothesis to~$\std(A_i)$, since not necessarily~$A_i$ is a Cayley permutation.}, as wanted. Otherwise, let~$i<j$. Then~$b\in\out{12}(A_i)$, $c\in\out{12}(A_j)$ and the elements~$b$ and~$c$, together with any copy of~$1$, realize an occurrence of~$231$ in~$\out{12}(\pi)$, as desired.

Conversely, suppose that~$\pi=\pi_1\cdots\pi_n$ is not~$12$-sortable, i.e.~$\out{12}(\pi)$ contains~$231$. We prove that~$\pi$ contains~$213$. Let~$bca$ be an occurrence of~$231$ in~$\out{12}(\pi)$. Observe that~$b$ must precede~$c$ in~$\pi$, since a non-inversion in the output necessarily comes from a non-inversion in the input, being the stack~$12$-avoiding. However, $b$ is extracted before~$c$ enters. Let~$x$ be the next element of the input when~$b$ is extracted. Since the stack is~$12$-avoiding, then the top~$b$ is greater than or equal to any other element contained in the~$12$-stack. Thus~$x<b$ and also~$x\neq c$, since~$c>b$. Finally, the triple~$bxc$ is an occurrence of~$213$ in~$\pi$, as desired.
\end{proof}

\begin{theorem}\label{theorem_necess_class_Cayley}
Let~$\sigma$ be a Cayley permutation and suppose that~$\sigma\neq 12$. If~$\hat{\sigma}$ avoids~$231$, then~$\Sort^{\Cay}(\sigma)$ is not a class.
\end{theorem}
\begin{proof}
Let~$\sigma=\sigma_1\cdots\sigma_k$, with~$k\ge 2$. We show that there are two Cayley permutations~$\alpha$ and~$\beta$ such that~$\alpha\le\beta$, $\beta$ is~$\sigma$-sortable and~$\alpha$ is not~$\sigma$-sortable. Table~\ref{table_alpha_beta_Cayley} shows an example of such permutations for patterns~$\sigma$ of length two and for~$\sigma=231$. Now, suppose that~$\sigma$ has length at least three and~$\sigma\neq 231$. Then the Cayley permutation~$\alpha=132$ is not~$\sigma$-sortable. Indeed, $\out{\sigma}(\alpha)=\reverse(\alpha)=231$, since~$\alpha$ avoids~$\reverse(\sigma)$. Define~$\beta$ according to the following case by case analysis:

\begin{itemize}
\item Suppose that~$\sigma_1$ is the strict minimum of~$\sigma$, that is~$\sigma_1=1$ and~$\sigma_i\ge 2$ for each~$i\ge 2$. Define:
$$
\beta=\sigma'_k\cdots\sigma'_31\sigma'_2\sigma'_1,
$$
where~$\sigma'_i=\sigma_i+1$ for each~$i$. Note that~$\beta\in\Cay$ and~$1\sigma'_2\sigma'_1$ is an occurrence of~$132$ in~$\beta$. We prove that~$\beta$ is~$\sigma$-sortable by showing that~$\out{\sigma}(\beta)$ avoids~$231$. The action of the~$\sigma$-stack on input~$\beta$ is depicted in Figure~\ref{figure_sorting_beta_Cayley}. The first~$k-1$ elements of~$\beta$ are pushed into the~$\sigma$-stack, since~$\sigma$ has length~$k$. Then the~$\sigma$-stack contains~$1\sigma'_3\cdots\sigma'_k$, reading from top to bottom, and the next element of the input is~$\sigma'_2$. Note that~$\sigma'_2>1$, whereas~$\sigma_1<\sigma_2$, therefore~$\sigma'_2 1\sigma'_3\cdots\sigma'_k$ is not an occurrence of~$\sigma$ and so~$\sigma'_2$ is pushed. The next element of the input is now~$\sigma'_1$. Here~$\sigma'_1\sigma'_2\sigma'_3\cdots\sigma'_k$ is an occurrence of~$\sigma$, thus~$\sigma'_2$ is extracted before~$\sigma'_1$ enters. After this pop operation, the~$\sigma$-stack contains~$1\sigma'_3\cdots\sigma'_k$. Again we have~$\sigma'_1>1$, whereas~$\sigma_1<\sigma_2$, therefore~$\sigma'_1$ is pushed into the~$\sigma$-stack. The resulting string is:
$$
\out{\sigma}(\beta)=\sigma'_2\sigma'_11\sigma'_3\sigma'_4\cdots\sigma'_k.
$$
We show that~$\out{\sigma}(\beta)$ avoids~$231$. Note that~$\sigma'_2\sigma'_1\sigma'_3\sigma'_4\cdots\sigma'_k\simeq\hat{\sigma}$ avoids~$231$ by hypothesis. Moreover, the element~$1$ cannot be part of an occurrence of~$231$, because~$\sigma'_2>\sigma'_1$ and~$1$ is strictly less than the other elements of~$\beta$. Therefore~$\out{\sigma}(\beta)$ avoids~$231$, as desired.

\item Next suppose that~$\sigma_1$ is not the strict minimum of~$\sigma$, i.e. either~$\sigma_1\neq 1$ or~$\sigma_i=1$ for some~$i\ge 2$. Define
$$
\beta=\sigma''_k\cdots\sigma''_21\sigma''_12,
$$
where~$\sigma''_i=\sigma_i+2$ for each~$i$. Note that~$\beta\in\Cay$ and~$1\sigma''_22$ is an occurrence of~$132$ in~$\beta$. Consider the action of the~$\sigma$-stack on input~$\beta$. Again the first~$k-1$ elements of~$\beta$ are pushed into the~$\sigma$-stack. Then the~$\sigma$-stack contains~$\sigma''_2\cdots\sigma''_k$, reading from top to bottom, and the next element of the input is~$1$. Note that~$1\sigma''_2\cdots\sigma''_k$ is not an occurrence of~$\sigma$. Indeed~$1<\sigma''_i$ for each~$i$, while~$\sigma_1$ is not the strict minimum of~$\sigma$ by hypothesis. Therefore~$1$ enters the~$\sigma$-stack. The next element of the input is then~$\sigma''_1$, which realizes an occurrence of~$\sigma$ together with~$\sigma''_2\cdots\sigma''_k$. Thus~$1$ and~$\sigma''_2$ are extracted before~$\sigma''_1$ is pushed. Finally, the last element of the input is~$2$. Again~$2$ can be pushed into the~$\sigma$-stack, since~$2$ is strictly smaller than every element in the~$\sigma$-stack, whereas~$\sigma_1$ is not the strict minimum of~$\sigma$ by hypothesis. The resulting string is:
$$
\out{\sigma}(\beta)=1\sigma''_22\sigma''_1\sigma''_3\cdots\sigma''_k.
$$
Note that~$\sigma''_2\sigma''_1\sigma''_3\cdots\sigma''_k\simeq\hat{\sigma}$ avoids~$231$ by hypothesis. Finally, it is easy to realize that the elements~$1$ and~$2$ cannot be part of an occurrence of~$231$, similarly to the previous case. This completes the proof.
\end{itemize}
\end{proof}

\begin{table}
\centering
\def\arraystretch{1.1}
\begin{tabular}{lcc}
\toprule
$\sigma$ & $\sigma$\textbf{-sortable Cayley permutation} & \textbf{Non-}$\sigma$\textbf{-sortable pattern}\\
\midrule
11 & 3132 & 132\\
21 & 35241 & 3241\\
231 & 361425 & 1324\\
\bottomrule
\end{tabular}
\caption[Non-classes of~$\sigma$-sortable Cayley permutations.]{The case by case analysis of Theorem~\ref{theorem_necess_class_Cayley}}\label{table_alpha_beta_Cayley}
\end{table}

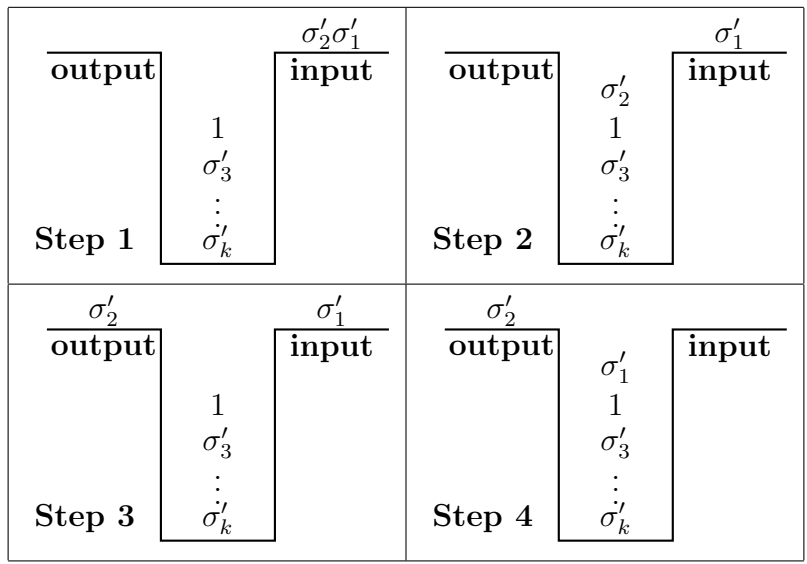
\begin{figure}
\centering
\def\arraystretch{5}
\begin{tabular}{|c|c|}
\hline
\begin{tikzpicture}[baseline=20pt]
\draw[thick] (0,3)--(1.5,3)--(1.5,0.2)--(3,0.2)--(3,3)--(4.5,3);
\node at (3.75,2.75){\textbf{input}};
\node at (0.75,2.75){\textbf{output}};
\node at (3.75,3.25){$\sigma'_2\sigma'_1$};
\node at (2.25,0.5){$\sigma'_k$};
\node at (2.25,1){$\vdots$};
\node at (2.25,1.5){$\sigma'_3$};
\node at (2.25,2){$1$};
\node at (0.5,0.5){\textbf{Step 1}};
\end{tikzpicture}
&
\begin{tikzpicture}[scale=1, baseline=20pt]
\draw[thick] (0,3)--(1.5,3)--(1.5,0.2)--(3,0.2)--(3,3)--(4.5,3);
\node at (3.75,2.75){\textbf{input}};
\node at (0.75,2.75){\textbf{output}};
\node at (3.75,3.25){$\sigma'_1$};
\node at (2.25,0.5){$\sigma'_k$};
\node at (2.25,1){$\vdots$};
\node at (2.25,1.5){$\sigma'_3$};
\node at (2.25,2){$1$};
\node at (2.25,2.5){$\sigma'_2$};
\node at (0.5,0.5){\textbf{Step 2}};
\end{tikzpicture}\\
\hline
\begin{tikzpicture}[scale=1, baseline=20pt]
\draw[thick] (0,3)--(1.5,3)--(1.5,0.2)--(3,0.2)--(3,3)--(4.5,3);
\node at (3.75,2.75){\textbf{input}};
\node at (0.75,2.75){\textbf{output}};
\node at (3.75,3.25){$\sigma'_1$};
\node at (2.25,0.5){$\sigma'_k$};
\node at (2.25,1){$\vdots$};
\node at (2.25,1.5){$\sigma'_3$};
\node at (2.25,2){$1$};
\node at (0.75,3.25){$\sigma'_2$};
\node at (0.5,0.5){\textbf{Step 3}};
\end{tikzpicture}
&
\begin{tikzpicture}[scale=1, baseline=20pt]
\draw[thick] (0,3)--(1.5,3)--(1.5,0.2)--(3,0.2)--(3,3)--(4.5,3);
\node at (3.75,2.75){\textbf{input}};
\node at (0.75,2.75){\textbf{output}};
\node at (2.25,0.5){$\sigma'_k$};
\node at (2.25,1){$\vdots$};
\node at (2.25,1.5){$\sigma'_3$};
\node at (2.25,2){$1$};
\node at (2.25,2.5){$\sigma'_1$};
\node at (0.75,3.25){$\sigma'_2$};
\node at (0.5,0.5){\textbf{Step 4}};
\end{tikzpicture}\\
\hline
\end{tabular}
\caption[The action of the~$\sigma$-stack described in Theorem~\ref{theorem_necess_class_Cayley}.]{The action of the~$\sigma$-stack on input~$\beta$ described in Theorem~\ref{theorem_necess_class_Cayley}.}\label{figure_sorting_beta_Cayley}
\end{figure}

\begin{corollary}\label{corollary_class_vs_nonclass_Cayley}
Let~$\sigma$ be a Cayley permutation of length three or more. Then~$\Sort^{\Cay}(\sigma)$ is not a permutation class if and only if~$\hat{\sigma}$ avoids~$231$. Otherwise, if~$\hat{\sigma}$ contains~$231$, then~$\Sort^{\Cay}(\sigma)$ is a class with basis either~$\{132,\reverse(\sigma)\}$, if~$\reverse(\sigma)$ avoids~$132$, or~$\{132\}$, otherwise.
\end{corollary}

Cayley permutations avoiding any classical permutation pattern of length three are enumerated by sequence A226316 in~\cite{Sl}. We end this section by analyzing the~$21$-machine. The~$11$-machine will be discussed in Section~\ref{section_operator_Cayley}, thus completing the analysis of~$\sigma$-machines for patterns of length two. The analogue of the~$21$-machine on classical permutations consists in applying a right-greedy algorithm on two stacks in series, which is precisely the well known case of West~$2$-stacksort~\cite{We2}. Recall from Theorem~\ref{theorem_west_2stack} that~$\Sort(21)=\Perm(2341,3\bar{5}241)$. The barred pattern~$3\bar{5}241$ can be represented as a mesh pattern, as shown in Figure~\ref{figure_mesh_west_Cayley}. In order to prove an analogous result for the~$12$-machine on Cayley permutations, we define mesh patterns on Cayley permutations (see~\cite{Ce}). To extend mesh patterns to strings that may contain repeated elements, we simply allow the shading of boxes that correspond to repeated elements. Instead of giving a formal definition, we refer the reader to~\cite{Ce} and to the example depicted in Figure~\ref{figure_mesh_west_Cayley}. We will use the term \textit{Cayley-mesh pattern} to denote mesh patterns on Cayley permutations. For the rest of this section, let~$\zeta$ be the Cayley-mesh pattern depicted in Figure~\ref{figure_mesh_west_Cayley}.

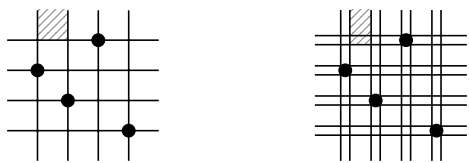
\begin{figure}
\centering
\begin{DrawPerm}
\meshBox{(1,4)}{(2,5)}
\fillPerm{3,2,4,1}{4.99}{4.99}
\end{DrawPerm}
\hspace{50pt}
\begin{tikzpicture}[scale=0.4, baseline=20pt]
\fill[NE-lines] (1.15,3.85) rectangle (1.85,4.15);
\fill[NE-lines] (1.15,4.15) rectangle (1.85,5);
\draw [semithick] (0,0.85) -- (5,0.85);
\draw [semithick] (0,1.15) -- (5,1.15);
\draw [semithick] (0,1.85) -- (5,1.85);
\draw [semithick] (0,2.15) -- (5,2.15);
\draw [semithick] (0,2.85) -- (5,2.85);
\draw [semithick] (0,3.15) -- (5,3.15);
\draw [semithick] (0,3.85) -- (5,3.85);
\draw [semithick] (0,4.15) -- (5,4.15);
\draw [semithick] (0.85,0) -- (0.85,5);
\draw [semithick] (1.15,0) -- (1.15,5);
\draw [semithick] (1.85,0) -- (1.85,5);
\draw [semithick] (2.15,0) -- (2.15,5);
\draw [semithick] (2.85,0) -- (2.85,5);
\draw [semithick] (3.15,0) -- (3.15,5);
\draw [semithick] (3.85,0) -- (3.85,5);
\draw [semithick] (4.15,0) -- (4.15,5);
\fillPerm{3,2,4,1}{0}{0}
\end{tikzpicture}
\caption[The Cayley-mesh pattern describing~$12$-sortable Cayley permutations.]{On the left, the barred pattern~$3\bar{5}241$, equivalent to the mesh pattern~$(3241,\{(1,4)\})$. On the right, the Cayley-mesh pattern~$\zeta$. The additional shaded box in~$\zeta$ keeps into account the case of an occurrence of~$3241$ that is part of an occurrence of~$34241$.}\label{figure_mesh_west_Cayley}
\end{figure}

\begin{lemma}\label{lemma_increas_stack_Cayley}
Let~$\pi=\pi_1\cdots\pi_n\in\Cay$. Suppose that~$\pi_i<\pi_j$, for some~$i<j$. Then~$\pi_i$ precedes~$\pi_j$ in~$\out{21}(\pi)$.
\end{lemma}
\begin{proof}
We have~$\pi_j\pi_i\simeq 21$, thus~$\pi_i$ must be extracted from the~$21$-stack before~$\pi_j$ enters.
\end{proof}

\begin{theorem}
We have:
$$
\Sort^{\Cay}(21)=\Cay(2341,\zeta).
$$
\end{theorem}
\begin{proof}
We can essentially repeat the argument used by West for classical permutations, but incorporating the additional shaded box in~$\zeta$, which corresponds to an occurrence of~$3241$ that is part of an occurrence of~$34241$. We sketch the proof below, leaving some technical details to the reader.

Suppose that~$\pi$ is~$21$-sortable. Suppose, for a contradiction, that~$\pi$ contains an occurrence~$bcda$ of~$2341$ and consider the action of the~$21$-stack on~$\pi$. By Lemma~\ref{lemma_increas_stack_Cayley}, $b$ is extracted from the~$21$-stack before~$c$ enters. Similarly, $c$ is extracted before~$d$ enters. Thus~$\out{21}(\pi)$ contains the occurrence~$bca$ of~$231$, a contradiction with~$\pi$ being~$21$-sortable. Otherwise, suppose that~$\pi$ contains an occurrence~$cbda$ of~$3241$. We show that there is an element~$x$ between~$c$ and~$b$ in~$\pi$ such that~$x\ge d$. If~$x<c$ for each entry~$x$ between~$c$ and~$b$, then~$b$ is pushed into the~$21$-stack before~$c$ is popped. This results in the occurrence~$bca$ of~$231$ in~$\out{21}(\pi)$, a contradiction with the fact that~$\pi$ is~$21$-sortable. Otherwise, suppose there is at least one element~$x$ between~$c$ and~$b$ in~$\pi$, with~$x\ge c$. If~$x=c$, we can repeat the same argument with~$xbda$ instead of~$cbda$. If~$c<x<d$, then~$cxda\simeq 2341$, which is impossible due to what said in the previous case. Therefore it has to be~$x\ge d$, as desired.

Conversely, suppose that~$\pi$ is not~$12$-sortable. Equivalently, let~$bca$ be an occurrence of~$231$ in~$\out{21}(\pi)$. We show that either~$\pi$ contains~$2341$ or~$\pi$ contains an occurrence~$cbda$ of~$3241$ such that~$x<d$ for each~$x$ between~$c$ and~$b$ in~$\pi$. Observe that~$a$ follows~$c$ and~$b$ in~$\pi$ due to Lemma~\ref{lemma_increas_stack_Cayley}. Suppose that~$b$ comes before~$c$ in~$\pi$. Note that~$c$ is extracted from the~$21$-stack before~$a$ enters. Let~$d$ be the next element of the input when~$c$ is extracted. Then~$d>c$ and~$bcda$ is an occurrence of~$2341$, as wanted. Otherwise, suppose that~$b$ follows~$c$ in~$\pi$, and thus~$\pi$ contains~$cba$. Since~$c$ is not extracted before~$b$ enters, it has to be~$x\le c$ for each~$x$ between~$c$ and~$b$ in~$\pi$. Moreover, $c$ is extracted before~$a$ enters. When~$c$ is extracted, the next element~$d$ of the input is such that~$d>c$. This results in an occurrence~$cbda$ of~$3241$ with the desired property.
\end{proof}

Observe that, due to the presence of the Cayley-mesh pattern~$\zeta$, the set~$\Sort^{\Cay}(21)$ is not a class. For instance, the~$21$-sortable Cayley permutation~$34241$ contains the non-sortable pattern~$3241$. The problem of enumerating~$\Sort^{\Cay}(21)$ remains to be solved.

\begin{openproblem}
Enumerate the set of~$21$-sortable Cayley permutations. The initial terms of the sequence are~$1,3,13,73,483,3547,27939,231395$ (not in~\cite{Sl}).
\end{openproblem}

Recall that West~$2$-stack sortable permutations are precisely those classical permutations that are~$21$-sortable. It would thus be interesting to find a combinatorial argument for the enumeration of~$\Sort^{\Cay}(21)$ that generalizes the one for West~$2$-stack sortable permutations (and allows us to recollect the classical result as a particular instance of this new, more general, approach).

\subsection{\texorpdfstring{Fully bijective~$\sigma$-stacks}{Fully-bijective sigma-stacks}
}\label{section_operator_Cayley}

In this section we regard a~$\sigma$-stack as a map~$\mapsigma{\sigma}:\Cay\to\Cay$ and investigate the properties of the resulting operator. This approach is slightly more general than the one adopted in Chapter~\ref{chapter_sorted_perm_fertilities}, where we studied the behavior of the restriction of~$\mapsigma{\sigma}$ to the set~$\Sort(\sigma)$ of~$\sigma$-sortable permutations. Recall (from Section~\ref{section_sorted_perm_fertilities}) that a classical permutation~$\sigma$ is said to be bijective if the map~$\mapsigma{\sigma}:\Sort(\sigma)\to\Perm(231)$ is bijective. Analogously, given a Cayley permutation~$\sigma$, we say that~$\sigma$ is \textit{bijective} if~$\mapsigma{\sigma}:\Sort^{\Cay}(\sigma)\to\Cay(231)$ is bijective. A Cayley permutation~$\sigma$ is \textit{fully bijective} if the map~$\mapsigma{\sigma}:\Cay\to\Cay$ is bijective. Notice that if~$\sigma$ is fully bijective, then~$\sigma$ is bijective and thus $\Sort^{\Cay}(\sigma)$ and~$\Cay(231)$ are Wilf-equivalent. The main goal of this section is to provide a characterization of fully bijective patterns.

We start by discussing the pattern~$\sigma=11$. The following is a useful decomposition lemma for the~$11$-stack.

\begin{lemma}\label{lemma_11_decom}
Let~$\pi=\pi_1\cdots\pi_n$ be a Cayley permutation. Suppose that~$\pi$ contains~$k+1$ occurrences~$\pi_1,\pi_1^{(1)},\dots,\pi_1^{(k)}$ of the integer~$\pi_1$, for some~$k\ge 0$. Write:
$$
\pi=\pi_1B_1\pi_1^{(1)}B_2\cdots\pi_1^{(k)}B_k.
$$
Then:
$$
\mapsigma{11}(\pi)=\mapsigma{11}(B_1)\pi_1\mapsigma{11}(B_2)\pi_1^{(1)}\cdots\mapsigma{11}(B_k)\pi_1^{(k)}.
$$
\end{lemma}
\begin{proof}
Consider the action of the~$11$-stack on input~$\pi$. Since~$x\neq\pi_1$ for each~$x\in B_1$, the sorting process of~$B_1$ is not affected by the presence of~$\pi_1$ at the bottom of the~$11$-stack. Then, when the next element of the input is the second occurrence~$\pi_1^{(1)}$ of~$\pi_1$, the~$11$-stack is emptied, since~$\pi_1\pi_1^{(1)}$ is an occurrence of the forbidden pattern~$11$. The initial elements of~$\mapsigma{11}(\pi)$ are thus~$\mapsigma{11}(B_1)\pi_1$. Finally, $\pi_1^{(1)}$ is pushed into the (empty)~$11$-stack and the same argument can be repeated.
\end{proof}

\begin{theorem}\label{theorem_11_bij}
The map~$(\reverse\circ\mapsigma{11})$ is an involution on~$\Cay$. Therefore the pattern~$11$ is fully bijective.
\end{theorem}
\begin{proof}
We proceed by induction on the length of input permutations. Let~$\pi=\pi_1\cdots\pi_n$ be a Cayley permutation of length~$n$. The case~$n=1$ is trivial. If~$n\ge 2$, write~$\pi=\pi_1B_1\pi_1^{(1)}B_2\cdots\pi_1^{(k)}B_k$ as in Lemma~\ref{lemma_11_decom}. Then, using the same lemma and the inductive hypothesis:
\begin{equation*}
\begin{split}
\left[\reverse\circ\mapsigma{11}\right]^2(\pi)=&\\
\left[\reverse\circ\mapsigma{11}\right]^2\left(\pi_1 B_1\pi_1^{(1)} B_2\cdots\pi_1^{(k)} B_k\right) =&\\
\left[\reverse\circ\mapsigma{11}\circ\reverse\right]\left(\mapsigma{11}(B_1)\pi_1\mapsigma{11}(B_2)\pi_1^{(1)}\cdots\mapsigma{11}(B_k)\pi_1^{(k)}\right) =&\\
\left[\reverse\circ\mapsigma{11}\right]\left(\pi_1^{(k)}\reverse(\mapsigma{11}(B_k))\cdots\pi_1^{(1)}\reverse(\mapsigma{11}(B_2))\pi_1\reverse(\mapsigma{11}(B_1))\right) =&\\
\reverse\left( \mapsigma{11}(\reverse(\mapsigma{11}(B_k)))\pi_1^{(k)}\cdots\mapsigma{11}(\reverse(\mapsigma{11}(B_2)))\pi_1^{(1)}\mapsigma{11}(\reverse(\mapsigma{11}(B_1)))\pi_1\right) =&\\
\pi_1\left[\reverse\circ\mapsigma{11}\right]^2 (B_1)\pi_1^{(1)}\left[\reverse\circ\mapsigma{11}\right]^2(B_2)\cdots\pi_1^{(k)}\left[\reverse\circ\mapsigma{11}\right]^2(B_k) =&\\
\pi_1 B_1\pi_1^{(1)} B_2\cdots\pi_1^{(k)} B_k =\pi &\\
\end{split}
\end{equation*}
Therefore~$(\reverse\circ\mapsigma{11})^2(\pi)=\pi$, as desired. Finally, the reverse map~$\reverse$ is bijective, thus~$\mapsigma{11}$ is a bijection on~$\Cay$ with inverse~$\reverse\circ\mapsigma{11}\circ\reverse$.
\end{proof}

An immediate consequence of Theorem~\ref{theorem_11_bij} is that every~$11$-sortable Cayley permutation~$\pi$ is obtained from a~$231$-avoiding Cayley permutation by applying~$\reverse\circ\mapsigma{11}\circ\reverse$, that is:
$$
\Sort^{\Cay}(11)=\reverse\circ\mapsigma{11}\circ\reverse(\Cay(231)).
$$
We wish to generalize this result by encoding the action of~$\mapsigma{\sigma}$ as a labeled Dyck path. In what follows, we always consider labeled Dyck paths where the label of each up step is equal to the label of its matching down step. This allows us to represent a labeled Dyck path as a pair~$\mathcal{P}=(P,\pi)$, where~$P$ is the underlying Dyck path and~$\pi$ is the string obtained by reading the labels of the up steps of~$P$ from left to right. Let~$\sigma$ be a Cayley permutation and let~$\pi$ be an input permutation for the~$\sigma$-stack. Define a labeled Dyck path~$\pathsigma{\sigma}(\pi)$ as follows, starting from the empty path:

\begin{itemize}
\item insert an up step~$\U$ labeled~$a$ whenever an element~$a$ is pushed into the~$\sigma$-stack;
\item insert a down step~$\D$ labeled~$a$ whenever an element~$a$ is extracted from the~$\sigma$-stack.
\end{itemize}

Equivalently, if~$P_{\sigma}(\pi)$ is the unlabeled Dyck path obtained by recording the push operations of the~$\sigma$-stack with~$\U$ and the pop operations with~$\D$, then~$\pathsigma{\sigma}(\pi)=(P_{\sigma}(\pi),\pi)$. It is easy to realize that~$P_{\sigma}(\pi)$ is a Dyck path. Indeed the number of push and pop operations performed by the~$\sigma$-stack on~$\pi$ is the same (it is equal to the length of~$\pi$), therefore the number of~$\U$ steps matches the number of~$\D$ steps (and thus the path ends on the~$x$-axis). Moreover, the path never goes below the~$x$-axis, since this would correspond to performing a pop operation when the~$\sigma$-stack is empty, which is not possible. An example of this construction, when~$\sigma=11$, is depicted in Figure~\ref{figure_Dyck_pathsigma}. We collect several properties of~$\pathsigma{\sigma}(\pi)$ in the following lemma, whose easy proof is omitted.

\begin{figure}
\centering
\begin{tikzpicture}[scale=0.75, baseline=20pt]
\draw [help lines] (0,0) -- (10,0);
\draw [thick] (0,0) -- (4,4);
\draw [thick] (4,4) -- (7,1);
\draw [thick] (7,1) -- (8,2);
\draw [thick] (8,2) -- (10,0);
\draw [dotted] (0.5,0.5) -- (9.5,0.5);
\draw [dotted] (1.5,1.5) -- (6.5,1.5);
\draw [dotted] (2.5,2.5) -- (5.5,2.5);
\draw [dotted] (3.5,3.5) -- (4.5,3.5);
\draw [dotted] (7.5,1.5) -- (8.5,1.5);
\filldraw (0,0) circle (3pt);
\filldraw (1,1) circle (3pt);
\filldraw (2,2) circle (3pt);
\filldraw (3,3) circle (3pt);
\filldraw (4,4) circle (3pt);
\filldraw (5,3) circle (3pt);
\filldraw (6,2) circle (3pt);
\filldraw (7,1) circle (3pt);
\filldraw (8,2) circle (3pt);
\filldraw (9,1) circle (3pt);
\filldraw (10,0) circle (3pt);
\node[above,left] at (0.75,0.75) {$4$};
\node[above,left] at (1.75,1.75) {$2$};
\node[above,left] at (2.75,2.75) {$1$};
\node[above,left] at (3.75,3.75) {$3$};
\node[below,left] at (4.75,3.25) {$3$};
\node[below,left] at (5.75,2.25) {$1$};
\node[below,left] at (6.75,1.25) {$2$};
\node[above,left] at (7.75,1.75) {$2$};
\node[below,left] at (8.75,1.25) {$2$};
\node[below,left] at (9.75,0.25) {$4$};
\end{tikzpicture}
\caption[A Dyck path encoding the action of the~$11$-machine.]{The Dyck path~$UUUUDDDUDD$ which encodes~$\mapsigma{11}(42132)$. Dotted lines connect matching steps, which have the same label.}\label{figure_Dyck_pathsigma}
\end{figure}
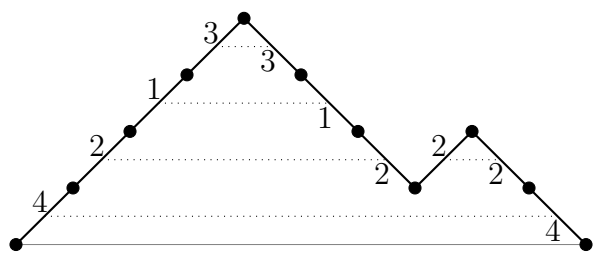

\begin{lemma}\label{lemma_pathsigma_prop}
Let~$\sigma$ be a Cayley permutation. Let~$\pi=\pi_1\cdots\pi_n$ be a Cayley permutation of length~$n$ and let~$\pathsigma{\sigma}(\pi)=(P_{\sigma}(\pi),\pi)$. Then:
\begin{enumerate}
\item The input~$\pi$ is obtained by reading the labels of the up steps of~$P_{\sigma}(\pi)$ from left to right. The output~$\out{\sigma}(\pi)$ is obtained by reading the labels of the down steps from left to right.

\item The height of~$P_{\sigma}(\pi)$ after each up (respectively down) step is equal to the number of elements contained in the~$\sigma$-stack after the corresponding push (respectively pop) operation.

\item A pop operation empties the~$\sigma$-stack if and only if the corresponding~$\D$ step of~$P_{\sigma}(\pi)$ ends on the~$x$-axis. Notice that the decomposition of~$\pi$ considered in Lemma~\ref{lemma_11_decom} corresponds to the decomposition of~$P_{\sigma}(\pi)$ obtained by considering the returns on the~$x$-axis.

\item The labels of the down steps are uniquely determined by the labels of the up steps. Conversely, the labels of the down steps uniquely determine the labels of the up steps. More precisely, matching steps have the same label. Indeed any element pushed into the~$\sigma$-stack by an up step is then popped by the matching down step.

\item Two consecutive steps of~$P_{\sigma}(\pi)$ form a valley~$\D\U$ if and only if, denoting by~$a$ the label of~$\D$ and~$b$ the label of~$\U$, $b$ plays the role of~$\sigma_1$ in an occurrence of~$\sigma$ that triggers the restriction of the~$\sigma$-stack, while~$a$ plays the role of~$\sigma_2$ in that same occurrence. Therefore the number of valleys of~$P_{\sigma}(\pi)$ is equal to the number of elements of~$\pi$ that trigger the restriction of the~$\sigma$-stack.

\item If~$\sigma_1=\sigma_2$, then the steps~$\D$ and~$\U$ in each valley~$\D\U$ have the same label.
\end{enumerate}
\end{lemma}

Recall from Section~\ref{section_lattice_paths} that the reverse path~$\reverse(P)$ of a Dyck path~$P$ is its symmetric with respect to the vertical line~$x=n$, where~$n$ is the semilength of~$P$. The following lemma shows that if~$\sigma_1=\sigma_2$, then the path that encodes the action of the~$\sigma$-stack on~$\pi$ is the reverse of the path that encodes the action of the~$\sigma$-stack on input~$\reverse(\mapsigma{\sigma}(\pi))$.

\begin{lemma}\label{lemma_pathsigma_reverse}
Let~$\sigma=\sigma_1\cdots\sigma_k$ be a Cayley permutation. Let~$\pi=\pi_1\cdots\pi_n$ be a Cayley permutation and let~$\gamma=\reverse(\mapsigma{\sigma}(\pi))$. Consider the two labeled Dyck paths~$\pathsigma{\sigma}(\pi)=(P_{\sigma}(\pi),\pi)$ and~$\pathsigma{\sigma}(\gamma)=(P_{\sigma}(\gamma),\gamma)$.
\begin{enumerate}
\item If~$\sigma_1=\sigma_2$, then~$P_{\sigma}(\pi)=\reverse(P_{\sigma}(\gamma))$.
\item If~$P_{\sigma}(\pi)=\reverse(P_{\sigma}(\gamma))$, then~$(\reverse\circ\mapsigma{\sigma})^2(\pi)=\pi$.
\end{enumerate}
\end{lemma}
\begin{proof} 
\begin{enumerate}
\item Suppose that~$\sigma_1=\sigma_2$. We proceed by induction on the number of valleys of~$P_{\sigma}(\pi)$. If~$P_{\sigma}(\pi)$ has zero valleys, then~$\pi$ avoids~$\reverse(\sigma)$ by point~$5.$ of Lemma~\ref{lemma_pathsigma_prop}. Therefore~$\mapsigma{\sigma}(\pi)=\reverse(\pi)$ and~$\gamma=\reverse^2(\pi)=\pi$. Since~$P_{\sigma}(\pi)=\U^n\D^n$ is a pyramid, the thesis follows immediately since each pyramid is equal to its reverse.

Now suppose that~$P_{\sigma}(\pi)$ has at least one valley. Let~$P_{\sigma}(\pi)=p_1\cdots p_{2n}$ and write~$P_{\sigma}(\pi)=\U^i\U^j\D^j\U^l\D Q$, where~$p_{i+2j},p_{i+2j+1}$ is the leftmost valley and~$Q=p_{i+2j+l+2}\cdots p_n$ is the remaining suffix of~$P_{\sigma}(\pi)$ (see Figure~\ref{figure_pathsigma_prefix}). Observe that the label of both~$p_{i+2j}$ and~$p_{i+2j+1}$ is equal to~$\pi_{i+1}$ as a consequence of items~$4.$, $5.$ and~$6.$ of Lemma~\ref{lemma_pathsigma_prop}. Item~$5.$ also implies that~$p_{i+2j+1}$ plays the role of~$\sigma_1$ in an occurrence of~$\sigma$ that triggers the restriction of the~$\sigma$-stack. More precisely, as soon as~$\pi_{i+j}$ is pushed (i.e. after the up step~$p_{i+j}$ in~$P_{\sigma}(\pi)$), $\pi_{i+j+1}$ is the next element of the input. Since the next segment of the path is~$\D^j$, $j$ pop operations are performed before pushing~$\pi_{i+j+1}$. This means that the element~$\pi_{i+1}$, corresponding to the last down step, plays the role of~$\sigma_2$ in an occurrence of~$\sigma$, while~$\pi_{i+j+1}$ plays the role of~$\sigma_1$. Moreover, there are~$k-2$ elements in the~$\sigma$-stack that play the role of~$\sigma_3,\dots,\sigma_k$. Since the elements in the~$\sigma$-stack correspond to the labels of the initial prefix~$\U^i$, $\pi_1\cdots\pi_i$ contains an occurrence of~$\sigma_k\cdots\sigma_3$ (claim~I). Then, after that~$j$ pop operations are performed, the~$\sigma$-stack contains~$\pi_i\cdots\pi_1$, reading from top to bottom, and the elements~$\pi_{i+j+1},\pi_{i+j+2},\dots,\pi_{i+j+l}$ are pushed (claim~II). Now, write:
$$
\pi=\underbrace{\pi_1\cdots\pi_i}_{A}\ \underbrace{\pi_i+1\cdots\pi_{i+j}}_{B}\ \underbrace{\pi_{i+j+1}\cdots\pi_{i+j+l}}_{C}\ \underbrace{\pi_{i+j+l+1}\cdots\pi_n}_{D},
$$
where the elements of~$A$ correspond to the initial prefix~$\U^i$ of~$P_{\sigma}(\pi)$, $B$ corresponds to~$\U^j$, $C$ to~$\U^l$ and~$D$ to the remaining up steps. Consider the string:
$$
ACD=\pi_1\cdots\pi_i\pi_{i+j+1}\cdots\pi_n,
$$
obtained by removing the segment~$B=\pi_{i+1}\cdots\pi_{i+j}$ from~$\pi$. Let~$\tilde{\pi}=\std(ACD)$ be the standardization of~$ACD$. Note that~$\pathsigma{\sigma}(\tilde{\pi})$ is obtained from~$\pathsigma{\sigma}(\pi)$ by cutting out the pyramid~$\U^j\D^j$, which corresponds to the removed segment~$B$. Indeed the elements contained in the~$\sigma$-stack when~$\pi_i$ enters are exactly the same as the elements contained in the~$\sigma$-stack when~$\pi_{i+j+1}$ is pushed, thus we can safely cut out the pyramid~$\U^j\D^j$ without affecting the sorting procedure. Therefore:
$$
\mapsigma{\sigma}(\pi)=\reverse(B)\mapsigma{\sigma}(\tilde{\pi})
$$
and
$$
\gamma=\reverse(\mapsigma{\sigma}(\pi))=\reverse(\mapsigma{\sigma}(\tilde{\pi}))B.
$$
Now, since~$P_{\sigma}(\tilde{\pi})$ has one valley less than~$P_{\sigma}(\pi)$, by the inductive hypothesis we have~$P_{\sigma}(\tilde{\pi})=\reverse(P_{\sigma}(\tilde{\gamma}))$, where~$\tilde{\gamma}=\reverse(\mapsigma{\sigma}(\tilde{\pi}))$. The only difference bewteen~$P_{\sigma}(\pi)$ and~$P_{\sigma}(\tilde{\pi})$ is the removed pyramid~$\U^j\D^j$. Therefore, if we show that~$P_{\sigma}(\gamma)$ is obtained from~$P_{\sigma}(\tilde{\gamma})$ by reinserting the pyramid~$\U^j\D^j$ in the same place, the thesis follows. We have~$\gamma=\reverse(\mapsigma{\sigma}(\tilde{\pi}))B$ and~$\tilde{\gamma}=\reverse(\mapsigma{\sigma}(\tilde{\pi}))$. Consider the last push operation performed by the~$\sigma$-stack when processing~$\tilde{\gamma}$, which corresponds to the last up step of~$\pathsigma{\sigma}(\tilde{\gamma})$. Note that, since~$P_{\sigma}(\tilde{\pi})=\reverse(P_{\sigma}(\tilde{\gamma})$, this is also the first down step of~$P_{\sigma}(\tilde{\pi})$, and thus the first pop operation performed when processing~$\tilde{\pi}$. Therefore the elements contained in the~$\sigma$-stack when the last push operation is performed, while processing~$\tilde{\gamma}$, are~$\pi_{i+j+l}\cdots\pi_{i+j+1}\pi_i\cdots\pi_1$, reading from top to bottom. If we sort~$\gamma$ instead of~$\tilde{\gamma}$, we have to process the additional segment~$B$. Now, the first element of~$B$ is~$\pi_{i+1}$. As a consequence of claim~I, $\pi_{i+1}$ realizes an occurrence of~$\sigma$ together with~$\pi_{i+j+1}$ (which plays the role of~$\sigma_2$) and other~$k-2$ elements in~$\pi_1\cdots\pi_i$. The only difference is that, contrary to what happened when sorting~$\pi$, the role of~$\pi_{i+1}$ and~$\pi_{i+j+1}$ are interchanged: here the hypothesis~$\sigma_1=\sigma_2$ is relevant. As a result, before pushing the first element~$\pi_{i+1}$ of~$B$, we have to pop each element of the~$\sigma$-stack up to~$\pi_{i+j+1}$, $\pi_{i+j+1}$ included. Then, the~$\sigma$-stack contains~$\pi_i\cdots\pi_1$, reading from top to bottom. Therefore we can push~$\pi_{i+1}=\pi_{i+j+1}$ and the remaining elements of~$B$, this time because of claim~II. This means that~$P_{\sigma}(\gamma)$ is obtained by inserting a pyramid~$\U^j\D^j$ immediately before the last~$i$ down steps of~$P_{\sigma}(\tilde{\gamma})$, as desired.

\item By hypothesis, $P_{\sigma}(\gamma)=\reverse(P_{\sigma}(\pi))$, therefore the word~$w$ obtained by reading the labels of the down steps of~$P_{\sigma}(\gamma)$ (from left to right) is~$w=\reverse(\pi)$. By definition of~$\pathsigma{\sigma}(\gamma)$, we also have~$w=\mapsigma{\sigma}(\gamma)$. Therefore:
$$
\reverse(\pi)=\mapsigma{\sigma}(\gamma)=\mapsigma{\sigma}(\reverse(\mapsigma{\sigma}(\pi)))$$
and the thesis follows by applying the reverse operator to both sides of the equality.
\end{enumerate}
\end{proof} 

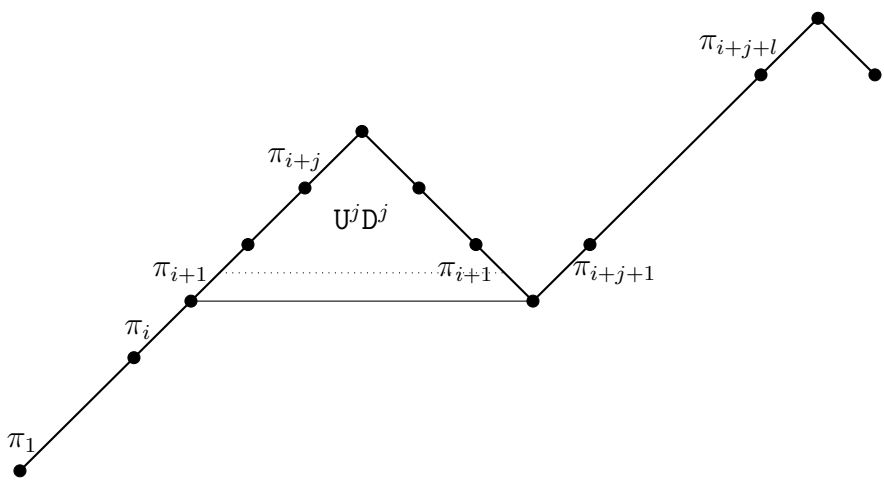
\begin{figure}
\centering
\begin{tikzpicture}[scale=0.75, baseline=20pt]
\draw [thick] (0,0) -- (6,6);
\draw [thick] (6,6) -- (9,3);
\draw [thick] (9,3) -- (14,8);
\draw [thick] (14,8) -- (15,7);
\filldraw (0,0) circle (3pt);
\filldraw (2,2) circle (3pt);
\filldraw (3,3) circle (3pt);
\filldraw (4,4) circle (3pt);
\filldraw (5,5) circle (3pt);
\filldraw (6,6) circle (3pt);
\filldraw (7,5) circle (3pt);
\filldraw (8,4) circle (3pt);
\filldraw (9,3) circle (3pt);
\filldraw (10,4) circle (3pt);
\filldraw (13,7) circle (3pt);
\filldraw (14,8) circle (3pt);
\filldraw (15,7) circle (3pt);
\node at (6,4.5) {$\U^j\D^j$};
\node[above,left] at (2.5,2.5){$\pi_{i}$};
\node[above,left] at (5.5,5.5){$\pi_{i+j}$};
\node[above,left] at (13.5,7.5){$\pi_{i+j+l}$};
\node[above,left] at (0.5,0.5){$\pi_{1}$};
\node[above,left] at (3.5,3.5){$\pi_{i+1}$};
\node[left] at (8.5,3.5){$\pi_{i+1}$};
\node[right] at (9.5,3.5){$\pi_{i+j+1}$};
\draw [dotted] (3.5,3.5) -- (8.5,3.5);
\draw [dotted] (1,1) -- (2,2);
\draw [thin] (3,3) -- (9,3);
\end{tikzpicture}
\caption[The path~$P_{\sigma}(\pi)$ of Lemma~\ref{lemma_pathsigma_reverse}.]{The (prefix of the) path~$P_{\sigma}(\pi)$ mentioned in the proof of Lemma~\ref{lemma_pathsigma_reverse}.}\label{figure_pathsigma_prefix}
\end{figure}

As a consequence of what proved so far in this section, we obtain the desired characterization of fully bijective Cayley permutations.

\begin{theorem}\label{theorem_pathsigma_bij}
Let~$\sigma=\sigma_1\cdots\sigma_k\in\Cay$. Then~$\sigma$ is fully bijective if and only if~$\sigma_1=\sigma_2$.
\end{theorem}
\begin{proof}
Suppose that~$\sigma_1\neq\sigma_2$. Then~$\hat{\sigma}\neq\sigma$ and so~$\reverse(\sigma)\neq\reverse(\hat{\sigma})$. Finally, it is easy to realize that:
$$
\mapsigma{\sigma}(\reverse(\sigma))=\hat{\sigma}=\mapsigma{\sigma}((\reverse(\hat{\sigma}))),
$$
therefore~$\mapsigma{\sigma}$ is not injective on~$\Cay$.

Conversely, suppose that~$\sigma_1=\sigma_2$. By Lemma~\ref{lemma_pathsigma_reverse}, we have that~$(\reverse\circ\mapsigma{\sigma})^2$ is the identity on~$\Cay$, therefore~$\reverse\circ\mapsigma{\sigma}$ is bijective. Finally, since the reverse map~$\reverse$ is bijective, $\mapsigma{\sigma}$ is a bijection too, as desired.
\end{proof}

\begin{remark}\label{remark_R_circ_sigma_bij}
As we pointed out in the proof of Lemma~\ref{lemma_pathsigma_reverse}, the hypothesis~$\sigma_1=\sigma_2$ guarantees that the Dyck path obtained when sorting~$\pi$ is equal to the reverse of the path obtained when sorting~$\reverse(\mapsigma{\sigma}(\pi))$. This is sufficient for the operator~$\mapsigma{\sigma}$ in order to be bijective on~$\Cay$. More precisely, the crucial property is that the roles of~$\sigma_1$ and~$\sigma_2$ are interchanged in the paths associated to~$\pi$ and~$\reverse(\mapsigma{\sigma}(\pi))$: this is precisely where the hypothesis~$\sigma_1=\sigma_2$ plays a role. Due to this reason, an analogous argument can be repeated on stacks avoiding the pair of patterns~$(\sigma,\hat{\sigma})$. Indeed every time an occurrence of~$\sigma$ triggers the~$(\sigma,\hat{\sigma})$-stack on input~$\pi$, an occurrence of~$\hat{\sigma}$, where the roles of~$\sigma_1$ and~$\sigma_2$ are interchanged, triggers the~$(\sigma,\hat{\sigma})$-stack on input~$\reverse(\out{\sigma_1,\sigma_2}(\pi))$. The converse statement is true as well. Therefore the operator~$\mapsigma{\sigma,\hat{\sigma}}$ is bijective on~$\Cay$. Similarly, since classical permutations are mapped into classical permutations by any operator~$\mapsigma{\Sigma}$, the operator~$\mapsigma{\sigma,\hat{\sigma}}$ associated to the classical~$(\sigma,\hat{\sigma})$-machine is bijective on~$\Perm$. This result was generalized by Berlow in~\cite{Ber}: the map~$\mapsigma{\Sigma}$ is bijective if and only if for every~$\sigma\in\Sigma$ we have~$\hat{\sigma}\in\Sigma$ as well.
\end{remark}

\begin{remark}\label{remark_enumeration_via_paths}
The encoding of the action of a~$\sigma$-stack as Dyck paths could theoretically lead to the enumeration of~$\Sort^{\Cay}(\sigma)$. Indeed, due to Lemma~\ref{lemma_pathsigma_prop}, the number of~$\sigma$-sortable Cayley permutations of length~$n$ is equal to the number of labeled Dyck paths of semilength~$n$ such that:
\begin{itemize}
\item Reading the labels of the down steps from left to right yields a~$231$-avoiding Cayley permutations.
\item Each valley of the path corresponds to an element of the input permutation that triggers the restriction of the~$\sigma$-stack, as described in item~$5.$ of Lemma~\ref{lemma_pathsigma_prop}.
\end{itemize}
A natural question would thus be the following. Given a Dyck path~$P$, are there any parameters of~$P$ that allow us to describe (and enumerate) the set of Cayley permutations that can be used to suitably label~$P$? In other words, can we describe those Cayley permutations where the action of the~$\sigma$-stack is encoded by the same Dyck path~$P$?
\end{remark}

\section{\texorpdfstring{The~$\sigma$-machine on ascent sequences}{The sigma-machine on ascent sequences}}\label{section_asc_seq_intro}

We spend the last section of this chapter by discussing~$\sigma$-machines on classical ascent sequences~$\Ascseq$ and modified ascent sequences~$\Modasc$ (see Section~\ref{section_sequences_integers}). Recall that ascent sequences are a bijective encoding of Fishburn permutations~$\Fish=\Perm(\fishpattern)$, where~$\fishpattern$ is the bivincular pattern~$\fishpattern=(231,\lbrace 1\rbrace,\lbrace 1\rbrace)$. The maps that link the sets~$\Ascseq$, $\Modasc$ and~$\Fish$, as well as the pattern~$\fishpattern$, are depicted again (for convenience) in Figure~\ref{figure_fishpat_diagram_ascseq_recall}.

\begin{figure}
\begin{minipage}{7cm}
\centering
$
\fishpattern=
\begin{DrawPerm}
\meshBox{(1,0)}{(2,4)}
\meshBox{(0,1)}{(4,2)}
\fillPerm{2,3,1}{3.99}{3.99}
\end{DrawPerm}
$
\end{minipage}
\begin{minipage}{7cm}
\centering
$
\xymatrix{
\Ascseq \ar@{->}[rr]^{\phi} \ar@{->}[d]_{\mod}& & \Fish\\
\Modasc\ar@{->}[urr]_{\phi'} & &\\
}
$
\end{minipage}
\caption[]{The pattern~$\fishpattern$, on the left. How the bijections~$\phi$, $\mod$ and~$\phi'$ are related, on the right.}\label{figure_fishpat_diagram_ascseq_recall}
\end{figure}

Recently, Claesson and the current author~\cite{CeCl} initiated the development of a theory of transport of patterns between Fishburn permutations and ascent sequences. One of their goals is to achieve a better understanding of the notion of pattern involvement on the sets~$\Ascseq$ and~$\Modasc$, which proved to be rather challenging. The transport relies on a high-level generalization of~$\phi'$, which they call the \textit{Burge transpose}. One of the main results of~\cite{CeCl} is an explicit construction of a set of patterns~$\mathcal{B}(p)$ such that~$\phi'$ maps the set~$\Modasc(\mathcal{B}(p))$ of modified ascent sequences avoiding every pattern in~$\mathcal{B}(p)$ to the set~$\Fish(p)$ of Fishburn permutations avoiding~$p$. The set~$\mathcal{B}(p)$ is called the \textit{Fishburn basis} of~$p$. We refer the reader to~\cite{CeCl} for the definition of Burge transpose and for a detailed construction of~$\mathcal{B}(p)$.

\begin{theorem}[\cite{CeCl}, Transport of patterns from~$\Fish$ to~$\Modasc$]\label{theorem_transport_modasc_fish}
For any permutation~$p$, we have:
$$
\Fish(p)=\phi'\bigl(\Modasc(\mathcal{B}(p))\bigr).
$$
Therefore~$\Fish(p)$ and~$\Modasc(\mathcal{B}(p))$ are Wilf-equivalent subsets of~$\Cay$.
\end{theorem}

The current author aims to add one more piece to the general picture by analyzing the behavior of~$\sigma$-machines on ascent and modified ascent sequences. It is worth noticing that, in some cases, we can use the results obtained in the previous section on Cayley permutations, as we show in the following result.

\begin{theorem}\label{theorem_Cayley_class_to_ascseq}
Let~$X\in\lbrace\Ascseq,\Modasc\rbrace$ and let~$\sigma\in\Cay\cap X$. If~$\Sort^{\Cay}(\sigma)=\Cay(A)$, for a set of patterns~$A$, then~$\Sort^{X}(\sigma)=X(A)$.
\end{theorem}
\begin{proof}
Let~$w\in X$ and let~$w'=\std(w)$ be the standardization of~$w$. Notice that~$w'\in\Cay$. It is easy to observe that, for every pattern~$y$, $w$ contains~$y$ if and only if~$w'$ contains~$y$. More precisely, any subsequence~$w_{i_1}\cdots w_{i_k}$ of~$w$ is order isomorphic to the corresponding subsequence~$w'_{i_1}\cdots w'_{i_k}$ of~$w'=\std(w)$. An immediate consequence is that the action of the~$\sigma$-stack on~$w$ is identical to the action of the~$\sigma$-stack on~$w'$. Therefore~$w$ is~$\sigma$-sortable if and only if~$w'$ is~$\sigma$-sortable, which in turn is equivalent to~$w'\in\Cay(A)$ by hypothesis. Finally, $w'\in\Cay(A)$ if and only if~$w\in X(A)$, thus the thesis follows.
\end{proof}

Due to Corollary~\ref{corollary_class_vs_nonclass_Cayley} and Theorem~\ref{theorem_Cayley_class_to_ascseq}, if~$\sigma\in\Cay\cap\Ascseq$ and~$\hat{\sigma}$ contains~$231$, then~$\Sort^{\Ascseq}(\sigma)$ is a class with basis~$\lbrace 132,\reverse(\sigma)\rbrace$. An analogous result holds for modified ascent sequences. On the other hand, if~$\Sort^{\Cay}(\sigma)$ is not a class, not necessarily the same holds for~$\Sort^{\Ascseq}(\sigma)$ and~$\Sort^{\Modasc}(\sigma)$.

\subsection{Classical ascent sequences}\label{section_asc_seq_stack}

We start by recalling some useful results from the literature. The following lemma was proved by Duncan and Steingrimsson in~\cite{DS}.

\begin{lemma}\cite{DS}\label{lemma_ascseq_is_RGF}
The set~$\Ascseq(y)$ consists solely of {\rgfs} if and only if~$y\le 12123$.
\end{lemma}

\begin{lemma}\label{lemma_avoids123_maxatmost2}
Let~$x\in\Ascseq(123)$. Then~$\max(x)\le 2$.
\end{lemma}
\begin{proof}
It follows from Lemma~\ref{lemma_ascseq_is_RGF} and the definition of {\rgf}.
\end{proof}

Next we show that the pattern~$\sigma=11$ is an instance where~$\Sort^{\Ascseq}(\sigma)$ is a permutation class, whereas~$\Sort^{\Cay}(\sigma)$ is not.

\begin{theorem}\label{theorem_ascseq_11}
We have:
$$
\Sort^{\Ascseq}(11)=\Ascseq(1213,1223).
$$
\end{theorem}
\begin{proof}
Let~$x\in\Ascseq$. Observe that the first element~$x_1=1$ is extracted from the~$11$-stack if and only if the next element of the input is equal to~$1$, which in this case replace~$x_1$ at the bottom of the~$11$-stack. Thus the last element~$x_{last}$ of~$\out{11}(x)$ is~$x_{last}=1$. This fact will be repeatedly used for the rest of this proof.

Suppose initially that~$x$ contains an occurrence~$x_ix_jx_kx_{\ell}$ of~$1213$. Without losing generality, we can assume that~$x_ix_jx_k$ is the leftmost~$121$ in any occurrence of~$1213$ where~$x_{\ell}$ plays the role of~$3$. We wish to prove that~$x$ is not~$11$-sortable by showing that~$\out{11}(x)$ contains~$231$. If~$x_i$ is contained in the~$11$-stack when~$x_k$ is the next element of the input, then, since~$x_kx_i$ is an occurrence of~$11$, every element up to~$x_i$ must be extracted before pushing~$x_k$. Therefore~$x_j$ is extracted before~$x_k$ enters and~$\out{11}(x)$ contains the occurrence~$x_jx_{\ell}x_{last}$ of~$231$, as wanted. Otherwise, suppose that~$x_i$ is extracted before~$x_j$ enters the~$11$-stack. Let~$y$ be the next element of the input when~$x_i$ is extracted. Consider the following two cases.

\begin{itemize}
\item $yx_i$ is an occurrence of~$11$ (and thus~$x_i$ must be extracted); in this case we can repeat the same argument replacing~$x_i$ with~$y$ until we fall in the next case.
\item $y\neq x_i$ and there is an entry~$y'$ contained in the~$11$-stack such that~$yy'$ is an occurrence of~$11$ (thus the top element~$x_i$ must be extracted). If~$y<x_i$, then~$y'x_iy$ is an occurrence of~$121$ (with~$x_i<x_{\ell}$) that precedes~$x_ix_jx_k$, which is impossible due to our choice of~$x_ix_jx_k$. On the other hand, suppose that~$y>x_i$. Then it must be~$x_i>1$, or else~$x_1x_jx_k$ would be an occurrence of~$121$ that precedes~$x_ix_jx_k$, which is again impossible. Finally, we get the desired occurrence~$x_iyx_{last}$ of~$231$ in~$\out{11}(x)$, as desired.
\end{itemize}

Next suppose that~$x$ contains an occurrence~$x_ix_jx_kx_{\ell}$ of~$1223$. Then~$x_j$ must be extracted from the~$11$-stack before~$x_k$ enters (since~$x_jx_k$ is an occurrence of~$11$) and thus~$\out{11}(x)$ contains an occurrence~$x_jx_{\ell}x_{last}$ of~$231$, which is what we wanted.

Conversely, suppose that~$x$ avoids~$1213$ and~$1223$. Suppose, for a contradiction, that~$x$ is not~$11$-sortable. Equivalently, let~$x_i=b$, $x_j=c$ and~$x_k=a$ be three elements of~$x$ that result in an occurrence~$x_ix_jx_k=bca$ of~$231$ in~$\out{11}(x)$. Since~$x$ avoids~$1213$, then~$x$ is a {\rgf} by Lemma~\ref{lemma_ascseq_is_RGF}. Moreover, again due to the avoidance of~$1213$, the prefix~$x_1\cdots x_j$ of~$x$ must be weakly increasing, that is:
$$
x=1^{t_1}2^{t_2}\cdots a^{t_a}\cdots b^{t_b}\cdots(c-1)^{t_{c-1}}c^{t_c}x_j\cdots,
$$
for some integers~$t_u\ge 1$, $u=1,\dots,c-1$, and~$t_c\ge 0$. Moreover, since~$x$ avoids~$1223$, it must be~$t_u=1$ for each~$2\le u\le c-1$. Now, it is easy to observe that as soon as~$x_j=c$ enters the~$11$-stack, the content of the~$11$-stack is~$(c-1)\cdots b\cdots a\cdots 321$, reading from top to bottom. Therefore, if~$i>j$ then~$x_j$ is extracted before~$x_i$ enters, since~$x_i=b$ forms an occurrence of~$11$ together with the other copy of~$b$ in the~$11$-stack. But this is impossible since~$x_i$ precedes~$x_j$ in~$\out{11}(x)$. Thus it must be~$i<j$. But then~$x_j$ enters the~$11$-stack above~$x_i$, which is again impossible since we supposed that~$x_i$ precedes~$x_j$ in~$\out{11}(x)$.
\end{proof}

Due to Lemma~\ref{lemma_ascseq_is_RGF}, we have~$\Ascseq(1213,1223)=\RGF(1213,1223)$. The enumeration of~$\RGF(1213,1223)$ can be found in~\cite{JMS}, where the authors determine all the Wilf-equivalence classes of pairs of patterns of length four. The arising sequence is A005183 in~\cite{Sl}.

Next we analyze the~$12$-machine. Since~$\Sort^{\Cay}(12)=\Cay(213)$, we can apply Theorem~\ref{theorem_Cayley_class_to_ascseq}.

\begin{theorem}\label{theorem_ascseq_12}
We have:
$$
\Sort^{\Ascseq}(12)=\Ascseq(213).
$$
\end{theorem}

Notice that~$\Ascseq(213)=\Ascseq(1213)$ due to Lemma~\ref{lemma_RGF_prop} and Lemma~\ref{lemma_ascseq_is_RGF}. The enumeration of~$\Ascseq(213)$ can be found in \cite{DS}.

\begin{theorem}\label{theorem_ascseq_121}
We have:
$$
\Sort^{\Ascseq}(121)=\Ascseq(213).
$$
\end{theorem}
\begin{proof}
Let~$x\in\Ascseq$. Suppose that~$x\ge 213$. We show that~$x$ is not~$121$-sortable. Due to Lemma~\ref{lemma_ascseq_is_RGF}, $x\ge 213$ if and only if~$x\ge 1213$. Let~$x_ix_jx_kx_l$ be the leftmost occurrence of~$1213$ in~$x$. Observe that~$x_j$ must be extracted from the~$121$-stack before~$x_k$ enters. Therefore~$x_jx_lx_1$ is an occurrence of~$231$ in~$\out{121}(x)$, as wanted.

Conversely, suppose that~$x$ is not~$121$-sortable and let~$bca$ be an occurrence of~$231$ in~$\out{121}(x)$. Let~$b=x_i$, $c=x_j$ and~$a=x_k$, for some~$i,j,k$. Suppose, for a contradiction, that~$x$ avoids~$213$. Then~$x$ is a {\rgf} due to Lemma~\ref{lemma_ascseq_is_RGF}. Since~$x$ avoids~$213$, the elements before~$x_j=c$ in~$x$ must be in weakly increasing order. More precisely, it must be:
$$
x=1^{t_1}2^{t_2}\cdots a^{t_a}\cdots b^{t_b}\cdots c^{t_c}x_j,
$$
for some integers~$t_u\ge 1$, $u=1,\dots,c-1$, and~$t_c\ge 0$. Now, if~$i>j$, then~$x_j=c$ is extracted before~$x_i$ enters, since otherwise~$x_ix_jb$ would be an occurrence of~$121$ in the~$12$-stack. But this contradicts our assumption that~$x_i,x_j,x_k$ results in an occurrence of~$231$ in~$\out{121}(x)$. Therefore~$i<j$ and, by hypothesis, $x_i=b$ is extracted from the~$121$-stack before~$x_j=c$ enters. But this is impossible, since the prefix of~$x$ up to~$x_j$ is weakly increasing.
\end{proof}

Notice that the set of~$121$-sortable Cayley permutations is not a permutation class due to Corollary~\ref{corollary_class_vs_nonclass_Cayley}. By Theorems~\ref{theorem_ascseq_12} and \ref{theorem_ascseq_121}, we have $\Sort^{\Ascseq}(12)=\Sort^{\Ascseq}(121)$. However, the operations performed by the~$12$-stack and the~$121$-stack are not always the same. For example, $\out{12}(12132)=23211$, whereas~$\out{121}(12132)=22311$.

\begin{theorem}\label{theorem_contains123_class}
Let~$\sigma\in\Ascseq$ and suppose that~$\sigma$ contains~$123$. Then~$\Sort^{\Ascseq}(\sigma)=\Ascseq(132)$.
\end{theorem}
\begin{proof} 
Let~$x\in\Ascseq$. Suppose that~$x\ge 132$. We show that~$x$ is not~$\sigma$-sortable. If~$x$ avoids~$\reverse(\sigma)$, then~$\out{\sigma}(x)=\reverse(x)$ contains~$\reverse(132)=231$, which means that~$x$ is not~$\sigma$-sortable. On the other hand, suppose that~$x$ contains~$\reverse(\sigma)$. Notice that~$\out{\sigma}(x)$ contains~$\hat{\sigma}$ due to Remark~\ref{remark_hat_Cayley}. Moreover, since~$\sigma$ contains~$123$, it is easy to observe that~$\hat{\sigma}$ contains an occurrence~$ab$ of~$12$ such that~$a>1$. Let~$y_1$ and~$y_2$ be the elements that correspond to~$a$ and~$b$ in such occurrence of~$12$ in~$\out{\sigma}(x)$. Notice also that the last element of~$\out{\sigma}(x)$ is~$x_1=1$. Therefore~$y_1y_2x_1$ is an occurrence of~$231$ in~$\out{\sigma}(x)$. Thus~$x$ is not~$\sigma$-sortable.

Conversely, we show that if~$x$ avoids~$132$, then~$x$ is~$\sigma$-sortable. If~$x$ avoids~$\reverse(\sigma)$, then~$\out{\sigma}(x)=\reverse(x)$ avoids~$231$ and thus~$x$ is~$\sigma$-sortable. Otherwise, suppose that~$x\ge\reverse(\sigma)$. Since~$\sigma\ge 123$, $x$ contains an occurrence~$x_{i_1}x_{i_2}x_{i_3}$ of~$\reverse(123)=321$. But then~$x_1x_{i_1}x_{i_2}$ would be an occurrence of~$132$ in~$x$, a contradiction.
\end{proof}

\begin{theorem}\label{theorem_avoids123_nonclass}
Let~$\sigma$ be an ascent sequence of length at least four and suppose that~$\sigma$ avoids~$123$. Then~$\Sort^{\Ascseq}(\sigma)$ is not a class.
\end{theorem}
\begin{proof}
Let~$\sigma=\sigma_1\cdots\sigma_k$, with~$k\ge 4$. Observe that the ascent sequence~$1232$ is not~$\sigma$-sortable. Indeed~$1232$ avoids~$\reverse(\sigma)$, since~$\sigma$ has length at least four and~$\reverse(1232)=2321$ is not an ascent sequence. Thus~$\out{\sigma}(1232)=\reverse(1232)=2321$ contains~$231$. We shall define an ascent sequence~$\alpha$ such that~$\alpha$ contains~$1232$ and~$\alpha$ is~$\sigma$-sortable, thus showing that~$\Sort^{\Ascseq}(\sigma)$ is not a class. Due to Lemma~\ref{lemma_avoids123_maxatmost2}, we have~$\max(\sigma)\le 2$ for each~$i$. We distinguish the following cases.

\begin{itemize}
\item Suppose that~$\max(\sigma)=1$, i.e.~$\sigma=1\cdots 1=1^k$, for some~$k\ge 1$. Define:
$$
\alpha=1^{k-1}2312.
$$
Note that~$\alpha\ge 1232$ and~$\alpha$ is an ascent sequence. Finally, it is easy to check that:
$$
\out{\sigma}(\alpha)=32121^{k-1},
$$
which avoids~$231$. Thus~$\alpha$ is~$\sigma$-sortable.

\item Suppose that~$\max(\sigma)=2$ and the last element is~$\sigma_k=1$. Define:
$$
\alpha=\sigma_k\cdots\sigma_23\sigma_12.
$$
Since~$\max(\sigma)=2$, there is an index~$m$ such that~$\sigma_m=2$. Notice that by our assumptions it must be~$m\neq k$ and~$m\neq 1$, thus~$\alpha$ contains~$1232$. Also~$3\le\asc(\sigma_k\cdots\sigma_2)+2$ (and~$\sigma_k=1$), thus~$\alpha$ is an ascent sequence. Finally, an easy computation shows that:
$$
\out{\sigma}(\alpha)=3\sigma_21\sigma_1\sigma_3\cdots\sigma_k,
$$
which avoids~$231$ (for example, because the initial~$3$ is the only element greater than~$2$ in~$\out{\sigma}(\alpha)$).

\item Suppose that~$\max(\sigma)=2$ and the last element is~$\sigma_k=2$. Similarly to the previous case, define:
$$
\alpha=1\sigma_k\cdots\sigma_23\sigma_12.
$$
Due to the insertion of the initial~$1$, $\alpha$ is again an ascent sequence and~$\alpha$ contains~$1232$. Finally, we have that:
$$
\out{\sigma}(\alpha)=3\sigma_21\sigma_1\sigma_3\cdots\sigma_k 1,
$$
which avoids~$231$. We leave the details to the reader.
\end{itemize}
\end{proof}

\begin{corollary}\label{corollary_class_nonclass}
Let~$\sigma$ be an ascent sequence. If~$\sigma\in\lbrace 11,12,121\rbrace$, then~$\Sort^{\Ascseq}(\sigma)$ is a class. In all the other cases, $\Sort^{\Ascseq}(\sigma)$ is a class if and only if~$\sigma\ge 123$. Moreover, if~$\sigma\ge 123$, then~$\Sort^{\Ascseq}(\sigma)=\Ascseq(132)$.
\end{corollary}
\begin{proof}
Patterns~$\sigma$ of length at most three, except~$123$, are discussed in Table~\ref{table_short_patterns_ascseq} and theorems~\ref{theorem_ascseq_11}, \ref{theorem_ascseq_12} and \ref{theorem_ascseq_121}. Patterns of greater length and the pattern~$123$ are discussed in Theorems~\ref{theorem_contains123_class} and \ref{theorem_avoids123_nonclass}.
\end{proof}

It is easy to observe that, in accordance with Corollary~\ref{corollary_class_vs_nonclass_Cayley} and Theorem~\ref{theorem_Cayley_class_to_ascseq}, if~$\hat{\sigma}$ contains~$231$, then~$\sigma$ contains~$123$ and~$\Sort^{\Ascseq}(\sigma)$ is indeed a class, as stated in Corollary~\ref{corollary_class_nonclass}.

\begin{table}
\centering
\def\arraystretch{1.1}
\begin{tabular}{lcc}
\toprule
$\sigma$ & $\sigma$\textbf{-sortable ascent sequence} & \textbf{Non-}$\sigma$\textbf{-sortable pattern}\\
\midrule
111 & 112312 & 1232\\
112 & 121312 & 1232\\
122 & 122312 & 1232\\
\bottomrule
\end{tabular}
\caption[Non-classes of~$\sigma$-sortable ascent sequences.]{Patterns~$\sigma$ of length at most three where~$\Sort^{\Ascseq}(\sigma)$ is not a class.}\label{table_short_patterns_ascseq}
\end{table}

\subsection{Modified ascent sequences}\label{section_modasc_seq_stack}

The following two results can be found in~\cite{CeCl}.
\begin{lemma}\cite{CeCl}\label{lemma_modasc_ascent_tops}
Let~$x\in\Cay_n$ be a Cayley permutation. Then~$x$ is a modified ascent sequence if and only if the following two conditions hold:
\begin{enumerate}
\item $x_1=1$;
\item an entry~$x_i=k>1$ is the leftmost occurrence of the integer~$k$ in~$x$ if and only if~$x_{i-1}<x_i$ (that is~$x_i$ is an ascent top).
\end{enumerate}
\end{lemma}

\begin{theorem}\cite{CeCl}\label{theorem_modasc_mesh_patterns}
Let~$\mathfrak{a}$ and~$\mathfrak{b}$ the Cayley-mesh patterns depicted in Figure~\ref{figure_mesh_patterns_modasc}. Then:
$$
\Modasc=\Cay(\mathfrak{a},\mathfrak{b}).
$$
\end{theorem}

\begin{figure}
\centering
$
\mathfrak{a}\,=\,
\begin{tikzpicture}[scale=0.50, baseline=19pt]
\fill[NE-lines] (2.15,0) rectangle (2.85,3);
\draw [semithick] (0,0.85) -- (4,0.85);
\draw [semithick] (0,1.15) -- (4,1.15);
\draw [semithick] (0,1.85) -- (4,1.85);
\draw [semithick] (0,2.15) -- (4,2.15);
\draw [semithick] (0.85,0) -- (0.85,3);
\draw [semithick] (1.15,0) -- (1.15,3);
\draw [semithick] (1.85,0) -- (1.85,3);
\draw [semithick] (2.15,0) -- (2.15,3);
\draw [semithick] (2.85,0) -- (2.85,3);
\draw [semithick] (3.15,0) -- (3.15,3);
\filldraw (1,2) circle (5pt);
\filldraw (2,1) circle (5pt);
\filldraw (3,2) circle (5pt);
\end{tikzpicture}
\hspace{50pt}
\mathfrak{b}\,=\,
\begin{tikzpicture}[scale=0.50, baseline=19pt]
\fill[NE-lines] (1.15,0) rectangle (1.85,3);
\fill[NE-lines] (0,0.85) rectangle (0.85,1.15);
\draw [semithick] (0,0.85) -- (3,0.85);
\draw [semithick] (0,1.15) -- (3,1.15);
\draw [semithick] (0,1.85) -- (3,1.85);
\draw [semithick] (0,2.15) -- (3,2.15);
\draw [semithick] (0.85,0) -- (0.85,3);
\draw [semithick] (1.15,0) -- (1.15,3);
\draw [semithick] (1.85,0) -- (1.85,3);
\draw [semithick] (2.15,0) -- (2.15,3);
\filldraw (1,2) circle (5pt);
\filldraw (2,1) circle (5pt);
\end{tikzpicture}
$
\caption[Cayley-mesh patterns characterizing modified ascent sequences.]{Cayley-mesh patterns such that~$\Modasc=\Cay(\mathfrak{a},\mathfrak{b})$.}\label{figure_mesh_patterns_modasc}
\end{figure}

Theorem~\ref{theorem_modasc_mesh_patterns} is essentially a reformulation of Lemma~\ref{lemma_modasc_ascent_tops} in terms of Cayley-mesh patterns. The avoidance of~$\mathfrak{a}$ implies that every ascent top is the leftmost occurrence of the corresponding integer. Conversely, to avoid~$\mathfrak{b}$ implies that each entry that is not an ascent top is not the leftmost occurrence of the corresponding integer.

\begin{lemma}\label{lemma_modasc213_1213}
We have:
$$
\Modasc(213)=\Modasc(1213).
$$
Moreover, the set~$\Modasc(213)$ is enumerated by the Catalan numbers.
\end{lemma}
\begin{proof}
We start by showing that~$\Modasc(213)=\Modasc(1213)$. Let~$x=x_1\cdots x_n\in\Modasc$. It is enough to show that if~$x\ge 213$, then~$x\ge 1213$. Let~$x_ix_jx_k$ be an occurrence of~$213$ in~$x$. Let~$j'$ be the index of the leftmost occurrence of the integer~$x_j$ in~$x$. If~$j'<i$, then~$x_{j'}x_ix_jx_k$ is an occurrence of~$1213$. If~$i<j'\le j$, then by Lemma~\ref{lemma_modasc_ascent_tops} it must be~$x_{j'-1}<x_j'$. Therefore we can repeat the same argument on the occurrence~$x_ix_{j'-1}x_k$ of~$213$, until we either find an occurrence of~$1213$ or a contradiction.

Now, by Theorem~\ref{theorem_transport_modasc_fish}, the set~$\Fish(3124)$ is Wilf-equivalent to the set~$\Modasc(\mathcal{B}(3124))$, where the Fishburn basis of~$3124$ is~$\mathcal{B}(3124)=\lbrace 1213,2314\rbrace$ (see again~\cite{CeCl}). In~\cite{GW}, the set~$\Fish(3124)$ is shown to be enumerated by the Catalan numbers. Finally, due to what proved above we have~$\Modasc(1213)=\Modasc(213)$. Thus:
$$
\Modasc(1213,2314)=\Modasc(213,2314)=\Modasc(213),
$$
and the thesis follows.
\end{proof}

\begin{theorem}\label{theorem_modasc_11}
We have:
$$
\Sort^{\Modasc}(11)=\Modasc(1213,1223).
$$
\end{theorem}
\begin{proof}
The proof of the inclusion~$\Sort^{\Modasc}(11)\subseteq\Modasc(1213,1223)$ is identical to the analogous inclusion of Theorem~\ref{theorem_ascseq_11}.

Conversely, let~$x\in\Modasc$ and suppose that~$x$ is not~$11$-sortable. Let~$x_i=b$, $x_j=c$ and~$x_k=a$ be the three elements of~$x$ that result in the leftmost occurrence~$x_ix_jx_k=bca$ of~$231$ in~$\out{11}(x)$, with~$a<b<c$. We wish to show that~$x$ contains~$1213$ or~$1223$. We distinguish two cases, according to whether~$i<j$ or~$i>j$.

\begin{itemize}
\item Suppose that~$i<j$. Note that~$x_i=b$ is extracted from the~$11$-stack before~$x_j=c$ enters. Let~$y$ be the next element of the input when~$x_i$ is extracted. If~$y=b$, then~$x_1x_iyc\simeq 1223$. If~$y\neq b$, then there must be another copy~$y'$ of the integer~$y$ in the~$11$-stack. If~$y<b$, then~$y'x_iyc\simeq 1213$. Otherwise, if~$y>b$ then~$x_iy'x_k$ is an occurrence of~$231$ that precedes~$x_ix_jx_k$ in~$\out{11}(x)$, which is a contradiction.

\item Suppose that~$i>j$, that is~$x_j=c$ precedes~$x_i=b$ in~$x$. Since~$x_i$ precedes~$x_j$ in~$\out{11}(x)$, $x_j$ must be contained in the~$11$-stack when~$x_i$ enters. Let~$i'$ be the index of the first occurrence of the integer~$b$ in~$x$. If~$i'<j$, then~$x_{i'}$ is extracted from the~$11$-stack before~$x_j$ enters. Otherwise both~$x_{i'}=b$ and~$x_j$ (which is above~$x_{i'}$) would be extracted (at most) when~$x_i=b$ is the next element of the input, since~$x_i=x_{i'}$. But this is impossible due the hypothesis that~$x_i$ precedes~$x_j$ in~$\out{11}(x)$. Consider the instant when~$x_{i'}$ is extracted and let~$y$ be the next element of the input when this happens. If~$y=b$, then~$x_1x_{i'}yc\simeq 1223$, as desired. If~$y\neq x_{i'}$, then there must be another copy of the integer~$y$, say~$y'$, contained in the~$11$-stack. If~$y'<b$, then~$y'x_{i'}yc\simeq 1213$. Finally, if~$y'>b$ then~$x_{i'}y'x_1$ is an occurrence of~$231$ in~$\out{11}(x)$ that precedes~$x_ix_jx_k$, contradicting our choice of~$i,j,k$.
\end{itemize}
\end{proof}

\begin{theorem}\label{theorem_modasc_fibonacci_odd}
The set~$\Modasc(1213,1223)$ is enumerated by the odd index Fibonacci numbers (sequence~A001519 in~\cite{Sl}).
\end{theorem}
\begin{proof}
Due to Theorem~\ref{theorem_transport_modasc_fish}, we have
$$
\Fish(1324)=\phi'\bigl(\Modasc(1223,1324)\bigr)\quad\text{and}\quad\Fish(3124)=\phi'\bigl(\Modasc(1213,2314)\bigr).
$$
Recall also that~$\Modasc(1213)=\Modasc(213)$, as proved in Lemma~\ref{lemma_modasc213_1213}. Thus, since~$213\le 1324$ and~$213\le 2314$, we have:
$$
\Modasc(1223,1324,1213,2314)=\Modasc(1223,213)=\Modasc(1213,1223)
$$
and
$$
\Fish(1324,3124)=\phi'\bigl(\Modasc(1223,1324,1213,2314)\bigr)=\phi'\bigl(\Modasc(1213,1223)\bigr).
$$ 
Let~$F(n)=\Fish_n(1324,3124)$ and let~$f(n)=|F(n)|$. We show that the coefficients~$f(n)$ satisfy~$f(1)=1$, $f(2)=2$ and~$f(n+1)=3f(n)-f(n-1)$, for~$n\ge 2$, which is a very well known recurrence for the odd index Fibonacci numbers. Let~$G(n,k)$ be the set:
$$
G(n,k)=\lbrace p\in F(n): \ltrmaxsize(p)=k\rbrace.
$$
Let~$g(n,k)=|G(n,k)|$. Notice that~$F(n)=\dot{\bigcup}_{k}G(n,k)$ and thus~$f(n)=\sum_{k=1}^{n}g(n,k)$. We show that, for any~$n\ge 1$:
$$
\begin{cases}
g(n+1,1)=f(n)\\
g(n+1,2)=f(n)\\
g(n+1,k+1)=g(n,k),\quad k\ge 2.
\end{cases}
$$
\begin{itemize}
\item Let us start by proving the first equation~$g(n+1,1)=f(n)$. We provide a bijection~$\alpha:F(n)\to G(n+1,1)$. Given~$p\in F(n)$, define~$\alpha(p)=1\ominus p$. Equivalently, $\alpha(p)$ is obtained from~$p$ by adding an initial maximum~$n+1$. It is easy to realize that~$\alpha(p)\in G(n+1,1)$ for each~$p\in F(n)$ and~$\alpha$ is injective. Finally, if~$q\in G(n+1,1)$, then~$q_1=n+1$ and the removal of~$n+1$ from~$q$ yields a permutation~$p\in F(n)$, thus~$\alpha$ is surjective too.

\item Similarly, we define a bijection~$\beta:F(n)\to G(n+1,2)$ by suitably adding a new maximum~$n+1$ to a permutation~$p\in F(n)$. If~$p\in G(n,1)$, then~$p_1=n$ and we set~$\beta(p)=n(n+1)p_2\cdots p_n$. Otherwise, if~$p\in G(n,k)$, for some~$k\ge 2$, let:
$$
p=m_1A_1m_2A_2\cdots m_kA_k
$$
be the ltr-max decomposition of~$p$. Then define~$\beta(p)$ by
$$
\beta(p)=m_1A_1(n+1)m_2A_2\cdots m_kA_k.
$$
It is easy to realize that~$\beta(p)$ avoids~$\fishpattern$, $1324$ and~$3124$. The case~$k=1$ is trivial. On the other hand, if~$k\ge 2$ then an occurrence of any of the listed patterns in~$\beta(p)$ should involve the new element~$n+1$, either as a~$4$ in an occurrence of~$1324$ or~$3124$ or as a~$3$ in an occurrence of~$\fishpattern$. But then, in all these cases, the element~$m_2$ would play the same role in an occurrence of the same pattern in~$p$, which is impossible since~$p\in F(n)$. A similar analysis shows that the removal of~$n+1$ from a permutation in~$G(n+1,2)$ yields a permutation in~$F(n)$, thus~$\beta$ is bijective and~$g(n+1,2)=f(n)$, as wanted.

\item Next suppose that~$k\ge 2$. We provide a bijection~$\gamma:G(n,k)\to G(n+1,k+1)$. Let~$p\in G(n,k)$ and write again:
$$
p=m_1A_1\cdots m_{k-1}A_{k-1}m_kA_k.
$$
Let~$A_k=a_1\cdots a_t$ and define:
$$
\gamma(p)=m_1A_1\cdots m_{k-1}A_{k-1}m'_{k-1}m'_ka'_1\cdots a'_t,
$$
where~$m'_k=m_k+1$ and~$a'_i=a_i+1$, if~$a_i>m_{k-1}$, or~$a'_i=a_i$, if~$a_i<m_k$. In other words, $\gamma(p)$ is obtained from~$p$ by inserting~$m_{k-1}+1$ immediately before~$m_k=n$ and suitably rescaling the other elements. Only those elements contained in the last block~$A_k$, $m_k$ included, eventually need to be rescaled. Notice that~$\gamma(p)$ has~$k+1$ ltr-maxima and~$\gamma$ is injective by construction. The proof that~$\gamma(p)$ avoids~$\fishpattern$, $1324$ and~$3124$ is identical to the previous cases, so we omit it. Finally, in the resulting permutation~$\gamma(p)$, the new element~$m_{k-1}+1$ is the~$k$-th ltr-maximum. Thus the inverse map~$\gamma^{-1}:G(n+1,k+1)\to G(n,k)$ is obtained by removing the~$k$-th ltr-maximum, which again does not create an occurrence of one of the forbidden patterns.
\end{itemize}

Due to what proved above, using induction, we have:
\begin{equation*}
\begin{split}
f(n+1)=\sum_{k=1}^{n+1}g(n+1,k)=&\\
g(n+1,1)+g(n+1,2)+\sum_{k=3}^{n+1}g(n+1,k)=&\\
f(n)+f(n)+\sum_{k=3}^{n+1}g(n,k-1)=&\\
f(n)+f(n)+\sum_{j=2}^{n}g(n,j)=&\\
f(n)+f(n)+\left[f(n)-g(n,1)\right]=&\\
3f(n)-f(n-1),
\end{split}
\end{equation*}
as desired.
\end{proof}

The next result for the pattern~$12$ is again a corollary of Theorem~\ref{theorem_Cayley_class_to_ascseq}.

\begin{theorem}\label{theorem_modasc_12}
We have:
$$
\Sort^{\Modasc}(12)=\Modasc(213).
$$
\end{theorem}

\begin{theorem}\label{theorem_modasc_121}
We have:
$$
\Sort^{\Modasc}(121)=\Modasc(213).
$$
\end{theorem}
\begin{proof}
Due to Lemma~\ref{lemma_modasc213_1213}, we have~$\Modasc(213)=\Modasc(1213)$. Let~$x\in\Modasc$ and suppose that~$x\ge 1213$. We show that~$x$ is not~$121$-sortable. Let~$x_{i_1}x_jx_{i_2}x_k$ be an occurrence of~$1213$ in~$x$. We can assume that~$x_{i_1}x_jx_{i_2}$ is the leftmost occurrence of~$121$ such that the element that plays the role of~$2$ is smaller than~$x_k$. If~$x_j$ is extracted from the~$121$-stack before~$x_k$ enters, then~$\out{121}(x)$ contains~$x_jx_kx_1\simeq 231$, thus~$x$ is not~$121$-sortable. Otherwise, suppose that~$x_j$ is still in the~$121$-stack when~$x_k$ enters. Observe that, since~$x_{i_2}x_jx_{i_1}\simeq 121$, the element~$x_{i_1}$ is extracted before~$x_j$ enters. Otherwise~$x_{i_1}$ and~$x_j$ would be extracted from the~$121$-stack at most when~$x_{i_2}$ is the next element of the input (which contradicts the assumption that~$x_j$ is still in the~$121$-stack when~$x_k$ enters). Let~$y_1$ be the next element of the input when~$x_{i_1}$ is extracted. Since the~$121$-stack restriction is triggered, there are two elements in the~$121$-stack, say~$z$ and~$y_2$, with~$z$ above~$y_2$, such that~$y_1zy_2\simeq 121$. Due to our choice of~$i_1,j,i_2$, it must be~$z>x_k$ (and thus~$z\neq x_{i_1}$). Notice that it cannot be~$x_{i_1}=1$, since in that case the~$121$-stack would contain~$x_{i_1}zx_1\simeq 121$. Then~$x_{i_1}>1$ and~$\out{121}(x)$ contains an occurrence~$x_{i_1}zx_1$ of~$231$, as wanted.

Conversely, suppose that~$x$ is not~$121$-sortable. Equivalently, suppose there are three elements~$x_i=b$, $x_j=c$ and~$x_k=a$ such that~$x_ix_jx_k=bca$ is an occurrence of~$231$ in~$\out{121}(x)$. We show that~$x$ contains~$213$. We can assume that~$bca$ is the leftmost occurrence of~$231$ in~$\out{121}(x)$. We distinguish two cases, according to whether~$i<j$ or~$j>i$.

\begin{itemize}
\item Suppose that~$i<j$, i.e.~$x_i=b$ precedes~$x_j=c$ in~$x$. By hypothesis~$x_i$ is extracted from the~$121$-stack before~$x_j$ enters. Therefore, at that moment, the~$121$-stack contains two elements~$zy_1$, with~$z$ above~$y_1$, such that~$y_2zy_1\simeq 121$, where~$y_2$ is the next element of the input. If~$y_2<b$, then~$x_iy_2x_j$ is an occurrence of~$213$ in~$x$, as wanted. If instead~$y_2=b$, then the stack contains an occurrence~$x_izy_1$ of~$121$, which is forbidden. Finally, if~$y_1>b$, then also~$x>b$ and thus~$x_izx_k$ is an occurrence of~$231$ in~$\out{121}(x)$ that precedes~$x_ix_jx_k$, which is impossible due to our choice of~$i,j,k$. 

\item Suppose that~$j<i$, i.e.~$x_j=c$ precedes~$x_i=b$ in~$x$. Note that~$x_j$ must be in the~$121$-stack when~$x_i$ enters. Moreover, since~$x_ix_jx_k$ is the leftmost occurrence of~$231$ in~$\out{121}(x)$, $x_i=b$ is the first occurrence of the integer~$b$ that is extracted from the~$121$-stack. Let~$i'$ be the index of the first occurrence of the integer~$b$ in~$x$. If~$i'<j$, then~$x_{i'}x_jx_i\simeq 121$. Therefore at least one between~$x_{i'}$ and~$x_j$ must be extracted from the~$121$-stack before~$x_i$ enters (otherwise we would have~$x_ix_jx_{i'}\simeq 121$ inside the~$121$-stack). As said before, $x_j$ is still contained in the~$121$-stack when~$x_i$ enters, therefore~$x_{i'}$ is the one that has been extracted before. But then~$x_{i'}x_jx_k$ is an occurrence of~$231$ in~$\out{121}(x)$ that precedes~$x_ix_jx_k$, which is impossible. We can thus assume that~$i'>j$. Now, consider the element~$x_{i'-1}$. Due to Lemma~\ref{lemma_modasc_ascent_tops}, we have~$x_{i'-1}<x_{i'}$. If~$x_{i'-1}=1$, then~$x_{i'-1}x_jx_1\simeq 121$ and thus~$x_j$ is extracted from the~$121$-stack before~$x_i$ enters, which is a contradiction. If instead~$x_{i'-1}>1$, we can repeat the same argument, but using the first occurrence~$x_w$ of~$x_{i'-1}$ in~$x$ in place of~$x_{i'}$ (if~$w<j$, then~$x_j$ is extracted too soon and otherwise we consider~$x_{w-1}$). Sooner or later this would result in a contradiction.
\end{itemize}
\end{proof}

The proof of the next result is analogous to the proof of Theorem~\ref{theorem_contains123_class}, and it is left to the reader.

\begin{theorem}\label{theorem_modasc_class_suff_123}
Let~$\sigma=\sigma_1\cdots\sigma_k$ be a modified ascent sequence of length at least three and suppose that~$\sigma\ge 123$. Then~$\Sort^{\Modasc}(\sigma)=\Modasc(132)$.
\end{theorem}

\begin{theorem}\label{theorem_modasc_class_suff_122}
Let~$\sigma=\sigma_1\cdots\sigma_k$ be a modified ascent sequence of length at least three. If~$\sigma$ avoids~$123$ and~$\sigma_1\sigma_2\sigma_3\simeq 122$, then~$\Sort^{\Modasc}(\sigma)=\Modasc(132,\reverse(\sigma)\oplus 1)$.
\end{theorem}
\begin{proof}
Suppose that~$\sigma$ avoids~$123$ and let~$\sigma_2=\sigma_3=t$, for some~$t>1$. Observe that~$t=\max(\sigma)$, otherwise~$\sigma_1\sigma_2=1t$ would realize an occurrence of~$123$ together with the maximum of~$\sigma$. Let~$x$ be a modified ascent sequence.

We firts show that if~$x$ is not~$\sigma$-sortable, then~$x\notin\Modasc(132,\reverse(\sigma))\oplus 1)$, thus proving the inclusion~$\Sort^{\Modasc}(\sigma)\supseteq\Modasc(132,\reverse(\sigma)\oplus 1)$. Suppose that the three elements~$x_{i}=a$, $x_{j}=b$ and~$x_{k}=c$ result in an occurrence~$bca$ of~$231$ in~$\out{\sigma}(x)$. If~$j>k$, then~$x_1x_kx_j$ is an occurrence of~$132$ in~$x$, as wanted. Otherwise, suppose that~$j<k$ and~$x_j$ is extracted from the~$\sigma$-stack before~$x_k$ enters. When~$x_j$ is extracted, there must be~$k-1$ elements~$\sigma'_2,\dots,\sigma'_k$ in the~$\sigma$-stack, reading from top to bottom, such that~$\sigma'_1\sigma'_2\cdots\sigma'_k$ is an occurrence of~$\sigma$, where~$\sigma'_1$ is the next element of the input. Now, if~$\sigma'_2<c$, then~$\sigma'_u<c$ for each~$u$, since~$\sigma_2=\max(\sigma)$. Therefore~$\sigma'_k\cdots\sigma'_2\sigma'_1c$ is an occurrence of~$\reverse(\sigma)\oplus 1$, as wanted. If instead~$\sigma'_2>c$, then~$x_1\sigma'_2c$ is an occurrence of~$132$ in~$x$. Finally, suppose that~$\sigma'_2=c$. Thus~$x_k=c$ is not the leftmost occurrence of the integer~$c$ in~$x$, since~$x_k$ follows~$\sigma'_2=c$ in~$x$. By Lemma~\ref{lemma_modasc_ascent_tops}, it must be~$x_{k-1}\ge x_k$ (and thus~$x_{k-1}\neq\sigma'_1$). If~$x_{k-1}>x_k$, then~$x_1x_{k-1}x_k$ is an occurrence of~$132$. If~$x_{k-1}=x_k$, then we can repeat the same argument, but using~$x_{k-1}$ instead of~$x_k$. Sooner or later this will result in either an occurrence of~$132$ or a contradiction.

Conversely, we shall prove the inclusion~$\Sort^{\Modasc}(\sigma)\subseteq\Modasc(132,\reverse(\sigma)\oplus 1)$ by showing that if~$x$ contains either~$132$ or~$\reverse(\sigma)\oplus 1$, then~$x$ is not~$\sigma$-sortable. Suppose initially that~$x\ge132$. Without losing generality, choose the leftmost occurrence~$x_1x_ix_j$ of~$132$ in~$x$. If~$x_j$ enters the~$\sigma$-stack above~$x_i$, then~$\out{\sigma}(x)$ contains an occurrence~$x_jx_ix_1$ of~$231$, thus~$x$ is not~$\sigma$-sortable. Suppose instead that~$x_i$ is extracted from the~$\sigma$-stack before~$x_j$ enters. At that moment, the next element of the input~$\sigma'_1$ forms an occurrence of~$\sigma$ together with some elements~$\sigma'_2,\dots,\sigma'_k$ contained in the~$\sigma$-stack. Notice that~$\sigma'_2=\sigma'_3>1$, since~$\sigma_1\sigma_2\sigma_3\simeq 122$. Due to our choice of~$x_1x_ix_j$ as leftmost occurrence of~$132$ in~$x$, it must be~$\sigma'_u\le x_j$ for each~$u\le k-1$. Since~$\sigma_1<\sigma_2$, we also have~$\sigma'_1<x_j$ and thus~$\sigma'_2\neq x_i$. If~$\sigma'_2<x_j$, then~$\out{\sigma}(x)$ contains an occurrence~$\sigma'_2x_jx_1$ of~$231$ and we are done. Therefore we can assume~$\sigma'_2=x_j$ (and thus~$\sigma'_2=\sigma'_3=x_j$). Due to Lemma~\ref{lemma_modasc_ascent_tops}, it must be~$x_{j-1}\ge x_j$. Also~$x_{j-1}\ge x_i$, again due to our choice of~$x_1x_ix_j$ as leftmost occurrence of~$132$, and thus~$x_{j-1}\neq\sigma'_1$. If~$x_{j-1}>x_i$, then~$x_ix_{j-1}x_1$ is an occurrence of~$231$ in~$\out{\sigma}(x)$. Finally, let~$x_{j-1}=x_i$. Then again~$x_{j-2}\ge x_{j-1}$ due to Lemma~\ref{lemma_modasc_ascent_tops} and we can repeat the same argument on~$x_{j-2}$. Sooner or later, since the next element of the input is~$\sigma'_1\neq x_{j-1}$, either we will find an occurrence of~$132$ or a contradiction. This proves that if~$x\ge 132$, then~$x$ is not~$\sigma$-sortable. Next suppose that~$x$ avoids~$132$, but~$x$ contains an occurrence~$\sigma'_k\cdots\sigma'_2\sigma'_1m$ of~$\reverse(\sigma)\oplus 1$. If the elements~$\sigma'_k\cdots\sigma'_2$ are still in the~$\sigma$-stack when~$\sigma'_1$ is the next element of the input, then~$\sigma'_2$ is extracted and~$\out{\sigma}(x)$ contains an occurrence~$\sigma'_2mx_1$ of~$231$. Otherwise, there must have been a previous occurrence, say~$\sigma''_k\cdots\sigma''_2\sigma''_1$, of~$\reverse(\sigma)$ in~$x$ such that at least one element amongst~$\sigma'_k,\dots,\sigma'_2$ is extracted when~$\sigma''_1$ is the next element of the input, together with~$\sigma''_2$. Observe that it must be~$\sigma''_2\le m$, since we are assuming that~$x$ avoids~$132$. If~$\sigma''_2<m$, then again~$\sigma''_2mx_1$ is an occurrence of~$231$ in~$\out{\sigma}(x)$. Finally, if~$\sigma''_2=m$, then~$x_1\sigma''_2\sigma'_2$ is an occurrence of~$132$, a contradiction.
\end{proof}

\begin{lemma}\label{lemma_reverse_is_modasc}
Let~$x=x_1\cdots x_k$ be a modified ascent sequence and suppose that~$x$ avoids~$123$. If~$x_k=1$, then~$\reverse(x)$ is a modified ascent sequence. Otherwise, if~$x_k>1$, then~$1\reverse(x)$ is a modified ascent sequence.
\end{lemma}
\begin{proof}
Let~$m=\max(x)$. Since~$x$ avoids~$123$ (and~$x_1=1$), the elements of~$x$ that are greater than~$1$ are in weakly decreasing order. Therefore, by Lemma~\ref{lemma_modasc_ascent_tops}, we have:
$$
x=1^{i_1}m^{j_1}1^{i_2}(m-1)^{j_2}\cdots 1^{i_{m-1}}2^{j_{m-1}}1^{i_m},
$$
where~$i_u\ge 1$ for each~$u<m$, $i_m\ge 0$ and~$j_v\ge 1$ for each~$v$. Therefore, again by Lemma~\ref{lemma_modasc_ascent_tops}, if~$x_k=1$ (or equivalently~$i_m\ge 1$) then~$\reverse(x)$ is a modified ascent sequence. Similarly, if~$x_k>1$, then we obtain a modified ascent sequence by inserting an additional~$1$ at the beginning of~$\reverse(x)$.
\end{proof}

\begin{theorem}\label{theorem_modasc_class_nec}
Let~$\sigma=\sigma_1\cdots\sigma_k$ be a modified ascent sequence of length at least four. If~$\sigma$ avoids~$123$ and~$\sigma_1\sigma_2\sigma_3$ is not an occurrence of~$122$, then~$\Sort^{\Modasc}(\sigma)$ is not a class.
\end{theorem}
\begin{proof}
Observe that the modified ascent sequence~$\alpha=1312$ is not~$\sigma$-sortable, since~$\out{\sigma}(\alpha)=\reverse(\alpha)=2131\ge 231$, for each~$\sigma$ of length four or more. We wish to construct a~$\sigma$-sortable modified ascent sequence~$\beta$, with~$\beta\ge\alpha$. We distinguish some cases. Being the other cases similar, we give a detailed proof for the first one only.

\begin{itemize}
\item Suppose that~$\sigma_2=1$ and~$\sigma_k=1$. Define:
$$
\alpha=\sigma_k\cdots\sigma_3\sigma_2(m+2)\sigma_1(m+1),
$$
where~$m=\max(\sigma)$. Notice that~$\sigma_k(m+2)\sigma_1(m+1)$ is an occurrence of~$1312$ in~$\alpha$. Now, an easy computation shows that:
$$
\out{\sigma}(\alpha)=(m+2)\sigma_2(m+1)\sigma_1\sigma_3\cdots\sigma_k.
$$
Observe that~$\out{\sigma}(\alpha)$ avoids~$231$. Indeed~$m+2$ and~$m+1$ are not part of an occurrence of~$231$ (since~$\sigma_2=1$). Moreover, suppose, for a contradiction, that~$\sigma_3\cdots\sigma_k$ contains an occurrence~$\sigma_{i_i}\sigma_{i_2}\sigma_{i_3}$ of of~$231$. Then~$\sigma_1\sigma_{i_2}\sigma_{i_3}$ is an occurrence of~$123$ in~$\sigma$, which contradicts the hypothesis. Finally, we shall prove that~$\alpha$ is a modified ascent sequence. Notice that~$\reverse(\sigma)$ is a modified ascent sequence by Lemma~\ref{lemma_reverse_is_modasc}. The only additional elements are~$m+2$, which is placed immediately after~$\sigma_2=1$ and before~$\sigma_1=1$, and~$m+2$, which is placed at the end, immediately after~$\sigma_1=1$. Therefore~$\alpha$ is a modified ascent sequence by Lemma~\ref{lemma_modasc_ascent_tops}.

\item If~$\sigma_2=1$ and~$\sigma_k>1$, then define~$\alpha$ exactly as in the previous case, but inserting an additional~$1$ at the beginning (as in Lemma~\ref{lemma_reverse_is_modasc}).

\item Suppose that~$\sigma_2>1$ and~$\sigma_k=1$. Then~$\sigma_2=m$ is equal to the maximum value of~$\alpha$, because~$\alpha$ avoids~$123$. Since~$\sigma_1\sigma_2\sigma_3$ is not an occurrence of~$122$, Lemma~\ref{lemma_modasc_ascent_tops} implies that~$\sigma_2$ is the only occurrence of~$m$ in~$\sigma$. Define:
$$
\alpha=\sigma_k\cdots\sigma_3(m+1)\sigma_1\sigma_2.
$$
Then~$\alpha$ contains~$1312$, $\alpha$ is a modified ascent sequence and:
$$
\out{\sigma}(\alpha)=(m+1)\sigma_2\sigma_1\sigma_3\cdots\sigma_k,
$$
which avoids~$231$. We leave the details to the reader.

\item If~$\sigma_2>1$ and~$\sigma_k>1$, then define~$\alpha$ exactly as in the previous case, but inserting an additional~$1$ at the beginning (as in Lemma~\ref{lemma_reverse_is_modasc}).
\end{itemize}
\end{proof}

\begin{table}
\centering
\def\arraystretch{1.1}
\begin{tabular}{lcc}
\toprule
$\sigma$ &~$\sigma$\textbf{-sortable modified sequence} & \textbf{Non-}$\sigma$\textbf{-sortable pattern}\\
\midrule
111 & 11312 & 1312\\
112 & 121413 & 1312\\
\bottomrule
\end{tabular}
\caption[Non-classes of~$\sigma$-sortable modified ascent sequences.]{Patterns~$\sigma$ of length at most three where~$\Sort^{\Modasc}(\sigma)$ is not a class.}\label{table_short_patterns_modasc}
\end{table}

The following corollary, which is an immediate consequence of the results proven so far in this section, provides a characterization of the patterns~$\sigma$ where~$\Sort^{\Modasc}(\sigma)$ is a class.

\begin{corollary}\label{corollary_class_nonclass_modasc}
Let~$\sigma=\sigma_1\cdots\sigma_k$ be a modified ascent sequence. If~$\sigma\in\lbrace 11,12\rbrace$, then~$\Sort^{\Modasc}(\sigma)$ is a class. In all the other cases, $\Sort^{\Modasc}(\sigma)$ is a class if and only if~$\sigma\ge 123$ or~$\sigma_1\sigma_2\sigma_3\simeq 122$. Moreover, if~$\sigma\ge 123$, then~$\Sort^{\Modasc}(\sigma)=\Modasc(132)$. If instead~$\sigma$ avoids~$123$ and~$\sigma_1\sigma_2\sigma_3\simeq 122$, then~$\Sort^{\Modasc}(\sigma)=\Modasc(132,\reverse(\sigma)\oplus 1)$.
\end{corollary}
\begin{proof}
The patterns~$11$, $12$, $121$, $111$ and~$112$ were solved in Theorem~\ref{theorem_modasc_11}, Theorem~\ref{theorem_modasc_12}, Theorem~\ref{theorem_modasc_121} and Table~\ref{table_short_patterns_modasc}. The remaining cases were considered in Theorems~\ref{theorem_modasc_class_suff_123}, \ref{theorem_modasc_class_suff_122} and \ref{theorem_modasc_class_nec}.
\end{proof}

Similarly to what observed in the previous section, in accordance with Corollary~\ref{corollary_class_vs_nonclass_Cayley} and Theorem~\ref{theorem_Cayley_class_to_ascseq}, if~$\hat{\sigma}$ contains~$231$, then~$\sigma$ contains~$123$ and~$\Sort^{\Modasc}(\sigma)$ is a class, as in Corollary~\ref{corollary_class_nonclass_modasc}.

\appendix
\newgeometry{top=2.5cm,bottom=2.5cm}

\chapter{Index of integer sequences}\label{appendix_OEIS}

\centering
{\small
\begin{tabular}{lll}
\toprule
\textbf{OEIS} & \textbf{Sequence} & \textbf{References}\\
\midrule
\href{https://oeis.org/A000079}{A000079} & Powers of~$2$ & Appendix~\ref{appendix_basis_size_two}, Table~\ref{table_sorted_perms}\\
\href{https://oeis.org/A000108}{A000108} & Catalan numbers & Example~\ref{example_pat_involv}, Table~\ref{table_pairs_of_pat}, Table~\ref{table_sorted_perms}\\
\href{https://oeis.org/A000110}{A000110} & Bell numbers & Section~\ref{section_sequences_integers}\\
\href{https://oeis.org/A000124}{A000124} & Central polygonal numbers & Table~\ref{table_sorted_perms}\\
\href{https://oeis.org/A000670}{A000670} & Fubini numbers & Section~\ref{section_sequences_integers}\\
\href{https://oeis.org/A001006}{A001006} & Motzkin numbers & Section~\ref{section_lattice_paths}\\
\href{https://oeis.org/A001519}{A001519} & Odd indexed Fibonacci numbers & Appendix~\ref{appendix_basis_size_two}, Table~\ref{table_decr_pattern}, Theorem~\ref{theorem_modasc_fibonacci_odd}\\
\href{https://oeis.org/A002057}{A002057} & $4$-th convolution of Catalan numbers & Proposition~\ref{proposition_fourth_convolution}\\
\href{https://oeis.org/A006318}{A006318} & Large Schr\"oder numbers & Section~\ref{section_lattice_paths}, Table~\ref{table_pairs_of_pat}\\
\href{https://oeis.org/A007317}{A007317} & Binomial transform of Catalan numbers & Chapter~\ref{chapter_pattern132}, Table~\ref{table_pairs_of_pat}\\
\href{https://oeis.org/A009766}{A009766} & Catalan triangle (ballot numbers) & Section~\ref{section_123_132}\\
\href{https://oeis.org/A011782}{A011782} & & Table~\ref{table_decr_pattern}\\
\href{https://oeis.org/A022493}{A022493} & Fishburn numbers & Section~\ref{section_sequences_integers}\\
\href{https://oeis.org/A024175}{A024175} & & Table~\ref{table_decr_pattern}\\
\href{https://oeis.org/A033184}{A033184} & Catalan triangle transposed & Section~\ref{section_123_132}\\
\href{https://oeis.org/A080937}{A080937} & & Appendix~\ref{appendix_basis_size_two}, Table~\ref{table_decr_pattern}\\
\href{https://oeis.org/A102407}{A102407} & & Table~\ref{table_pairs_of_pat}\\
\href{https://oeis.org/A115139}{A115139} & Catalan polynomials & Section~\ref{section_decreasing_pattern}\\
\href{https://oeis.org/A116845}{A116845} & & Appendix~\ref{appendix_basis_size_two}\\
\href{https://oeis.org/A124302}{A124302} & & Appendix~\ref{appendix_basis_size_two}, Table~\ref{table_decr_pattern}\\
\href{https://oeis.org/A129591}{A129591} & & Theorem~\ref{theorem_first_el_dec_enumer}\\
\href{https://oeis.org/A202062}{A202062} & & Table~\ref{table_unsolved_patterns}\\
\href{https://oeis.org/A294790}{A294790} & Subtract~$n$ from partial sums & Chapter~\ref{chapter_pattern123}\\
 & of partial sums of Catalan numbers & \\
\bottomrule
\end{tabular}
}

\chapter{\texorpdfstring{Non-principal classes~$\Sort(\sigma)$}{Non-principal classes Sort(sigma)}}\label{appendix_basis_size_two}

\centering
{\small
\begin{tabular}{p{10mm}clr}
\toprule
$\sigma$ & \textbf{G.F.} & \textbf{Sequence}~$\lbrace\fsigma{\sigma}_n\rbrace_n$ & \textbf{OEIS}\\
\midrule
321 & \vphantom{\bigg|}~$\frac{1-t}{1-2t}$ & 1, 2, 4, 8, 16, 32, 64, 128, 256, 512 & A000079\\
\midrule
3214\newline
4213\newline
4312\newline
4321 & \vphantom{\bigg|}~$\frac{1-2t}{1-3t+t^2}$ & 1, 2, 5, 13, 34, 89, 233, 610, 1597, 4181 & A001519\\
\midrule
32145 & \vphantom{\bigg|}~$\frac{-3t^4+9t^3-12t^2+6t-1}{(t-1)(t^2-3t+1)^2}$ & 1, 2, 5, 14, 41, 121, 355, 1032, 2973, 8496 & A116845\\
52134 & \vphantom{\bigg|}~$\frac{(1-t)(2t-1)^2}{t^4-9t^3+12t^2-6t+1}$ & 1, 2, 5, 14, 41, 121, 355, 1033, 2986, 8594 & \\
54123 & \vphantom{\bigg|}~$\frac{1-4t+5t^2-3t^3}{t^4 -6t^3+8t^2-5t+1}$ & 1, 2, 5, 14, 41, 121, 356, 1044, 3057, 8948 & \\
32154\newline
42135\newline
43125\newline
43215\newline
52143\newline
53124\newline 
53214\newline
54132\newline
54213\newline
54312\newline
54321\newline & \vphantom{\bigg|}~$\frac{t^2-3t+1}{3t^2-4t+1}$ & 1, 2, 5, 14, 41, 122, 365, 1094, 3281, 9842 & A124302\\
\bottomrule
\end{tabular}
}

\centering
{\small
\begin{tabular}{p{10mm}clr}
\toprule
$\sigma$ & \textbf{G.F.} & \textbf{Sequence}~$\lbrace\fsigma{\sigma}_n\rbrace_n$ & \textbf{OEIS}\\
\midrule
321645\footnotemark\newline
$\vdots$\newline
654321 & \vphantom{\bigg|}~$\frac{3t^2-4t+1}{1-5t+6t^2-t^3}$ & 1, 2, 5, 14, 42, 131, 417, 1341, 4334, 14041 & A080937\\
421356\newline
431256\newline
432156 & \vphantom{\bigg|}~$\frac{2t^5-16t^4+29t^3-23t^2+8t-1}{9t^5-33t^4+46t^3-30t^2+9t-1}$ & 1, 2, 5, 14, 42, 131, 416, 1329, 4247, 13544 & \\
631245\newline
632145 & \vphantom{\bigg|}~$\frac{t^5-7t^4+17t^3-17t^2+7t-1}{4t^5-16t^4+29t^3-23t^2+8t-1}$ & 1, 2, 5, 14, 42, 131, 416, 1329, 4247, 13545 & \\
621345 & \vphantom{\bigg|}~$\Fsigma{621345(t)}$\footnotemark & 1, 2, 5, 14, 42, 131, 414, 1304, 4065, 12530 & \\
521346\newline
651243\newline
651324\newline
652134 & \vphantom{\bigg|}~$\frac{t^4-9t^3+12t^2-6t+1}{5t^4-17t^3+17t^2-7t+1}$ & 1, 2, 5, 14, 42, 131, 416, 1329, 4248, 13560 & \\
541236\newline
641235\newline
654123 & \vphantom{\bigg|}~$\frac{t^4-6t^3+8t^2-5t+1}{4t^4-11t^3+12t^2-6t+1}$ & 1, 2, 5, 14, 42, 131, 416, 1330, 4261, 13658 & \\
321465\newline
321546 & \vphantom{\bigg|}~$\frac{t^4-12t^3+16t^2-7t+1}{6t^4-23t^3+22t^2-8t+1}$ & 1, 2, 5, 14, 42, 131, 416, 1328, 4233, 13430 & \\
621354\newline
621435 & \vphantom{\bigg|}~$\frac{t^4-8t^3+12t^2-6t+1}{(t^2-3t+1)(4t^2-4t+1)}$ & 1, 2, 5, 14, 42, 131, 416, 1328, 4234, 13446 & \\
651234 & \vphantom{\bigg|}~$\frac{1-6t+14t^2-17t^3+10t^4-4t^5}{t^6-9t^5+20t^4-27t^3+19t^2-7t+1}$ & 1, 2, 5, 14, 42, 131, 414, 1306, 4094, 12766 & \\
321456 & \vphantom{\bigg|}~$\Fsigma{321456}(t)$\footnotemark & 1, 2, 5, 14, 42, 131, 414, 1304, 4063, 12497 & \\
\bottomrule
\end{tabular}
}
\small{
\footnotetext[1]{321654, 421365, 431265, 432165, 521436, 531246, 532146, 541326, 542136, 543126, 543216, 621534, 621543, 631254, 632154, 641325, 642135, 643125, 643215, 651423, 651432, 652143, 653124, 653214, 654132, 654213, 654312, 654321.}
\footnotetext[2]{\vphantom{\bigg|}~$\Fsigma{621345(t)}=\frac{(1-t)^3(2t-1)^3}{t^7-24t^6+74t^5-109t^4+89t^3-41t^2+10t-1}$.}
\footnotetext[3]{\vphantom{\bigg|}~$\Fsigma{321456}(t)=\frac{1-11t+50t^2-122t^3+175t^4-152t^5+79t^6-25t^7+4t^8}{(1-t)^3(t^2-3t+1)^3}$.}
}

\chapter{\texorpdfstring{Enumerative data for~$\Sort(\sigma)$}{Enumerative data for Sort(sigma)}}\label{appendix_sortable}

\centering
{\small
\begin{tabular}{ll}
\toprule
$\sigma$ & \textbf{Sequence}~$\lbrace\fsigma{\sigma}_n\rbrace_n$\\
\midrule
123 & 1, 2, 5, 13, 35, 99, 295, 920, 2975, 9892, 33605\\
132 & 1, 2, 5, 15, 51, 188, 731, 2950, 12235, 51822, 223191\\
213 & 1, 2, 5, 16, 62, 273, 1307, 6626, 35010, 190862, 1066317\\
231 & 1, 2, 6, 23, 102, 496, 2569, 13934, 78295, 452439, 2674769\\
312 & 1, 2, 5, 15, 52, 201, 843, 3764, 17659, 86245, 435492\\
321 & 1, 2, 4, 8, 16, 32, 64, 128, 256, 512, 1024\\
\midrule
1234 & 1, 2, 5, 14, 40, 113, 319, 918, 2731, 8438, 27011\\
1243 & 1, 2, 5, 14, 41, 122, 366, 1108, 3397, 10586, 33618\\
1324 & 1, 2, 5, 14, 42, 134, 455, 1640, 6229, 24692, 101205\\
1342 & 1, 2, 5, 14, 42, 132, 429, 1430, 4862, 16796, 58786\\
1423 & 1, 2, 5, 14, 44, 154, 588, 2396, 10237, 45284, 205608\\
1432 & 1, 2, 5, 14, 43, 144, 521, 2010, 8156, 34402, 149496\\
2134 & 1, 2, 5, 14, 45, 170, 740, 3567, 18408, 99505, 555982\\
2143 & 1, 2, 5, 14, 44, 157, 634, 2844, 13829, 71318, 383825\\
2314 & 1, 2, 5, 15, 53, 215, 972, 4767, 24837, 135434, 764875\\
2341 & 1, 2, 5, 14, 42, 132, 429, 1430, 4862, 16796, 58786\\
2413 & 1, 2, 5, 15, 52, 201, 842, 3745, 17435, 84119, 417617\\
2431 & 1, 2, 5, 14, 42, 132, 429, 1430, 4862, 16796, 58786\\
3124 & 1, 2, 5, 14, 44, 155, 603, 2541, 11401, 53758, 263847\\
3142 & 1, 2, 5, 14, 42, 132, 429, 1430, 4862, 16796, 58786\\
3214 & 1, 2, 5, 13, 34, 89, 233, 610, 1597, 4181, 10946\\
3241 & 1, 2, 5, 14, 42, 132, 429, 1430, 4862, 16796, 58786\\
3412 & 1, 2, 5, 15, 53, 214, 954, 4562, 22929, 119512, 640367\\
3421 & 1, 2, 5, 15, 53, 214, 954, 4562, 22929, 119512, 640367\\
4123 & 1, 2, 5, 14, 42, 135, 467, 1731, 6803, 28031, 119976\\
4132 & 1, 2, 5, 14, 43, 144, 522, 2027, 8334, 35894, 160531\\
4213 & 1, 2, 5, 13, 34, 89, 233, 610, 1597, 4181, 10946\\
4231 & 1, 2, 5, 14, 42, 132, 429, 1430, 4862, 16796, 58786\\
4312 & 1, 2, 5, 13, 34, 89, 233, 610, 1597, 4181, 10946\\
4321 & 1, 2, 5, 13, 34, 89, 233, 610, 1597, 4181, 10946\\
\bottomrule
\end{tabular}
}

\centering
{\small
\begin{tabular}{ll}
\toprule
$\sigma$ & \textbf{Sequence}~$\lbrace\fsigma{\sigma}_n\rbrace_n$\\
\midrule
12345 & 1, 2, 5, 14, 42, 129, 391, 1158, 3384, 9924\\
12354 & 1, 2, 5, 14, 42, 130, 405, 1257, 3883, 11980\\
12435 & 1, 2, 5, 14, 42, 130, 405, 1257, 3883, 11980\\
12453 & 1, 2, 5, 14, 42, 132, 429, 1430, 4862, 16796\\
12534 & 1, 2, 5, 14, 42, 131, 417, 1341, 4335, 14059\\
12543 & 1, 2, 5, 14, 42, 131, 417, 1341, 4335, 14059\\
13245 & 1, 2, 5, 14, 42, 131, 420, 1388, 4765, 17094\\
13254 & 1, 2, 5, 14, 42, 132, 430, 1447, 5032, 18110\\
13425 & 1, 2, 5, 14, 42, 132, 429, 1430, 4862, 16796\\
13452 & 1, 2, 5, 14, 42, 132, 429, 1430, 4862, 16796\\
13524 & 1, 2, 5, 14, 42, 132, 429, 1430, 4862, 16796\\
13542 & 1, 2, 5, 14, 42, 132, 429, 1430, 4862, 16796\\
14235 & 1, 2, 5, 14, 42, 133, 444, 1566, 5841, 22989\\
14253 & 1, 2, 5, 14, 42, 132, 429, 1430, 4862, 16796\\
14325 & 1, 2, 5, 14, 42, 132, 432, 1475, 5272, 19756\\
14352 & 1, 2, 5, 14, 42, 132, 429, 1430, 4862, 16796\\
14523 & 1, 2, 5, 14, 42, 132, 429, 1430, 4862, 16796\\
14532 & 1, 2, 5, 14, 42, 132, 429, 1430, 4862, 16796\\
15234 & 1, 2, 5, 14, 42, 135, 470, 1773, 7170, 30636\\
15243 & 1, 2, 5, 14, 42, 134, 456, 1657, 6403, 26098\\
15324 & 1, 2, 5, 14, 42, 134, 456, 1657, 6404, 26117\\
15342 & 1, 2, 5, 14, 42, 132, 429, 1430, 4862, 16796\\
15423 & 1, 2, 5, 14, 42, 133, 443, 1552, 5717, 22092\\
15432 & 1, 2, 5, 14, 42, 133, 443, 1552, 5717, 22092\\
21345 & 1, 2, 5, 14, 42, 136, 493, 2043, 9547, 48738\\
21354 & 1, 2, 5, 14, 42, 135, 474, 1851, 8061, 38601\\
21435 & 1, 2, 5, 14, 42, 135, 474, 1851, 8061, 38601\\
21453 & 1, 2, 5, 14, 42, 132, 429, 1430, 4862, 16796\\
21534 & 1, 2, 5, 14, 42, 134, 458, 1696, 6857, 30246\\
21543 & 1, 2, 5, 14, 42, 134, 458, 1696, 6857, 30246\\
23145 & 1, 2, 5, 14, 43, 146, 552, 2316, 10642, 52641\\
23154 & 1, 2, 5, 14, 43, 145, 539, 2208, 9896, 47917\\
23415 & 1, 2, 5, 14, 42, 132, 429, 1430, 4862, 16796\\
23451 & 1, 2, 5, 14, 42, 132, 429, 1430, 4862, 16796\\
23514 & 1, 2, 5, 14, 42, 132, 429, 1430, 4862, 16796\\
23541 & 1, 2, 5, 14, 42, 132, 429, 1430, 4862, 16796\\
24135 & 1, 2, 5, 14, 43, 144, 523, 2045, 8530, 37583\\
24153 & 1, 2, 5, 14, 42, 132, 429, 1430, 4862, 16796\\
24315 & 1, 2, 5, 14, 42, 132, 429, 1430, 4862, 16796\\
24351 & 1, 2, 5, 14, 42, 132, 429, 1430, 4862, 16796\\
\bottomrule
\end{tabular}
}

\centering
{\small
\begin{tabular}{ll}
\toprule
$\sigma$ & \textbf{Sequence}~$\lbrace\fsigma{\sigma}_n\rbrace_n$\\
\midrule
24513 & 1, 2, 5, 14, 42, 132, 429, 1430, 4862, 16796\\
24531 & 1, 2, 5, 14, 42, 132, 429, 1430, 4862, 16796\\
25134 & 1, 2, 5, 14, 43, 145, 534, 2119, 8921, 39327\\
25143 & 1, 2, 5, 14, 43, 144, 522, 2028, 8352, 36088\\
25314 & 1, 2, 5, 14, 42, 132, 429, 1430, 4862, 16796\\
25341 & 1, 2, 5, 14, 42, 132, 429, 1430, 4862, 16796\\
25413 & 1, 2, 5, 14, 42, 132, 429, 1430, 4862, 16796\\
25431 & 1, 2, 5, 14, 42, 132, 429, 1430, 4862, 16796\\
31245 & 1, 2, 5, 14, 42, 135, 471, 1793, 7400, 32692\\
31254 & 1, 2, 5, 14, 42, 134, 457, 1675, 6602, 27852\\
31425 & 1, 2, 5, 14, 42, 132, 429, 1430, 4862, 16796\\
31452 & 1, 2, 5, 14, 42, 132, 429, 1430, 4862, 16796\\
31524 & 1, 2, 5, 14, 42, 132, 429, 1430, 4862, 16796\\
31542 & 1, 2, 5, 14, 42, 132, 429, 1430, 4862, 16796\\
32145 & 1, 2, 5, 14, 41, 121, 355, 1032, 2973, 8496\\
32154 & 1, 2, 5, 14, 41, 122, 365, 1094, 3281, 9842\\
32415 & 1, 2, 5, 14, 42, 132, 429, 1430, 4862, 16796\\
32451 & 1, 2, 5, 14, 42, 132, 429, 1430, 4862, 16796\\
32514 & 1, 2, 5, 14, 42, 132, 429, 1430, 4862, 16796\\
32541 & 1, 2, 5, 14, 42, 132, 429, 1430, 4862, 16796\\
34125 & 1, 2, 5, 14, 43, 145, 538, 2188, 9650, 45495\\
34152 & 1, 2, 5, 14, 42, 132, 429, 1430, 4862, 16796\\
34215 & 1, 2, 5, 14, 43, 145, 538, 2188, 9650, 45495\\
34251 & 1, 2, 5, 14, 42, 132, 429, 1430, 4862, 16796\\
34512 & 1, 2, 5, 14, 42, 132, 429, 1430, 4862, 16796\\
34521 & 1, 2, 5, 14, 42, 132, 429, 1430, 4862, 16796\\
35124 & 1, 2, 5, 14, 43, 144, 522, 2027, 8332, 35849\\
35142 & 1, 2, 5, 14, 42, 132, 429, 1430, 4862, 16796\\
35214 & 1, 2, 5, 14, 43, 144, 522, 2027, 8332, 35849\\
35241 & 1, 2, 5, 14, 42, 132, 429, 1430, 4862, 16796\\
35412 & 1, 2, 5, 14, 42, 132, 429, 1430, 4862, 16796\\
35421 & 1, 2, 5, 14, 42, 132, 429, 1430, 4862, 16796\\
41235 & 1, 2, 5, 14, 42, 133, 445, 1578, 5924, 23418\\
41253 & 1, 2, 5, 14, 42, 132, 429, 1430, 4862, 16796\\
41325 & 1, 2, 5, 14, 42, 134, 456, 1658, 6422, 26314\\
41352 & 1, 2, 5, 14, 42, 132, 429, 1430, 4862, 16796\\
41523 & 1, 2, 5, 14, 42, 132, 429, 1430, 4862, 16796\\
41532 & 1, 2, 5, 14, 42, 132, 429, 1430, 4862, 16796\\
42135 & 1, 2, 5, 14, 41, 122, 365, 1094, 3281, 9842\\
\bottomrule
\end{tabular}
}

\centering
{\small
\begin{tabular}{ll}
\toprule
$\sigma$ & \textbf{Sequence}~$\lbrace\fsigma{\sigma}_n\rbrace_n$\\
\midrule
42153 & 1, 2, 5, 14, 42, 132, 429, 1430, 4862, 16796\\
42315 & 1, 2, 5, 14, 42, 132, 429, 1430, 4862, 16796\\
42351 & 1, 2, 5, 14, 42, 132, 429, 1430, 4862, 16796\\
42513 & 1, 2, 5, 14, 42, 132, 429, 1430, 4862, 16796\\
42531 & 1, 2, 5, 14, 42, 132, 429, 1430, 4862, 16796\\
43125 & 1, 2, 5, 14, 41, 122, 365, 1094, 3281, 9842\\
43152 & 1, 2, 5, 14, 42, 132, 429, 1430, 4862, 16796\\
43215 & 1, 2, 5, 14, 41, 122, 365, 1094, 3281, 9842\\
43251 & 1, 2, 5, 14, 42, 132, 429, 1430, 4862, 16796\\
43512 & 1, 2, 5, 14, 42, 132, 429, 1430, 4862, 16796\\
43521 & 1, 2, 5, 14, 42, 132, 429, 1430, 4862, 16796\\
45123 & 1, 2, 5, 14, 43, 146, 550, 2279, 10216, 48660\\
45132 & 1, 2, 5, 14, 43, 145, 538, 2187, 9628, 45205\\
45213 & 1, 2, 5, 14, 43, 145, 538, 2187, 9628, 45205\\
45231 & 1, 2, 5, 14, 42, 132, 429, 1430, 4862, 16796\\
45312 & 1, 2, 5, 14, 43, 145, 538, 2187, 9628, 45205\\
45321 & 1, 2, 5, 14, 43, 145, 538, 2187, 9628, 45205\\
51234 & 1, 2, 5, 14, 42, 131, 421, 1403, 4893, 17932\\
51243 & 1, 2, 5, 14, 42, 132, 432, 1476, 5288, 19908\\
51324 & 1, 2, 5, 14, 42, 132, 432, 1476, 5288, 19908\\
51342 & 1, 2, 5, 14, 42, 132, 429, 1430, 4862, 16796\\
51423 & 1, 2, 5, 14, 42, 133, 443, 1552, 5718, 22113\\
51432 & 1, 2, 5, 14, 42, 133, 443, 1552, 5718, 22113\\
52134 & 1, 2, 5, 14, 41, 121, 355, 1033, 2986, 8594\\
52143 & 1, 2, 5, 14, 41, 122, 365, 1094, 3281, 9842\\
52314 & 1, 2, 5, 14, 42, 132, 429, 1430, 4862, 16796\\
52341 & 1, 2, 5, 14, 42, 132, 429, 1430, 4862, 16796\\
52413 & 1, 2, 5, 14, 42, 132, 429, 1430, 4862, 16796\\
52431 & 1, 2, 5, 14, 42, 132, 429, 1430, 4862, 16796\\
53124 & 1, 2, 5, 14, 41, 122, 365, 1094, 3281, 9842\\
53142 & 1, 2, 5, 14, 42, 132, 429, 1430, 4862, 16796\\
53214 & 1, 2, 5, 14, 41, 122, 365, 1094, 3281, 9842\\
53241 & 1, 2, 5, 14, 42, 132, 429, 1430, 4862, 16796\\
53412 & 1, 2, 5, 14, 42, 132, 429, 1430, 4862, 16796\\
53421 & 1, 2, 5, 14, 42, 132, 429, 1430, 4862, 16796\\
54123 & 1, 2, 5, 14, 41, 121, 356, 1044, 3057, 8948\\
54132 & 1, 2, 5, 14, 41, 122, 365, 1094, 3281, 9842\\
54213 & 1, 2, 5, 14, 41, 122, 365, 1094, 3281, 9842\\
54231 & 1, 2, 5, 14, 42, 132, 429, 1430, 4862, 16796\\
54312 & 1, 2, 5, 14, 41, 122, 365, 1094, 3281, 9842\\
54321 & 1, 2, 5, 14, 41, 122, 365, 1094, 3281, 9842\\
\bottomrule
\end{tabular}
}

\chapter{Listings}\label{appendix_listings}

\begin{algorithm}
$Stack:=\emptyset$\;
\While{$i\leq n$}
{
    \If{$Stack=\emptyset$ or $\pi_i < TOP(Stack)$}
    {
        $\textnormal{execute}$ S\;
        $i:=i+1$\;
    }
    \Else{$\textnormal{execute}$ O\;}
}
\While{$Stack \neq \emptyset$}
{$\textnormal{execute}$ O\;}
\caption{Stacksort. Here~$Stack$ is the stack, $TOP(Stack)$ is the current top of the stack, $\pi=\pi_1\cdots\pi_n$ is the input permutation.}\label{listing_stacksort}
\end{algorithm}

\begin{algorithm}
$Stack_I:=\emptyset$\;
$Stack_{\sigma}:=\emptyset$\;
$i:=1$\;
\While{$i\leq n$}
{
    \If{$\sigma\nleq Stack_{\sigma}\circ\pi_i$}
    {
        $\textnormal{execute}$ $P_{\sigma}$\;
        $i:=i+1$\;
    }
    \ElseIf{$Stack_I =\emptyset$ or $TOP(Stack_{\sigma})<TOP(Stack_I )$}
    {$\textnormal{execute}$ $P_{I}$\;}
    {$\textnormal{execute}$ O\;}
}
\While{$Stack_{\sigma}\neq\emptyset$}
{
    \eIf{$Stack_I =\emptyset$ or $TOP(Stack_{\sigma})<TOP(Stack_I )$}
    {$\textnormal{execute}$ $P_{I}$\;}
    {$\textnormal{execute}$ O\;}
}
\While{$Stack_I\neq\emptyset$}
{$\textnormal{execute}$ O\;}
\caption{The~$\sigma$-machine. Here~$Stack_{\sigma}$ is the~$\sigma$-avoiding stack, $Stack_I$ is the increasing stack, $P_{\sigma}$ means pushing into~$Stack_{\sigma}$, $P_I$ means pushing into~$Stack_I$, O means moving~$TOP(Stack_I)$ into the output, $\circ$ is the concatenation operation.}\label{listing_sigma_avoiding}
\end{algorithm}

\end{document}